\documentclass[a4paper, 11pt, leqno]{amsart}

\usepackage{setspace}
\usepackage{a4wide}
\usepackage{amsthm}
\usepackage{enumitem}

\usepackage{amssymb,amsmath,amsfonts}
\usepackage{graphicx,color}
\usepackage{epstopdf}
\usepackage{enumerate}
\usepackage{caption}
\usepackage{subfig}
\usepackage{multirow}
\usepackage{pgfplots}
\usepackage{mathrsfs}
\usepackage{tikz}
\usepackage{booktabs}
\usepackage{colortbl}

\usepackage[active]{srcltx}
\usepackage{comment}
\usepackage[mathscr]{euscript}
\usepackage{placeins}
\usepackage[symbol]{footmisc}
\usetikzlibrary{shapes.geometric}

\graphicspath{graphics/}

\newcommand{\cO}{\mathcal{O}}
\newcommand{\cF}{\mathcal{F}}
\newcommand{\cT}{\mathcal{T}}

\newcommand{\cL}{\mathcal{L}}
\newcommand{\cE}{\mathcal{E}}

\newcommand{\bE}{\mathbb{E}}
\newcommand{\bL}{\mathbb{L}}
\newcommand{\bP}{\mathbb{P}}

\newcommand{\bH}{\mathbb{H}}

\newcommand{\bVh}{{\mathbb V}_h}
\newcommand{\p}{\partial}
\newcommand{\dd}{{\rm d}}
\newcommand{\ds}{{\rm d}s}
\newcommand{\dW}{{\rm d}W}

\newcommand{\uli}{\overline{u}}

\newcommand{\eli}{\overline{e}}

\newcommand{\supT}{\mathop {\sup} \limits_{0 \le t \le T}}

\definecolor{Green}{rgb}{0.0, 0.5, 0.}
\definecolor{Red}{rgb}{0.7, 0.2, 0}
\newcommand\del[1]{}

\numberwithin{equation}{section}

\newtheorem{theorem}{Theorem}[section]
\newtheorem{definition}[theorem]{Definition}

\newtheorem{lemma}[theorem]{Lemma}

\newtheorem{remark}{Remark}
\newtheorem{example}{Example}
\newtheorem{scheme}{Scheme}

\numberwithin{figure}{section}

\newcommand{\blue}{\color{blue}}

\newcommand{\red}{\color{black}}

\begin{document}

\author{\makebox[\textwidth]
{Xiaobing Feng$^\dag$, Akash Ashirbad Panda$^\ddagger$, \and Andreas Prohl$^\star$}
}

\thanks{\\$^\dag$Department of Mathematics,  The University of Tennessee, Knoxville, TN 37996, U.S.A. ({\tt xfeng@utk.edu}); The work of this author was partially supported by the NSF grant DMS-2012414.
	\\$^\ddagger$Mathematisches Institut, Universit\"at T\"ubingen, Auf der 
 			Morgenstelle 10, D-72076, T\"ubingen, Germany ({\tt panda@na.uni-tuebingen.de}).
	\\$^\star$Mathematisches Institut, Universit\"at T\"ubingen, Auf der 
 			Morgenstelle 10, D-72076, T\"ubingen, Germany ({\tt prohl@na.uni-tuebingen.de}).	
}

\title{{Higher order time discretization for the stochastic semilinear 
 		wave equation with multiplicative noise}}

 \begin{abstract}
In this paper, a higher order time-discretization scheme is proposed, where the iterates approximate the solution of the stochastic semilinear wave equation driven by multiplicative noise with general drift and diffusion.  We employ variational method for its error analysis and prove an improved convergence order of $\frac 32$ for the approximates of the solution. The core of the analysis is H\"older continuity in time and moment bounds for the solutions of the continuous and the discrete problem. Computational experiments are also presented.
%
\end{abstract}

\maketitle

\noindent
\textbf{Keywords and phrases:}
 		{stochastic semilinear wave equations, 
 		multiplicative noise,
		time discretization,
 		stability analysis,
		convergence rates,
 		strong approximation.}

\medskip

\noindent	
\textbf{AMS subject classification (2020):} { 65M12, 65C20, 60H10, 60H15, 60H35, 65C30.}
 	
\pagestyle{myheadings}
\thispagestyle{plain}
\markboth{XIAOBING FENG, AKASH A. PANDA, AND ANDREAS PROHL}{HIGHER ORDER DISCRETIZATION OF THE STOCHASTIC SEMILINEAR WAVE EQUATION} 
 	
\medskip

\section{Introduction}\label{sec-1}

\del{
In the field of fluid dynamics, electromagnetics, and acoustics, the physical phenomena such as mechanical vibrations and wave motions are commonly observed. These phenomena are usually modelled by hyperbolic partial differential equations (PDEs). One of the examples of hyperbolic PDEs is the wave equation, which describes the propagation of variety of waves (e.g. water waves, sound waves, and seismic waves). The wave equation subject to random perturbations gained a lot of attention during 1960's due to their applications in various fields such as oceanography, quantum mechanics and gene therapy, to name a few; we give several engineering examples that we aim to numerically access in Section \ref{sec-appli}. For an introduction to the theory of stochastic wave equation (SWE) and its applications, we refer to \cite[Chapter 5]{Chow_2015}, and references therein.
}

Let ${\mathcal O} \subset {\mathbb R}^d$,  for $1 \leq d \leq 3$ be a bounded domain. We consider the numerical approximation of the following stochastic semilinear wave equation perturbed by multiplicative noise of It\^o type:
\begin{align}\label{stoch-wave1:1}
\begin{cases}
\p_t^2 u + A  u = F(u,\p_t u) + \sigma(u, \p_t u)\,\p_t W \qquad
&\mbox{\rm in}\  (0,T) \times \cO\,, \\ 
u(0,\cdot) = u_0\, , \quad \p_t u(0,\cdot) = v_0 \qquad &\mbox{\rm in}\  \cO\,,\\
u(t,\cdot) =0  \qquad &\mbox{\rm on}\  \p \cO, \,\forall \,  t \in (0,T)\,,
\end{cases}
\end{align}
where $A$ is a strongly elliptic second order differential operator of the form
\begin{equation}\label{stoch-wave1:1b} A \varphi(x)= - \sum_{i, j=1}^d \frac{\partial} {\partial x_i} \Bigl( a_{ij}(x) \frac{\partial}{\partial x_j} \varphi(x)\Bigr)  \qquad \forall\, x \in {\mathcal O}\,,
\end{equation}
with suitably smooth coefficients $a_{ij}(x)$, where $a^{ij} = a^{ji} \ \forall i, j,$ and for every non-zero $\xi \in {\mathbb R}^d$, $\sum_{i, j}^d a_{ij}(x) \xi_i \xi_j \geq \gamma |\xi|^2$, for some constant $\gamma > 0$. Here, $F$ and $\sigma$ are Lipschitz in both arguments.
  Let $\mathfrak{P}:= (\Omega, \mathcal{F}, \mathbb{F}, \mathbb{P})$ be a filtered probability space, and $\{W(t)\}_{t\geq0}$  be a finite dimensional Wiener process defined on it; the initial data $u_0$ and $v_0$ are given $\cF_0-$measurable random variables. 
  
A strong variational solution to (\ref{stoch-wave1:1}) exists, see {\em e.g.}~\cite[Sec.~6.8]{Chow_2015}, and is usually constructed via the reformulation of (\ref{stoch-wave1:1})$_1$ as a first order system by setting $v = \partial_t u$,
  \begin{align}\label{stoch-wave1:1a}
\begin{cases}
&\dd u  = v\,  \dd t\,  \\
&\dd v  = \bigl(-A u  + F (u , v) \bigr)\, \dd t +  \sigma(u,v) \,\dd W(t)\,,
\end{cases}
\end{align}
and then using a Faedo-Galerkin method, related uniform energy bounds, and a compactness argument; see Definition \ref{def-Strong_sol}, and Appendix \ref{Lem-EE} below. A prototype example is $A = -\Delta$, for which we associate the following energy functional 
\begin{equation}\label{energ1}
{\mathcal E}(u,v) := {\mathcal E}_{\tt kin}(v) + {\mathcal E}_{\tt ela}(u) = \frac{1}{2} \int_{\mathcal O} \vert v(x)\vert^2 \dd x + \frac{1}{2} \int_{\mathcal O} \vert \nabla u(x)\vert^2 \dd x\,,
\end{equation}
where the first term represents the kinetic energy, and the second the elastic energy of the propagating wave with pointwise elongation $u: [0,T] \times \overline{\mathcal O} \times \Omega \rightarrow {\mathbb R}$. ---
  We begin the further discussion with an example to motivate the effect of noise.

\begin{example}\label{noise-effect}
 Let $\cO = (0, 1),$ $T=1,$ $A = -\Delta$, $F\equiv0$ in \eqref{stoch-wave1:1}, and $W$ be of the form \eqref{w-form}, with $M =3$, and  $e_j(x) = \sqrt{2} \sin(j\pi x)$. --- The first line in Fig.~\ref{energy} displays single trajectories of $u$ for different $\sigma \equiv \sigma(u, v).$ For $\sigma \equiv 0$ both, the amplitude and wavelength remain constant over time in snapshot {\rm (A)}, as does  $\mathcal{E}(u, v)$ in  {\rm ({D})}. For $\sigma(u, v)=\frac{1}{2}u$, the amplitude of a single wave realization in snapshot {\rm (B)} changes --- as do the trajectory-wise energy parts in {\rm (E)} ---, while the wavelength remains constant over time. The computation of the (approximate) total energy uses ${\tt MC}=10^3$  Monte-Carlo simulations in snapshot {\rm (G)}: it is  conserved, and close to {\rm (D)}.

For $\sigma(u, v)=\frac{1}{2}v$ both, the wavelength and frequency of a single trajectory are heavily affected, see  snapshot {\rm (C)}, and {\rm (F)}, where
only  $t \mapsto \mathcal{E}_{\tt ela}(u(t, \omega))$ is smooth. In contrast, the
dynamics of ${\mathbb E}_{\tt MC}[{\mathcal E}(u,v)]$ in {\rm (H)} recovers the
exchange of elastic and kinetic energy parts, but the total energy is not conserved any more. The proper resolution of snapshot {\rm (H)} required $5$ times more Monte-Carlo simulations than for {\rm (G)}.
\del{
The components of the energy of a wave are kinetic and potential, which depends on the amplitude and the frequency of the wave. The total mechanical energy of the wave is the sum of its potential energy and kinetic energy, which is defined as
\[\mathcal{E}(u, v):= \frac{1}{2} \int_0^1 (\nabla u)^2 \,dx + \frac{1}{2} \int_0^1 (v)^2 \,dx.\] 
Time-changes of energy profiles are depicted in Figure \ref{energy}. The three different panels of this figure display the evolution of the above mentioned energies for different functions $\sigma=\sigma(u,v)$. 
}
\begin{figure}[htbp!]
  \centering
  \subfloat[$\sigma(u, v)=0$]{\includegraphics[width=0.33\textwidth]{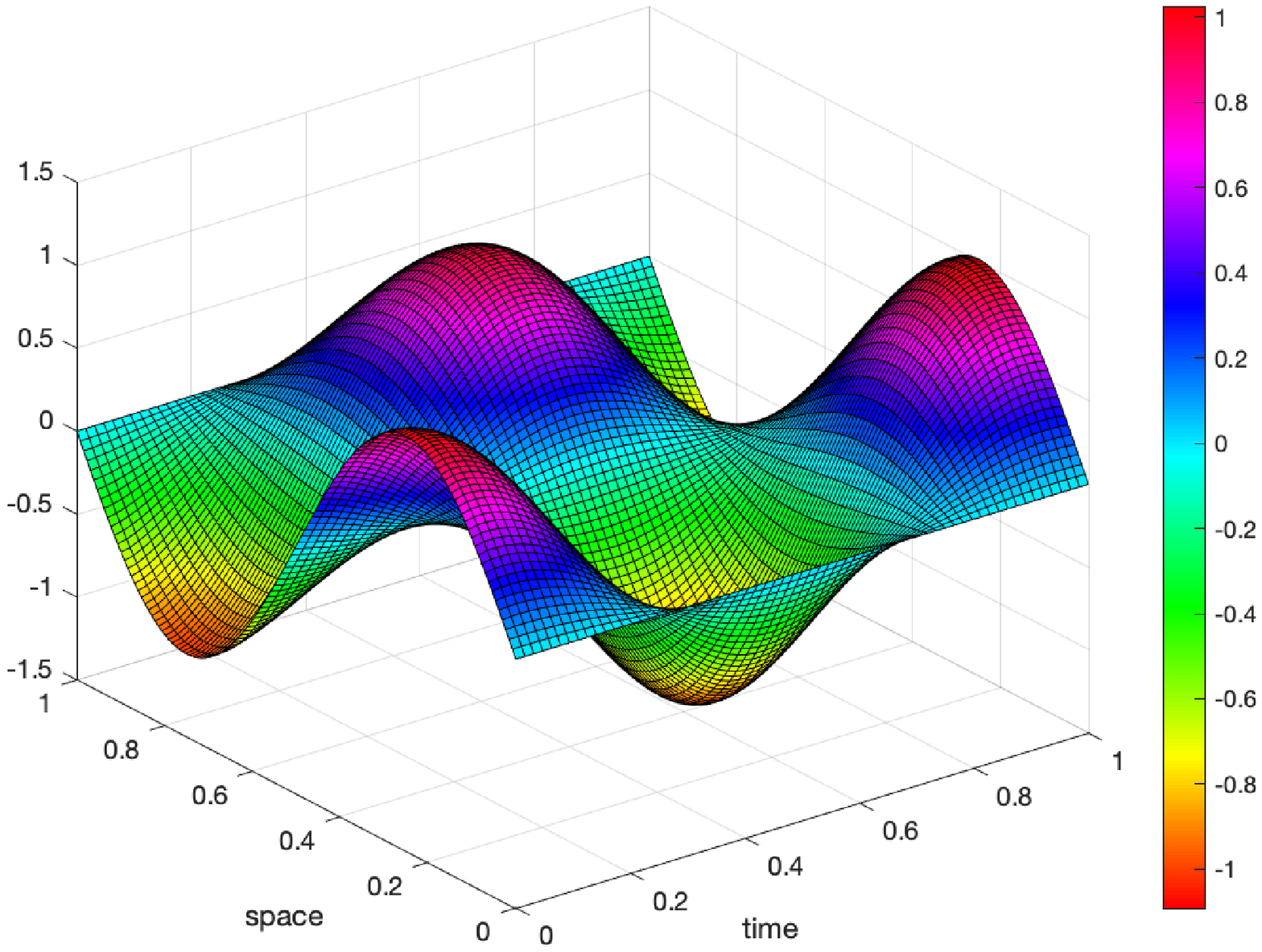}\label{fig:f4}}
\hfill
\subfloat[$\sigma(u, v)= \frac{1}{2} u$]{\includegraphics[width=0.33\textwidth]{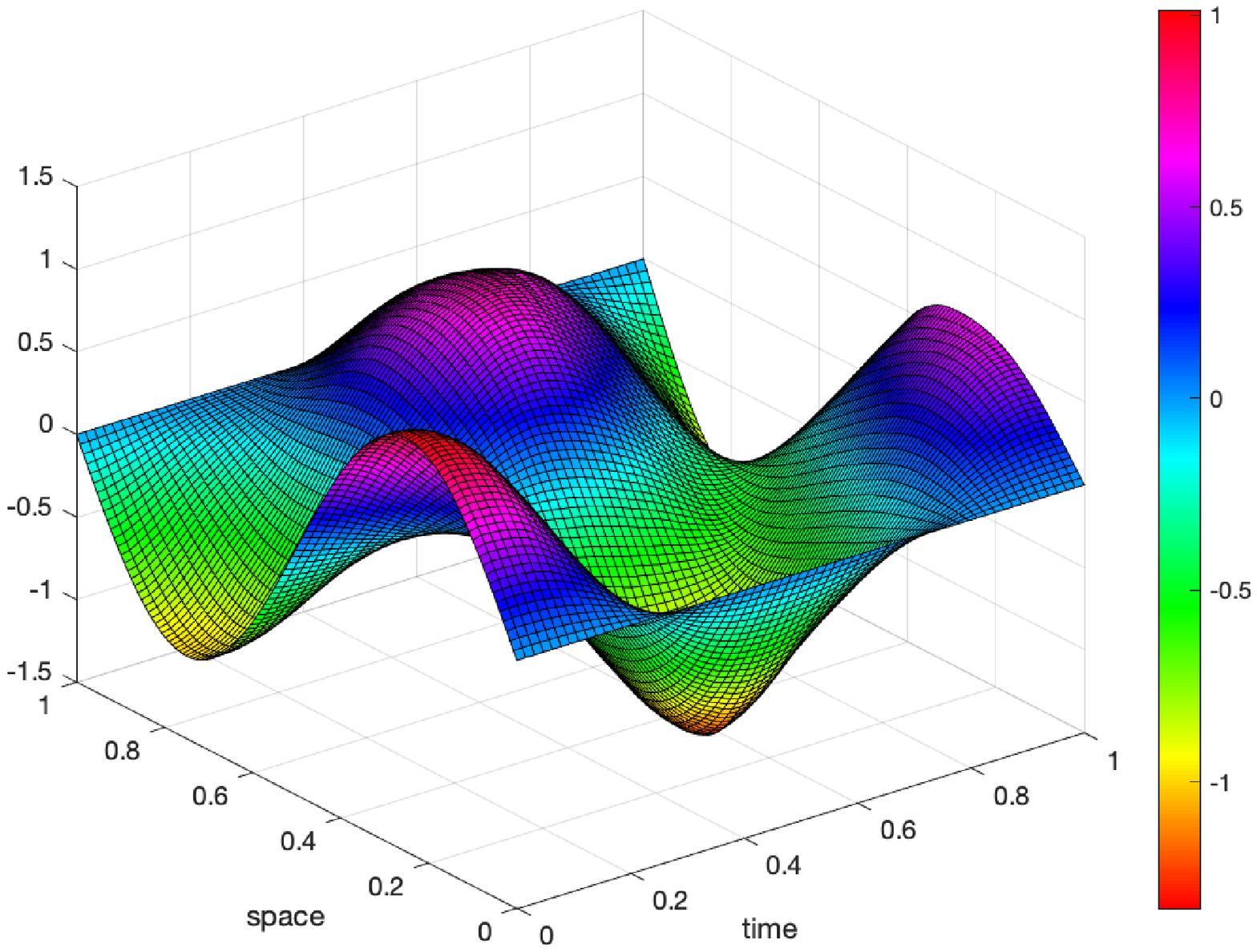}\label{fig:f5}}
\hfill
\subfloat[$\sigma(u, v)= \frac{1}{2} v$]{\includegraphics[width=0.33\textwidth]{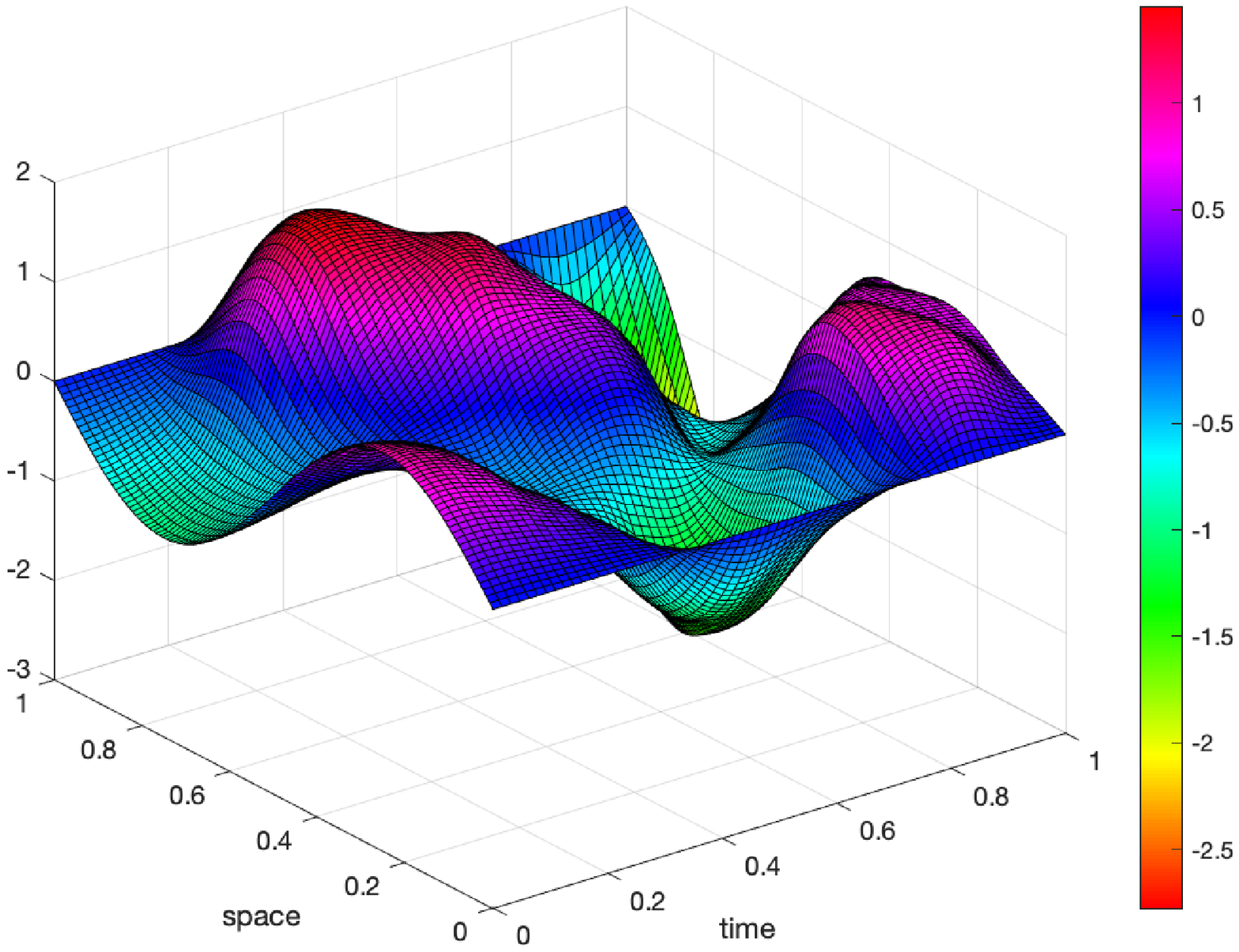}\label{fig:f6}}
\par
  \subfloat[$\sigma(u, v)=0$]{\includegraphics[width=0.33\textwidth]{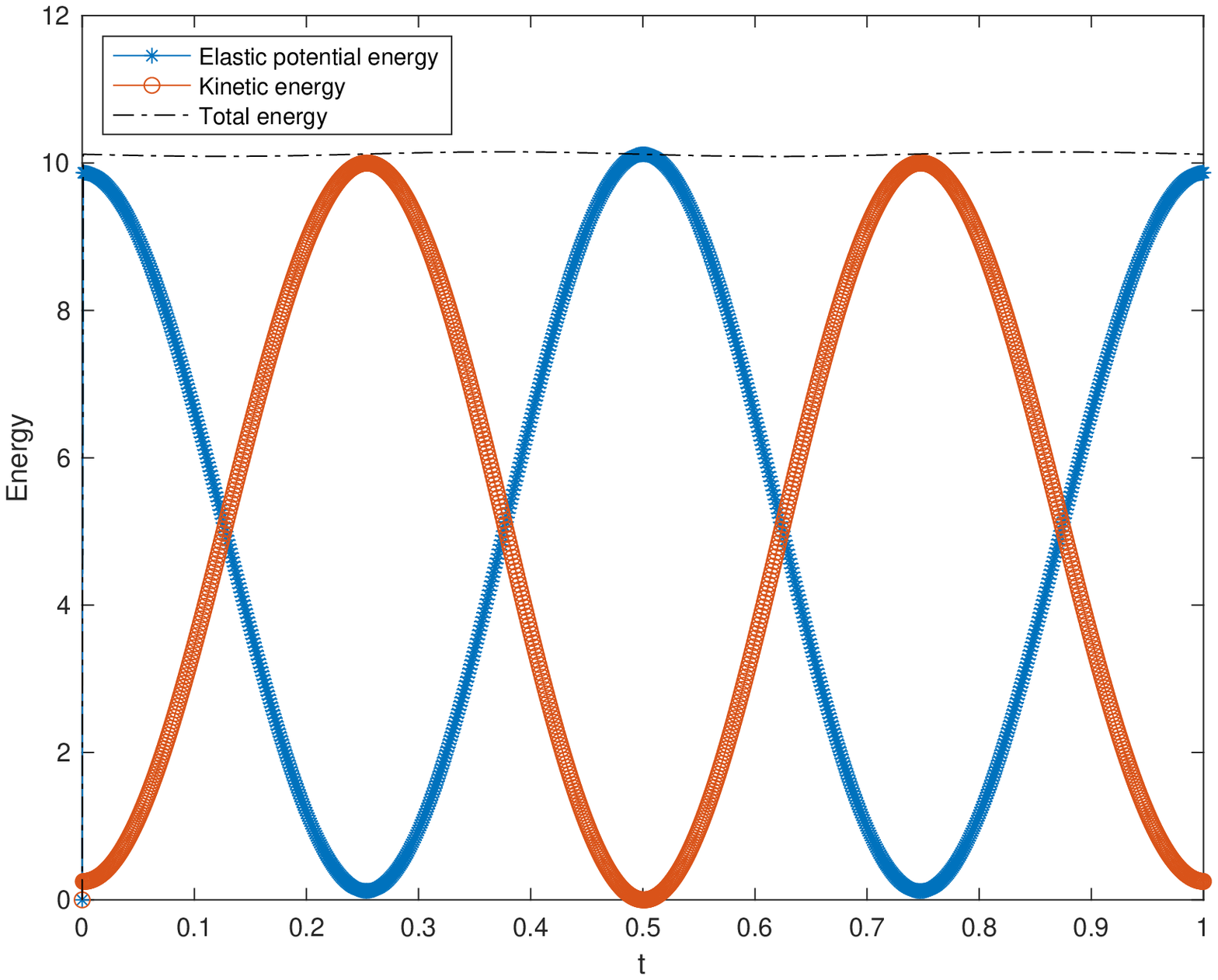}\label{f-f1}}
  \hfill
  \subfloat[$\sigma(u, v)= \frac{1}{2} u$]{\includegraphics[width=0.33\textwidth]{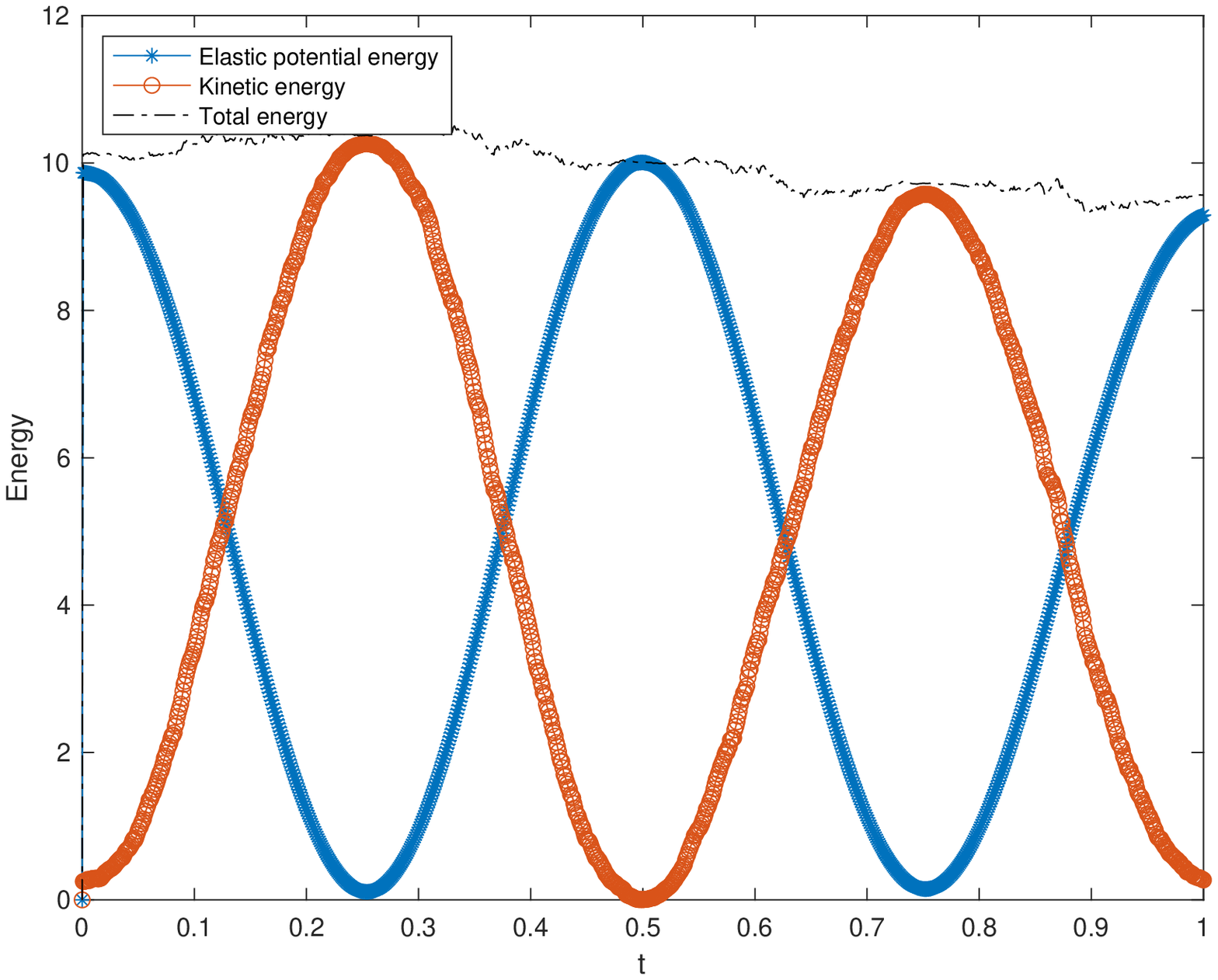}\label{f-f2}}
 \hfill
 \subfloat[$\sigma(u, v)= \frac{1}{2} v, \beta=\frac{1}{4}$]{\includegraphics[width=0.33\textwidth]{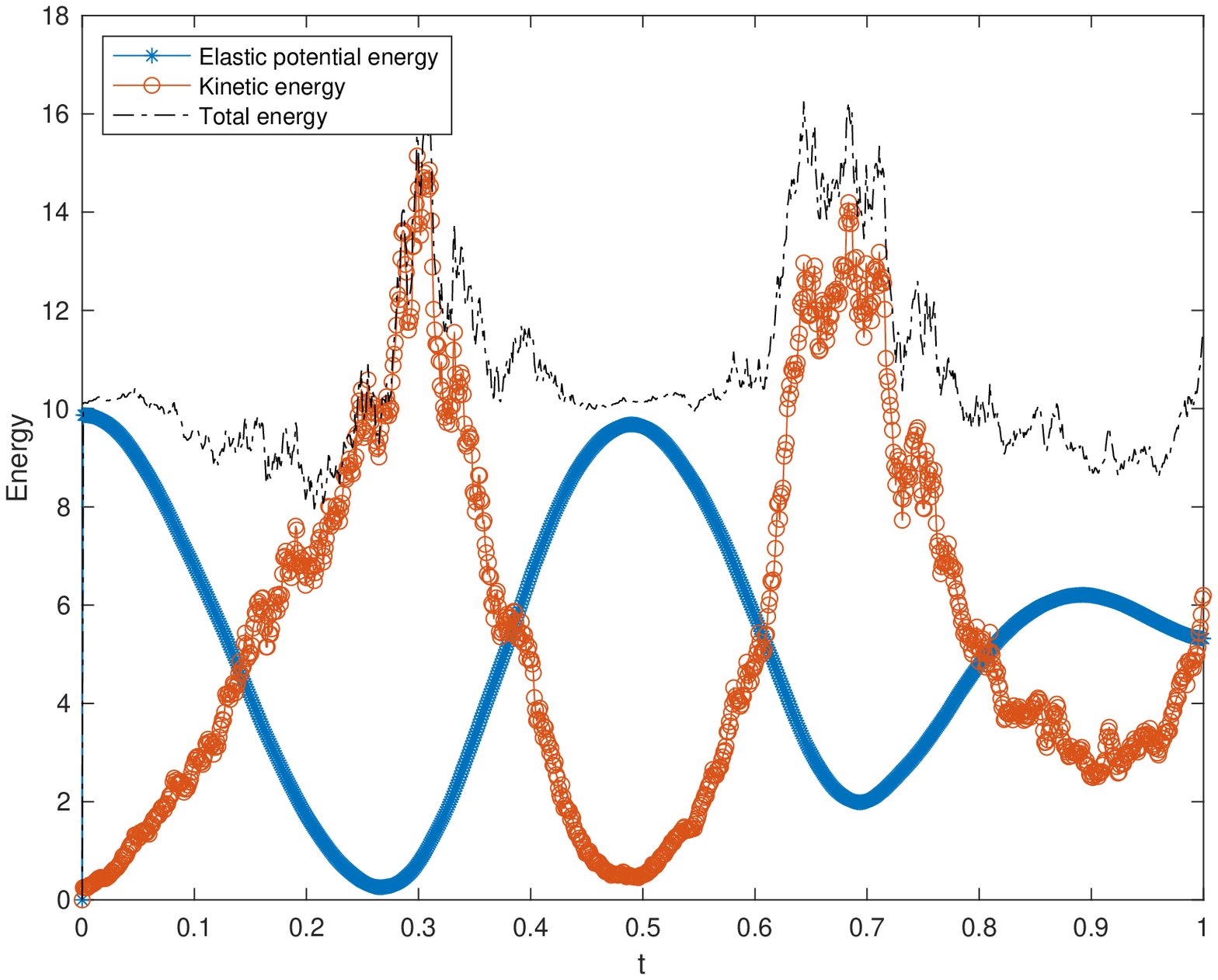}\label{f-f3}}
\par
  \subfloat[$\sigma(u, v)= \frac{1}{2} u,$ with ${\tt MC}=10^3$]{\includegraphics[width=0.33\textwidth]{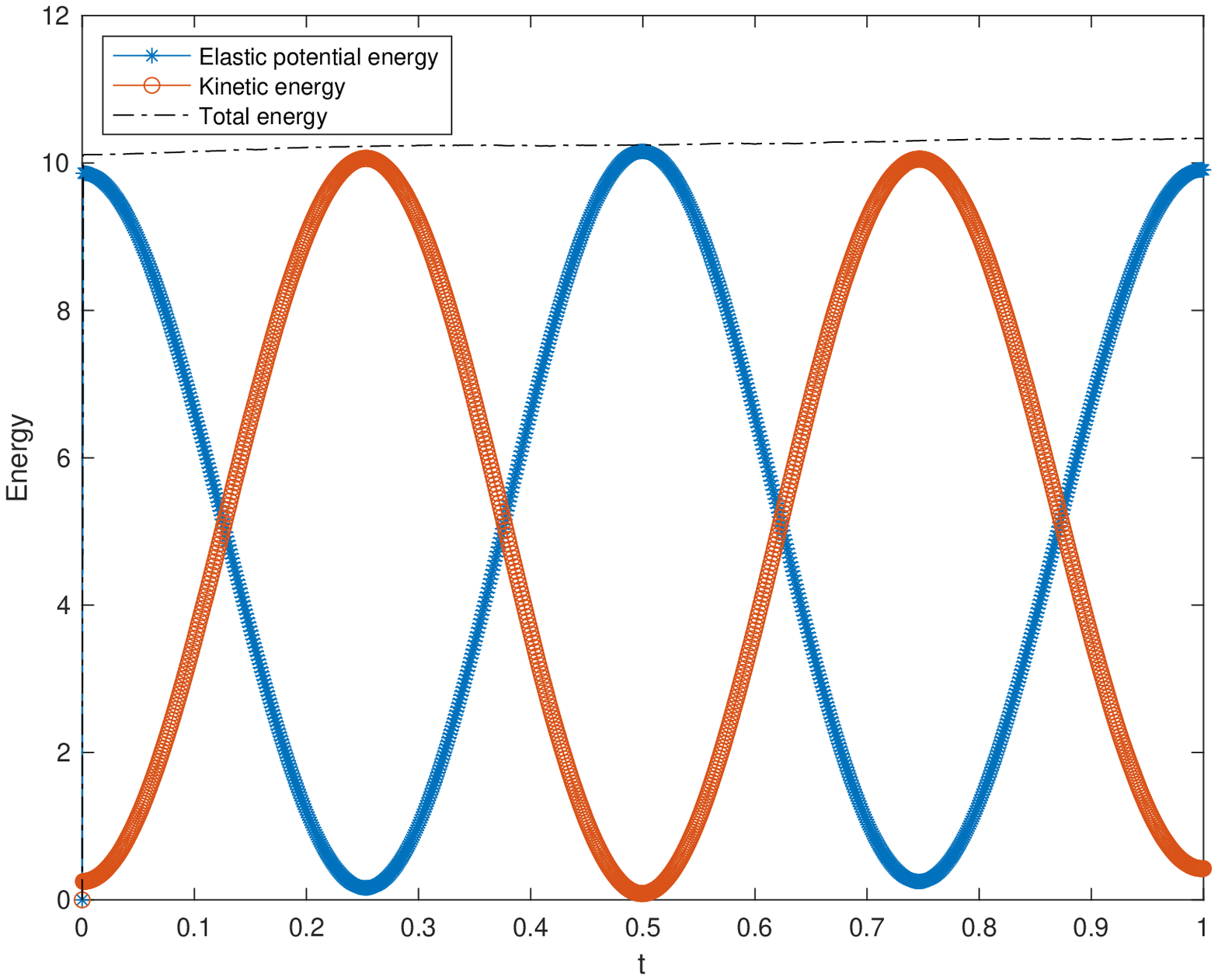}\label{f-f4}}
  \subfloat[$\sigma(u, v)=\frac{1}{2} v,$ with ${\tt MC}=5 \times 10^3$]{\includegraphics[width=0.33\textwidth]{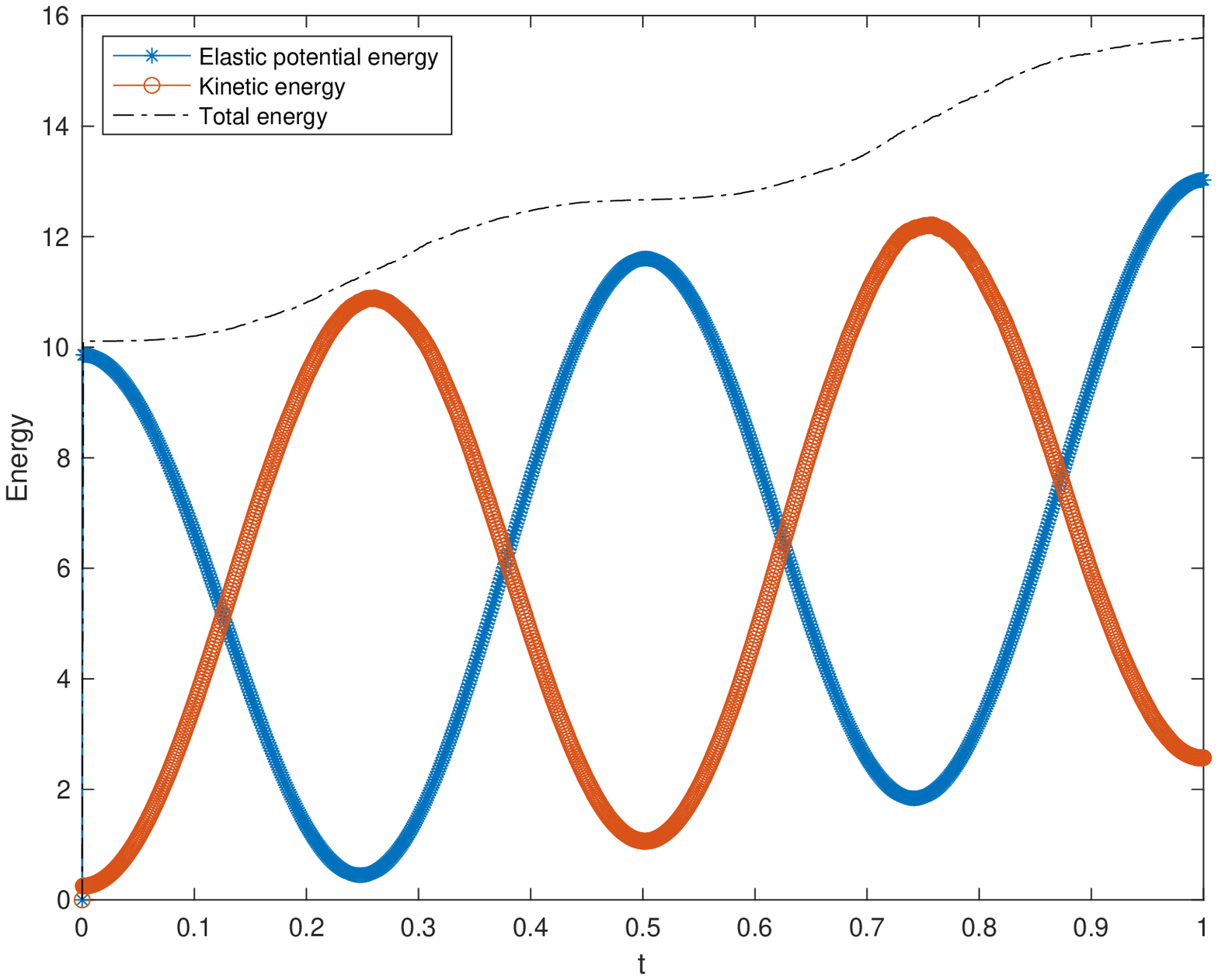}\label{f-f5}}

\caption{{\bf (Example \ref{noise-effect})} (1st line) Single trajectory of $u$ from \eqref{stoch-wave1:1}, simulated via $(\widehat{\alpha}, \beta)-$scheme ($\widehat{\alpha} = 1$). (2nd line) Corresponding elastic ($\mathcal{E}_{\tt ela}$), kinetic ($\mathcal{E}_{\tt kin}$), total energy ($\mathcal{E}$), for mesh sizes $h=2^{-7}$ and $k=2^{-10}$. (3rd line) Plots $t \mapsto \mathbb{E}_{\tt MC} [\mathcal{E}(u(t), v(t))]$, with ${\tt MC}=10^3$ in  snapshot (G) and ${\tt MC}=5 \times 10^3$ in (H).}
\label{energy}
\end{figure}
%

\end{example}
The first works to numerically solve (\ref{stoch-wave1:1}) are \cite{Walsh_2006} and \cite{Sar+San_2006}, where (semi-)discrete schemes were constructed based on the solution concept of a mild solution for (\ref{stoch-wave1:1}): 
in \cite{Walsh_2006}, which considered ${\mathcal O}= {\mathbb R}$, $A=-\Delta$, Lipschitz nonlinearities $F \equiv F(u)$ and $\sigma \equiv \sigma(u)$, and white noise, a strong convergence rate ${\mathcal O}(k^{1/2})$ was shown for an explicit finite difference scheme, where the temporal step size $k$ is equal to the mesh size $h$ of the Cartesian spatial mesh; the error analysis uses
the Green's function, which is explicitly known in this case, and hence used the mild solution concept for this Cauchy problem.
%
%
%

A further development in this direction is  \cite{CQ_2016},
where ${\mathcal O}=(0,1), A= -\Delta$  in (\ref{stoch-wave1:1}),
and the authors used the explicit representation of (discrete) Green's function, such that its implementation crucially hinges on
the availability of eigenvalues and eigenfunctions of the Laplacian; see also \cite[Sec.~5.3]{Chow_2015}, and \cite{Hochbruck_2010}. The stable scheme then allows independent choices of $k$ and $h$, and the proof of \cite[Thm.~4.1]{CQ_2016} provides convergence rates both, in terms of spatial and temporal discretization. We also mention \cite{CLS_2013}, where ${\mathcal O}$ is a bounded convex domain with polygonal boundary, and $A= -\Delta$  in (\ref{stoch-wave1:1}); the space-time discretization was proposed with the explicit knowledge of the related (discrete) semigroup, whose efficient implementation again hinges on the knowledge of the related eigenvalues and eigenfunctions. Later, in \cite{ACLW_2016} the authors addressed the multiplicative noise case with $\sigma \equiv \sigma(u)$, where $\sigma$ and also the nonlinearity $F \equiv F(u)$ were assumed to be zero on the boundary. The above mentioned works did not address the case when $F \equiv F(u, v)$ and $\sigma \equiv \sigma(u, v)$.

In engineering applications for elastic and acoustic wave propagations which may be described by (\ref{stoch-wave1:1}), the considered domains ${\mathcal O} \subset {\mathbb R}^d$ are typically complicated, and/or the propagating medium is heterogeneous, 
 with layers, anisotropies, cavities ({\em e.g.} in seismology, or material testing), or may even be random. Moreover, models of type (\ref{stoch-wave1:1}) often require   non-constant and non-self-adjoint operators, such as those in (\ref{stoch-wave1:1b}), which may even have random coefficients. Therefore, such engineering problems often exclude the efficient use of semigroup based methods through spectral theory as discussed above. This motivates us to aim for the following goals in this work:
\begin{itemize}
\item[1)] Use an implicit method in time (below referred to as $(\widehat{\alpha}, \beta)-$method, where $\widehat{\alpha}=0$; see Scheme \ref{Hat-scheme}) to approximate \eqref{stoch-wave1:1} with $F \equiv F(u, \partial_t u)$ and $\sigma \equiv \sigma(u, \partial_t u)$, and employ variational methods for its error analysis. This part is motivated by  \cite{Dup_1973} for the deterministic linear wave equation, {\em i.e.}, $F \equiv \sigma \equiv 0$. For {finite dimensional noise} of type \eqref{w-form}, we use energy arguments to obtain ${\mathcal O}(k)$ for the temporal error --- which coincides with the order obtained in \cite[Thm.~4]{ACLW_2016} and \cite[Thm.~4.1]{CQ_2016} for an exponential integrator, in the case $\sigma \equiv \sigma(u)$, $F \equiv F(u)$ and trace-class noise in (\ref{stoch-wave1:1}). We obtain ${\mathcal O}(k^{\frac 12})$ for the temporal error in the general case $\sigma \equiv \sigma(u, v)$ and $F \equiv F(u, v)$, which has not been addressed in the existing literature. 
\item[2)] For $\sigma \equiv \sigma(u)$ and $F \equiv F(u)$, in fact, we improve the $(\widehat{\alpha}, \beta)-$method  to a higher-order method which yields improved convergence order ${\mathcal O}(k^{3/2})$ for approximates of $u$ in ${\mathbb L}^2$; see Theorem \ref{lem:scheme1:con}. The additional term that arises for $\widehat{\alpha} = 1$ is {\em motivated} by It\^o's formula, and uses
increments
\begin{equation}\label{w-tilde}
\widetilde{\Delta_n W}:=  \int_{t_n}^{t_{n+1}} (s-t_n) \,\mathrm{d} W(s) = \int_{t_n}^{t_{n+1}} s \, {\mathrm d}W(s) - t_n \Delta_n W\, .
\end{equation}
\del{
We show that this term approximates the explicit 
increment $\widehat{\Delta_n W} := k \Delta_n W$ in such a way that the overall order ${\mathcal O}(k^{3/2})$ is not affected.
}
\item[3)] Computational experiments in Section~\ref{sec-6} show that these results are sharp {\em w.r.t.} the used noise, {\em i.e.}, there are examples for $\sigma \equiv \sigma(u,v)$ where 
the error converges only in order ${\mathcal O}(k)$ --- rather than ${\mathcal O}(k^{3/2})$ in the case $\sigma \equiv \sigma(u)$.
\end{itemize}
In this work, we focus on proper time discretizations for (\ref{stoch-wave1:1}), which we consider to be the essential part of an overall discretization,
and leave a related finite element error analysis for future work.
The results will be derived for (\ref{stoch-wave1:1}) with $A = -\Delta$ to simplify the technical setup, but easily generalize to $A$ in (\ref{stoch-wave1:1b}), even with random coefficients there. Moreover, the $(\widehat{\alpha}, \beta)-$method is neither a spectral Galerkin method nor does its implementation hinge on related semigroups.

\smallskip

While being inspired by the second order time-stepping scheme of \cite{Dup_1973} for the deterministic wave equation, where $u^{n, \frac 12} := \frac{1}{2} (u^{n+1} + u^{n-1})$, we propose the following scheme for (\ref{stoch-wave1:1}):
\begin{scheme}\label{scheme--2}
{\bf ($(\widetilde{\alpha},\beta)-$scheme)} Fix $\widetilde{\alpha} \in \{0,1\}$ and $\beta \in [0, 1/2)$.
Let $\{ t_n\}_{n=0}^N$ be a mesh of size $k>0$ covering $[0, T],$ and $\{ (u^n, v^n)_{n=0, 1}\}$ be given $\mathcal{F}_{t_n}$-measurable, $[{\mathbb H}^1_0]^2$-valued r.v's.
For every $n \ge 1$, find $[{\mathbb H}^1_0]^2$-valued, 
${\mathcal F}_{t_{n+1}}$-measurable r.v's~$(u^{n+1}, v^{n+1})$ such that ${\mathbb P}$-a.s.
\begin{eqnarray}
\label{scheme2--1}
\quad &&(u^{n+1}-u^{n}, \phi) = k(v^{n+1},\phi) 
\quad 
\forall \phi \in \mathbb{L}^2\, , \\  \nonumber
\quad &&(v^{n+1}-v^{n}, \psi) = - k \big(\nabla \widetilde{u}^{n,\frac 12}, \nabla\psi \big)  
+ \Bigl(\sigma(u^{n}, v^{n- \frac 12})\, \Delta_n W, \psi\Bigr)  
\\  \label{scheme2--2}
\quad &&\qquad\qquad\qquad\quad+ \widetilde{\alpha} \,\Bigl( D_u\sigma(u^{n}, v^{n-\frac 12})\, v^{n}\, \widetilde{\Delta_n W}, \psi\Bigr) \\ \nonumber
\quad &&\qquad\qquad\qquad\quad+ \frac{k}{2}\Bigl(3F(u^n, v^n) 
- F(u^{n-1}, v^{n-1}), \psi\Bigr) \quad \forall \psi \in \mathbb{H}^1_0\, ,
\end{eqnarray}
where 
\begin{align}\label{tilde1/2}
\widetilde{u}^{n,\frac 12} \equiv {{\widetilde{u}^{n,\frac 12}_\beta}} := \frac{1+ \beta\,k^{\beta}}{2} u^{n+1} + \frac{1- \beta\,k^{\beta}}{2} u^{n-1}\,,
\end{align}
and
\begin{align*}
\Delta_{n} W := W(t_{n+1}) - W(t_n) \ \ \mbox{ and } \ \ v^{n- \frac 12} := \frac 12 (v^n + v^{n-1})\, .
\end{align*}
\end{scheme}

Note that $\widetilde{u}^{n,\frac 12} = u^{n,\frac 12}$ for $\beta=0$. 
Also, in the case when $F \equiv F(u)$, $\sigma \equiv \sigma(u)$ and $\beta=0$, the $(\widetilde{\alpha},\beta)-$scheme simplifies to (for $n \geq 1$)
\begin{eqnarray}
\label{scheme2:1}
\quad &&(u^{n+1}-u^{n}, \phi) = k(v^{n+1},\phi) 
\quad 
\forall \phi \in \mathbb{L}^2\, , \\  
\quad &&(v^{n+1}-v^{n}, \psi) = - k \big(\nabla u^{n, \frac 12}, \nabla\psi \big)  
+ \Bigl(\sigma(u^{n}) \Delta_n W, \psi\Bigr)+ \widetilde{\alpha}  \,\Bigl(D_u\sigma(u^{n}) v^{n} \widetilde{\Delta_n W}, \psi\Bigr) 
\label{scheme2:2}\\
\quad &&\qquad\qquad\qquad\quad+ \frac{k}{2}\Bigl(3F(u^n) 
- F(u^{n-1}), \psi\Bigr) \quad \forall \psi \in \mathbb{H}^1_0\, \nonumber
. 
\end{eqnarray}

Scheme~\ref{scheme--2} involves the increment $\widetilde{\Delta_n W}$ from \eqref{w-tilde}. For their implementation, we approximate it by {\em $\widehat{\Delta_n W}$} defined as
\begin{align}\label{W-hat}
\widehat{\Delta_n W} := k W(t_{n+1}) - k^2 \sum_{\ell=1}^{k^{-1}} W(t_{n,\ell})\,,
\end{align}
where $\big\{ W(t_{n, \ell}) \big\}_{\ell = 1}^{k^{-1}}$ is the piecewise affine approximation of $W$ on $[t_{n},t_{n+1}]$ on an equidistant mesh $\{ t_{n,\ell}\}_{\ell=1}^{k^{-1}},$ of size $k^2:= t_{n, \ell+1} - t_{n, \ell}$.
To motivate this approximation, we first use It\^o's formula to restate $\widetilde{\Delta_n W}$ as
\begin{equation}\label{DnW-tilde}
\begin{split}
\widetilde{\Delta_n W} &= \Bigl(t_{n+1} W(t_{n+1}) - t_{n} W(t_{n})\Bigr) -
\int_{t_n}^{t_{n+1}} W(s)\, {\mathrm d}s - t_n \Delta_n W 
\\ &= \int_{t_n}^{t_{n+1}} \bigl[W(t_{n+1}) - W(s)\bigr]\, {\rm d}s = k W(t_{n+1}) - \int_{t_n}^{t_{n+1}} W(s)\, {\rm d}s
\end{split}
\end{equation}
and we approximate the last integral on the right-hand side of \eqref{DnW-tilde} by $k^2 \sum_{\ell=1}^{k^{-1}} W(t_{n,\ell})$. Thus, we have the following implementable scheme:
\begin{scheme}\label{Hat-scheme}
{\bf ($(\widehat{\alpha}, \beta)-$scheme)} Consider Scheme \ref{scheme--2}. We refer to \eqref{scheme2--1}--\eqref{scheme2--2} as $(\widehat{\alpha}, \beta)-$scheme, when $\widetilde{\alpha}$ and $\widetilde{\Delta_n W}$ are replaced by $\widehat{\alpha}$ and $\widehat{\Delta_n W}$, respectively. 
\end{scheme}

\medskip

The following example motivates that the convergence rate for the
$(1, 0)-$scheme is boosted from ${\mathcal O}(k)$ to ${\mathcal O}(k^{3/2})$,
 in case $\sigma \equiv \sigma(u)$ and $F \equiv F(u)$.
\begin{example}\label{exm-alp0}
Let $\cO=(0, 1)$, $T=1$, $A = -\Delta$, $F \equiv 0$, $\sigma(u)=2 \sin(u)$
in \eqref{stoch-wave1:1a}. Let
$$u_0(x) = \sin(2 \pi x) \qquad \mbox{and} \qquad  v_0(x) = \sin(3 \pi x)\,,$$
 and $W$ as in Example \ref{noise-effect}. 
Fig. \ref{5.1} displays convergence studies for  the scheme \eqref{scheme2:1}--\eqref{scheme2:2}$:$
for $\widehat{\alpha} = 0$, the plots {\rm (A) -- (C)} show $\bL^2$-errors in $u$, $\nabla u$, evidencing convergence order ${\mathcal O}(k)$, and those for $v$ evidence convergence order ${\mathcal O}(k^{1/2})$. For $\widehat{\alpha} = 1$, the convergence order improves to ${\mathcal O}(k^{3/2})$ for $u$, $\nabla u$, and ${\mathcal O}(k)$ for $v$; see  plots {\rm (D) -- (F)}.
See Section \ref{sec-6} for more details.
\begin{figure}[h!]
  \centering
  \subfloat[$\bL^2$-error for $u$ ($\widehat{\alpha}=0$ case)]{\includegraphics[width=0.33\textwidth]{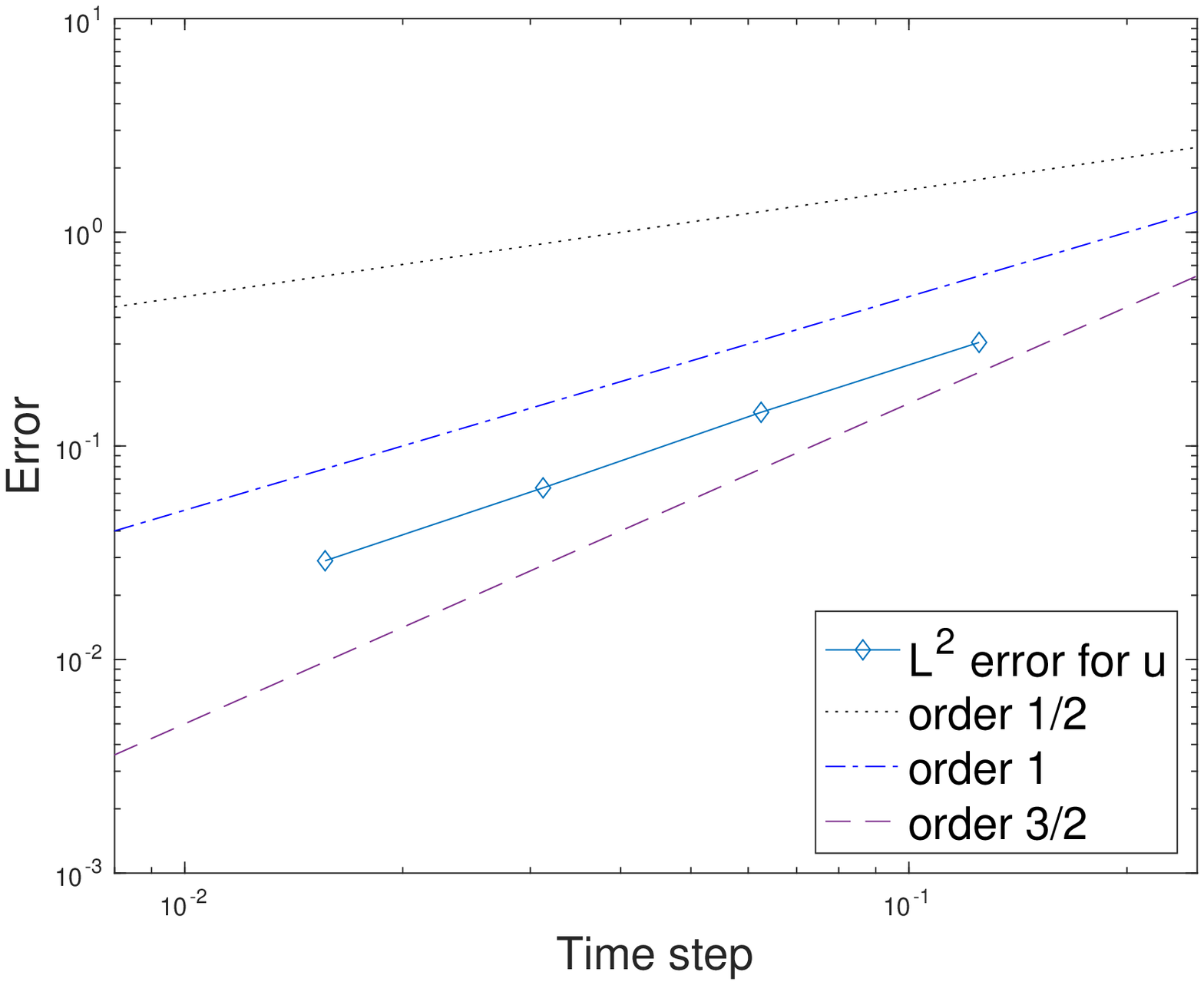}\label{fig:f1}}
  \hfill
  \subfloat[$\bL^2$-error for $\nabla u$ ($\widehat{\alpha}=0$ case)]{\includegraphics[width=0.33\textwidth]{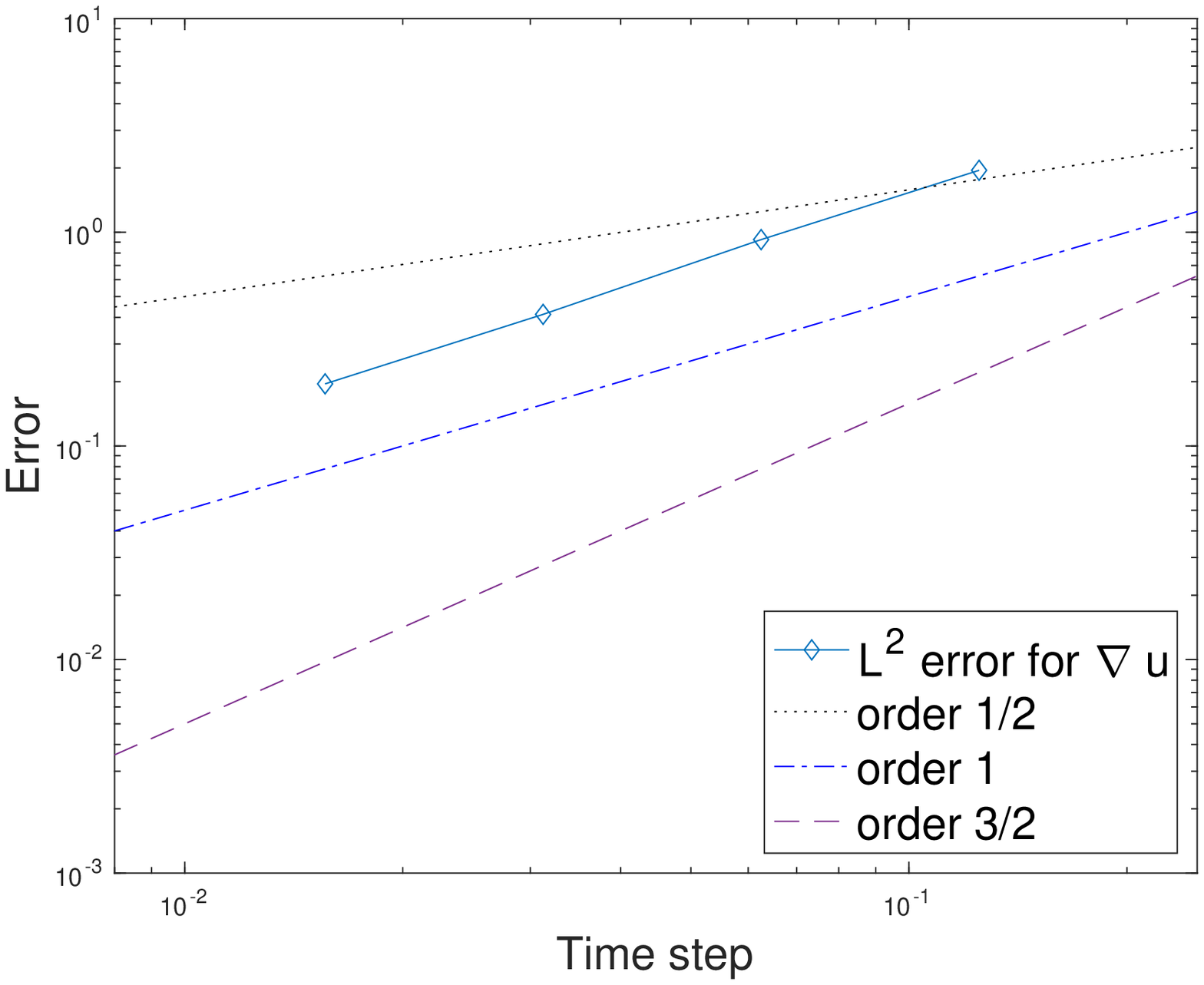}\label{fig:f2}}
  \hfill
  \subfloat[$\bL^2$-error for $v$ ($\widehat{\alpha}=0$ case)]{\includegraphics[width=0.33\textwidth]{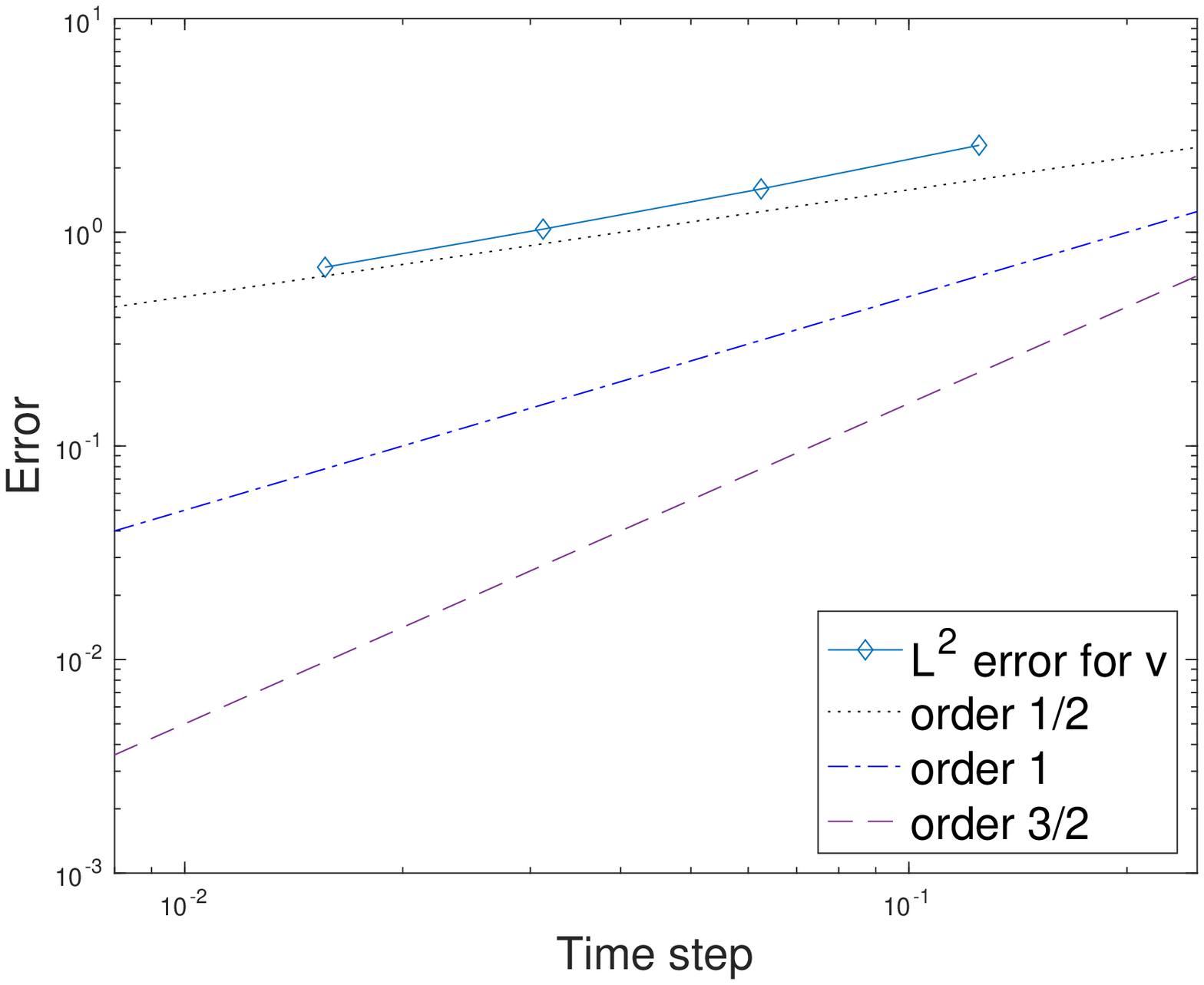}}\label{fig:f3}
 \par
 \subfloat[$\bL^2$-error for $u$ ($\widehat{\alpha}=1$ case)]{\includegraphics[width=0.33\textwidth]{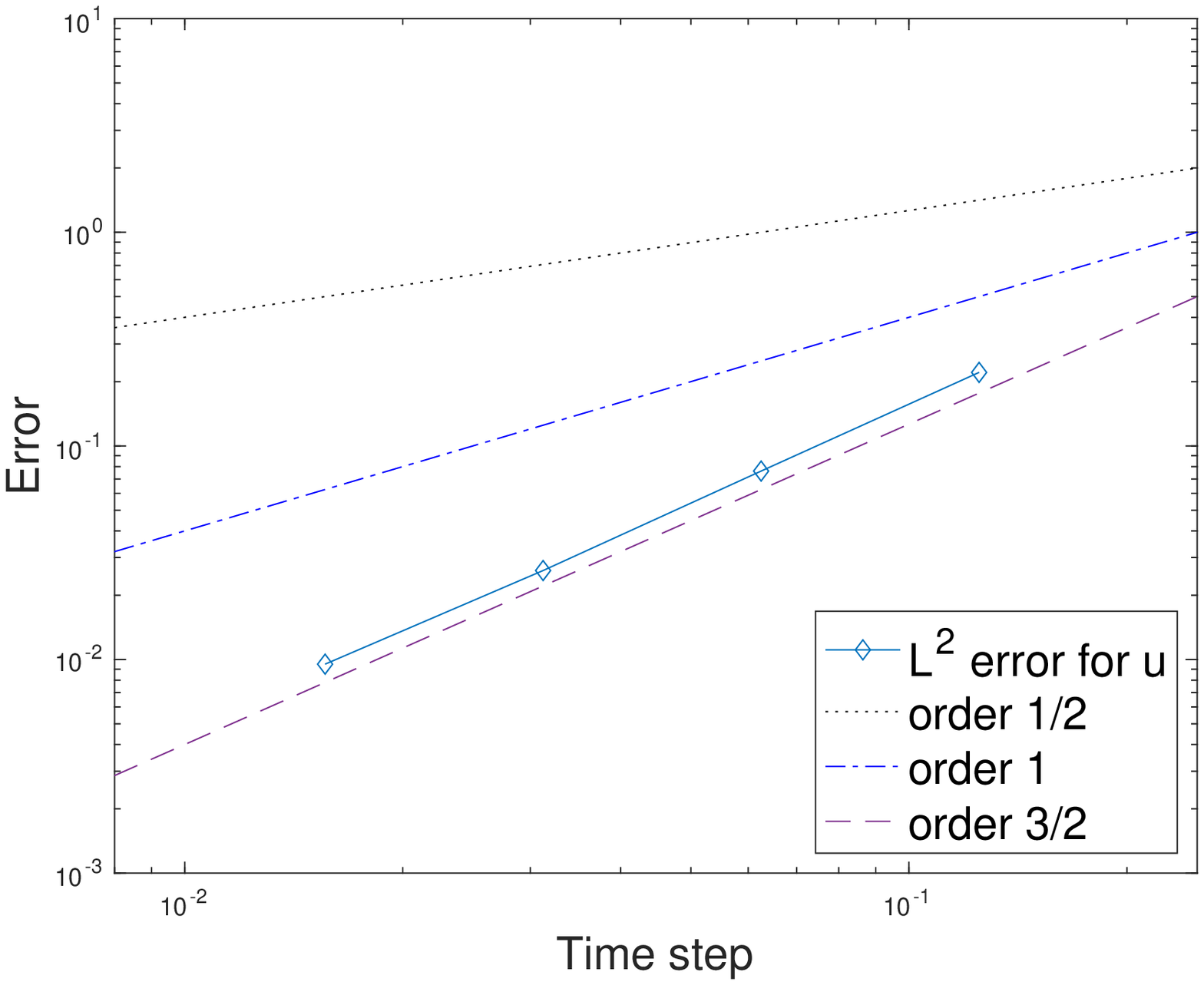}\label{fig4}}
  \hfill
  \subfloat[$\bL^2$-error for $\nabla u$ ($\widehat{\alpha}=1$ case)]{\includegraphics[width=0.33\textwidth]{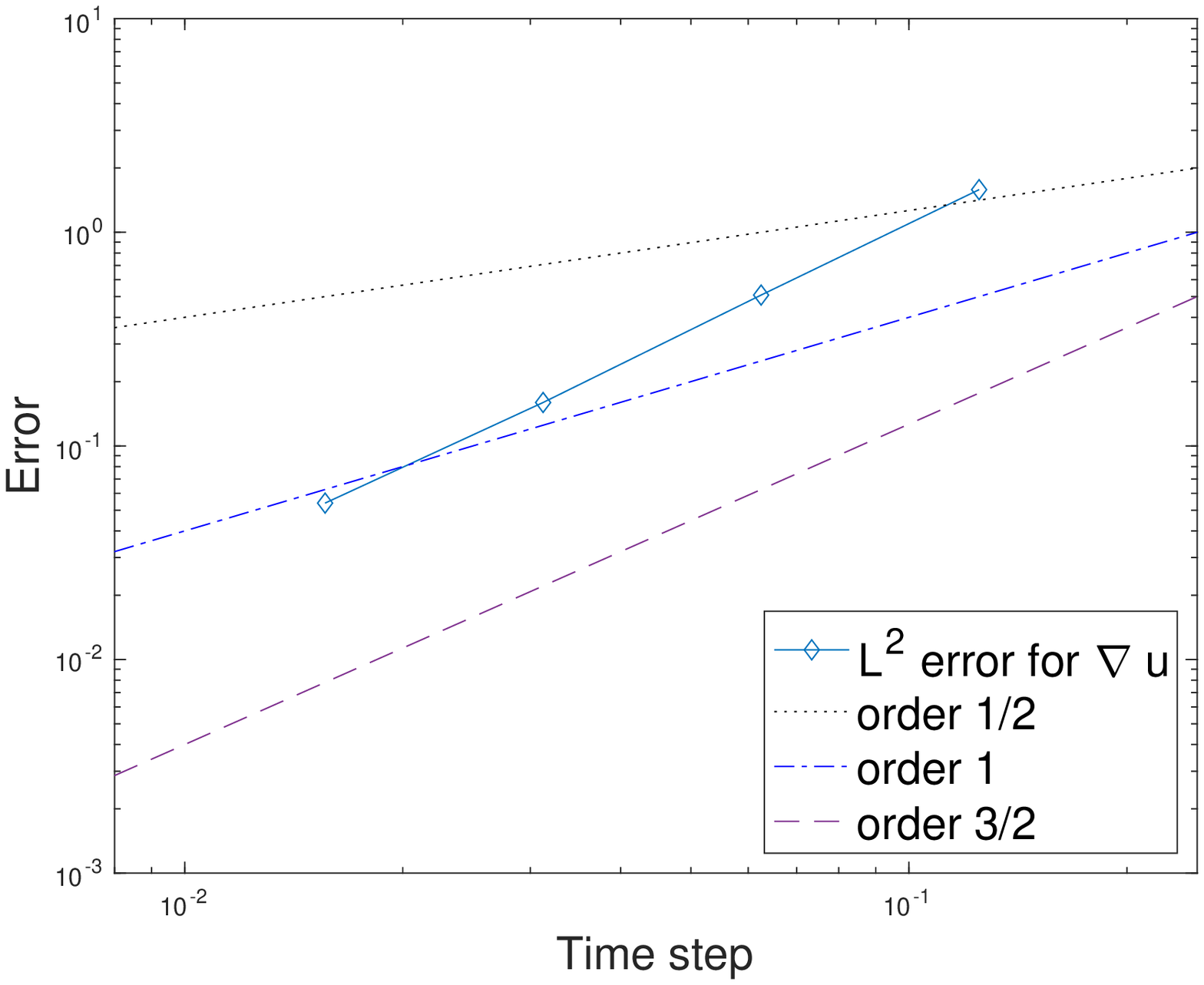}\label{fig5}}
  \hfill
  \subfloat[$\bL^2$-error for $v$ ($\widehat{\alpha}=1$ case)]{\includegraphics[width=0.33\textwidth]{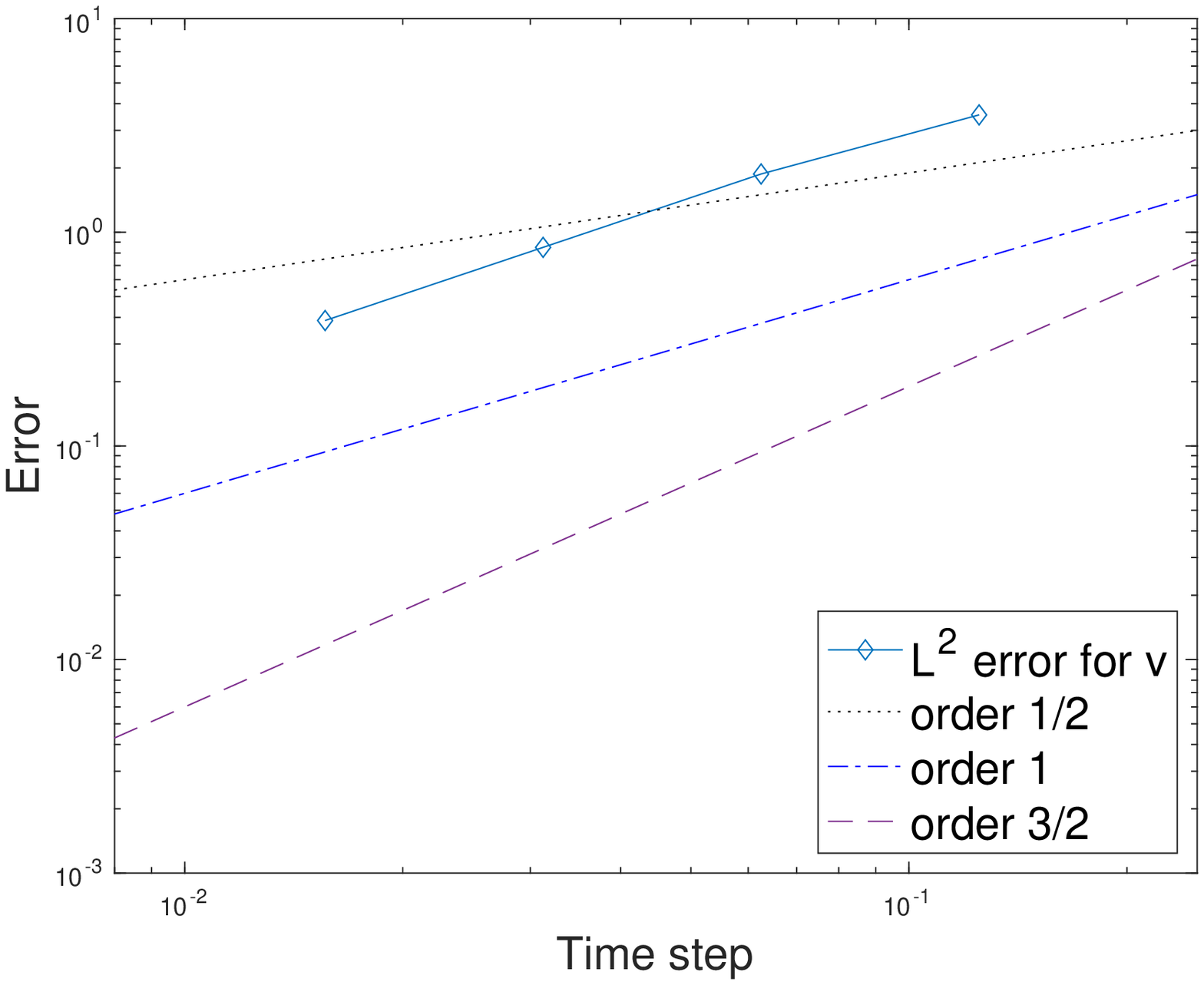}\label{fig6}}
  \caption{{\bf (Example \ref{exm-alp0})} Temporal rates of convergence for the scheme \eqref{scheme2:1}--\eqref{scheme2:2} with $F\equiv 0$ and $\sigma(u) = 2 \sin(u)$; $\widehat{\alpha}=0$ in {\rm (A)}, {\rm (B)}, {\rm (C)}, and $\widehat{\alpha}=1$ in {\rm (D)}, {\rm (E)}, {\rm (F)}; discretization parameters: $h= 2^{-7}, k=\{2^{-3}, \cdots, 2^{-6}\}$, ${\tt MC}=3000$.} 
\label{5.1}
\end{figure}

\end{example}

\medskip

The rest of the paper is organized as follows.  In Section \ref{Nota+Prel}, we  precise the data requirements in \eqref{stoch-wave1:1a} with $A = -\Delta$ and provide the structure assumptions on $F$ and $\sigma$. In Section~\ref{Mat-fra}, we recall the concept of a strong variational  solution for the problem \eqref{stoch-wave1:1a} and discuss its regularity. 
In Section~\ref{sec-3}, we prove stability results for the $(\widehat{\alpha}, \beta)-$scheme. In Section~\ref{sec-4}, we prove strong rates of convergence for the above mentioned schemes. In Section~\ref{sec-6}, we present comparative computational studies which evidence the role of noise in various cases and validates the proved error estimate results.

\bigskip

\section{Preliminaries and Assumptions}\label{Nota+Prel}

\subsection{Notation and useful results} 
Let $(L^p(\cO), \|\cdot\|_{L^p})$ and $(W^{m, p}(\cO), \|\cdot\|_{W^{m, p}})$ be the Lebesgue and Sobolev spaces respectively, endowed with usual norms for $m \in \mathbb{N}$ and $1 \leq p \leq \infty$. We denote ${\mathbb L}^p := L^p(\cO)$ and  ${\mathbb W}^{m,p} := W^{m, p}(\cO)$. For $p=2$, let $(\cdot, \cdot)$ be the inner-product in ${\mathbb L}^2$, and  ${\mathbb H}^m := {\mathbb W}^{m, 2}$. We define ${\mathbb H}^1_0 := \{u\in {\mathbb H}^1: u |_{\partial \cO} = 0\}$. 

Let $\mathbb{X}, \mathbb{Y}$ be two separable Hilbert spaces. Let $\mathcal{L}(\mathbb{X}, \mathbb{Y})$ denote the space of all linear maps from $\mathbb{X}$ to $\mathbb{Y}$, and $\mathcal{L}_m(\mathbb{X}, \mathbb{Y})$ denotes the space of all multi-linear maps from $\mathbb{X} \times \cdots \times \mathbb{X}$  ($m$-times) to $\mathbb{Y}$ for $m \geq 2$. Throughout this paper, for some $\Phi: \bH^1_0 \times \bH^1_0 \to \mathbb{L}^2$, we use the notation $D_u \Phi(u, v) \in \mathcal{L}(\bH^1_0, \bL^2)$ for the Gateaux derivative w.r.t $u$, whose action is seen as
\[h \mapsto D_u \Phi(u, v)(h), \quad \mbox{for } h \in \bH^1_0\, .\] 
We denote the second derivative w.r.t. $u$ by $D_u^2 \Phi(u, v) \in \mathcal{L}_2(\bH^1_0, \bL^2)$, whose action can be seen as
\[(h, k) \mapsto D_u^2 \Phi(u, v)(h, k) := (D_u^2 \Phi(u, v) h)(k) \quad \mbox{for } (h, k) \in [\bH^1_0]^2\, .\]
Similarly, we define $D_v \Phi(u, v), D_v^2 \Phi(u, v)$.

\smallskip

\subsubsection{A quadrature formula}

\del{
In our paper, we use an important trapezoid inequality in order to find a relation between the It\^o integral $\widetilde{\Delta_n W}$ and $\widehat{\Delta_n W}.$ In \cite{Dragomir_2011}, we have the generalisation of the classical trapezoid inequality for integrators of bounded variation and H\"older continuous integrands. We state the result here.
\begin{lemma}\label{Dra-thm}
Let $f:[a, b] \to \mathbb{C}$ be a $p-\eta$-H\"older type function, that is, it satisfies the condition
\[|f(x) - f(y)| \leq \eta\,|x-y|^p, \quad \forall x, y\in [a, b],\]
where $\eta>0$ and $p\in (0, 1]$ are given, and $z:[a, b] \to \mathbb{C}$ is a function of bounded variation on $[a, b]$. Then we have the inequality:
\begin{align}
\Big| \frac{f(a)+f(b)}{2} \cdot [z(a)-z(b)] - \int_a^b f(t)\,dz(t)\Big| \leq \frac{1}{2^{p}} \,\eta\,(b-a)^p \,TV_{[a,b]}(z)\, ,
\end{align}
where $TV_{[a,b]}(z)$ denotes the total variation of $z$ on the interval $[a, b].$
\end{lemma}
}

\noindent
The following quadrature formula will be crucially used in our error analysis (see \cite[Thm. 2]{DM_2000}).
\begin{lemma}\label{quadrature1}
	Let $f \in {C}^{1,\gamma}([0,T]; {\mathbb R})$, for some $\gamma \in (0,1]$. 
	Then there holds
	\begin{align*}
	\Bigl|\frac{f(0) + f(T)}{2} 
	-\frac{1}{T}\int_0^T f(\xi)  \, \dd \xi\Bigr|
	\le \frac{\widetilde{C}}{(\gamma+2)(\gamma + 3)} T^{1+\gamma},
	\end{align*}
	where $\widetilde{C}>0$ satisfies
	\begin{equation}\label{cond1} 
	\bigl|f'(t)-f'(s)\bigr|
	\le \widetilde{C} |t-s|^{\gamma} \qquad \forall \, s,t \in [0,T]\, .
	\end{equation}
\end{lemma}

\smallskip

\subsection{Assumptions}

In this section, we list all the assumptions and hypotheses that are imposed throughout this paper.

\subsubsection{Domain and initial data}
We make the following assumptions.\\

\noindent
{\bf (A1)} \label{Domain}
Let $\cO\subset {\mathbb R}^d,$ for $1 \leq d\leq 3$,
be a bounded domain 
\begin{itemize}
\item[$(i)$]
with $\partial \cO$ of class $\mathrm{C}^1,$ and $(u_0, v_0) \in \bH^1_0 \times \bL^2\, ,$
\item[$(ii)$]
with $\partial \cO$ of class $\mathrm{C}^2,$ and $(u_0, v_0) \in ( \bH^1_0 \cap \bH^2) \times \bH^1_0\, ,$
\item[$(iii)$]
with $\partial \cO$ of class $\mathrm{C}^3,$ and $(u_0, v_0) \in (\bH^1_0 \cap \bH^3) \times (\bH^1_0 \cap \bH^2)\, .$
\item[$(iv)$]
with $\partial \cO$ of class $\mathrm{C}^4,$ and $(u_0, v_0) \in (\bH^1_0 \cap \bH^4) \times (\bH^1_0 \cap \bH^3)\, .$
\end{itemize}

\subsubsection{Probability set-up}

For simplicity, let $W$ be a finite-dimensional Wiener process.\\

\noindent
{\bf (A2)} \label{pro-space}
Let $\mathfrak{P}:=\big( \Omega, \mathcal{F}, \{\mathcal{F}_t\}_{t \geq 0}, \mathbb{P}\big)$ be a stochastic basis with a complete filtration $\{\mathcal{F}_t\}_{t \geq 0} \subseteq \mathcal{F}$. For some $M \in \mathbb{N}$, let $W$ be a $\mathbb{K}$-valued Wiener process on $\mathfrak{P}$ of the form 
\begin{align}\label{w-form}
W(t, x, \omega) := \sum_{j = 1}^M  \beta_j(t, \omega) e_j(x)\,, 
\end{align}
where $\mathbb{K} \subseteq \bH^1_0 \cap \bH^3$ is a Hilbert space, and $\big\{\beta_j(t, \omega) ; t \geq 0 \big\}$ are mutually independent  Brownian motions relative to $\{\mathcal{F}_t\}_{t \geq 0}$, and $\{e_j\}_{j=1}^M$ be an orthonormal basis of $\mathbb{K}$.

\del{
Note that the series does not converge in $\bL^2,$ but in a larger Hilbert space $U$ such that $\bL^2 \hookrightarrow U$ is Hilbert-Schmidt. For all $s, t\in [0, T]$ and $g, h \in \bL^2$ we have
\[\bE \big[ (W(s)g \cdot W(t)h) \big] = (s \wedge t) (g, h)_{\bL^2}\,.\]
This means that the covariance operator of $W$ is $\mbox{Id}: \bL^2 \to \bL^2$.  We shall always assume that $W(t)$ is adapted to the given filtration $(\mathcal{F}_t)_{t \geq 0}$. For a given linear operator $\Phi$ from $\bL^2$ to some Hilbert space $\mathbb{K},$ the Wiener process $\Phi\,\dd W = \sum_{k \in \mathbb{N}} \beta_k \Phi \,e_k$ is well-defined in $\mathbb{K}$ provided we have $\Phi \in \mathcal{L}_2(\bL^2, \mathbb{K}).$ 
}

\subsubsection{The nonlinearity of the model}

Let $F: [{\mathbb H}^1_0]^2 \rightarrow \bL^2$ and $\sigma: [{\mathbb H}^1_0]^2 \rightarrow {\mathbb H}^1_0$. \\

\noindent
{\bf (A3)} \label{as-Fs}
Assume $F(u, v) = F_1(u) + F_2(v)$ and $\sigma(u, v) = \sigma_1(u) + \sigma_2(v)$, such that $F_2(v)$ and $\sigma_2(v)$ are affine in $v$. For any $u, \tilde{u} \in \bH^1_0$, there is a constant $C_{\tt L}>0$ such that the Lipschitz condition holds:
\begin{align*}
&\|F_1(u) - F_1(\tilde{u}) \|_{\bL^2} + \|\sigma_1(u) - \sigma_1(\tilde{u}) \|_{\bL^2} \leq C_{\tt L} \,\| u - \tilde{u}\|_{\bL^2}\,.
\end{align*} 

\smallskip

\noindent
{\bf (A4)} \label{as-bd-Fs}
There exists a constant $C_g>0$ such that
\begin{align*}
&\|D_u^m F_1(\cdot)\|_{L^{\infty}(\bH^1_0; \cL_m(\bH^1_0, \bL^2))} + \|D_u^m \sigma_1(\cdot)\|_{L^{\infty}(\bH^1_0; \cL_m(\bH^1_0, \bH^1_0))} \le C_g \quad (m=1, 2, 3)\,.
\end{align*}
By the assumption {\bf (A4)}, we deduce that $F(u, v) \in \bH^1$. Since $F(u, v)$ is not assumed to be zero on the boundary, we introduce the following notation.

\smallskip

\noindent
{\bf (A5)}
Let $\widehat{F}(u, v) := F(u, v) - F(0, 0) = F_1(u) + F_2(v) - F(0, 0)$, and assume $F(0, 0) \in L^2(0, T; \bH^m)$ for $m=1, 2, 3$.

\bigskip

\section{Definition and properties of solution}\label{Mat-fra}

\noindent
We recall the concept of a strong  variational  solution for \eqref{stoch-wave1:1a} with $A = -\Delta$ and establish stability results in higher spatial norms, and bounds in temporal H\"older norms. 
\begin{definition}\label{def-Strong_sol}
Assume {\bf (A1)$_{i}$}, {\bf (A2)} and {\bf (A3)}. We call the tuple $(u, v)$ a  strong variational solution of \eqref{stoch-wave1:1a} with $A = -\Delta$ on the interval $[0, T]$ if
\begin{itemize}
\item[$(i)$]
 \ $(u, v) \ \mbox{is an} \ \bH^1_0 \times \bL^2\mbox{-valued}, \{\mathcal{F}_t \}\mbox{-adapted process};$
\item[$(ii)$]
 \ $(u, v) \in L^2\big(\Omega; C([0, T]; \bH^1_0)\big) \times L^2\big(\Omega; C([0, T]; \bL^2)\big);$
and 
{\small{
\begin{align}\label{strong1}
(u(t),\phi) &= \int_0^t (v ,\phi)\, \ds + (u_0,\phi)\, \quad \forall \phi \in \bL^2\, ,\\
(v(t),\psi) &= -\int_0^t \Bigl[(\nabla u , \nabla \psi) +\bigl(F(u,v), \psi \bigr)\Bigr] \ds +\int_0^t \bigl(\psi,\sigma (u,v)\, \dW(s) \bigr) 
+ (v_0,\psi)\, \quad \forall \psi \in \bH^1_0\, , \label{strong2} 
\end{align}
}}
holds for each $t \in [0, T]$ $\bP$-a.s..
\item[$(iii)$]
There exists a constant $C>0$, depending on $T, C_{\tt L}$ and initial data such that there holds $\mathbb P$-a.s.
\begin{align*}
\bE\Bigl[\supT \mathcal{E}(u(t), v(t)) \Bigr] \le C\, .
\end{align*}
\end{itemize}
\end{definition}
\noindent
The existence of a unique strong variational solution was shown in \cite[Sec. 6.8, Thm. 8.4]{Chow_2015}.

\smallskip

\begin{lemma}\label{lem:L2}
Let  $(u, v)$ be the strong (variational) solution to the problem \eqref{strong1}--\eqref{strong2}. For $p \in \mathbb{N}$, there holds $\mathbb P$-a.s.
\begin{itemize}
\item[$(i)$]
under the hypotheses {\bf (A1)$_{i}$}, {\bf (A2)}, and {\bf (A3)}, the $\{\mathcal{F}_t \}_{t \geq 0}-$adapted process \\$(u, v) \in L^{2p} \big(\Omega; L^{\infty}(0, T; \bH^{1} \times \bL^{2}) \big)$, and there exists $K_1 \equiv K_1(p)>0$, such that
\begin{align}
\label{lem:L2:1}
&\bE\Bigl[\supT \Big(  \|u(t)\|_{{\mathbb H}^1}^{2p} + \|v(t)\|_{{\mathbb L}^2}^{2p} \Big) \Bigr] 
\le K_1\, ;
\end{align} 
\item[$(ii)$]
under the hypotheses {\bf (A1)$_{ii}$}, {\bf (A2), (A3)}, and {\bf (A4), (A5)} for $m=1$, the $\{\mathcal{F}_t \}_{t \geq 0}-$adapted process $(u, v) \in L^{2p} \big(\Omega; L^{\infty}(0, T; \bH^{2} \times \bH^{1}) \big)$, and there exists $K_2 \equiv K_2(p)>0$, such that
\begin{align}
\label{lem:H1:1}
&\bE \Bigl[\supT \Big( \| u(t)\|_{\bH^2}^{2p} + \| v(t)\|_{\bH^1}^{2p} \Bigr) \Bigr] \le K_2\, ;
\end{align} 
\item[$(iii)$]
under the hypotheses {\bf (A1)$_{iii}$}, {\bf (A2), (A3)}, and {\bf (A4), (A5)}  for $m=1, 2$, the $\{\mathcal{F}_t \}_{t \geq 0}-$ adapted process $(u, v) \in L^{2p} \big(\Omega; L^{\infty}(0, T; \bH^{3} \times \bH^{2}) \big)$, and there exists $K_3 \equiv K_3(p)>0$, such that
\begin{align}
\label{lem:H2:1}
&\bE \Bigl[ \supT \Big( \|u(t)\|_{\bH^3}^{2p} + \|v(t)\|_{\bH^2}^{2p} \Big) \Bigr] \le K_3\, ;
\end{align} 
\item[$(iv)$]
under the hypotheses {\bf (A1)$_{iv}$}, {\bf (A2), (A3)}, and {\bf (A4), (A5)} for $m=1, 2, 3$, the $\{\mathcal{F}_t \}_{t \geq 0}-$ adapted process $(u, v) \in L^{2p} \big(\Omega; L^{\infty}(0, T; \bH^{4} \times \bH^{3}) \big)$, and there exists $K_4 \equiv K_4(p)>0$, such that
\begin{align}
\label{lem:H3:1}
&\bE \Bigl[ \supT \Big( \|u(t)\|_{\bH^4}^{2p} + \|v(t)\|_{\bH^3}^{2p} \Big) \Bigr] \le K_4\, .
\end{align} 
\end{itemize}
\end{lemma}	

\begin{proof}
The proof is given in Appendix \ref{Lem-EE}.
\end{proof}

\subsection{H\"older continuity in time}\label{Hoe-cont}

In this subsection, we derive temporal H\"older continuity estimates for the solution pair $(u, v)$ of the problem \eqref{strong1}--\eqref{strong2} with respect to different norms, which will be useful in the error analysis in later section.

\begin{lemma} \label{lem:Holder}
Let $(u, v)$ be the strong (variational) solution to the problem \eqref{strong1}--\eqref{strong2}. Then, for any $s,t \in [0,T]$, we have for $p \ge 1$ 
\begin{itemize}
\item[$(i)$]
under the hypotheses {\bf (A1)$_{i}$}, {\bf (A2)}, and {\bf (A3)}, there holds $\mathbb P$-a.s.
\[\left( \bE \Bigl[ \sup_{s \leq r \leq t} \|u(r)-u(s)\|_{{\mathbb L}^2}^{2p} \Bigr] \right)^{1/2p} \le C(K_1)\, |t-s| \, ;\]
\item[$(ii)$]
under the hypotheses {\bf (A1)$_{ii}$}, {\bf (A2), (A3)}, and {\bf (A4), (A5)} for $m=1$, there holds $\mathbb P$-a.s.
\[ \left( \bE \Bigl[ \sup_{s \leq r \leq t} \|u(r)- u(s)\|_{{\mathbb H}^1}^{2p} \Bigr] \right)^{1/2p} + \bE \Bigl[ \sup_{s \leq r \leq t} \|v(r)-v(s)\|_{{\mathbb L}^2}^2 \Bigr] \le C(K_2)\, |t-s|\, ;\]
\item[$(iii)$]
under the hypotheses {\bf (A1)$_{iii}$}, {\bf (A2), (A3)}, and {\bf (A4), (A5)}  for $m=1, 2$, there holds $\mathbb P$-a.s.
\[ \left( \bE \Bigl[ \sup_{s \leq r \leq t} \| u(r)- u(s) \|_{{\mathbb H}^2}^{2p} \Bigr] \right)^{1/2p} + \bE \Bigl[ \sup_{s \leq r \leq t} \| v(r)- v(s)\|_{{\mathbb H}^1}^2 \Bigr] \le C(K_3)\, |t-s| \, ,\]
\item[$(iv)$]
under the hypotheses {\bf (A1)$_{iv}$}, {\bf (A2), (A3)}, and {\bf (A4), (A5)} for $m=1, 2, 3$, there holds $\mathbb P$-a.s.
\[\left( \bE \Bigl[ \sup_{s \leq r \leq t} \|u(r)- u(s)\|_{{\mathbb H}^3}^{2p} \Bigr] \right)^{1/2p} +\bE \Bigl[ \sup_{s \leq r \leq t} \|v(r)-v(s)\|_{{\mathbb H}^2}^2 \Bigr]
\le C(K_4)\, |t-s|\, ,\]
\end{itemize}
where the positive constants $C(K_i)$ for $i=1, \cdots, 4$, depend on the constants $K_i$, defined in Lemma \ref{lem:L2}. 
\end{lemma}
\begin{proof}
The proof is given in Appendix \ref{Hoe-con}.
\end{proof}

\bigskip

\section{Discrete Stability Analysis for the $(\widehat{\alpha}, \beta)-$scheme}\label{sec-3}

If compared to the term $-\Delta {u}^{n,\frac{1}{2}}$, the term $-\Delta\widetilde{u}^{n,\frac{1}{2}}$ in the $(\widehat{\alpha}, \beta)-$scheme fortifies stability properties of the method: in fact, the identity
\begin{equation}\label{tidt1}
\widetilde{u}^{n, \frac 12} = {u}^{n, \frac 12} + \beta \frac{k^{\beta}}{2} \big(u^{n+1} - u^{n-1} \big) =
{u}^{n, \frac 12} + \beta k^{1+\beta} v^{n+\frac{1}{2}} 
\end{equation}
creates an additional  `numerical dissipation' term scaled by $\beta k^{2+\beta}$ in \eqref{stoch-wave1:1a}, which suffices to control general
noises $\sigma \equiv \sigma(u,v)$, in case $0 < \beta <\frac{1}{2}$ (see Lemma \ref{lem:scheme1:stab} below); for $\sigma \equiv \sigma(u)$ only, the scheme \eqref{scheme2:1}--\eqref{scheme2:2} yields a stable scheme.

\smallskip

In this section, we discuss the discrete stability analysis for the $(\widehat{\alpha}, \beta)-$scheme and we make a remark on the stability results of the scheme \eqref{scheme2:1}--\eqref{scheme2:2} as this is a sub-case of the $(\widehat{\alpha}, \beta)-$scheme. We recall \eqref{energ1}, where the energy functional is stated. 

\smallskip

\noindent
{\bf (B1)} For the stability results, we need the following assumptions on the iterates $(u^1, v^1)$:
\begin{itemize}
\item[$(i)$]
Along with {\bf (A1)}$_{i}$, assume $(u^1, v^1) \in L^{2^{p}}\big(\Omega; [\bH^1_0 ]^2 \big)$ for $p \geq 1$.
\item[$(ii)$]
Along with {\bf (A1)}$_{ii}$, assume $(u^1, v^1) \in L^{2^{p}}\big(\Omega; [\bH^1_0 \cap \bH^2 ]^2 \big)$ for $p \geq 1$.
\end{itemize}

{{
\begin{lemma}\label{lem:scheme1:stab}
Let $\widehat{\alpha} \in \{0, 1\}$. Assume {\bf (A1)$_{ii}$}, {\bf (A2), (A3)}, {\bf (A4)} for $m=1$, and {\bf (B1)$_{i}$}. Then, there exists an $\big[ \bH^1_0 \big]^2$-valued 
$\{ (\cF_{t_{n}})_{0 \le n \le N} \}$-adapted solution $\{(u^{n}, v^{n}); \, 0 \leq n \leq N\}$ of the $(\widehat{\alpha}, \beta)-$scheme. Moreover, for {\red{$0 < \beta < \frac{1}{2}$}}, and $k \leq k_0(C_{\tt L}, C_g)$ sufficiently small, there exists a constant ${\red{C_1 \equiv C_1(\beta)}}  > 0$ that does not depend on $k>0$ such that 
\begin{equation}\label{energy1}
\max_{1 \le n \le N-1} \bE \Bigl[ \cE(u^{n+1},v^{n+1}) \Bigr] + {\beta {k^{2+\beta}\, \sum_{n=1}^{N-1} \bE \Big[ \| \nabla v^{n+ \frac 12} \|_{\bL^2}^2 \Big]}} \le C_1\, .
\end{equation}
In addition, there exists a further constant $C_{2,p} \equiv C_{2,p}(\beta) >0$ such that we have
\begin{equation}\label{high-moment}
\max_{1 \le n \le N-1} \bE\bigl[ {\mathcal E}^{2^p}(u^{n+1},v^{n+1}) \bigr] \leq C_{2, p} \qquad (p \geq 1)\, .
\end{equation}
Additionally, assume $\sigma_2(v) \equiv 0 \equiv F_2(v)$ in {\bf (A3)} and $\beta=0$. For $k \leq k_0(C_{\tt L}, C_g)$ sufficiently small, there exists a constant $C_3  > 0$ independent of $k>0$ such that
\begin{align}\label{energy2}
\max_{1 \le n \le N} \bE\Bigl[ \| u^{n}\|^2_{\bL^2}\Bigr] + \frac 14 \bE \Big[ k \sum_{j=1}^n  \bigl\Vert \nabla u^{j} \bigr\Vert^2_{{\mathbb L}^2}  \Big]  \le C_3\, .
\end{align}
There exists further constant $C_{4, p}>0$ such that we have
\begin{align}\label{energy2-himo}
\max_{1 \le n \le N} \bE\Bigl[ \| u^{n}\|^{2^{p}}_{\bL^2}\Bigr] \leq C_{4, p} \qquad (p \geq 1)\, .
\end{align}
\end{lemma}
\noindent
The following remark discusses specific problems to derive this stability result. 
\begin{remark}\label{rema2}
{\bf 1.}~The derivation of (discrete) stability estimates for a (temporal) discretization  for the stochastic wave equation \eqref{stoch-wave1:1a} --- like Scheme \ref{scheme--2} ---   differs from corresponding tasks for parabolic SPDEs, which is due to the {\em conservation of energy} in the deterministic case. In this case (where  $\sigma \equiv 0$), the test function $v^{n+1/2}$ is `natural' to deduce \eqref{energy1}; it is used in \cite{Dup_1973} as well, and exploits the (third) binomial formula, such that the first term in \eqref{scheme2--2} becomes
$$\bigl(v^{n+1} - v^{n}, \frac{1}{2} [v^{n+1} + v^{n}] \bigr) = \frac{1}{2} \bigl( \Vert v^{n+1} \Vert^2_{{\mathbb L}^2} - \Vert v^{n}\Vert^2_{{\mathbb L}^2}\bigr)\, ;$$
see also \eqref{S1S:1} below. Conceptional difficulties now appear in the stochastic case where $\sigma \neq 0$ --- see {\em e.g.} the term $\mathscr{J}^n_1$  in \eqref{S1S:1}. A well-known strategy in a setting of parabolic SPDEs would be to employ ${\mathbb E}\big[\Delta_n W \big] = 0$ to conclude
\begin{eqnarray}\label{probl1}\mathscr{J}^n_1 &=& {\mathbb E}\Bigl[\Bigl(\sigma (u^n, v^{n- \frac 12}) \,\Delta_nW, \frac{1}{2}[v^{n+1} - v^{n}] \Bigr)\Bigr] \\ \nonumber
&\leq& C_{\delta}\, {\mathbb E}\Bigl[\Vert\sigma (u^n, v^{n- \frac 12}) \,\Delta_nW
\Vert_{{\mathbb L}^2}^2\Bigr] + \delta  \, {\mathbb E}\bigl[ \Vert v^{n+1} - v^n\Vert^2_{{\mathbb L}^2}\bigr]\qquad (\delta >0)\, ,
\end{eqnarray}
and to then absorb the last term by a corresponding one on the left-hand side --- which would arise if
$v^{n+1}$ instead would have been chosen as  test function. We avoid the estimation in \eqref{probl1} by using the equation \eqref{scheme2--2} to replace
$v^{n+1} - v^{n}$ in $\mathscr{J}^n_1$; see \eqref{sig-v-n+1} below.

\smallskip

\noindent
{\bf 2.}~The last term in \eqref{tidt1} is the reason to evaluate $\sigma$ resp. $D_u \sigma$ at $(u^{n},v^{n-\frac 12})$ in \eqref{scheme2--2} --- instead of {\em e.g.}~at $(u^n, v^n)$; see also the left-hand side of \eqref{test-mult},
and the estimation of the terms $\mathscr{J}^{n, 1}_{1,2}$, $\mathscr{J}^{n, 2}_{1,1}$, and $\mathscr{J}^{n, 2}_{1,2}$ in the proof of Lemma \ref{lem:scheme1:stab}.

\smallskip

\noindent
{\bf 3.}~The inequality \eqref{high-moment} assembles higher moment estimates, which  will be used in Section \ref{sec-4} to derive improved rates of convergence; see Theorem \ref{lem:scheme1:con}.

\smallskip

\noindent
{\bf 4.}~The following two inequalities will be used frequently in the proof of discrete stability; see parts {\bf 1a)} and {\bf 1b)} below. The identity \eqref{DnW-tilde} leads to 
\begin{align*}
{\mathbb E}\left[ \vert \widetilde{\Delta_n W}\vert^2 \right]  \leq k \int_{t_n}^{t_{n+1}} \bE \Big[ |W(t_{n+1}) - W(s)|^2 \Big] \,\dd s \leq C k^3\, ,
\end{align*}
and by the identity \eqref{W-hat}, for $q=1, 2$ we infer
\begin{align*}
{\mathbb E}\left[ \vert \widehat{\Delta_n W}\vert^{2q} \right] &\leq C k^{2q} \,\bE \big[ |W(t_{n+1})|^{2q} \big] + C k^{2q+1} \,\sum_{\ell =1}^{k^{-1}} \bE \big[ |W(t_{n, \ell})|^{2q} \big] 
\leq C k^{3q} + C k^{4q} \leq C k^{3q}\, .
\end{align*}
The estimation of the distance between $\widetilde{\Delta_n W}$ and $\widehat{\Delta_n W}$ is useful in the convergence analysis in the next section, which is derived in \eqref{dist-tilde-hat}.

\smallskip

\noindent
{\bf 5.}~If $\sigma_2(v) \equiv 0 \equiv F_2(v)$ in {\bf (A3)} and $\beta=0$ hold, we consider the scheme \eqref{scheme2:1}--\eqref{scheme2:2}, where to verify discrete stability is easier; in fact, the identity (\ref{S1S:1}) simplifies considerably since both, $F$ and $\sigma$ do not depend on $v$ any more; see {\em e.g.}
%
%
%
the estimate for $\mathscr{J}^{n, 1}_{1, 1}$ in \eqref{413a}. For the higher moment estimates, we multiply the corresponding \eqref{frakE} of the scheme \eqref{scheme2:1}--\eqref{scheme2:2} with $\mathfrak{E}(u^{n+1}, v^{n+1})$ only, which is different from the general case; see step {\bf 2)} below.

\smallskip

\noindent
{\bf 6.}~In the proof of \eqref{energy2}, under the hypotheses $\sigma_2(v) \equiv 0 \equiv F_2(v)$ in {\bf (A3)}, $\beta=0$ and {\bf (A4)} for $m=1$, we combine both equations of the scheme \eqref{scheme2:1}--\eqref{scheme2:2} to write  a single equation for $u^{\ell}$ and sum over the first $n$ steps; see \eqref{sum_ul}. If we would apply the same approach for the general $\sigma \equiv \sigma(u^{\ell}, v^{\ell - 1/2})$, we could use the first equation of the $(\widehat{\alpha}, \beta)-$scheme to replace $v^{\ell - 1/2}$ by $\frac{1}{2k} \big( u^{\ell}- u^{\ell-2} \big)$. Thus, to estimate \eqref{kln2-sigmau} in general case, we use the growth condition to write
\begin{align*}
k^2 \sum^n_{\ell=1} \bE \Big[\|\sigma(u^{\ell}, (1/2k) ( u^{\ell}- u^{\ell-2}))\|_{\bL^2}^2 \Big] \leq C_{\tt L} k^2 \sum^n_{\ell=1} \bE\Big[ 1+ \|\nabla u^{\ell} \|_{\bL^2}^2 + \frac{1}{4 k^2} \big(\|u^{\ell} \|_{\bL^2}^2 + \|u^{\ell-2} \|_{\bL^2}^2 \big) \Big]\, ,
\end{align*}
which, by to the last term, is not in suitable form to apply the discrete Gronwall lemma. Thus, the approach to prove \eqref{energy2} is not useful for the general $\sigma \equiv \sigma(u, v)$.

\end{remark}

\begin{proof}[Proof of Lemma \ref{lem:scheme1:stab}]
The ${\mathbb P}$-a.s.~solvability easily follows from Lax-Milgram lemma, and {\bf (A3)}. Using the $\mathbb{L}^2$-regularity theory for elliptic equations on regular domains (see \cite[Sec.~15.5]{Gil+Tru}), the system in Scheme \ref{scheme--2} holds strongly $\mbox{Leb} \otimes \mathbb{P}-$a.s. The proof of Lemma \ref{lem:scheme1:stab} is split into the following three steps {\bf 1)} -- {\bf 3)}.

\smallskip

\noindent
{\bf 1) Proof of $(\ref{energy1})$.}
We use the test function $2k v^{n+1/2} = u^{n+1} - u^{n-1}$  in (\ref{scheme2--2}),
and identity (\ref{tidt1}) to  get
\begin{equation}\label{S1S:1}
\begin{split}
&\frac{1}{2} {\mathbb E}\Bigl[\|v^{n+1}\|^2_{\bL^2}  -\|v^{n}\|^2_{\bL^2}  \Bigr] + k {\mathbb E}\Bigl[ \big(\nabla  u^{n, \frac 12}, \nabla v^{n+\frac 12} \big)\Bigr] + \beta\,k^{2+\beta} 
{\mathbb E}\bigl[\| \nabla v^{n+ \frac 12} \|_{\bL^2}^2\bigr]
\\ &= {\mathbb E}\Bigl[\Bigl(\sigma (u^n, v^{n- \frac 12}) \,\Delta_nW, v^{n+\frac 12} \Bigr)
+ \,\widehat{\alpha} \,\Bigl(D_u\sigma(u^n, v^{n-\frac 12}) v^n \,\widehat{\Delta_n W}, v^{n+\frac 12} \Bigr) 
\\ &\qquad+ \frac{k}{2}\Bigl( 3F(u^n, v^n) - F(u^{n-1}, v^{n-1}), v^{n+\frac 12}\Bigr)\Bigr] =: \sum_{i=1}^3 \bE \big[ \mathscr{J}^n_i \big]\, . 
\end{split}
\end{equation}
 Next, we use (\ref{scheme2--1}) in strong form, sum it for two subsequent steps, and multiply this equation with $-\Delta u^{n, \frac 12}$; we then arrive at
\begin{align}\label{1eq-test}
\frac 14 \Big[ \|\nabla u^{n+1}\|_{\bL^2}^2 - \|\nabla u^{n-1}\|_{\bL^2}^2 \Big] = k \Bigl(\nabla  u^{n, \frac 12}, \nabla v^{n+\frac 12} \Bigr)\,.
\end{align}
Since the right-hand side of \eqref{1eq-test} is equal to 
the second term on the left-hand side of \eqref{S1S:1} we conclude that
\begin{equation}\label{test-mult}
\begin{split}
&\frac{1}{2} {\mathbb E}\Bigl[\|v^{n+1}\|^2_{\bL^2}  -\|v^{n}\|^2_{\bL^2}  \Bigr] + \frac 14 {\mathbb E}\Big[ \|\nabla u^{n+1}\|_{\bL^2}^2 - \|\nabla u^{n-1}\|_{\bL^2}^2 \Big] 
\\ &\quad+ \beta \,k^{2+\beta} {\mathbb E} \big[ \| \nabla v^{n+ \frac 12} \|_{\bL^2}^2 \big] = \sum_{i=1}^3 \bE \big[ \mathscr{J}^n_i \big]\, . 
\end{split}
\end{equation}
Now we estimate each term on the right-hand side of \eqref{test-mult}. Using properties of the increments $\Delta_n W$, we get $\bE\bigl[\bigl(\sigma (u^n, v^{n-\frac 12})\Delta_n W, v^{n} \bigr)\bigr]=0$. Using this we infer
\begin{align*}
\bE \big[ \mathscr{J}^n_1 \big] =\bE\Bigl[\Bigl(\sigma (u^n, v^{n-\frac 12})\Delta_n W, v^{n+ \frac 12} \Bigr)\Bigr] &= \bE\Bigl[\Bigl(\sigma (u^n, v^{n- \frac 12})\Delta_n W, \frac{v^{n+1} + v^n}{2} \Bigr) \Bigr]
\\ &=\frac 12 \bE\Bigl[\Bigl(\sigma (u^n, v^{n- \frac 12})\Delta_n W, v^{n+1} - v^n \Bigr) \Bigr] \, .
\end{align*} 
We use  (\ref{scheme2--2}) --- in modified form as stated in Scheme \ref{Hat-scheme} --- to replace $v^{n+1} -v^n$. Hence
{\small{
\begin{equation}\label{sig-v-n+1}
\begin{split}
 \bE\big[ \mathscr{J}^n_1 \big] &= \frac 12 \bE\Bigl[\Bigl(\sigma (u^n, v^{n- \frac 12})\Delta_n W, k \Delta u^{n, \frac 12} \Bigr) \Bigr] + \frac 12 \bE\Bigl[\Bigl(\sigma (u^n, v^{n- \frac 12}) \Delta_n W, \beta k^{2+\beta} \Delta v^{n+ \frac 12} \Bigr) \Bigr]
\\ &\quad+ \frac 12 \bE \Big[ \|\sigma (u^n, v^{n- \frac 12})\|_{\bL^2}^2 |\Delta_n W|^2 \Big] 
\\ &\quad+ \frac{\widehat{\alpha}}{2}\, \bE\Bigl[\Bigl(\sigma (u^n, v^{n- \frac 12})\Delta_n W, D_u\sigma(u^{n}, v^{n- \frac 12}) v^{n} \,\widehat{\Delta_n W} \Bigr) \Bigr] 
\\ &\quad+ \frac{k}{4}\,\bE\Bigl[\Bigl(\sigma (u^n, v^{n- \frac 12})\Delta_n W, \bigl[ 3F(u^n, v^n) - F(u^{n-1}, v^{n-1}) \bigr] \Big) \Bigr] =: \sum_{i=1}^5 \mathscr{J}^{n, i}_1\, .
\end{split}
\end{equation}
}}
In the following parts {\bf a)}--{\bf c)}, we independently bound
$\bE \big[ \mathscr{J}^{n}_1 \big]$ through $\bE \big[ \mathscr{J}^{n}_3 \big]$ in (\ref{S1S:1}). 

\smallskip

\noindent
{\bf a) Estimation of $\bE\big[ \mathscr{J}^{n}_1 \big]$ in (\ref{sig-v-n+1}).}  We estimate the five terms $\mathscr{J}^{n, i}_1, \,i=1, \cdots, 5,$ on the right-hand side of \eqref{sig-v-n+1}. Let $\underline{D_u \sigma} \equiv D_u \sigma(u^n, v^{n- \frac 12}) \in \mathcal{L}(\bH^1_0, \bH^1_0)$ and $\underline{D_v \sigma} \equiv D_v \sigma(u^n, v^{n- \frac 12}) \in \mathcal{L}(\bH^1_0, \bH^1_0)$. By integration by parts and using $\sigma (u^n, v^{n-\frac 12}) = 0$ on $\partial {\mathcal O}$ we infer
\begin{equation}\label{j^n,1_1est}
\begin{split}
\mathscr{J}^{n, 1}_1 &= \frac 12 \bE\Bigl[-\Bigl( \underline{D_u \sigma}\, \nabla u^n \Delta_n W, k \nabla u^{n, \frac 12} \Bigr) \Bigr] + \frac 12 \bE\Bigl[-\Bigl( \underline{D_v \sigma} \, \nabla v^{n- \frac 12} \Delta_n W, k \nabla u^{n, \frac 12} \Bigr) \Bigr] 
\\ &= \mathscr{J}^{n, 1}_{1, 1} +\mathscr{J}^{n, 1}_{1, 2}\, .
\end{split}
\end{equation}

\noindent
Using {\bf (A4)} for $m=1$ and the It\^o isometry we get
\begin{equation}\label{413a}
\begin{split}
\mathscr{J}^{n, 1}_{1, 1} &\leq C_g^2\, \bE \Big[ \|\nabla u^n\|_{\bL^2}^2 |\Delta_n W|^2\Big] + C k^2 \,\bE \Big[ \|\nabla u^{n, \frac 12}\|_{\bL^2}^2  \Big]
\\ &\leq C_g^2\, k\, \bE \big[ \|\nabla u^n\|_{\bL^2}^2 \big] + C k^2\, \bE \Big[ \|\nabla u^{n+1}\|_{\bL^2}^2 +\|\nabla u^{n-1}\|_{\bL^2}^2  \Big]\, .
\end{split}
\end{equation}
Using {\bf (A4)} for $m=1$, the independence property of the increment $\Delta_n W$, the It\^o isometry  and the identity $2k v^{n+1/2} = u^{n+1}-u^{n-1}$, we estimate
\begin{equation}\label{delv-sigma-1/2}
\begin{split}
\mathscr{J}^{n, 1}_{1, 2} &= \frac 12 \bE\Bigl[-\Bigl( \underline{D_v \sigma} \, \nabla v^{n- \frac 12} \Delta_n W, \frac k2\, \nabla \big[ u^{n+1} - u^{n-1} \big] \Bigr) \Bigr] 
\\ &= \frac 12 \bE\Bigl[-\Bigl( k^{1- \frac{\beta}{2}}\underline{D_v \sigma} \,\nabla v^{n- \frac 12} \Delta_n W, k^{1+ \frac{\beta}{2}} \,\nabla v^{n+ \frac 12}\Bigr) \Bigr] 
\\ &\leq \frac{3}{8 \beta} \,C_g^2\,  \,k^{1+(2- \beta)}\, \bE \Big[ \|\nabla v^{n- \frac 12}\|_{\bL^2}^2 \Big] + \frac{\beta}{6} \,k^{2+ \beta}\, \bE \Big[ \|\nabla v^{n+\frac 12}\|_{\bL^2}^2 \Big]\,.
\end{split}
\end{equation}
The terms on the right-hand sides of (\ref{413a}) and (\ref{delv-sigma-1/2})
may now be controlled by those on the left-hand side of (\ref{test-mult}) after summation over $1 \le n \le N-1$, provided that
$k \leq \big( \frac{4}{3 C_g^2}\big)^{\frac{1}{1-2\beta}}$ for $\beta < \frac 12$ is sufficiently small.

\smallskip

\noindent
Now we turn to $\mathscr{J}^{n, 2}_{1}$ in (\ref{sig-v-n+1}):  integration by parts and using the fact that $\sigma (u^n, v^{n-\frac 12}) = 0$ on $\partial {\mathcal O}$ we get
\begin{equation}\label{sig-v-n+1_}
\begin{split}
\mathscr{J}^{n, 2}_{1} &= \frac 12\, \bE\Bigl[-\Bigl( k^{1+\frac{\beta}{2}}\underline{D_u \sigma} \, \nabla u^n \,\Delta_n W, \beta\,k^{1+\frac{\beta}{2}}\, \nabla v^{n+\frac 12} \Bigr) \Bigr] 
\\ &\quad+ \frac 12 \,\bE\Bigl[-\Bigl( k^{1+\frac{\beta}{2}} \underline{D_v \sigma} \, \nabla v^{n- \frac 12} \Delta_n W, \beta\,k^{1+\frac{\beta}{2}}\,\nabla v^{n+\frac 12} \Bigr) \Bigr]  =: \mathscr{J}^{n, 2}_{1, 1} +\mathscr{J}^{n, 2}_{1, 2}\, .
\end{split}
\end{equation}
Using {\bf (A4)} for $m=1$ and the independence property of $\Delta_n W$ we estimate
\begin{align*}
\mathscr{J}^{n, 2}_{1, 1} \leq C \beta \,k^{3+\beta}\,\bE \big[ \|\nabla u^n\|_{\bL^2}^2 \big] +  \frac \beta6\, k^{2+ \beta}\, \bE \Big[ \|\nabla v^{n+\frac 12}\|_{\bL^2}^2 \Big]\,,
\end{align*}
where the second term in the right-hand side can be 
be controlled by the corresponding term on the left-hand side of (\ref{test-mult}).
 Again, using {\bf (A4)} for $m=1$ we obtain for the second term in (\ref{sig-v-n+1_})
\begin{align*}
\mathscr{J}^{n, 2}_{1, 2} \leq \beta \frac 38 C_g^2\,k\,k^{2+\beta}\,\bE \big[ \|\nabla v^{n-\frac 12}\|_{\bL^2}^2 \big] +  \frac \beta6\, k^{2+ \beta}\, \bE \Big[ \|\nabla v^{n+\frac 12}\|_{\bL^2}^2 \Big]\,,
\end{align*}
where the right-hand side can be managed with the left-hand side of \eqref{test-mult} for $\beta <1/2$. 

\smallskip

\noindent
We continue with the next term $\mathscr{J}^{n, 3}_1$ in \eqref{sig-v-n+1}:
by It\^o isometry and {\bf (A3)},
\begin{align*}
\mathscr{J}^{n, 3}_1 &= \frac 12 \bE \Big[ \|\sigma (u^n, v^{n- \frac 12})\|_{\bL^2}^2 |\Delta_n W|^2 \Big] 
\leq C k \, \bE \Big[ 1+ \|\nabla u^{n}\|_{\bL^2}^2 + \|v^{n}\|_{\bL^2}^2 + \|v^{n-1}\|_{\bL^2}^2 \Big]\, ,
\end{align*}
where $C>0$ depends on $C_{\tt L}$.
Next comes $\mathscr{J}^{n, 4}_1$: using {\bf (A3)}, {\bf (A4)} for $m=1$ and item {\bf 4.}~of Remark \ref{rema2}, we infer
\begin{equation}
\begin{split}
\mathscr{J}^{n, 4}_1 &\leq C \bE \Big[ \|\sigma (u^n, v^{n- \frac 12})\|_{\bL^2}^2 |\Delta_n W|^2 \Big] + \frac{\widehat{\alpha}^2}{4} C_g^2 \, \bE \Big[ \|v^{n} \|_{\bL^2}^2 \big|\widehat{\Delta_n W} \big|^2 \Big]
\\ &\leq C k \, \bE \Big[ 1+ \|\nabla u^{n}\|_{\bL^2}^2 + \|v^{n}\|_{\bL^2}^2 + \|v^{n-1}\|_{\bL^2}^2 \Big] + \frac{\widehat{\alpha}^2}{4} C_g^2 \, k^3\,\bE \big[ \|v^{n} \|_{\bL^2}^2 \big]\,,
\end{split}
\end{equation}
where $C>0$ depends on $C_{\tt L}$. The last term is $\mathscr{J}^{n, 5}_1$:
by {\bf (A3)} we obtain
\begin{equation}\label{j^n,1_5est}
\begin{split}
\mathscr{J}^{n, 5}_1 &\leq C \bE \Big[ \|\sigma (u^n, v^{n- \frac 12})\|_{\bL^2}^2 |\Delta_n W|^2 \Big] + C k^2 \, \bE \Big[ \|F(u^n, v^n)\|_{\bL^2}^2 + \|F(u^{n-1}, v^{n-1})\|_{\bL^2}^2 \Big]
\\ &\leq C k \, \bE \Big[ 1+ \|\nabla u^{n}\|_{\bL^2}^2 + \|\nabla u^{n-1}\|_{\bL^2}^2 + \|v^{n}\|_{\bL^2}^2 + \|v^{n-1}\|_{\bL^2}^2 \Big] \,,
\end{split}
\end{equation}
where the constant $C>0$ depends on $C_{\tt L}$. Thus, the estimate of $\mathscr{J}^n_1$ through
those of  $\mathscr{J}^{n, 1}_1$ through $\mathscr{J}^{n, 5}_1$ is complete.

\smallskip

\noindent
{\bf b) Estimation of $\bE\big[ \mathscr{J}^{n}_2 \big]$ in (\ref{S1S:1}).}   By {\bf (A4)} for $m=1$, item {\bf 4.}~of Remark \ref{rema2}, and the independence property of $\widehat{\Delta_n W}$,
\begin{equation*}
\begin{split}
 \mathscr{J}^n_2  &\leq \widehat{\alpha}^2 \frac 1k \,\bE\Bigl[ \| D_u\sigma(u^n, v^{n-\frac 12}) v^n \|_{\bL^2}^2 \big|\widehat{\Delta_n W} \big|^2 \Bigr] + C k\,\bE \Big[ \| v^{n+ \frac 12}\|_{\bL^2}^2\Bigr] \\
&\le C_g^2 \,\widehat{\alpha}^2\,k^2 \,\bE \big[ \| v^{n}\|_{\bL^2}^2\bigr]
+ C k\,\bE \Big[ \| v^{n+1}\|_{\bL^2}^2 + \| v^{n}\|_{\bL^2}^2\Bigr]\, .
\end{split}
\end{equation*}
\smallskip

\noindent
{\bf c) Estimation of $\bE\big[ \mathscr{J}^{n}_3 \big]$ in (\ref{S1S:1}).} By  {\bf (A3)} we estimate
\begin{equation}\label{S1S:4}
\begin{split}
 \mathscr{J}^n_3 &= k \bE \Big[ \Bigl(F(u^n, v^n) , v^{n+\frac 12}\Bigr) \Big] +\frac{k}{2}\, \bE \Big[\Bigl(F(u^n, v^n) - F(u^{n-1}, v^{n-1}), v^{n+ \frac 12}\Bigr) \Big]
\\ 
& \le C k\, \bE \big[ \|v^{n+\frac 12}\|^2_{\bL^2} \big] + C \,k\, \bE \Bigl[ 1+ \| \nabla u^{n}\|_{\bL^2}^2 + \| \nabla u^{n-1}\|_{\bL^2}^2 + \| v^{n}\|_{\bL^2}^2 + \| v^{n-1}\|_{\bL^2}^2 \Bigr]\, .
\end{split}
\end{equation}
Now, we may use the parts {\bf a)} through {\bf c)} to bound the terms on the right-hand side of \eqref{test-mult}. Summation over all $1 \le n \le N-1$,
for 
%
$k \leq k_0 \equiv k_0(C_{\tt L}, C_g)$  sufficiently small, leads to
{{
\begin{equation}\label{accu-all-1}
\begin{split}
&\frac{1}{4} \bE \Big[ \cE(u^N, v^N) \Big] +  \beta \frac 14 \,k^{2+\beta}\, \bE \big[ \| \nabla v^{N+ \frac 12} \|_{\bL^2}^2\big]
\\ &\leq \frac{1}{4} \bE \big[ \cE(u_0, v^1) \big] + \beta \frac{k^{2+\beta}}{4} {\mathbb E}\bigl[ \Vert \nabla v^{1/2}\Vert_{\bL^2}^2\bigr] + \frac 14 \bE \big[ \|\nabla u^1\|_{\bL^2}^2 \big] + C k \sum_{n=1}^{N-1} \bE\big[ \cE(u^n, v^n) \big]  \, .
\end{split}
\end{equation}
}}
By {\bf (B1)$_{i}$}, the implicit version of the discrete Gronwall lemma  then shows (\ref{energy1}) for $\beta\in(0, \frac{1}{2})$.

\del{
\noindent
Now taking the sum over $n=1, \cdots, N-1$,  and rearranging we infer that
\begin{equation}\label{eq-4.9}
\begin{split}
&\bE \big[ \cE(u^{N},v^{N}) \big] + \frac{1}{2} \sum_{n=1}^{N-1} \bE\big[ \|v^{n+1}-v^{n}\|^2_{\bL^2} \big]
\\ &\leq (1+C\,k)\,\bE \big[ \cE(u^{1},v^{1}) \big] + C \,k \sum_{n=1}^{N-1} \bE \big[ \cE(u^{n+1},v^{n+1}) \big] + C\,k\, \bE \big[ \cE(u^{0},v^{0}) \big] \, ,
\end{split}
\end{equation}
}

\medskip

\noindent
{\bf 2) Proof of \eqref{high-moment} for $p=1$.} To simplify technicalities, we put $F \equiv 0$.
%
%
%
%
Let us denote $\mathfrak{E}(u^{n+1}, v^{n+1}) := \frac 14 \big[ \|\nabla u^{n+1}\|^2_{\bL^2} + 2 \|v^{n+1}\|^2_{\bL^2}\big]$. Arguing as before \eqref{test-mult}
then leads to 
\begin{eqnarray}\label{frakE}
&&\big[ \mathfrak{E}(u^{n+1}, v^{n+1}) - \mathfrak{E}(u^{n-1}, v^{n}) \big]  +
\beta k^{2+\beta} \Vert \nabla v^{n+\frac 12}\Vert^2_{{\mathbb L}^2} \\ \nonumber
&&\qquad =  \Big(\sigma(u^{n}, v^{n- \frac 12}) \Delta_n W, v^{n + \frac 12} \Big) + \widehat{\alpha}\,\Big( D_u\sigma(u^{n}, v^{n- \frac 12}) v^{n} \,\widehat{\Delta_n W}, v^{n + \frac 12} \Big)\, .
\end{eqnarray}
Now fix $\frac{1}{4} \leq \delta_1, \delta_2 \leq 1$, then
multiply (\ref{frakE}) with $$\delta_1 \mathfrak{E}(u^{n+1}, v^{n+1}) +
\delta_2 \Bigl( \mathfrak{E}(u^{n+1}, v^{n+1}) + \mathfrak{E}(u^{n-1}, v^{n})\Bigr) \, ,$$
and take the expectation to get
{\small{
\begin{equation}\label{sum-mult-ene}
\begin{split}
&\frac{\delta_1 + 2\delta_2}{2} \bE \Big[ \mathfrak{E}^2(u^{n+1}, v^{n+1})  - \mathfrak{E}^2(u^{n-1}, v^{n}) \Big] + \frac{\delta_1}{2} \bE \Big[ \big| \mathfrak{E}(u^{n+1}, v^{n+1})  - \mathfrak{E}(u^{n-1}, v^{n}) \big|^2 \Big]
\\ &\quad+ \beta k^{2+\beta} {\mathbb E}\Bigl[ \Vert \nabla v^{n+\frac 12}\Vert^2_{{\mathbb L}^2} \big[ (\delta_1 + \delta_2) \mathfrak{E}(u^{n+1}, v^{n+1})+\delta_2\mathfrak{E}(u^{n-1}, v^{n}) \big] \Bigr] 
\\ &= \bE \Big[ \Big(\sigma(u^{n}, v^{n-\frac 12}) \Delta_n W, v^{n + \frac 12} \Big) \cdot \big[(\delta_1 + \delta_2) \mathfrak{E}(u^{n+1}, v^{n+1})+\delta_2\mathfrak{E}(u^{n-1}, v^{n}) \big] \Big] 
\\ &\quad+ \widehat{\alpha}\,\bE \Big[ \Big( D_u\sigma(u^{n}, v^{n-\frac 12}) v^{n} \,\widehat{\Delta_n W}, v^{n + \frac 12} \Big) \cdot \big[ (\delta_1 + \delta_2) \mathfrak{E}(u^{n+1}, v^{n+1})+\delta_2 \mathfrak{E}(u^{n-1}, v^{n}) \big] \Big] \\
& =: \mathscr{K}^{n,1} +  \mathscr{K}^{n,2}\,.
\end{split}
\end{equation}
}}
We independently estimate the terms $\mathscr{K}^{n,1}$ and $\mathscr{K}^{n,2}$.

\smallskip

\noindent
{\bf a) Estimation of $\mathscr{K}^{n,1}$ in (\ref{sum-mult-ene}).} This term may be written as the sum of two others:
\begin{equation}\label{1.21-eq}
\begin{split}
\mathscr{K}^{n,1}&= (\delta_1 + \delta_2) \bE \Big[ \Big(\sigma(u^{n}, v^{n-\frac 12}) \Delta_n W, v^{n + \frac 12} \Big) \cdot \big(\mathfrak{E}(u^{n+1}, v^{n+1}) - \mathfrak{E}(u^{n-1}, v^{n}) \big) \Big] 
\\ &\qquad+ (\delta_1+2\delta_2) \bE \Big[ \Big(\sigma(u^{n}, v^{n-\frac 12}) \Delta_n W, v^{n + \frac 12} \Big) \cdot \mathfrak{E}(u^{n-1}, v^{n}) \Big] := \mathscr{K}^{n,1}_1 + \mathscr{K}^{n,1}_2\,.
\end{split}
\end{equation}
We consider $\mathscr{K}^{n,1}_1$ first. By ${\mathbb E}\bigl[\vert \Delta_n W\vert^4 \bigr] = {\mathcal O}(k^2)$, and {\bf (A3)} we find
{\small{
\begin{equation*}
\begin{split}
\mathscr{K}^{n,1}_1 &\leq C_{\delta_1}\, \bE\Big[ \|\sigma(u^{n}, v^{n-\frac 12}) \Delta_n W \|_{\bL^2}^2 \|v^{n + \frac 12} \|_{\bL^2}^2 \Big] + \frac{\delta_1}{4} \bE\Big[ \big| \mathfrak{E}(u^{n+1}, v^{n+1})  - \mathfrak{E}(u^{n-1}, v^{n}) \big|^2 \Big] 
\\ & \leq \frac{C_{\delta_1}}{k} \bE\Big[ \|\sigma(u^{n}, v^{n-\frac 12})\|_{\bL^2}^4 |\Delta_n W |^4 \Big] + C_{\delta_1} k\, \bE\Big[ \|v^{n + \frac 12} \|_{\bL^2}^4 \Big] + \frac{\delta_1}{4} \bE\Big[ \big| \mathfrak{E}(u^{n+1}, v^{n+1})  - \mathfrak{E}(u^{n-1}, v^{n}) \big|^2 \Big] 
\\ &\leq C_{\delta_1}k\, \bE\Big[ 1 + \sum_{\ell=-1}^1\mathfrak{E}^2(u^{n+\ell}, v^{n+\ell})\Big] + \frac{\delta_1}{4} \bE\Big[ \big| \mathfrak{E}(u^{n+1}, v^{n+1})  - \mathfrak{E}(u^{n-1}, v^{n}) \big|^2 \Big] \, ,
\end{split}
\end{equation*}
}}
where the last term on the right-hand side can be absorbed on the left-hand side of \eqref{sum-mult-ene}. 

\noindent
We continue with $\mathscr{K}^{n,1}_2$: on
using ithe ndependence property of $\Delta_n W$, and equation \eqref{scheme2--2},
{\small{
\begin{equation*}
\begin{split}
\mathscr{K}^{n,1}_2 &\leq 3 \,\Bigl\vert\bE \Big[ \Big(\sigma(u^{n}, v^{n-\frac 12}) \Delta_n W, v^{n+1} - v^n \Big) \cdot \mathfrak{E}(u^{n-1}, v^{n}) \Big] \Bigr\vert
\\ &= 3 \, \Bigl\vert \bE \Big[ \Big(\sigma(u^{n}, v^{n-\frac 12}) \Delta_n W, k \Delta \widetilde{u}^{n, \frac 12} \Big) \cdot \mathfrak{E}(u^{n-1}, v^{n}) \Big] \Bigr\vert + 3\, \Bigl\vert \bE \Big[ \|\sigma(u^{n}, v^{n-\frac 12}) \Delta_n W \|_{\bL^2}^2\, \mathfrak{E}(u^{n-1}, v^{n}) \Big]\Bigr\vert
\\ &\quad+ 3 \, \Bigl\vert\bE \Big[ \Big(\sigma(u^{n}, v^{n-\frac 12}) \Delta_n W, D_u \sigma(u^n, v^{n-\frac 12}) v^n \widehat{\Delta_n W} \Big) \cdot \mathfrak{E}(u^{n-1}, v^{n}) \Big]\Bigr\vert =: \mathscr{K}^{n,1}_{2, 1} + \mathscr{K}^{n,1}_{2, 2} + \mathscr{K}^{n,1}_{2, 3} \, .
\end{split}
\end{equation*}
}}
We split $\mathscr{K}^{n,1}_{2, 1} := \mathscr{K}^{n,1, {\tt A}}_{2, 1} + \mathscr{K}^{n,1, {\tt B}}_{2, 1}$ because of (\ref{tidt1}); here, $\mathscr{K}^{n,1, {\tt A}}_{2, 1}$ is as $\mathscr{K}^{n,1}_{2, 1}$, where  $\widetilde{u}^{n, \frac{1}{2}}$ is replaced by ${u}^{n, \frac{1}{2}}$. We use
integration by parts, and the fact that $\sigma (u^n, v^{n-\frac 12}) = 0$ on $\partial {\mathcal O}$, {\bf (A4)} for $m=1$, the independence property of $\Delta_n W$ and that
${\mathbb E}\bigl[\vert \Delta_n W\vert^4 \bigr] = {\mathcal O}(k^2)$ to conclude
\begin{equation*}
\begin{split}
\mathscr{K}^{n,1,{\tt A}}_{2, 1} &= \frac 92 \, \Bigl\vert \bE \Big[ -\Big(\nabla \sigma(u^{n}, v^{n-\frac 12}) \Delta_n W, k \nabla [u^{n+1} - u^{n-1}] \Big) \cdot \mathfrak{E}(u^{n-1}, v^{n}) \Big] \Bigr\vert
\\ &= 9 \,\Bigl\vert\bE \Big[ -\Big(\nabla \sigma(u^{n}, v^{n-\frac 12}) \Delta_n W, k^{1- \frac{\beta}{2}}  k^{1+ \frac{\beta}{2}}\, \nabla v^{n+\frac 12} \Big) \cdot \mathfrak{E}(u^{n-1}, v^{n}) \Big] \Bigr\vert
\\ &\leq \frac{C_{\delta_2}}{\beta} k^{3-\beta}\,\bE \Big[\| \underline{D_u \sigma}\, \nabla u^n \|_{\bL^2}^2  \mathfrak{E}(u^{n-1}, v^{n}) \Big]  + \frac{C_{\delta_2}}{\beta} k^{3-\beta}\,\bE \Big[\|\underline{D_v \sigma}\, \nabla v^{n-\frac 12}\|_{\bL^2}^2  \mathfrak{E}(u^{n-1}, v^{n}) \Big] 
\\ &\quad+ \frac{\beta \delta_2}{4} k^{2+\beta} \,\bE \Big[ \|\nabla v^{n+\frac 12}\|_{\bL^2}^2 \,\mathfrak{E}(u^{n-1}, v^{n}) \Big]
\\ &\leq C_{\delta_2} C_{\beta} C_g^2 k^{3-\beta} \bE\Big[ \mathfrak{E}^2(u^{n}, v^{n}) + \mathfrak{E}^2(u^{n-1}, v^{n}) \Big] + C_{\delta_2} C_{\beta} C_g^2 \,k^{3-\beta} \,\bE \Big[ \|\nabla v^{n-\frac 12}\|_{\bL^2}^2 \,\mathfrak{E}(u^{n-1}, v^{n}) \Big] 
\\ &\quad+ \frac{\beta \delta_2}{4} k^{2+\beta} \,\bE \Big[ \|\nabla v^{n+\frac 12}\|_{\bL^2}^2 \,\mathfrak{E}(u^{n-1}, v^{n}) \Big]\,,
\end{split}
\end{equation*}

\noindent
where $C_{\beta} >0$ is a constant dependent on $\beta$ for $0<\beta<\frac{1}{2}$. We use a similar idea to estimate $\mathscr{K}^{n,1,{\tt B}}_{2, 1}$,
\begin{equation*}
\begin{split}
\mathscr{K}^{n,1,{\tt B}}_{2, 1} &= \frac{9}{2} \,\Bigl\vert \bE \Big[ -\Big(\nabla \sigma(u^{n}, v^{n-\frac 12}) \Delta_n W,  \beta k^{2+\beta} \nabla v^{n+\frac{1}{2}} \Big) \cdot \mathfrak{E}(u^{n-1}, v^{n}) \big)\Bigr] \Bigr\vert
\\ &\leq \frac{C_{\delta_2}}{\beta} k^{3+\beta}\,\bE \Big[\| \underline{D_u \sigma}\, \nabla u^n \|_{\bL^2}^2  \mathfrak{E}(u^{n-1}, v^{n}) \Big]  + \frac{C_{\delta_2}}{\beta} k^{3+\beta}\,\bE \Big[\|\underline{D_v \sigma}\, \nabla v^{n-\frac 12}\|_{\bL^2}^2  \mathfrak{E}(u^{n-1}, v^{n}) \Big] 
\\ &\quad+ \beta \frac{\delta_2}{4} k^{2+\beta} \,\bE \Big[ \|\nabla v^{n+\frac 12}\|_{\bL^2}^2 \,\mathfrak{E}(u^{n-1}, v^{n}) \Big]
\\ &\leq C_{\delta_2} C_{\beta} C_g^2\, k^{3+\beta}\,\bE \Big[ \mathfrak{E}^2(u^{n}, v^{n}) + \mathfrak{E}^2(u^{n-1}, v^{n}) \Big]  + C_{\delta_2} C_{\beta} C_g^2\, k^{3+\beta}\,\bE \Big[\| \nabla v^{n-\frac 12}\|_{\bL^2}^2  \mathfrak{E}(u^{n-1}, v^{n}) \Big] 
\\ &\quad+ \beta \frac{\delta_2}{4} k^{2+\beta} \,\bE \Big[ \|\nabla v^{n+\frac 12}\|_{\bL^2}^2 \,\mathfrak{E}(u^{n-1}, v^{n}) \Big]\, ,
\end{split}
\end{equation*}
where the last two terms in the right-hand sides of $\mathscr{K}^{n,1,{\tt A}}_{2, 1}$ and $\mathscr{K}^{n,1,{\tt B}}_{2, 1}$ may be controlled by those on the left-hand side of \eqref{sum-mult-ene} after summation over $1 \le n \le N-1$, provided that $k$ is sufficiently small and $\beta < \frac 12$ .

\noindent
Similarly, using {\bf (A3)}, and ${\mathbb E}\bigl[\vert \Delta_n W\vert^4 \bigr] = {\mathcal O}(k^2)$ we estimate
\begin{equation*}
\mathscr{K}^{n,1}_{2, 2} \leq \frac Ck\, \bE \Big[ \|\sigma(u^{n}, v^{n-\frac 12})\|_{\bL^2}^4 |\Delta_n W|^4 \Big] + C k\, \bE\Big[ \mathfrak{E}^2(u^{n-1}, v^{n}) \Big]
\leq C  k\, \bE \Big[ 1+ \sum_{\ell=-1}^0\mathfrak{E}^2(u^{n+\ell}, v^{n+\ell})\Big]\,.
\end{equation*}
Using ${\mathbb E}\bigl[\vert \widehat{\Delta_n W}\vert^4 \bigr] = {\mathcal O}(k^6)$, and {\bf (A3)} gives
\begin{equation*}
\begin{split}
\mathscr{K}^{n,1}_{2, 3} &\leq \frac Ck\, \bE \Big[ \| \sigma(u^{n}, v^{n-\frac 12}) \Delta_n W\|_{\bL^2}^2 \,\|D_u \sigma(u^n, v^{n-\frac 12}) v^n \widehat{\Delta_n W}\|_{\bL^2}^2  \Big] + C k\, \bE\Big[ \mathfrak{E}^2(u^{n-1}, v^{n}) \Big]
\\ &\leq \frac Ck\, \bE \Big[ \| \sigma(u^{n}, v^{n-\frac 12})\|_{\bL^2}^4 |\Delta_n W|^4 \Big]  + \frac{C_g^4}{k} \bE \Big[  \| v^n\|_{\bL^2}^4 \big| \widehat{\Delta_n W} \big|^4 \Big] + C k\, \bE\Big[ \mathfrak{E}^2(u^{n-1}, v^{n}) \Big]
\\ &\leq C  k\, \bE\Big[1+ \sum_{\ell=-1}^0\mathfrak{E}^2(u^{n+\ell}, v^{n+\ell})\Big]\,.
\end{split}
\end{equation*}
{\bf b) Estimation of $\mathscr{K}^{n,2}$ in (\ref{sum-mult-ene}).} 
By {\bf (A4)} for $m=1$ and using the fact that ${\mathbb E}\bigl[\vert \widehat{\Delta_n W}\vert^4 \bigr] = {\mathcal O}(k^6)$, we infer
\begin{equation*}
\begin{split}
\mathscr{K}^{n,2} &\leq \frac{\widehat{\alpha}^2}{k} \bE \Big[ \|D_u\sigma(u^{n}, v^{n-\frac 12}) v^{n} \,\widehat{\Delta_n W} \|_{\bL^2}^2\, \|v^{n + \frac 12}\|_{\bL^2}^2 \Big] 
\\ &\quad+ Ck \,\bE\Big[ \mathfrak{E}^2(u^{n+1}, v^{n+1}) \Big] + Ck \,\bE\Big[ \mathfrak{E}^2(u^{n-1}, v^{n}) \Big]
\\ &\leq \frac{\widehat{\alpha}^4}{k^3} \bE \Big[ \|D_u\sigma(u^{n}, v^{n-\frac 12}) v^{n}\|_{\bL^2}^4 \big| \widehat{\Delta_n W}\big|^4 \Big] + Ck\, \bE \Big[ \|v^{n + \frac 12}\|_{\bL^2}^4 \Big] + Ck \,\bE\Big[ \sum_{\ell=-1}^1\mathfrak{E}^2(u^{n+\ell}, v^{n+\ell}) \Big]
 \\ &\leq \widehat{\alpha}^4\,C_g^4\,k^3 \,\bE \Big[ \mathfrak{E}^2(u^{n}, v^{n}) \Big] +  Ck \,\bE\Big[ \sum_{\ell=-1}^1\mathfrak{E}^2(u^{n+\ell}, v^{n+\ell}) \Big]\,.
\end{split}
\end{equation*}
Now we insert the estimates from parts {\bf a)} and {\bf b)} into (\ref{sum-mult-ene}), and sum over $1 \leq n \leq N-1$. Then, for all $k \leq k_0 \equiv k_0(C_{\tt L}, C_g)$ there exists $C \equiv C(\beta)>0$ for $\beta\in (0, \frac{1}{2})$ such that the 
assertion \eqref{high-moment} for $p=1$ follows from the implicit version of the discrete Gronwall lemma.

\medskip

\noindent
{\bf 3) Proof of  \eqref{high-moment} for $p \geq 2$.} Starting from the identity (\ref{sum-mult-ene}), we multiply $\delta_1 {\mathfrak E}^{2^{p-1}}(u^{n+1},v^{n+1}) + \delta_2 \big[{\mathfrak E}^{2^{p-1}}(u^{n+1},v^{n+1}) + {\mathfrak E}^{2^{p-1}}(u^{n-1},v^{n}) \big]$ on both sides, and then take expectations. We may then follow the same argument as in {\bf 2)} to settle the assertion.
\medskip

\noindent
{\bf 4) Proof of \eqref{energy2}.}
Let $\widehat{\alpha} = 1$. Suppose $\sigma_2(v) \equiv 0 \equiv F_2(v)$ in {\bf (A3)} and $\beta=0$. We combine both equations in the scheme \eqref{scheme2:1}--\eqref{scheme2:2} to get
\begin{equation}\label{ul-comb}
\begin{split}
\bigl[u^{\ell+1} - u^\ell\bigr] - \bigl[u^\ell - u^{\ell-1}\bigr] &=  k^2 \Delta u^{\ell,1/2} + k \,\sigma(u^\ell)\Delta_\ell W  + \widehat{\alpha} k \, D_u\sigma(u^\ell)v^{\ell} \widehat{\Delta_{\ell} W}
\\ &\quad  +  \frac{k^2}{2}\, \bigl[ 3F(u^n) -F(u^{n-1})\bigr]
\end{split}
\end{equation}
for all $1 \leq \ell \leq N$. Now sum over the first $n$ steps, and define $\uli^{n+1} := \sum_{\ell=1}^n {u}^{\ell+1}$ to get 
\begin{equation}\label{sum_ul}
\begin{split}
\bigl[u^{n+1} - u^n\bigr]- k^2 \Delta \uli^{n,1/2} &= \bigl[ u^1 - u^0\bigr] 
+ k  \sum^n_{\ell=1} \sigma(u^\ell) \Delta_{\ell} W + \widehat{\alpha} k \sum^n_{\ell=1} D_u\sigma(u^\ell)v^{\ell} \widehat{\Delta_{\ell} W}
\\ &\quad +\frac{k^2}{2}  \sum^n_{\ell=1} \bigl[3F(u^\ell) - F(u^{\ell-1})\bigr] \, .
\end{split}
\end{equation}
Multiply both sides with $u^{n+1/2}$ and use integration by parts to get
{\small{
\begin{equation}\label{uL2-energy}
\begin{split}
&\frac 12 \Big[ \|u^{n+1}\|_{\bL^2}^2 - \|u^{n}\|_{\bL^2}^2 \Big] + k^2 \big( \nabla \uli^{n,1/2}, \nabla u^{n+1/2} \big) 
\\ &= \big( u^1 - u^0, u^{n+1/2} \big) + k  \Big(\sum^n_{\ell=1} \sigma(u^\ell) \Delta_{\ell} W, u^{n+1/2} \Big) + \widehat{\alpha} k \Big(\sum^n_{\ell=1}\ D_u\sigma(u^\ell)v^{\ell} \widehat{\Delta_{\ell} W}, u^{n+1/2} \Big) 
\\ &\quad+ \frac{k^2}{2}  \sum^n_{\ell=1} \Big( 3F(u^\ell) - F(u^{\ell-1}), u^{n+1/2} \Big) =: \mathfrak{K}^n_{1} + \mathfrak{K}^n_{2} + \mathfrak{K}^n_{3} + \mathfrak{K}^n_{4}\, .
\end{split}
\end{equation}
}}
We observe that the last term in the left-hand side may be written as
{\small{
\begin{align*}
k^2 \big( \nabla \uli^{n,1/2}, \nabla u^{n+1/2} \big) = \frac{k^2}{4} \Big( \nabla [ \uli^{n+1} + \uli^{n-1}], \nabla [ \uli^{n+1} - \uli^{n-1}] \Big) = \frac{k^2}{4} \Big[ \|\nabla \uli^{n+1}\|_{\bL^2}^2 - \|\nabla \uli^{n-1}\|_{\bL^2}^2 \Big]\,.
\end{align*}
}}
Taking expectation on both sides leads to
\begin{align}\label{uL2-dest}
\frac 12 \bE\Big[ \|u^{n+1}\|_{\bL^2}^2 - \|u^{n}\|_{\bL^2}^2 \Big] + \frac{k^2}{4} \bE \Big[ \|\nabla \uli^{n+1}\|_{\bL^2}^2 - \|\nabla \uli^{n-1}\|_{\bL^2}^2 \Big] = \sum_{j=1}^4 \bE\big[\mathfrak{K}^n_{j} \big]\, .
\end{align}
Since $u^1 - u_0 = k v^1$, by {\bf (B1)$_{i}$} we infer
\begin{align}\label{uL2-desta}
\bE\big[\mathfrak{K}^n_{1} \big] \leq \frac 1k \bE \big[ \|u^1 - u_0\|_{\bL^2}^2 \big] + C k \,\bE \big[ \|u^{n+1/2} \|_{\bL^2}^2 \big] \leq Ck + C k \,\bE \big[ \|u^{n+1/2} \|_{\bL^2}^2 \big] \, .
\end{align}
Using the It\^o isometry and {\bf (A3)} we infer
\begin{equation}\label{kln2-sigmau}
\begin{split}
\bE\big[\mathfrak{K}^n_{2} \big] &\leq k \sum^n_{\ell=1} \bE \Big[ \|\sigma(u^{\ell}) \|_{\bL^2}^2 |\Delta_{\ell} W|^2 \Big] + C k \,\bE \big[ \|u^{n+1/2} \|_{\bL^2}^2 \big] 
\\ &\leq C k^2 C_{\tt L}^2 \sum^n_{\ell=1}\, \bE \Big[ 1 + \|u^{\ell} \|_{\bL^2}^2 \Big] + C k \,\bE \big[ \|u^{n+1/2} \|_{\bL^2}^2 \big] \, .
\end{split}
\end{equation}
Using item {\bf 4.} of Remark \ref{rema2}, and {\bf (A4)} for $m=1$ we infer
\begin{equation}\label{kn3-hat}
\begin{split}
\bE\big[\mathfrak{K}^n_{3} \big] &\leq k \sum^n_{\ell=1} \bE \Big[ \|D_u \sigma(u^{\ell}) v^{\ell}\|_{\bL^2}^2 \big|\widehat{\Delta_{\ell} W} \big|^2 \Big] + C k \,\bE \big[ \|u^{n+1/2} \|_{\bL^2}^2 \big] 
\\ &\leq k^4 C_{g}^2 \sum^n_{\ell=1}\, \bE \big[ \|v^{\ell} \|_{\bL^2}^2 \big] + C k \,\bE \big[ \|u^{n+1/2} \|_{\bL^2}^2 \big] \, .
\end{split}
\end{equation}
Since $v^{\ell} = \frac 1k \big[ u^{\ell} - u^{\ell-1}\big]$, we further estimate \eqref{kn3-hat} by
\begin{align*}
 &\leq k^2 C_{g}^2 \sum^n_{\ell=1}\, \bE \big[ \|u^{\ell}\|_{\bL^2}^2 + \| u^{\ell-1} \|_{\bL^2}^2\big] + C k \,\bE \big[ \|u^{n+1/2} \|_{\bL^2}^2 \big]\,.
\end{align*}
Using {\bf (A3)} we estimate $\bE\big[\mathfrak{K}^n_{4} \big]$ by
\begin{equation}
\begin{split}
\bE\big[\mathfrak{K}^n_{4} \big] &\leq  C k^2 C_{\tt L}^2 \sum^n_{\ell=1} \bE \Big[ 1 +\|u^{\ell} \|_{\bL^2}^2 + \| u^{\ell-1} \|_{\bL^2}^2 \Big] + C k \,\bE \big[ \|u^{n+1/2} \|_{\bL^2}^2 \big] \, .
\end{split}
\end{equation}
We insert these estimates into
\eqref{uL2-dest} and sum over $1 \leq n \leq N-1$. Then, for all $k \leq k_0 \equiv k_0(C_{\tt L}, C_{g})$
and by the implicit version of the discrete Gronwall lemma, there exists a constant $C>0$ such that the assertion \eqref{energy2} holds.

\medskip

\noindent
{\bf 5) Proof of \eqref{energy2-himo} for $p=1$.} 
To simplify technicalities, we put $F \equiv 0$. Let us denote $\widetilde{\mathfrak{E}}(u^{n}, \uli^n) := \Big[ \frac 12 \|u^{n}\|^2_{\bL^2} + \frac{k^2}{4} \|\nabla \uli^{n}\|_{\bL^2}^2 \Big]$. Then we can rewrite \eqref{uL2-dest} as
\begin{align}\label{uL2-dest-E}
\widetilde{\mathfrak{E}}(u^{n+1}, \uli^{n+1}) - \widetilde{\mathfrak{E}}(u^{n}, \uli^{n-1})   = \mathfrak{K}^n_{1} + \mathfrak{K}^n_{2} + \mathfrak{K}^n_{3}\, .
\end{align}
Multiply both sides with $\widetilde{\mathfrak{E}}(u^{n+1}, \uli^{n+1})$, using binomial formula and taking expectation we obtain
\begin{equation}\label{E-tilde-u}
\begin{split}
&\frac 12 \bE\Big[\widetilde{\mathfrak{E}}^2(u^{n+1}, \uli^{n+1}) - \widetilde{\mathfrak{E}}^2(u^{n}, \uli^{n-1}) \Big] + \frac 12 \bE \Big[\big| \widetilde{\mathfrak{E}}(u^{n+1}, \uli^{n+1}) - \widetilde{\mathfrak{E}}^2(u^{n}, \uli^{n-1}) \big|^2 \Big]
\\ &= \bE \Big[ \mathfrak{K}^n_{1}\, \widetilde{\mathfrak{E}}(u^{n+1}, \uli^{n+1}) \Big] + \bE \Big[ \mathfrak{K}^n_{2}\, \widetilde{\mathfrak{E}}(u^{n+1}, \uli^{n+1}) \Big]  +\bE \Big[ \mathfrak{K}^n_{3} \,\widetilde{\mathfrak{E}}(u^{n+1}, \uli^{n+1}) \Big] \,.
\end{split}
\end{equation}
Using Young's inequality, and arguing similarly to (\ref{uL2-desta}) shows 
\begin{equation}
\begin{split}
\bE \Big[ \mathfrak{K}^n_{1}\, \widetilde{\mathfrak{E}}(u^{n+1}, \uli^{n+1}) \Big] 
 &\leq \frac{1}{k^3} \bE \Big[ \|u^1 - u_0\|_{\bL^2}^4 \Big] +  k^2 \bE \big[ \|u^{n+1/2} \|_{\bL^2}^4 \big] + k\,\bE \Big[ \widetilde{\mathfrak{E}}^2(u^{n+1}, \uli^{n+1}) \Big]
\\ &\leq Ck + Ck\, \bE \Big[ \widetilde{\mathfrak{E}}^2(u^{n+1}, \uli^{n+1}) + \widetilde{\mathfrak{E}}^2(u^{n}, \uli^{n}) \Big]\,.
\end{split}
\end{equation}
By adding and subtracting $\widetilde{\mathfrak{E}}(u^{n}, \uli^{n-1})$, and using {\bf (A3)}, we estimate the second term on the right-hand side of \eqref{E-tilde-u} by
\begin{equation}
\begin{split}
&\bE \Big[ \mathfrak{K}^n_{2}\, \big( \widetilde{\mathfrak{E}}(u^{n+1}, \uli^{n+1}) - \widetilde{\mathfrak{E}}(u^{n}, \uli^{n-1}) \big) \Big] + \bE \Big[ \mathfrak{K}^n_{2}\, \widetilde{\mathfrak{E}}(u^{n}, \uli^{n-1}) \Big] 
\\ &\leq \bE \big[ |\mathfrak{K}^n_{2}|^2 \big] + \frac 14 \,\bE \Big[ \big| \widetilde{\mathfrak{E}}^2(u^{n+1}, \uli^{n+1}) - \widetilde{\mathfrak{E}}(u^{n}, \uli^{n-1}) \big|^2 \Big] 
\\ &\quad+ C k^2 C_{\tt L}^2 \sum^n_{\ell=1}\, \bE \Big[ 1 + \|u^{\ell} \|_{\bL^2}^4 \Big] + C k\, \bE \big[\|u^{n+1/2} \|^4_{\bL^2} \big] +C k\, \bE \big[ \widetilde{\mathfrak{E}}^2(u^{n}, \uli^{n-1})  \big] 
\\ &\leq C k^2 C_{\tt L}^2 \sum^n_{\ell=1}\, \bE \Big[ 1 + \widetilde{\mathfrak{E}}^2(u^{\ell}, \uli^{\ell}) \Big] + C k\, \bE \big[ \widetilde{\mathfrak{E}}^2(u^{n}, \uli^{n-1}) \big] 
\\ &\quad+ \frac 14 \,\bE \Big[ \big| \widetilde{\mathfrak{E}}^2(u^{n+1}, \uli^{n+1}) - \widetilde{\mathfrak{E}}(u^{n}, \uli^{n-1}) \big|^2 \Big] \,,
\end{split}
\end{equation}
where the last term in the right-hand side may be absorbed on the left-hand side of \eqref{E-tilde-u}.

\noindent
By item {\bf 4.} of Remark \ref{rema2}, and {\bf (A4)} for $m=1$ we estimate
\begin{equation}
\begin{split}
&\bE \Big[ \mathfrak{K}^n_{3}\, \widetilde{\mathfrak{E}}(u^{n+1}, \uli^{n+1}) \Big] \leq \frac 1k \,\bE \big[ |\mathfrak{K}^n_{3}|^2 \big] + Ck\, \bE \Big[ \widetilde{\mathfrak{E}}^2(u^{n+1}, \uli^{n+1}) \Big]
\\ &\leq \tilde{C}_g^4 \sum^n_{\ell=1} \bE \Big[ \|v^{\ell}\|_{\bL^2}^4 \big|\widehat{\Delta_{\ell} W} \big|^4 \Big] + C k \,\bE \big[ \|u^{n+1/2} \|_{\bL^2}^4 \big] + Ck\, \bE \Big[ \widetilde{\mathfrak{E}}^2(u^{n+1}, \uli^{n+1}) \Big] \\
&\leq C_g^4 k^6 \frac{1}{k^4} \sum^n_{\ell=1} \bE \Big[ \|u^{\ell}\|_{\bL^2}^4 + \|u^{\ell-1}\|_{\bL^2}^4 \Big] + Ck\, \bE \Big[ \widetilde{\mathfrak{E}}^2(u^{n}, \uli^{n}) + \widetilde{\mathfrak{E}}^2(u^{n+1}, \uli^{n+1}) \Big]
\\ &\leq C k^2 \sum^n_{\ell=1} \bE \Big[ \widetilde{\mathfrak{E}}^2(u^{\ell}, \uli^{\ell}) + \widetilde{\mathfrak{E}}^2(u^{\ell-1}, \uli^{\ell-1}) \Big] + Ck\, \bE \Big[ \widetilde{\mathfrak{E}}^2(u^{n}, \uli^{n}) + \widetilde{\mathfrak{E}}^2(u^{n+1}, \uli^{n+1}) \Big]\,.
\end{split}
\end{equation}
%
%
Now we insert the estimates into (\ref{E-tilde-u}), and sum over $1 \leq n \leq N-1$. Then, for all $k \leq k_0 \equiv k_0(C_{\tt L}, C_g)$
assertion \eqref{energy2-himo} for $p=1$ follows from the implicit version of the discrete Gronwall lemma.

\medskip

\noindent
{\bf 6) Proof of \eqref{energy2-himo} for $p \geq 2$.} Starting from the identity (\ref{E-tilde-u}), we multiply ${\widetilde{\mathfrak E}}^{2^{p-1}}(u^{n+1},v^{n+1})$ in both sides, and then take the expectation. We may then follow the same argument as in {\bf 5)} to settle the assertion.

\end{proof}
}}

\medskip

\medskip

\section{Strong Rates of Convergence  for $(\widehat{\alpha}, \beta)-$scheme}\label{sec-4}

We prove convergence rate ${\mathcal O}(k^{1/2})$  for the iterates $\{ (u^n, v^n)\}_{n\geq1}$ of the $(\widehat{\alpha}, \beta)-$scheme for $\widehat{\alpha} \in \{0, 1\}$; if additionally $\sigma_2(v) \equiv 0 \equiv F_2(v)$ in {\bf (A3)} holds, we may put $\beta = 0$, and 
\begin{itemize}
\item[a)] the convergence rate improves to ${\mathcal O}(k)$ for iterates $\{ u^n\}_{n\geq 1}$ in case $\widehat{\alpha}=0$, and 
\item[b)] to ${\mathcal O}(k^{3/2})$ in case $\widehat{\alpha}=1$. 
\end{itemize}
For the convergence analysis, we need the following assumption on $(u^1, v^1)$.\\ 

\noindent
{\bf (B2)} Along with {\bf (A1)}$_{ii}$ and {\bf (B1)}$_{ii}$, let $u^0 = u(0)$ and $v^0 = v(0)$, and $(u^1, v^1)$ satisfy
\begin{align*}
\Big( \bE \Big[ \|u(t_1) - u^1\|_{\bH^1}^2 + \|v(t_1) - v^1\|_{\bL^2}^2 \Big] \Big)^{1/2} = \mathcal{O}\big(k^{1/2} \big).
\end{align*}

\smallskip

\begin{theorem}\label{lem:scheme1:con}
Let $(u,v)$ be the strong solution of  (\ref{stoch-wave1:1a}) with $A=-\Delta$. Let $\{ (u^n, v^n)\}_{n\geq1}$ be the iterates from $(\widehat{\alpha}, \beta)-$scheme for $k \leq k_0(C_{\tt L}, C_g)$ sufficiently small, $\widehat{\alpha} \in \{0, 1\}$, and {\red{$0 < \beta < \frac{1}{2}$}}. Then, under the hypotheses  {\bf (A1)$_{iii}$}, {\bf (A2), (A3)}, and {\bf (A4), (A5)}  for $m=1, 2$, and {\bf (B2)}, there exists {\red{$C \equiv C(\beta) >0$}} such that
\begin{align}\label{eq-5.32}
\max_{1 \le n \le N} \Big( \bE\Bigl[ \| u(t_n) - u^n\|^2_{\bH^1} + \|v(t_n) - v^n\|^2_{\bL^2} \Bigr] \Big)^{1/2} \le C k^{1/2}\,.
\end{align}
For the following, additionally suppose $\sigma_2(v) \equiv 0 \equiv F_2(v)$ in {\bf (A3)} and that the initial data $u^0, u^1, v^0$ satisfy
{{
\begin{equation}\label{ini-1} 
\begin{split}
&\bigg( {\mathbb E}\Bigl[ \Vert u(t_1) - u^1\Vert^2_{{\mathbb L}^2} \Bigr]  
\\ &\quad+ \frac{1}{k^2}\,  {\red{{\mathbb E} \Bigl[ \big\Vert k v_0 - (u^1 - u^0)  + \frac{k^2}{2}  \Delta u_0 + k^2 F(u_0) + k \sigma(u_0) \Delta_0 W \big\Vert_{{\mathbb L}^2}^2 \Bigr] }}  \bigg)^{1/2} = \mathcal{O}\big(k^{3/2} \big)\, .
\end{split}
\end{equation}
}}
\begin{itemize}
\item[$(i)$]
Consider the $(0, 0)-$scheme and assume {\bf (A1)$_{iii}$}, {\bf (A2), (A3)}, and {\bf (A4), (A5)}  for $m=1, 2$, and {\bf (B2)}. Then there exists $C >0$ such that
\begin{align}\label{pi}
\max_{1 \le n \le N} \Big( \bE\Bigl[ \| u(t_n, \cdot) - u^{n}\|^2_{\bL^2}\Bigr] \Big)^{1/2} + \frac 12 \bigg(\bE \bigg[ k \sum_{j=1}^n  \bigl\Vert \nabla  \big[ u(t_{j}, \cdot)- u^{j} \big] \bigr\Vert^2_{{\mathbb L}^2}  \bigg] \bigg)^{1/2} \le C k\, .
\end{align}
\end{itemize}
\begin{itemize}
\item[$(ii)$]
Consider the $(1, 0)-$scheme and assume {\bf (A1)$_{iv}$}, {\bf (A2), (A3)}, and {\bf (A4), (A5)}  for $m=1, 2, 3$, and {\bf (B2)}. Then, there exists $C>0$ such that
\begin{align}\label{pii}
\max_{1 \le n \le N} \Big( \bE\Bigl[ \| u(t_n, \cdot) - u^{n}\|^2_{\bL^2}\Bigr] \Big)^{1/2} + \frac 12 \bigg(\bE \bigg[ k \sum_{j=1}^n  \bigl\Vert \nabla  \big[ u(t_{j}, \cdot)- u^{j} \big] \bigr\Vert^2_{{\mathbb L}^2}  \bigg] \bigg)^{1/2} \le C k^{3/2}\, .
\end{align}
\end{itemize}
\end{theorem}
\smallskip

\noindent
The following remark discusses the realizability of (\ref{ini-1}), and key tools to verify this theorem.
\begin{remark}
{{
{\bf 1.}~{{In Section \ref{sec-6}, we choose $(u^0, v^0) = (u(0), v(0))$, together with 
\begin{align}\label{u1v1choice}
\begin{cases}
{\red{u^1 = u_0 +k\,v_0 + \frac{k^2}{2} \Delta u_0 + k^2 F(u_0) + (k+k^2) \,\sigma(u_0) \Delta_0 W}}\,,\\
v^1 = v_0 + k \,\sigma(u_0) \Delta_0 W\, .
\end{cases}
\end{align}
By the choice of $v^1$, the verification of {\bf (B2)} is straight-forward. We now prove that \eqref{ini-1} holds in this case: first, we consider  \eqref{stoch-wave1:1a} in integral form on $[0,t_1]$,
\begin{align}\label{sto-wave1:2a}
\begin{cases}
&u(t_1)  = u_0 + \int_0^{t_1} v(s)\,\dd s  \\
&v(s)  = v_0+\int_0^{s} \Delta u(\tau)  \dd \tau + \int_0^{s} F(u(\tau))  \dd \tau + \int_0^{s} \sigma(u(\tau)) \dd W(\tau)\,, \quad 0 \leq \tau \leq s\, ,
\end{cases}
\end{align}
and insert (\ref{sto-wave1:2a})$_2$ into (\ref{sto-wave1:2a})$_1$;  a change of order of integration then gives
{\small{
\begin{equation}\label{stoch-wave1:3a}
\begin{split}
u(t_1) &= u_0 + t_1 v_0 +  \int_0^{t_1} \int_0^s \Delta u(\tau)\, \dd \tau \dd s +  \int_0^{t_1} \int_0^s F(u(\tau))\, \dd \tau \dd s + \int_0^{t_1} \int_0^s \sigma(u(\tau)) \dd W(\tau) \, \dd s
\\ & = u_0 + k v_0 + \int_0^{t_1} \int_{\tau}^{t_1} \dd s \,\Delta u(\tau)\, \dd \tau +  \int_0^{t_1} \int_{\tau}^{t_1} \dd s \,F(u(\tau)) \dd \tau + \int_0^{t_1} \int_{\tau}^{t_1}  \dd s \,\sigma(u(\tau)) \dd W(\tau)\, .
\end{split}
\end{equation}
}}
\smallskip

\noindent
%
%
Thus, 
\begin{equation}\label{u(t_1)}
\begin{split}
u(t_1) &= u_0 + k v_0 +  \int_0^{t_1} (t_1 - \tau) \Delta u(\tau)\,\dd \tau +  \int_0^{t_1} (t_1 - \tau) F(u(\tau))\,\dd \tau 
\\ &\quad+ \int_0^{t_1} (t_1 - \tau) \sigma(u(\tau)) \dd W(\tau)\, .
\end{split}
\end{equation}
\noindent
Subtracting \eqref{u1v1choice}$_{1}$ from \eqref{u(t_1)} we infer
\begin{align*}
u(t_1) - u^1 &= \int_0^{t_1} (t_1 - \tau) \Delta u(\tau)\,\dd \tau - \frac{k^2}{2} \Delta u_0 - k^2 F(u_0) +  \int_0^{t_1} (t_1 - \tau) F(u(\tau))\,\dd \tau 
\\ &\quad+ \int_0^{t_1} (t_1 - \tau) \sigma(u(\tau)) \dd W(\tau) - (k+k^2) \sigma(u_0) \big(W(t_1) - W(0)\big)\,.
\end{align*}
By {\bf (A1)$_{iii}$, (A3)}, It\^o isometry, Lemma \ref{lem:L2} $(i), (ii)$ and Lemma \ref{lem:Holder} $(i)$ we infer
\begin{align}
\bE \Big[ \|u(t_1) - u^1\|_{\bL^2}^2 \Big] &\leq C k^2\, \int_0^{t_1} \bE \Big[ \| \Delta u(\tau)\|_{\bL^2}^2 \Big] \,{\rm d} \tau + C k^4 \,\bE \big[ \|u_0 \|_{\bH^2}^2 \big] + C k^4 \,\bE \big[ \|F(u_0) \|_{\bL^2}^2 \big] \nonumber
\\ &\quad+ C k^2\, \int_0^{t_1} \bE \Big[ \| F(u(\tau))\|_{\bL^2}^2 \Big] \,{\rm d} \tau   + k^3\, \bE \big[ \| \sigma(u_0)\|_{\bL^2}^2 \big] \label{eq_1.7}
\\ &\quad + C k^2\, \int_0^{t_1} \bE \Big[ \| \sigma(u(\tau)) - \sigma(u_0) \|_{\bL^2}^2 \Big]\, \dd s \leq C k^3 \, .\nonumber
\end{align}
Similarly, by \eqref{u1v1choice}$_{1}$, {\bf (A1)$_{i}$} and It\^o isometry we get
\begin{equation}\label{eq_1.8}
\begin{split}
&{\mathbb E} \Bigl[ \big\Vert k v_0 - (u^1 - u^0)+ \frac{k^2}{2} \Delta u_0 +k^2 F(u_0)+ k \,\sigma(u_0) \Delta_0 W \big\Vert_{{\mathbb L}^2}^2 \Bigr] 
\\ &\leq  k^4\, \bE \Big[ \|\sigma(u_0) \|_{\bL^2}^2 |\Delta_0 W|^2 \Big] \leq C k^5 \, .
\end{split}
\end{equation}
Thus, combining \eqref{eq_1.7} and \eqref{eq_1.8} we get the assertion \eqref{ini-1} for $u^1$.
\del{
To validate the choice of $v^1$ in \eqref{u1v1choice}$_{2}$, we use {\bf (A1)$_{i}$} and It\^o isometry to get
\begin{align*}
&{\mathbb E} \Bigl[ \big\Vert k v_0 - (u^1 - u^0) + \frac{k^2}{2} \Delta u_0 + k^2 F(u_0) + k \,\sigma(u_0) \Delta_0 W \big\Vert_{{\mathbb L}^2}^2 \Bigr]  
\\ &= {\mathbb E} \Bigl[\Vert k v_0 - k v^1 \Vert_{{\mathbb L}^2}^2 \Bigr] = k^2 {\mathbb E} \Bigl[\Vert v^1 - v_0 \Vert_{{\mathbb L}^2}^2 \Bigr] \leq  k^4\, \bE \Big[ \|\sigma(u_0) \|_{\bL^2}^2 |\Delta_0 W|^2 \Big] \leq C k^5 \, ,
\end{align*}
which settles the assertion \eqref{ini-1}.
}
}}
}}
\smallskip

\noindent
{\bf 2.}~For $\widetilde{\alpha} \neq 0$, the additional noise term in \eqref{scheme2:2} improves the accuracy of the $(\widehat{\alpha}, 0)-$scheme, where $\widetilde{\Delta_n W}$ is approximated by $\widehat{\Delta_n W}\,.$ By \eqref{W-hat}, \eqref{DnW-tilde}, and the fact that $t_{n,\ell+1} - t_{n,\ell}=k^2$, we estimate the distance between $\widetilde{\Delta_n W}$ and $\widehat{\Delta_n W}$ as
\begin{equation*}
\begin{split}
{\mathbb E}\Bigl[ \big\vert \widetilde{\Delta_n W} - \widehat{\Delta_n W} \big\vert^2 \Bigr] &= {\mathbb E}\biggl[ \Big| - \int_{t_n}^{t_{n+1}} W(s)\,{\rm d}s + k^2 \sum_{\ell=1}^{k^{-1}} W(t_{n,\ell}) \Big|^2 \biggr]
\\ &=  {\mathbb E}\biggl[ \Big| \sum_{\ell=1}^{k^{-1}} \int_{t_{n, \ell}}^{t_{n, \ell+1}} \big( W(s)- W(t_{n,\ell}) \big)\, {\rm d}s\Big|^2 \biggr]\,.
\end{split}
\end{equation*}
By the independence property of the increment $\Delta_n W$, we further estimate 
\begin{equation}\label{dist-tilde-hat}
\begin{split}
&\leq k \sum_{\ell=1}^{k^{-1}}  \int_{t_{n, \ell}}^{t_{n, \ell+1}} {\mathbb E}\Bigl[ \big| W(s)- W(t_{n,\ell}) \big|^2 \Bigr] \,{\rm d}s
\leq k \sum_{\ell=1}^{k^{-1}}  \int_{t_{n, \ell}}^{t_{n, \ell+1}} (s-t_{n,\ell}) \,{\rm d}s \leq C k^4\,.
\end{split}
\end{equation}
\smallskip

\noindent
{\bf 3.}~The basic estimate is \eqref{eq-5.32}, which will be given in part {\bf 1)} in the proof below. Its derivation uses the H\"older estimates in Lemma \ref{lem:Holder} for $(u,v)$ in {\em strong} norms. The strategy of proof 
is similar to the one used in the stability analysis for $(\widehat{\alpha}, \beta)-$scheme in
Section \ref{sec-3}; see item {\bf 1.}~in Remark \ref{rema2}: the central term to estimate is $T_4^{(n)}$ in \eqref{e1_1:9}, in which we replace the
increments $e^{n+1}_v-e^{n}_v$ via the {\em error equation \eqref{e1_1:4}} to obtain
terms which are scaled by $k$, or the stochastic increments ${\Delta_n W}$ and $\widetilde{\Delta_n W}$. The order limiting term then is $T_{4, 1}^{(n,4)}$ in \eqref{limiti}, which may be traced back to the noise term $\sigma$, which may depend on $v$ as well.
In this case (only), the {\em additional term $-k^{2+\beta} \Delta v^{n+1/2}$} in Scheme \ref{scheme--2} is needed to control the effect of noise: see the additional term on the left-hand side of \eqref{e1_1:6} to {\em e.g.}~bound the corresponding term in \eqref{er-T4,1}.

The verification of assertions \eqref{pi} and \eqref{pii} differs completely from this strategy: it starts
with the reformulation \eqref{scheme2:3} that leads to the error identity \eqref{error-eq}, which then is tested with $e_u^{n+1/2}$; the noise part may here be estimated in a straight manner.

\smallskip

\noindent
{\bf 4.}~Part {\bf 2)} in the proof below is conceptually motivated from arguments in \cite{Bak_76}; however, their realization in the stochastic setting differs considerably. We remark that estimate \eqref{eq-5.32} is needed  to verify assertion \eqref{pi} --- next to Lemma \ref{quadrature1} to bound the quadrature error of the trapezoidal rule for integrands with limited regularity; see term $I^{\ell, n}_7$ in \eqref{e1_2:5}.

\smallskip

\noindent
{\bf 5.}~If $\widetilde{\alpha} = 0$, the estimate \eqref{restri} for term $I_3^{\ell,n}$ in \eqref{e1_2:5} restricts the order, and assertion \eqref{pi} follows; the improvement \eqref{pii} uses $\widetilde{\alpha} = 1$, {\em s.t.} this term $I_3^{\ell,n}$ gives way to the sum of new terms  in \eqref{e2_2:3}, which are of higher order; see {\bf a)-- e)} in part {\bf 3)} in the proof below.

\smallskip

\noindent
{\bf 6.}~For $\sigma \equiv \sigma(u, v)$ {\em or} $F \equiv F(u, v)$, neither assertion \eqref{pi} nor \eqref{pii} in Theorem \ref{lem:scheme1:con} may be concluded, due to the restricted H\"older regularity properties of $v$  opposed to $u$.

In this setting, either $\sigma$ {\em or} $F$ in \eqref{scheme2:3}  in the proof below would depend on $v$ as well, and thus would modify corresponding terms in \eqref{error-eq}. For $\sigma \equiv \sigma(u,v)$, (a modified version of) {\bf (A3)} would additionally create a term $Ck^2 \sum_{\ell=1}^n {\mathbb E}\bigl[ \Vert e^{\ell}_v\Vert^2_{{\mathbb L}^2}\bigr]$ on the right-hand side of \eqref{dazu1}, which may not be handled  via Gronwall's lemma to lift the order. For $F \equiv F(u,v)$, the argument in \eqref{dazu2} fails, which rests on Lemma~\ref{quadrature1}, and the H\"older continuity of $v = \partial_t u$.

\smallskip

\noindent
{\bf 7.}~In the proof of \eqref{eq-5.32}, where $\sigma \equiv \sigma(u, v)$ and $F \equiv F(u, v)$, we do not require the discrete energy bounds proved in Lemma \ref{lem:scheme1:stab}. We only require the energy bounds proved in Lemma \ref{lem:L2}. This is possible, since we can add and subtract $\nabla u(t_n)$ or $v(t_n)$ whenever $L^2$-norm of $\nabla u^n$ or $v^n$ appears. However, the higher moment bounds in energy norm (proved in Lemma \ref{lem:scheme1:stab})  are required to show the improved convergence order $\cO(k^{3/2})$ in the proof of \eqref{pii} of Theorem \ref{lem:scheme1:con}. 
\end{remark}

\begin{proof}[Proof of Theorem \ref{lem:scheme1:con}]

{\bf 1) Proof of \eqref{eq-5.32}}.
 For simplicity, we here give the proof for $F \equiv 0$. Correspondingly, let $(u,v)$ solve \eqref{stoch-wave1:1a}, and 
$\{(u^n, v^n) \}_{n\in \mathbb{N}}$ solves $(\widehat{\alpha}, \beta)-$scheme. We denote by $e_{u}^{n}:= u(t_n)-u^n$  and $e_v^n:= v(t_n)-v^n$ error iterates, which are zero on the boundary and solve
\del{
By \eqref{strong1}--\eqref{strong2}, for all $t_n \in [0,T], n \ge 0$, by regularity theory we have
\begin{eqnarray}
\label{e1_1:1}
\bigl( \nabla [u(t_{n+1}) - u(t_n)], \phi\bigr) 
&=& \Bigl(\int_{t_n}^{t_{n+1}} \nabla [v - v(t_{n+1}) + v(t_{n+1})]\,\ds, \phi\Bigr)\, , \\ 
\nonumber
\bigl(v(t_{n+1})-v(t_n), \psi\bigr) 
&=& -\int_{t_n}^{t_{n+1}} (\nabla u, \nabla \psi ) \, \ds  
+ \int_{t_n}^{t_{n+1}} \bigl( F(u,v), \psi\bigr) \, \ds \\ \label{e1_1:2}
&& +\Bigl(\int_{t_n}^{t_{n+1}} \sigma (u, v) \, \dW(s),\psi\Bigr)\, .
\end{eqnarray}
Subtracting \eqref{scheme2:1}--\eqref{scheme2:2} from \eqref{e1_1:1}--\eqref{e1_1:2} and 
}
%
\begin{eqnarray}
\label{e1_1:3}
\quad &&e^{n+1}_u-e^n_u 
= k \,e^{n+1}_v + \int_{t_n}^{t_{n+1}}  \bigl(v(s)-v(t_{n+1})\bigr)\, \ds\, ,\\
\nonumber
\quad &&e^{n+1}_v - e^n_v  
= k \,\Delta e^{n,1/2}_u + \int_{t_n}^{t_{n+1}}  \Delta \bigg[ \frac{ 2 u(s) - [u(t_{n+1})+u(t_{n-1})]}{2} \bigg]\, \ds \\
\nonumber
\quad &&\qquad \qquad \quad - \beta k^{2+\beta}\,\Delta e_v^{n+\frac 12} + \beta \frac{k^{2+\beta}}{2}\,\Delta \bigl[v(t_{n+1}) + v(t_n)\bigr]\\ \label{e1_1:4}
\quad && \qquad \qquad \quad + \int_{t_n}^{t_{n+1}} \bigl[\sigma\bigl(u(s), v(s)\bigr) - \sigma\bigl(u^n, v^{n-\frac 12} \bigr) \bigr] \, \dW(s) - \widehat{\alpha}\,D_u\sigma(u^n, v^{n-\frac 12})\,v^n\, \widehat{\Delta_n W}\,.
\end{eqnarray}
We multiply \eqref{e1_1:4} with $e^{n+\frac 12}_v$ and use \eqref{e1_1:3} to get
\begin{equation}\label{e1_1:6}
\begin{split}
\frac 12 \Big[ \|e^{n+1}_v \|_{\bL^2}^2 - \|e^{n}_v \|_{\bL^2}^2 \Big]  + \frac 14 \Big[ \|\nabla e^{n+1}_u \|_{\bL^2}^2 - \|\nabla e^{n-1}_u \|_{\bL^2}^2 \Big] 
+ \beta k^{2+\beta} \Vert \nabla e^{n+\frac 12}_v\Vert^2_{{\mathbb L}^2}\leq \sum_{j=1}^5 T_j^{(n)}\, ,
\end{split}
\end{equation}
where
\begin{align*}
T_1^{(n)} &:= \int_{t_n}^{t_{n+1}} \Bigl(\nabla \bigl[v(s) -v(t_{n+1})\bigr], \nabla e^{n,\frac 12}_u \Bigr) \,\ds + \int_{t_{n-1}}^{t_{n}} \Bigl(\nabla \bigl[v(s) -v(t_{n})\bigr], \nabla e^{n, \frac 12}_u \Bigr)\,\ds\, ,
\\ T_2^{(n)} &:=  -\int_{t_n}^{t_{n+1}} \Bigl(\nabla \Big[ \frac{2 u(s) - [u(t_{n+1})+u(t_{n-1})]}{2} \Big], \nabla e^{n+ \frac 12}_v\Bigr)\, \ds \, ,
\\  T_3^{(n)} &:= \beta \frac{k^{2+\beta}}{2} \Big(\nabla \bigl[v(t_{n+1}) + v(t_{n})\bigr], \nabla e_v^{n+ \frac 12} \Big)\, ,
\\ T_4^{(n)} &:= \Bigl(\int_{t_n}^{t_{n+1}} \Bigl[\sigma\bigl(u(s),v(s)\bigr) - \sigma(u^n, v^{n-\frac 12})
\Bigr] \, \dW(s), e^{n+\frac 12}_v\Bigr) \, ,
\\  T_5^{(n)} &:=  -\,\widehat{\alpha}\,\Bigl(D_u\sigma(u^n, v^{n-\frac 12})v^n \,\widehat{\Delta_n W}, e^{n+\frac 12}_v \Bigr)\, .
\end{align*}

\noindent
We estimate the expectation of each term on the right-hand side of \eqref{e1_1:6}. 
By Lemma~\ref{lem:Holder} $(iii)$,  we infer
\begin{align*}\label{e1_1:7}
\bE \bigl[T_1^{(n)} + T_2^{(n)} \bigr] 
&\le Ck^2+ C k \, \bE \Bigl[\| \nabla e^{n+1}_u\|^2_{\bL^2} 
+ \| \nabla e^{n-1}_u\|^2_{\bL^2} \Bigr] + C k \, \bE \Bigl[ \|e^{n+1}_v\|^2_{\bL^2} +  \|e^{n}_v\|^2_{\bL^2}\Bigr]\, .
\end{align*}
We use Lemma \ref{lem:L2} $(ii)$  to estimate
\begin{align*}
\bE \bigl[T_3^{(n)} \bigr] &\leq \beta \frac{k^{2+\beta}}{6}  {\mathbb E}\bigl[\Vert \nabla e^{n+\frac 12}_v\Vert^2_{{\mathbb L}^2}\bigr] + 
\frac{C}{\beta} k^{2+\beta}\, .
\end{align*}

\noindent
By properties of $\Delta_n W$ we rewrite the term $T_4^{(n)}$ as
\begin{equation}\label{e1_1:9}
\begin{split}
\bE \bigl[T_4^{(n)} \bigr] &= \frac 12 \,\bE \Bigl[ \Bigl( \bigl[ \sigma\bigl(u(t_n), v(t_{n-1/2})\bigr) - \sigma(u^n, v^{n-1/2})\bigr]\Delta_{n}W, e^{n+1}_v - e^{n}_v\Bigr)\Bigr] 
\\ &\quad + \frac 12 \,\bE \Bigl[ \Bigl(\int_{t_n}^{t_{n+1}} \bigl[\sigma \bigl(u(s),v(s) \bigr) 
- \sigma\bigl(u(t_n),v(t_{n-1/2})\bigr)\bigr] \dW(s),e^{n+1}_v-e^{n}_v\Bigr)\Bigr] 
\\ &:=  T_{4, 1}^{(n)} + T_{4, 2}^{(n)}\, .
\end{split}
\end{equation}
In order to estimate $T_{4, 1}^{(n)}$, we use equation \eqref{e1_1:4}  to write
{\small{
\begin{equation*} 
\begin{split}
T_{4, 1}^{(n)} &= \frac 12 \,\bE \Bigl[ \Bigl( \bigl[ \sigma\bigl(u(t_n), v(t_{n-1/2})\bigr) - \sigma(u^n, v^{n-1/2})\bigr]\Delta_{n}W, k \Delta e_u^{n, 1/2}\Bigr)\Bigr] 
\\ \nonumber &\quad + \frac 12 \,\bE \biggl[ \biggl( \bigl[ \sigma\bigl(u(t_n), v(t_{n-1/2})\bigr) - \sigma(u^n, v^{n-1/2})\bigr]\Delta_{n}W,  \int_{t_n}^{t_{n+1}} \Delta \Big[  \frac{2u - [u(t_{n+1})+u(t_{n-1})]}{2} \Big] \, \ds \biggr)\biggr]
\\ \nonumber &\quad + \frac 12 \,\bE \Bigl[ \Bigl( \bigl[ \sigma\bigl(u(t_n), v(t_{n-1/2})\bigr) - \sigma(u^n, v^{n-1/2})\bigr]\Delta_{n}W, - k^{2+\beta}\,\Delta v^{n+\frac 12}\Bigr)\Bigr] 
\\  &\quad + \frac 12 \,\bE \biggl[ \biggl( \bigl[ \sigma\bigl(u(t_n), v(t_{n-1/2})\bigr) - \sigma(u^n, v^{n-1/2})\bigr]\Delta_{n}W, \int_{t_n}^{t_{n+1}} [\sigma(u,v) - \sigma(u^n, v^{n-\frac 12})] \, \dW(s)\biggr)\biggr]
\\ \nonumber &\quad + \frac{\widehat{\alpha}}{2} \,\bE \biggl[ \biggl( \bigl[ \sigma\bigl(u(t_n), v(t_{n-1/2})\bigr) - \sigma(u^n, v^{n-1/2})\bigr]\Delta_{n}W, D_u\sigma(u^n, v^{n-\frac 12})\,v^n\, \widehat{\Delta_n W}\biggr]
\\ \nonumber &:= T_{4, 1}^{(n,1)} + T_{4, 1}^{(n, 2)} + T_{4, 1}^{(n, 3)} + T_{4, 1}^{(n, 4)} + T_{4, 1}^{(n, 5)}\, .
\end{split}
\end{equation*}
}}
We consider $T_{4, 1}^{(n,1)}$ first; to properly address the dependence of $\sigma$ on $v$,  we first restate it with the help of (\ref{e1_1:3}) and use the fact that $\sigma\bigl(u(t_n), v(t_{n-1/2})\bigr)=\sigma(u^n, v^{n-1/2})=0$ on $\partial \mathcal{O}$ to obtain
{\small{
\begin{eqnarray*}T_{4, 1}^{(n,1)} &=&  \frac 12 \,\bE \Bigl[ \Bigl( \bigl[ \sigma\bigl(u(t_n), v(t_{n-1/2})\bigr) - \sigma(u^n, v^{n-1/2})\bigr]\Delta_{n}W, k \Delta [e_u^{n+1} - e^{n-1}_u]\Bigr)\Bigr] \\
&=&-\frac 12 \,\bE \Bigl[ \Bigl( \nabla \bigl[ \sigma\bigl(u(t_n), v(t_{n-1/2})\bigr) - \sigma(u^n, v^{n-1/2})\bigr]\Delta_{n}W, 2k^2 \nabla e_v^{n+1/2} \Bigr)\Bigr]
\\ && -\frac 12 \,\bE \Bigl[ \Bigl( \nabla \bigl[ \sigma\bigl(u(t_n), v(t_{n-1/2})\bigr) - \sigma(u^n, v^{n-1/2})\bigr]\Delta_{n}W, k\nabla {\mathcal R}^{n+1/2}_v\Bigr)\Bigr] =: T_{4, 1, {\tt A}}^{(n,1)} + T_{4, 1, {\tt B}}^{(n,1)} \,,
\end{eqnarray*}
}}
where ${\mathcal R}^{n+1/2}_v := \int_{t_n}^{t_{n+1}}  \bigl(v(s)-v(t_{n+1})\bigr)\, \ds + \int_{t_{n-1}}^{t_{n}}  \bigl(v(s)-v(t_{n})\bigr)\, \ds$.
By chain rule, and {\bf (A4)} for $m=1$ we obtain 
\begin{equation*}
\begin{split} 
T_{4, 1, {\tt A}}^{(n,1)} &\leq - \bE \Big[ C_g \big\{ 2\|\nabla u(t_n)\|_{\bL^2} + 2\|\nabla v(t_{n-1/2})\|_{\bL^2} + \|\nabla e_u^n\|_{\bL^2} + \|\nabla e_v^{n-1/2}\|_{\bL^2} \big\} |\Delta_n W| 
\\ &\qquad \quad \cdot k^{1-\frac{\beta}{2}} k^{1+\frac{\beta}{2}} \,\|\nabla e_v^{n+1/2}\|_{\bL^2} \Big]\, .
\end{split}
\end{equation*}
%
We apply Young's inequality, Ito isometry, \eqref{energy1} and Lemma \ref{lem:L2} $(i), (ii)$ to further bound $T_{4, 1, {\tt A}}^{(n,1)}$ by
\begin{equation}\label{er-T4,1}
\begin{split} 
T_{4, 1, {\tt A}}^{(n,1)} &\leq \frac{C_g^2}{\beta}\,k^{2- \beta} \,\bE \Big[ \big\{ 2 \|\nabla u(t_n)\|^2_{\bL^2} +2 \|\nabla v(t_{n-1/2})\|^2_{\bL^2} + \|\nabla e_u^n\|^2_{\bL^2} 
\\ &\qquad \qquad \qquad+ \|\nabla e_v^{n-1/2}\|^2_{\bL^2} \big\} |\Delta_n W|^2 \Big] + \frac{\beta}{6} k^{2+\beta} \,\bE \Big[ \|\nabla e_v^{n+1/2}\|^2_{\bL^2} \Big]
\\ &\leq C_{\beta} k^{3- \beta} + C_{\beta} k^{3- \beta} {\mathbb E}[\|\nabla e_u^n\|^2_{\bL^2}] + C_g^2 C_{\beta}\,k^{3- \beta} \,\bE \Big[ \|\nabla e_v^{n-1/2}\|^2_{\bL^2} \Big] 
\\ &\qquad+ \frac{\beta}{6} k^{2+\beta} \,\bE \Big[ \|\nabla e_v^{n+1/2}\|^2_{\bL^2} \Big]\, ,
\end{split}
\end{equation}
where $C_{\beta} \equiv C(\beta)>0$ is a constant for $\beta\in(0, \frac{1}{2})$ and the last two terms on the right-hand side may be absorbed on the left-hand side of (\ref{e1_1:6}) for $k \leq k_0$ sufficiently small, and $\beta < \frac{1}{2}$. Arguing similarly and by Lemma \ref{lem:Holder} $(iii)$ we infer
\begin{equation}\label{er-T4,1,B}
\begin{split} 
T_{4, 1, {\tt B}}^{(n,1)} &\leq C k^{3- \beta} +C_g^2\,k^{3- \beta} \,\bE \big[ \|\nabla e_v^{n-1/2}\|^2_{\bL^2} \big] + C k^{2+\beta}\, ,
\end{split}
\end{equation}
where the second term on right-hand side may be absorbed on the left-hand side of (\ref{e1_1:6}) for $k \leq k_0$ sufficiently small, and $\beta < \frac{1}{2}$.

\noindent
We now estimate $T_{4, 1}^{(n,2)}$: by properties of $\Delta_n W$,  {\bf (A3)} and Lemma \ref{lem:Holder} $(iii)$ we get
\begin{equation*}
\begin{split} 
T_{4, 1}^{(n,2)} &\leq C_{\tt L} k\, \bE \Big[ \big( \|\nabla u(t_{n}) - \nabla u^n\|_{\bL^2}^2 + \|v(t_{n-1/2}) - v^{n-1/2}\|_{\bL^2}^2 \big) |\Delta_n W|^2 \Big] 
\\ &\quad+ C \int_{t_n}^{t_{n+1}} \bE \Big[ \Big\| \Delta \Big[ u(s) -\frac{u(t_{n+1})+u(t_{n-1})}{2} \Big] \Big\|_{\bL^2}^2 \Big]
\\ &\leq Ck\,\bE \Big[ \|\nabla e_u^{n}\|_{\bL^2}^2 + \|e_v^{n}\|_{\bL^2}^2 + \|e_v^{n-1}\|_{\bL^2}^2 \Big] +C k^3\, .
\end{split}
\end{equation*}
Using similar arguments as for the estimate of \eqref{er-T4,1} we infer
\begin{equation*}
\begin{split} 
T_{4, 1}^{(n,3)} &\leq \frac{C_g^2}{\beta} k^{3+\beta} \bE \Big[ \|\nabla u(t_{n})\|_{\bL^2}^2 + 2 \|\nabla v(t_{n-1/2})\|_{\bL^2}^2 + \|\nabla u^{n}\|_{\bL^2}^2 + \|\nabla e_v^{n-1/2}\|_{\bL^2}^2 \Big] 
\\ &\quad + \frac{\beta}{6} k^{2+\beta}\, \bE \Big[ \|\nabla e_v^{n+1/2}\|_{\bL^2}^2 + \|\nabla v(t_{n+1/2})\|_{\bL^2}^2 \Big] 
\\ &\leq \frac{C_g^2}{\beta}\,k^{3+ \beta} \,\bE \Big[ \|\nabla e_v^{n-1/2}\|^2_{\bL^2} \Big] + \frac{\beta}{6} k^{2+\beta} \,\bE \Big[ \|\nabla e_v^{n+1/2}\|^2_{\bL^2} \Big] + C_{\beta} k^{2+\beta}\, .
\end{split}
\end{equation*}

\noindent
Using {\bf (A3)}, Lemma \ref{lem:Holder} $(ii)$, and properties of $\Delta_n W$, we estimate
\begin{equation}
\begin{split} 
T_{4, 1}^{(n,4)} &\leq C\, \bE \Big[ \| \sigma(u(t_n), v(t_{n-1/2})) - \sigma(u^n, v^{n-1/2})\|_{\bL^2}^2\,|\Delta_n W|^2 \Big]
\\ &\quad+ C \int_{t_n}^{t_{n+1}} \bE \Big[ \| \sigma\bigl(u(s), v(s)\bigr) - \sigma\bigl(u(t_n), v(t_{n-1/2})\bigr)\|_{\bL^2}^2 \Big]\,\rm{d}s
\\ \label{limiti}&\leq Ck\,\bE \Big[ \|\nabla e_u^{n}\|_{\bL^2}^2 + \|e_v^{n}\|_{\bL^2}^2 + \|e_v^{n-1}\|_{\bL^2}^2 \Big] 
\\ &\quad+ C_{\tt L} \int_{t_n}^{t_{n+1}} \bE \Big[ \|\nabla [u-u(t_n)]\|_{\bL^2}^2  + \|v-v(t_{n-1/2})]\|_{\bL^2}^2 \Big] 
\\ &\leq Ck\,\bE \Big[ \|\nabla e_u^{n}\|_{\bL^2}^2 + \|e_v^{n}\|_{\bL^2}^2 + \|e_v^{n-1}\|_{\bL^2}^2 \Big]  + C k^2\, .
\end{split}
\end{equation}
Using {\bf (A3)}, {\bf (A4)} for $m=1$, item {\bf 4.}~of Remark \ref{rema2}, and using Lemma \ref{lem:L2} $(i)$ (due to addition and subtraction of $v(t_n)$ term to $v^n$) we estimate
\begin{equation*}
\begin{split} 
T_{4, 1}^{(n,5)} &\leq k\, \bE \Big[ \|\sigma\bigl(u(t_n), v(t_{n-1/2})\bigr) - \sigma(u^n, v^{n-1/2}) \|_{\bL^2}^2 \Big] +\frac{\widehat{\alpha}^2}{4} k^3\,\bE \Big[ \|D_u\sigma(u^n, v^{n-\frac 12})\,v^n \|_{\bL^2}^2 \Big]
\\ &\leq C_{\tt L} k\,\bE \Big[ \|\nabla e_u^{n}\|_{\bL^2}^2 + \|e_v^{n- 1/2}\|_{\bL^2}^2 \Big] +\frac{\widehat{\alpha}^2}{4} C_g^2\, k^3\, \bE\big[ \|e_v^{n}\|_{\bL^2}^2 \big] + C k^3\, .
\end{split}
\end{equation*}
Similar arguments, in combination with the H\"older estimates in Section \ref{Hoe-cont} may be used to estimate $T_{4, 2}^{(n)}$  in \eqref{e1_1:9}. Now, we estimate the last term in the right-hand side of \eqref{e1_1:6}.

\noindent
Using {\bf (A4)} for $m=1$, It\^o isometry, and Lemma \ref{lem:L2} $(i)$ (due to addition and subtraction of $v(t_n)$ term to $v^n$),  we obtain
\begin{equation*}
\begin{split}
\bE\big[ T_5^{(n)} \big] &\leq \frac{\widehat{\alpha}^2}{k} \bE \Big[ \|D_u\sigma(u^n, v^{n-\frac 12}) v^n \|_{\bL^2}^2 \big|\widehat{\Delta_n W} \big|^2 \Big] + C k\,\bE \big[ \|e^{n+1}_v \|_{\bL^2}^2+\|e^{n}_v \|_{\bL^2}^2 \big] 
\\ &\leq \widehat{\alpha}^2\,C_g^2\,k^2\, \bE \big[ \|v^n \|_{\bL^2}^2 \big] + C k\,\bE \big[ \|e^{n+1}_v \|_{\bL^2}^2+\|e^{n}_v \|_{\bL^2}^2 \big] \leq C k^2+ C k\,\bE \big[ \|e^{n+1}_v \|_{\bL^2}^2+\|e^{n}_v \|_{\bL^2}^2 \big] \, .
\end{split}
\end{equation*}
We now insert these estimates into (\ref{e1_1:6}), for which we apply expectations, and sum over iteration steps. The implicit version of the discrete Gronwall lemma along with {\bf (B2)}
then yields the assertion, again provided $k \leq k_0$ is sufficiently small and $\beta \in (0, 1/2)$.

\medskip

\noindent
{\bf 2) Proof of \eqref{pi}.}
Suppose $\sigma_2(v) \equiv 0 \equiv F_2(v)$ in {\bf (A3)}, and $\widehat{\alpha} = 0$. We combine both equations in the $(0, 0)-$scheme,
\begin{equation}\label{scheme2:3}
\bigl[u^{\ell+1} - u^\ell\bigr] - \bigl[u^\ell - u^{\ell-1}\bigr] = k^2 \Delta u^{\ell,1/2}  
+  \frac{k^2}{2} \bigl[ 3F(u^n) -F(u^{n-1})\bigr] + k \sigma(u^\ell)\Delta_\ell W  
\end{equation}
for all $1 \leq \ell \leq N$.
%
Now sum over the first $n$ steps, and define $\uli^{n+1} := \sum_{\ell=1}^n {u}^{\ell+1}$. We arrive at  
\begin{equation}\label{e1_2:2}
\bigl[u^{n+1} - u^n\bigr]
- k^2 \Delta \uli^{n,1/2} = \bigl[ u^1 - u^0\bigr] 
+\frac{k^2}{2}  \sum^n_{\ell=1} \bigl[3F(u^\ell) - F(u^{\ell-1})\bigr] + k  \sum^n_{\ell=1} \sigma(u^\ell) \Delta_{\ell} W\, .
\end{equation}
We proceed correspondingly with \eqref{strong2}, which we integrate in time: thanks to \eqref{strong1}, we get ($0 \le \lambda \le \mu \le T$)
\begin{equation}\label{e1_2:3}
\begin{split}
&\bigl[u(\mu)  - u(\lambda)\bigr]
- \int_\lambda^\mu \int^s_0 \Delta u(\xi) \, \dd \xi\ds 
\\ &\quad = [\mu-\lambda]  v_0 + \int_\lambda^\mu \int^s_0  F \bigl(u(\xi)\bigr) \, \dd \xi\ds + \int_{\lambda}^{\mu} \int_0^s \sigma \bigl( u(\xi)\bigr)\, {\rm d}W(\xi){\rm d}s \, ,
\end{split}
\end{equation}
where $s \in [t_n, t_{n+1}]$.
Setting $\mu=t_{n+1}$, $\lambda=t_n$ in \eqref{e1_2:3}, subtracting \eqref{e1_2:2} from \eqref{e1_2:3} then leads to 
\begin{align}\label{error-eq}
\bigl[e^{n+1}_u -e^n_u\bigr] - k^2  \Delta \eli^{n,1/2}_u &= k v_0 - (u^1 - u^0)  + \underbrace{\int_{t_n}^{t_{n+1}} {\red{\int_0^s}} \sigma \bigl( u(\xi)\bigr)\, {\rm d}W(\xi)\,{\rm d}s - k  \sum^n_{\ell=1} \sigma(u^\ell) \Delta_{\ell} W}_{{\tt := I}} \nonumber
\\ &\quad+\underbrace{\int_{t_n}^{t_{n+1}} {\red{\int_{0}^{s} }}
 \Delta u(\xi) \,\dd \xi \,\dd s -  k^2 \sum_{\ell =1}^n \frac{\Delta \big[ u(t_{\ell+1}) + u(t_{\ell-1}) \big]}{2}}_{{\tt := II}}
 \\ &\quad+ \underbrace{\int_{t_n}^{t_{n+1}} {\red{\int^s_0}}  F \bigl(u(\xi)\bigr) \, \dd \xi\,\ds - \frac{k^2}{2}  \sum^n_{\ell=1} \bigl[3F(u^\ell) - F(u^{\ell-1})\bigr]}_{{\tt := III}}\,. \nonumber
\end{align}
We first rewrite the term ${\tt I}$ in a form which is more suitable to obtain error estimates. Consider the term 
{\red{
\begin{align}
\int_0^s \sigma \bigl( u(\xi)\bigr)\, {\rm d}W(\xi)\,{\rm d}s = \bigg( \sum_{\ell=0}^{n-1} \int_{t_{\ell}}^{t_{\ell+1}} + \int_{t_n}^{s} \bigg) \sigma \bigl( u(\xi)\bigr)\, {\rm d}W(\xi)\,{\rm d}s\,.
\end{align}
}}
By rearranging the terms, we can rewrite ${\tt I}$ as the sum of five terms which are suitable to obtain error estimates
\begin{equation}\label{sigma-rearr}
\begin{split}
{\tt I} &= k \sum^n_{\ell=1} \bigl[\sigma\bigl( u(t_\ell)\bigr)  - \sigma(u^\ell) \bigr] \Delta_{\ell} W  + \int_{t_n}^{t_{n+1}} \sum^{n-1}_{\ell=1} \int_{t_\ell}^{t_{\ell+1}} 
\bigl[\sigma \bigl( u(\xi) \bigr)  - \sigma\bigl(u(t_\ell)\bigr)\bigr] \dW(\xi)\ds
\\ &\quad + \int_{t_n}^{t_{n+1}} \int_{t_n}^{s} \big[ \sigma \bigl( u(\xi)\bigr) - \sigma \bigl( u(t_n)\bigr) \big] \, {\rm d}W(\xi)\,{\rm d}s + \int_{t_n}^{t_{n+1}} \int_{t_n}^{s} \sigma \bigl( u(t_n)\bigr)\, {\rm d}W(\xi)\,{\rm d}s
\\ &\quad + \int_{t_n}^{t_{n+1}} \int_{t_0}^{t_1} \sigma \bigl( u(\xi)\bigr) \, {\rm d}W(\xi)\,{\rm d}s \,.
\end{split}
\end{equation}
Term ${\tt II}$ gives the quadrature error, for which we aim to apply Lemma \ref{quadrature1}. This result can not directly be applied here as the second term  involves the evaluation of $u$ at times $t_{\ell+1}$ and $t_{\ell-1}$, which are at distance $2k$. Thus, we rewrite the following integral as ($t_0 =0$)
{\red{
\begin{equation}\label{split-deltau}
\begin{split}
\int_0^s \Delta u(\xi)\,\dd \xi &= \bigg(\frac{1}{2} \sum_{\ell=1}^n \int_{t_{\ell-1}}^{t_{\ell+1}} + \frac{1}{2} \int_{t_{0}}^{t_{1}} + \frac{1}{2} \int_{t_{n}}^{t_{n+1}} - \int_s^{t_{n+1}} \bigg) \Delta u(\xi)\,\dd \xi\,,
\end{split}
\end{equation}
}}
where the first term on the right-hand side is now suitable to use Lemma \ref{quadrature1}. 

\medskip

We are now ready for the error analysis. We multiply both sides of \eqref{error-eq} with ${e}_u^{n+1/2}$, and observe that

\begin{equation*}
\begin{split}
k^2 \Big( \nabla \eli^{n,1/2}_u, \nabla e_u^{n+1/2} \Big) 
&= \frac{k^2}{4} \Bigl( \nabla \eli^{n +1}_u + \nabla \eli^{n -1}_u, 
\nabla [\eli_u^{n+1} - \eli_u^{n-1}] \Bigl) 
= \frac{k^2}{4} \Bigl[ \| \nabla \eli^{n+1}_u\|^2_{\bL^2} - \| \nabla \eli^{n-1}_u\|^2_{\bL^2} \Bigr]\, .
\end{split}
\end{equation*}
Using this, and rearranging the terms on the right-hand side of \eqref{error-eq} leads to the following error equation
\begin{eqnarray}\nonumber
&&\frac{1}{2} \Bigl[ \|e_u^{n+1}\|^2_{\bL^2} - \|e_u^{n}\|^2_{\bL^2}\Bigr] 
+\frac{k^2}{4} \Bigl[ \| \nabla \eli^{n+1}_u\|^2_{\bL^2} - \| \nabla \eli^{n-1}_u\|^2_{\bL^2} \Bigr]
\\ \nonumber
&&\quad = \big( k v_0 - [u^1 - u^0], e^{n+1/2}_u \big) + k\Bigl( \sum^n_{\ell=1} \bigl[\sigma\bigl( u(t_\ell)\bigr)  - \sigma(u^\ell) \bigr] \Delta_{\ell} W,e_u^{n+1/2}\Bigr) \\ \nonumber
&&\qquad + \Bigl( \int_{t_n}^{t_{n+1}} \sum^{n-1}_{\ell=1} \int_{t_\ell}^{t_{\ell+1}} 
\bigl[\sigma \bigl( u(\xi) \bigr)  - \sigma\bigl(u(t_\ell)\bigr)\bigr] \dW(\xi)\ds, e^{n+1/2}_u \Bigr) \\ \nonumber
&&\qquad + \Bigl( \int_{t_n}^{t_{n+1}} {\red{\int_{t_n}^{s}}} \big[ \sigma \bigl(u(\xi)\bigr) - \sigma \bigl(u(t_n)\bigr) \big] \, \dW(\xi)\,\ds, e^{n+1/2}_u \Bigr) \\ \label{e1_2:5}&&\qquad + \Bigl( \int_{t_n}^{t_{n+1}} \int_{t_n}^{s} \sigma \bigl(u(t_n)\bigr) \, \dW(\xi)\,\ds, e^{n+1/2}_u \Bigr) + \Bigl( \int_{t_n}^{t_{n+1}} {\red{\int_{t_0}^{t_1}}} \sigma \bigl(u(\xi)\bigr) \, \dW(\xi)\,\ds, e^{n+1/2}_u \Bigr) \\ \nonumber
&&\qquad - \bigg( \int_{t_n}^{t_{n+1}} {\red{\int_{0}^{s}}} \nabla u(\xi) \, \dd \xi \, \ds - k^2 \sum_{\ell=1}^n \frac{\nabla \big[ u(t_{\ell+1}) + u(t_{\ell-1}) \big]}{2} , \nabla e^{n+1/2}_u \bigg) 
\\ \nonumber
&&\qquad + \int_{t_n}^{t_{n+1}} \sum^{n-1}_{\ell=1} \int_{t_\ell}^{t_{\ell+1}} 
\Bigl(F (u)  - \frac{1}{2} \bigl[3F\bigl(u(t_\ell)\bigr) 
- F\bigl(u(t_{\ell-1})\bigr)\bigr], e^{n+1/2}_u\Bigr) \, \dd \xi\ds
\\ \nonumber
&&\qquad + \frac{k^2}{2} \sum^{n-1}_{\ell=1} \Bigl( 3\bigl[F\bigl(u(t_\ell)\bigr) - F(u^\ell) \bigr] - \bigl[F(u(t_{\ell-1})) - F(u^{\ell-1})\bigr] , e_u^{n+1/2}\Bigr)
\\ \nonumber
&&\qquad + \int_{t_n}^{t_{n+1}} {\red{\int_{t_0}^{t_1}}} \Big( F\big(u(\xi) \big) , e^{n+1/2}_u\Bigr) \, \dd \xi\ds + \int_{t_n}^{t_{n+1}} {\red{\int_{t_n}^{s}}} \Big( F\big(u(\xi) \big), e^{n+1/2}_u\Bigr) \, \dd \xi\ds
\\ \nonumber
&&\qquad - \frac{k^2}{2} \Bigl( \big[3F(u^n) - F(u^{n-1})\bigr] , e_u^{n+1/2}\Bigr)
\\ \nonumber
&&\quad =:I^{n}_1+ I^{\ell, n}_2 +\ldots + I^{\ell, n}_{12}\, .
\end{eqnarray}
%
We take the expectation on both sides and estimate all the terms on the right-hand side of \eqref{e1_2:5} separately. We begin with the term $I^{\ell, n}_2$.

\smallskip

\noindent
{\bf a) Estimation of $\bE \bigl[I^{\ell, n}_2\bigr]$.}
 By It\^o isometry, and {\bf (A3)}, we have 
\begin{eqnarray}\label{dazu1}
\bE \bigl[I^{\ell, n}_2\bigr] 
&\le& Ck \, \bE \bigl[ \|e_u^{n+1/2}\|^2_{\bL^2} \bigr]
+ k\, \bE\Bigl[\Bigl\| \sum_{\ell=1}^n 
\bigl[\sigma\bigl( u(t_\ell)\bigr) 
- \sigma(u^\ell) \bigr] \Delta_{\ell} W\Bigr\|^2_{\bL^2}\Bigr] \\ \nonumber
&\le&Ck\,  \bE \bigl[ \|e_u^{n+1/2}\|^2_{\bL^2} \bigr] 
+ \tilde{C}_{\tt L} k^2 \sum_{\ell=1}^n   
\bE\bigl[ \|e_u^\ell\|^2_{\bL^2}\bigr]\, .
\end{eqnarray}

\smallskip

\noindent
{\bf b) Estimation of $\bE \bigl[I^{\ell, n}_3\bigr]$.}
The term $I^{\ell, n}_3$ can be controlled by It\^o isometry, {\bf (A3)}, and Lemma~\ref{lem:Holder} $(i)$ as
\begin{eqnarray}\nonumber
\bE \bigl[I^{\ell, n}_3\bigr]  
&=& \bE \Bigl[\Bigl( \int_{t_n}^{t_{n+1}} \sum_{\ell=1}^{n-1} \int_{t_\ell}^{t_{\ell+1}} \bigl[\sigma \bigl(u(\xi)\bigr)  - \sigma\bigl(u(t_\ell)\bigr)\bigr] \, \dW(\xi)\ds, e^{n+1/2}_u \Bigr)\Bigr]
\\  \label{restri}
&\le& 
\tilde{C}_{\tt L} \int_{t_n}^{t_{n+1}} \sum_{\ell=1}^{n-1}
\int_{t_\ell}^{t_{\ell+1}}  \bE \Bigl[ \bigl\| u(\xi) - u(t_\ell)\bigr\|^2_{\bL^2} \Bigr] \dd \xi \,\dd s + Ck \, \bE \bigl[\|e_u^{n+1/2}\|^2_{\bL^2}\bigr] 
\\ \nonumber
&\le& \tilde{C}_{\tt L} k^3
+Ck \, \bE \bigl[\|e_u^{n+1}\|^2_{\bL^2} + \|e_u^n\|^2_{\bL^2}\bigr]\, .
\end{eqnarray}

\smallskip

\noindent
{\bf c) Estimation of $\bE \bigl[I^{\ell, n}_4\bigr]$.}
We use the similar arguments as for $\bE \bigl[I^{\ell, n}_3\bigr]$ to get
\begin{align*}
\bE \bigl[I^{\ell, n}_4\bigr]  &= \bE \Bigl[\Bigl( \int_{t_n}^{t_{n+1}} \int_{t_n}^{s} \bigl[\sigma \bigl(u(\xi)\bigr)  - \sigma\bigl(u(t_n)\bigr)\bigr] \, \dW(\xi)\ds, e^{n+1/2}_u \Bigr)\Bigr] 
\\ &\leq \tilde{C}_{\tt L} \int_{t_n}^{t_{n+1}} \int_{t_n}^{s} \bE \Bigl[ \bigl\| u(\xi) - u(t_n)\bigr\|^2_{\bL^2} \Bigr] \dd \xi \,\dd s + Ck \, \bE \bigl[\|e_u^{n+1/2}\|^2_{\bL^2}\bigr] 
\\ &\leq C k^4 +Ck \, \bE \bigl[\|e_u^{n+1/2}\|^2_{\bL^2}\bigr] \,.
\end{align*}

\smallskip

\noindent
{\bf d) Estimation of $\bE \bigl[I^{\ell, n}_5\bigr]$.} By independence of stochastic increments,
$$
\bE \bigl[I^{\ell, n}_5\bigr] = k \bE \Big[ \Big( \sigma \big(u(t_n) \big) \big( W(s) - W(t_n)\big), e_u^{n+1/2} \Big)\Big] = \frac{k}{2} \bE \Big[ \Big( \sigma \big(u(t_n) \big) \big( W(s) - W(t_n)\big), e_u^{n+1} - e_u^n \Big)\Big]\,.
$$
Using \eqref{e1_1:3} we rewrite
\begin{align*}
\bE \bigl[I^{\ell, n}_5\bigr] &= \frac{k}{2} \bE \Big[ \Big( \sigma \big(u(t_n) \big) \big( W(s) - W(t_n)\big), k e_v^{n+1} \Big)\Big]
\\ &\quad+ \frac{k}{2} \bE \Big[ \Big( \sigma \big(u(t_n) \big) \big( W(s) - W(t_n)\big), \int_{t_n}^{t_{n+1}} \big( v(s) - v(t_{n+1})\big)\,\dd s \Big)\Big] = \bE \bigl[I^{\ell, n}_{5; 1}\bigr] + \bE \bigl[I^{\ell, n}_{5; 2}\bigr]
\end{align*}
Using {\bf (A3)}, \eqref{eq-5.32} and the It\^o isometry,
\begin{align*}
\bE \bigl[I^{\ell, n}_{5; 1}\bigr] \leq C k^2 \,\bE\Big[ \big\| \sigma \big(u(t_n) \big) \big\|_{\bL^2}^2 \Big] \,\bE\Big[ |W(s) - W(t_n)|^2 \Big] + C k^2 \bE\big[ \|e_v^{n+1}\|_{\bL^2}^2 \big] \leq C k^3\,.
\end{align*}
Using the similar arguments and by Lemma \ref{lem:Holder} $(ii)$ we infer
\begin{align*}
\bE \bigl[I^{\ell, n}_{5; 2}\bigr] \leq C k^2 \,\bE\Big[ \big\| \sigma \big(u(t_n) \big) \big\|_{\bL^2}^2 \Big] \,\bE\Big[ |W(s) - W(t_n)|^2 \Big] + C k^2 \bE\Big[ \|v(s) - v(t_{n+1})\|_{\bL^2}^2 \Big] \leq C k^3\,.
\end{align*}

\smallskip

\noindent
{\bf e) Estimation of $\bE \bigl[I^{\ell, n}_6\bigr]$.} We add and subtract the term $\sigma(u_0)$ to get
\begin{align*}
\bE \bigl[I^{\ell, n}_6\bigr] &= \bE \bigg[ \Bigl( \int_{t_n}^{t_{n+1}} \int_{t_0}^{t_1} \big( \sigma \bigl(u(\xi)\bigr) - \sigma \bigl(u_0 \bigr) \big) \, \dW(\xi)\,\ds, e^{n+1/2}_u \Bigr) \bigg]
\\ &\quad+  \bE \bigg[ \Bigl( \int_{t_n}^{t_{n+1}} \int_{t_0}^{t_1} \sigma \bigl(u_0\bigr) \, \dW(\xi)\,\ds, e^{n+1/2}_u \Bigr) \bigg] := \bE \bigl[I^{\ell, n}_{6; 1}\bigr] + \bE \bigl[I^{\ell, n}_{6; 2}\bigr] \,,
\end{align*}
where the term $\bE \bigl[I^{\ell, n}_{6; 1}\bigr]$ will follow the similar arguments as for $\bE \bigl[I^{\ell, n}_4\bigr]$ in step {\bf c)} to yield 
\begin{align*}
\bE \bigl[I^{\ell, n}_{6; 1}\bigr] \leq C k^4 +Ck \, \bE \bigl[\|e_u^{n+1/2}\|^2_{\bL^2}\bigr] \,.
\end{align*}
The term 
$
\bE \bigl[I^{\ell, n}_{6; 2}\bigr] = k \bE \Big[ \big( \sigma \bigl(u_0\bigr) \Delta_0 W, e^{n+1/2}_u \big) \Big] 
$
will be merged with $I^n_1$\,.

\smallskip

\noindent
{\bf f) Estimation of $\bE \bigl[I^{\ell, n}_7\bigr]$.}
We will apply the quadrature formula (Lemma \ref{quadrature1}) to handle this term as mentioned above for the term ${\tt II}$. To get suitable terms for the error bounds, we use the identity \eqref{split-deltau} and again rearrange the terms by splitting the integral $\int_{t_{n}}^{t_{n+1}} \cdot \,\dd s$ into $\int_{t_{n}}^{s} \cdot \,\dd s + \int_{s}^{t_{n+1}} \cdot \,\dd s$,  to get
\begin{equation}\label{nab-xi}
\begin{split}
\int_0^s \nabla u(\xi)\,\dd \xi &= \frac{1}{2} \Bigg( \overbrace{\sum_{\ell=1}^n \int_{t_{\ell-1}}^{t_{\ell+1}} \nabla u(\xi)\,\dd \xi}^{:= {\tt I}_7} +  \Big( \overbrace{\int_{t_{0}}^{t_{1}} + \int_{t_{n}}^{s} - \int_{s}^{t_{n+1}} \Big) \nabla u(\xi)\,\dd \xi}^{:= {\tt II}_7} \Bigg) \,.
\end{split}
\end{equation}
Below, we handle the terms separately. 

\smallskip

\noindent
${\bf f}_1$)
Lemma~\ref{quadrature1} (with $T=2k$) now applies to handle ${\tt I}_7$.
 For this, we choose $f(\xi) = \bE\bigl[ (\nabla u(\xi), \nabla e_u^{n+1/2})\bigr]$ for all  $\xi \in [t_n, t_{n+1}]$. Using integration by parts, by Lemma~\ref{lem:Holder} $(iv)$, we have
$\gamma = \frac{1}{2}$ in (\ref{cond1}),
\begin{equation*}\label{e1_2:8}
\Bigl|\bE\Bigl[\Bigl(\nabla [v(t)- v(s)], \nabla e_u^{n+1/2} \Bigr)\Bigr]\Bigr|
\le C \Bigl(\bE\bigl[ \| \nabla e_u^{n+1/2} \|^2_{\bL^2}\bigr]\Bigr)^{1/2} \vert t-s \vert^{\frac{1}{2}}\, .
\end{equation*}
We know that
\begin{align*}
\frac{1}{2k} \int_{t_{\ell-1}}^{t_{\ell+1}} f(\xi)\,\dd \xi - \frac{f(t_{\ell+1})+f(t_{\ell-1})}{2} = \frac{1}{2k} \bigg[ \int_{t_{\ell-1}}^{t_{\ell+1}} \Big\{ f(\xi) - \frac{f(t_{\ell+1})+f(t_{\ell-1})}{2} \Big\}\,\dd \xi \bigg]\,.
\end{align*}
As a consequence 
\begin{align*}\label{e1_2:9}
&\frac{1}{2k} \bigg| \int_{t_{\ell-1}}^{t_{\ell+1}} \bE\Bigl[\Bigl(\nabla \Big[u(\xi) -  \frac{u(t_{\ell+1}) + u(t_{\ell-1})}{2} \Big], \nabla e_u^{n+1/2} \Bigr)\Bigr] \, \dd \xi \bigg| 
\\ &= \bigg| \frac{1}{2k} \int_{t_{\ell-1}}^{t_{\ell+1}} f(\xi)\,\dd \xi - \frac{f(t_{\ell+1})+f(t_{\ell-1})}{2} \bigg|
\leq \frac{\tilde{C}}{(\frac{1}{2}+2)(\frac{1}{2}+3)} (2k)^{3/2} \Big( \bE\big[ \|\nabla e_u^{n+1/2}\|_{\bL^2}^2 \big] \Big)^{1/2}\,.
\end{align*} 
Using \eqref{eq-5.32}, we can finally bound the term with the integral from $t_n$ to $t_{n+1}$ by
\begin{align*}
\leq C k^{5/2} \Big( \bE\big[ \|\nabla e_u^{n+1/2}\|_{\bL^2}^2 \big] \Big)^{1/2} 
\leq C k^2 \bE\big[ \|\nabla e_u^{n+1/2}\|_{\bL^2}^2 \big] + C k^3\, \leq Ck^3\,.
\end{align*} 

\smallskip

\noindent
${\bf f}_2$)
Consider the difference of the last two terms in ${\tt II}_7$. By adding and subtracting the terms $\frac{1}{2}\int_{t_n}^{t_{n+1}} \int_{t_n}^s \nabla u(t_n) \,\dd \xi \,\dd s$ and $\frac{1}{2} \int_{t_n}^{t_{n+1}} \int^{t_{n+1}}_s \nabla u(t_n) \,\dd \xi \,\dd s$ we get
\begin{align*}
&\frac{1}{2} \int_{t_n}^{t_{n+1}}  \bigg( \int_{t_{n}}^{s} \nabla u(\xi)\,\dd \xi - \int_{s}^{t_{n+1}} \nabla u(\xi) \,\dd \xi \bigg)\,\dd s 
\\ &= \overbrace{\frac{1}{2}  \int_{t_n}^{t_{n+1}} \int_{t_n}^s \nabla \Big( u(\xi) - u(t_n) \Big)\,\dd \xi \,\dd s}^{{\tt II}_{7, a}} - \overbrace{\frac{1}{2}  \int_{t_n}^{t_{n+1}} \int_{s}^{t_{n+1}} \nabla \Big( u(\xi) - u(t_n) \Big)\,\dd \xi \,\dd s}^{{\tt II}_{7, b}}
\\ &\quad+ \frac{1}{2} \nabla u(t_n) \bigg( \int_{t_n}^{t_{n+1}} \int_{t_n}^s \dd \xi \,\dd s - \int_{t_n}^{t_{n+1}} \int_{s}^{t_{n+1}} \dd \xi \,\dd s \bigg)\,,
\end{align*}
where the last term on the right-hand side vanishes as $\int_{t_n}^{t_{n+1}} \big[ (s- t_n) - (t_{n+1}-s)\big] \,\dd s = 0$. So, we estimate the first two terms only. By standard estimation,
\begin{align*}
&\int_{t_n}^{t_{n+1}} \bE \Big[ \Big({\tt II}_{7, a}, \nabla e_u^{n+1/2} \Big) \Big]\,\dd s
\\ &\leq C\,k^2 \int_{t_n}^{t_{n+1}} \bE \Big[ \big\|\nabla \big[ u(\xi) - u(t_n)\big] \big\|_{\bL^2}^2 \Big]\,\dd \xi + C k \bE \big[ \|\nabla e_u^{n+1/2}\|_{\bL^2}^2 \big]
\\ &\leq C\,k^5 + C k \bE \big[ \|\nabla e_u^{n+1/2}\|_{\bL^2}^2 \big]
\end{align*}
thanks to Lemma \ref{lem:Holder} $(ii)$.
Similarly, one may handle the term that involves ${\tt II}_{7, b}$. The only term left to handle from the right-hand side of \eqref{nab-xi} is $\int_{t_{0}}^{t_{1}} \nabla u(\xi)\,\dd \xi$. We can rewrite the term $\int_{t_n}^{t_{n+1}} \int_{t_{0}}^{t_{1}} \nabla u(\xi)\,\dd \xi\,\dd s$ as $\int_{t_n}^{t_{n+1}} \int_{t_{0}}^{t_{1}} \nabla \big[ u(\xi) - u_0 \big] \,\dd \xi\,\dd s + k^2 \nabla u_0\,.$ To estimate the first term, we proceed as before to conclude
\begin{align*}
&\Bigg| \int_{t_n}^{t_{n+1}} \bE \bigg[ \Big(\frac{1}{2} \int_{t_{0}}^{t_{1}} \nabla \big[ u(\xi) - u_0 \big] \,\dd \xi, \nabla e_u^{n+1/2} \Big) \bigg]\,\dd s  \Bigg|
\\ &\leq C \int_{t_n}^{t_{n+1}} \int_{t_{0}}^{t_{1}}  \xi^2\,\dd \xi \,\dd s + C k \bE\big[ \|\nabla e_u^{n+1/2}\|_{\bL^2}^2\big] \leq C k^4 + C k \bE\big[ \|\nabla e_u^{n+1/2}\|_{\bL^2}^2\big]\,.
\end{align*}
Now, we only need to handle the term $\bE \big[ \big( \frac{1}{2} k^2 \Delta u_0, e_u^{n+1/2} \big)\big]\,.$ This term will now be combined with the term $I^n_1$.

%

\smallskip

\noindent
{\bf g) Estimation of $\bE \bigl[I^{\ell, n}_8\bigr]$.}
As a first step, we split it into two parts,
\begin{equation*}
\begin{split}
{\mathbb E}[I^{\ell, n}_8] &= 
\int_{t_n}^{t_{n+1}} \sum^{n-1}_{\ell=1} \int_{t_\ell}^{t_{\ell+1}} {\mathbb E}\Bigl[
\Bigl(F \bigl(u(\xi)\bigr) - \frac{1}{2} \bigl[F\bigl(u(t_{\ell})\bigr) + F\bigl(u(t_{\ell+1})\bigr) 
\bigr], e^{n+1/2}_u \Bigr)\Bigr] \, \dd \xi\ds \\ 
& \quad+ \frac{k^2}{2}\sum_{\ell=1}^n 
{\mathbb E}\Bigl[\Bigl( F\bigl( u(t_{\ell+1})\bigr) - 2 F\bigl(u(t_{\ell})\bigr) + F \bigl( u(t_{\ell-1})\bigr),e^{n+1/2}_u\Bigr)\Bigr] =: {\mathbb E} \big[ I^{\ell, n}_{8;1} + I^{\ell, n}_{8;2} \big]\, .
\end{split}
\end{equation*}
To handle these two terms, we use Lemma~\ref{quadrature1} with
$f(\xi) = \bE\bigl[ \bigl( F\bigl(u(\xi)\bigr), e_u^{n+1/2} \bigr)\bigr]$ where 
$\xi \in [t_n, t_{n+1}]$, and verify $\gamma = \frac{1}{2}$ in
(\ref{cond1}): by {\bf (A4)} for $m=1, 2$, the chain-rule and the mean-value theorem
\begin{equation}\label{dazu2}
\begin{split}
&\Bigl\vert \bigl(D_t F\bigl(u(t)\bigr) - D_t F\bigl(u(s)\bigr), e^{n+1/2}_u\bigr)\Bigr\vert 
\\ &=  \Bigl\vert \Bigl(D_u F\bigl(u(t)\bigr) v(t) - D_u F\bigl(u(s)\bigr) v(s), e^{n+1/2}_u\Bigr)
\Bigr\vert 
\\ &=\Bigl\vert \Big( \big(D_u F(u(t)) - D_u F(u(s)) \big) v(t) + D_u F(u(s)) (v(t)-v(s)), e^{n+1/2}_u \Big)\Bigr\vert
\\ &= \Bigl\vert \Big( \big(\underline{D_u^2 F}(u(t)-u(s)) \big) (v(t)) + D_u F(u(s)) (v(t)-v(s)), e^{n+1/2}_u \Big) \Bigr\vert
\\ &\leq \tilde{C}_g \,\Vert u(t) - u(s)\Vert_{{\mathbb H}^1} \Vert v(t)\Vert_{{\mathbb H}^1} \Vert e^{n+1/2}_u\Vert_{{\mathbb L}^2}+ \tilde{C}_g\, \Vert v(t) - v(s)\Vert_{{\mathbb L}^2} \Vert e^{n+1/2}_u\Vert_{{\mathbb L}^2}\, .
\end{split}
\end{equation}
where $\underline{D_u^2 F} := D_u^2 F(\widetilde{u}_{\rho})$ and $\widetilde{u}_{\rho} := \rho u(t) + (1-\rho) u(s)$, for some
$\rho \in [0,1]$. Lemma~\ref{lem:Holder} $(ii)$ then establishes $\gamma = \frac{1}{2}$ in (\ref{cond1}), and so Lemma~\ref{quadrature1} yields
$${\mathbb E}\bigl[ I^{\ell, n}_{8;1} \bigr]  \leq C k^{\frac{5}{2}}C \Bigl(\bE\bigl[ \bigl\| e_u^{n+1/2} \bigr\|^2_{\bL^2}\bigr]\Bigr)^{\frac{1}{2}} \leq Ck^4 + k\, \bE\bigl[ \bigl\| e_u^{n+1/2} \bigr\|^2_{\bL^2}\bigr]\, .$$
{{
In order to estimate ${\mathbb E}[I^{\ell, n}_{8;2}]$, we may write for some $\theta \in (0, 1)$
\begin{align*}
F(u(t_{\ell+1})) &= F(u(t_{\ell})) + D_u F(u(t_{\ell})) (u(t_{\ell+1}) - u(t_{\ell})) \\ &\quad+ \frac{1}{2} \Big( D_u^2 F\big(u(t_{\ell})+ \theta (u(t_{\ell+1}) - u(t_{\ell})) \big) (u(t_{\ell+1}) - u(t_{\ell})) \Big) \big(u(t_{\ell+1}) - u(t_{\ell}) \big) \,,
\end{align*}
and
\begin{align*}
F(u(t_{\ell-1})) &= F(u(t_{\ell})) + D_u F(u(t_{\ell})) (u(t_{\ell-1}) - u(t_{\ell})) \\ &\quad+ \frac{1}{2} \Big( D_u^2 F\big(u(t_{\ell})+ \theta (u(t_{\ell-1}) - u(t_{\ell})) \big) (u(t_{\ell-1}) - u(t_{\ell})) \Big) \big(u(t_{\ell-1}) - u(t_{\ell}) \big) \,.
\end{align*}
Then, adding the above two terms we get
\begin{align*}
&F\bigl( u(t_{\ell+1})\bigr) - 2 F\bigl( u(t_{\ell})\bigr) + F\bigl( u(t_{\ell-1}) \bigr)
\\ &= \underline{D_u F} \big(u(t_{\ell+1}) - 2 u(t_{\ell})+ u(t_{\ell-1}) \big) + \frac 12 \big(\underline{D_u^2 F} \,( u(t_{\ell+1}) - u(t_{\ell})) \big) (u(t_{\ell+1}) - u(t_{\ell})) 
\\ &\quad+ \frac 12 \big(\overline{D_u^2 F} \,( u(t_{\ell-1}) - u(t_{\ell}))\big)(u(t_{\ell-1}) - u(t_{\ell}))\, ,
\end{align*}
where $\underline{D_u F} := D_u F(u(t_{\ell}))$, $\underline{D_u^2 F} := D_u^2 F\big(u(t_{\ell})+ \theta (u(t_{\ell+1}) - u(t_{\ell})) \big)$ and $\overline{D_u^2 F} := D_u^2 F\big(u(t_{\ell})+ \theta (u(t_{\ell-1}) - u(t_{\ell})) \big)$. We begin with the first term on the right-hand side: first, by the mean value theorem, there exist $\zeta_1, \zeta_2  \in [0,1]$, such that
$$  u(t_{\ell+1}) - u(t_{\ell})  = k v\bigl( \zeta_1 t_{\ell}+[1-\zeta_1]t_{\ell+1}\bigr)\,, \qquad
- \bigl[u(t_{\ell}) - u(t_{\ell-1})\bigr]  = - k v\bigl( \zeta_2 t_{\ell-1}+[1-\zeta_2]t_{\ell}\bigr)\,.
$$
Hence, Lemma \ref{lem:Holder} $(ii)$ settles ${\mathcal O}(k^{\frac{3}{2}})$ for this term. If combined with Lemma \ref{lem:Holder} $(i)$, {\bf (A4)} for $m=1, 2$, we can conclude
\[{\mathbb E}[I^{\ell, n}_{8;2}]  \leq C {k^4} + k\, \bE\big[ \|e^{n+1/2}_u\big\|^2_{\bL^2} \big] \, .\]
}}
\del{
We use the mean-value theorem twice to conclude
$${\red{{\mathbb E}[I^{\ell, n}_{5;2}] \leq Ck^2 \sum_{\ell=1}^n {\mathbb E}\Bigl[ \Bigl(\Vert u (t_{\ell+1} ) - u (t_{\ell} )\Vert^2_{{\mathbb L}^2} + k\, \Vert d_t u (t_{\ell+1} ) - d_t u (t_{\ell} )\Vert_{{\mathbb L}^2} \Bigr)\Vert e^{n+1/2}_u\Vert_{{\mathbb L}^2}\Bigr].}}
$$
Using Lemma \ref{lem:Holder_L2_u} $(i)$, and the mean-value theorem one more time for the second summand in combination with Lemma \ref{lem:Holder_L2_u} $(i)$ and  \ref{lem:Holder_L2} $(i)$ we estimate further,
$$ \leq  Ck^2 \sum_{\ell=1}^n \bigl(k^2 + k^{\frac{3}{2}}\bigr)
\Bigl({\mathbb E}\Bigl[ \Vert e^{n+1/2}_u\Vert^2_{{\mathbb L}^2}\Bigr]\Bigr)^{\frac{1}{2}} \leq Ck^4 + k\, \bE\bigl[ \bigl\| e_u^{n+1/2} \bigr\|^2_{\bL^2}\bigr]\, .$$
}

\noindent
{\bf h) Estimation of $\bE \bigl[I^{\ell, n}_9\bigr]$.}
Finally, by {\bf (A3)} we infer
\begin{align*}
{\mathbb E}[I^{\ell, n}_9] \leq Ck^2 \biggl(\sum_{\ell=1}^n {\mathbb E}\Bigl[ \Vert e_u^{\ell}\Vert^2_{{\mathbb L}^2} + \Vert e_u^{\ell-1}\Vert^2_{{\mathbb L}^2}\Bigr]
+ {\mathbb E}\bigl[\Vert e^{n+1/2}_u\Vert_{{\mathbb L}^2}^2\bigr]\biggr)\, .
\end{align*}

\noindent
{\bf i) Estimation of $\bE \bigl[I^{\ell, n}_{10}\bigr]$.}
Adding and subtracting $F(u_0)$ to the term we get,
\begin{align*}
\bE \bigl[I^{\ell, n}_{10}\bigr] = \bE\bigg[\int_{t_n}^{t_{n+1}} \int_{t_0}^{t_1} \Big( \big[ F(u(\xi)) - F(u_0)\big], e^{n+1/2}_u \Big) \bigg] \,\dd \xi \,\dd s +  \bE\Big[ \big( k^2 F(u_0), e^{n+1/2}_u\big) \Big]\,.
\end{align*}
Using previous arguments as before, by {\bf (A3)} and Lemma \ref{lem:Holder} $(i)$ we bound the first term on the right-hand side by $C k^4 + C k\, {\mathbb E}\bigl[\Vert e^{n+1/2}_u\Vert_{{\mathbb L}^2}^2\bigr]\,.$ The second term on the right-hand side will now be combined with the term $I^n_1$.

\smallskip

\noindent
{\bf j) Estimation of $\bE \bigl[I^{\ell, n}_{11} + I^{\ell, n}_{12}\bigr]$.}  Here, we consider the estimation of $I^{\ell, n}_{11}$ and $I^{\ell, n}_{12}$ together. We subtract the term $\frac{1}{2} \big[ 3 F(u(t_n)) - F(u(t_{n-1})) \bigr]$ from $I^{\ell, n}_{11}$ and add the corresponding term with $I^{\ell, n}_{12}$. Then, the new decomposition can be estimated similarly as $\bE \big[I^{\ell, n}_{8}\big]$ and $\bE \big[I^{\ell, n}_{9} \big]$ (see steps {\bf g)} and {\bf h)}, respectively) to finally get 
\begin{align*}
\bE \bigl[I^{\ell, n}_{11} + I^{\ell, n}_{12}\bigr] \leq C k^4 + Ck \bigg({\mathbb E}\Bigl[ \Vert e_u^{n}\Vert^2_{{\mathbb L}^2} + \Vert e_u^{n-1}\Vert^2_{{\mathbb L}^2}\Bigr] + {\mathbb E}\bigl[\Vert e^{n+1/2}_u\Vert_{{\mathbb L}^2}^2\bigr]\biggr)\,.
\end{align*}

\smallskip

\noindent
{\bf k) Estimation of $\bE \bigl[I^{n}_1\bigr]$.}
To estimate this term, we need to combine the terms $k \sigma \bigl(u_0\bigr) \Delta_0 W, \,\frac{k^2}{2} \Delta u_0$ and $k^2 F(u_0)$, which are coming from the steps ${\bf e)}, {\bf f}_2)$ and ${\bf i)}$, respectively. By the assumption (\ref{ini-1}) we infer
\begin{align*}
&{\mathbb E} \Big[ \Big( k v_0 - [u^1 - u^0] + \frac{k^2}{2} \Delta u_0 + k^2 F(u_0)+ k \sigma \bigl(u_0\bigr) \Delta_0 W,  e_u^{n+1/2} \Big) \Big] 
\\ & \leq \frac Ck\, {\mathbb E} \Big[ \big\| k v_0 - [u^1 - u^0] + \frac{k^2}{2} \Delta u_0 + k^2 F(u_0) + k \sigma \bigl(u_0\bigr) \Delta_0 W \big\|_{\bL^2}^2  \Big] + C k\,  {\mathbb E} \Big[ \Vert e^{n+1/2}_u \Vert^2_{{\mathbb L}^2} \Big] 
\\ &\leq C k^4 +Ck\,  {\mathbb E} \Big[ \Vert e^{n+1/2}_u \Vert^2_{{\mathbb L}^2} \Big]\,. 
\end{align*}

\smallskip

Now we combine all the above estimates in (\ref{e1_2:5}) in summarized form, then the implicit version of the discrete Gronwall lemma yields  assertion \eqref{pi}.

\medskip

\noindent
{\bf 3) Proof of \eqref{pii}.} Similar to (\ref{e1_2:2}), we have for $\widehat{\alpha} = 1$ 
\begin{equation}\label{tilde-comb-sum}
\begin{split}
&\big[u^{n+1} - u^n \big] - k^2 \Delta \uli^{n,1/2} 
- \big[ u^1 - u^0 \big] 
\\ &= k \sum^n_{\ell=1} \sigma(u^\ell) \Delta_{\ell} W  
+ \widehat{\alpha} k \sum^n_{\ell=1} D_u\sigma(u^\ell)v^{\ell} \widehat{\Delta_{\ell} W}+\frac{k^2}{2}  \sum^n_{\ell=1} \bigl[3F(u^\ell) - F(u^{\ell-1})\bigr]\, . 
\end{split}
\end{equation}
So the additional term on the right-hand side of the error equation \eqref{error-eq} is
\begin{align}\label{dist-th}
\widehat{\alpha} k\sum_{\ell=1}^n  
- D_u\sigma(u^\ell) v^{\ell}  \widetilde{\Delta_{\ell} W} + 
\widehat{\alpha} k\sum_{\ell=1}^n  
- D_u\sigma(u^\ell) v^{\ell} \bigl[
(\widehat{\Delta_{\ell} W}- \widetilde{\Delta_{\ell} W})\bigr] := \mathfrak{I}^{\ell,n}_{3,{\tt A}}
+ \mathfrak{I}^{\ell,n}_{3,{\tt B}}\,. 
\end{align}
We now follow the argumentation in {\bf 2)}:
multiplication with ${e}_u^{n+1/2}$ of the modified error equation (\ref{error-eq}) then leads to (\ref{e1_2:5}), where $\mathfrak{I}^{\ell,n}_{3,{\tt A}}$ is merged with $I^{\ell, n}_3$. Then, the sum $I^{\ell, n}_3+I^{\ell, n}_4 + I^{\ell, n}_5$ may be rewritten as the following sums
{\small{
\begin{equation}\label{e2_2:3}
\begin{split}
&\underbrace{\Bigl( \int_{t_n}^{t_{n+1}} \sum^{n-1}_{\ell=1} \int_{t_\ell}^{t_{\ell+1}} \Bigl[ D_u \sigma \bigl(u(t_\ell) \bigr)v(t_\ell) - D_u \sigma(u^\ell) v^{\ell} \Bigr](\xi - t_\ell) \, \dW(\xi)\ds, e_u^{n+1/2}\Bigr)}_{:= {I}^{\ell,n}_{3,{\tt A}_1}}\\ 
&\quad + \underbrace{\Bigl( \int_{t_n}^{t_{n+1}} \sum^{n-1}_{\ell=1} \int_{t_\ell}^{t_{\ell+1}} \Bigl[\sigma \bigl( u(\xi)\bigr)  - \sigma\bigl(u(t_\ell)\bigr) - D_u \sigma\bigl(u(t_\ell)\bigr) v(t_\ell)(\xi - t_\ell)\Bigr] \, \dW(\xi)\ds, e_u^{n+1/2} \Bigr)}_{:= {I}^{\ell,n}_{3,{\tt A}_2}} \\
&\quad + \underbrace{\Bigl( \int_{t_n}^{t_{n+1}} \int_{t_n}^{s} \Bigl[ D_u \sigma \bigl(u(t_n) \bigr)v(t_n) - D_u \sigma(u^n) v^{n} \Bigr](\xi - t_n) \, \dW(\xi)\ds, e_u^{n+1/2}\Bigr)}_{:= {I}^{\ell,n}_{3,{\tt A}_3}} \\ 
&\quad + \underbrace{\Bigl( \int_{t_n}^{t_{n+1}} \int_{t_n}^{s} \Bigl[\sigma \bigl( u(\xi)\bigr)  - \sigma\bigl(u(t_n)\bigr) - D_u \sigma\bigl(u(t_n)\bigr) v(t_n)(\xi - t_n)\Bigr] \, \dW(\xi)\ds, e_u^{n+1/2} \Bigr)}_{:= {I}^{\ell,n}_{3,{\tt A}_4}}\,.
\end{split}
\end{equation}
}}
We independently bound the other error terms in (\ref{e1_2:5}) in this modified setting:

\smallskip

\noindent
{\bf a)} To bound ${\mathbb E}[I^{\ell, n}_{3;{\tt A}_1}]$ in \eqref{e2_2:3}, we use It\^o isometry, the mean-value theorem, {\bf (A4)} for $m=1, 2$,  to get
\begin{equation}\label{I3;1_Itois}
\begin{split}
\bE\bigl[I^{\ell, n}_{3;{\tt A}_1}\bigr] &\leq k \bE \Bigl[ \frac{1}{k} \cdot \Big\| \int_{t_{n}}^{t_{n+1}} \sum^{n-1}_{\ell=1} \int_{t_{\ell}}^{t_{\ell+1}} \bigl[ D_u \sigma \bigl(u(t_{\ell}) \bigr)v(t_{\ell}) - D_u\sigma(u^{\ell}) v^{\ell} \bigr]
\\ &\qquad \qquad \times (\xi - t_{\ell}) \, \dW(\xi)\, \ds \Big\|_{\bL^2}^2 \Big] + k \, \bE\bigl[\|e_u^{n+1/2}\|^2_{\bL^2}\bigr]
\\ &\leq \int_{t_{n}}^{t_{n+1}} \sum^{n-1}_{\ell=1} \bE \Big[ \int_{t_{\ell}}^{t_{\ell+1}} \big\| D_u \sigma \bigl(u(t_{\ell}) \bigr)v(t_{\ell}) 
- D_u\sigma(u^{\ell}) v^{\ell}  \big\|^2_{\bL^2} 
\\ &\qquad \qquad \qquad \times (\xi - t_{\ell})^2 \,d \xi \,\ds \Big] + k \, \bE\bigl[\|e_u^{n+1/2}\|^2_{\bL^2}\bigr]
\\ &\le C k^4 \,\sum_{\ell=1}^{n-1} \bE\Bigl[\bigl\| D_u\sigma\bigl( u(t_{\ell})\bigr) v( t_{\ell}) - D_u\sigma(u^{\ell}) v^{\ell} \bigr\|^2_{\bL^2}  \Bigr]  + k \, \bE\bigl[\|e_u^{n+1/2}\|^2_{\bL^2}\bigr] 
\\ &\le Ck^4\,\sum_{\ell=1}^{n-1} \Bigl(\bE \Big[ \|e_v^{\ell} \|^2_{\bL^2} \Big]
+{\mathbb E} \Big[\widetilde{I^{\ell, n}_{3;{\tt A}_1}} \Big] \Bigr) + k \, \bE\Bigl[\|e_u^{n+1/2}\|^2_{\bL^2}\Bigr]\,,
\end{split}
\end{equation}
where $\widetilde{I^{\ell, n}_{3;{\tt A}_1}} := 
 \bigl\|[ D_u\sigma\bigl( u(t_{\ell})\bigr) - D_u\sigma(u^{\ell})] v\bigl( t_{\ell}\bigr)\bigr\|^2_{\bL^2}$. 
In order to handle the first term in the right-hand side, we estimate \eqref{I3;1_Itois} further by
 \begin{align*}
 \le Ck^2 \sum_{\ell=1}^{n-1} \Bigl(\bE \Big[ \|e_u^{\ell} \|^2_{\bL^2}+\|e_u^{\ell-1} \|^2_{\bL^2} \Big]
+ k^2 {\mathbb E} \Big[\widetilde{I^{\ell, n}_{3;{\tt A}_1}} \Big] \Bigr)+ k \, \bE\Bigl[\|e_u^{n+1/2}\|^2_{\bL^2}\Bigr]\,.
 \end{align*}
 We estimate the second term in the right-hand side as
 \begin{equation}\label{eq-embe}
 \begin{split}
C k^4 \sum_{\ell=1}^{n-1} {\mathbb E}\Big[\widetilde{I^{\ell, n}_{3;{\tt A}_1}} \Big]
&\leq C k^4 \sum_{\ell=1}^{n-1}
\bE\Bigl[ \|e_u^{\ell}\|_{\bL^2}^2 \|v(t_{\ell})\|_{\bL^{\infty}}^2 \Bigr] 
\\ &\leq C k \cdot k^3 \sum_{\ell=1}^{n-1}
\bE\Bigl[ \|e_u^{\ell}\|_{\bL^2} \| e_u^{\ell}\|_{\bL^2}  \|v(t_{\ell})\|_{\bL^{\infty}}^2 \Bigr] 
\\ &\leq Ck^2 \sum_{\ell=1}^{n-1} {\mathbb E} \Bigl[ \Vert e^{\ell}_u\Vert^2_{{\mathbb L}^2} \Bigr] + C k^6 \sum_{\ell=1}^{n-1} {\mathbb E} \Bigl[ \| e^{\ell}_u\Vert^4_{{\mathbb L}^2} \Bigr] + C k^6 \sum_{\ell=1}^{n-1} {\mathbb E} \Bigl[ \Vert v(t_{\ell})\Vert^8_{{\mathbb L}^{\infty}} \Bigr]\, ,
\end{split}
\end{equation}
where the last term on the right-hand side is bounded by $Ck^5$ due to Lemma \ref{lem:L2} $(iii)$. The second term on the right-hand side is bounded further by $C k^6 \sum_{\ell=1}^{n-1} {\mathbb E} \Bigl[ \Vert u(t_{\ell})\Vert^4_{{\mathbb L}^2} + \Vert u^{\ell}\Vert^4_{{\mathbb L}^2} \Bigr]$, which may be bounded by $Ck^5$, thanks to Lemma \ref{lem:L2} $(i)$ for $p=2$, and \eqref{energy2-himo}.

\smallskip

\noindent
{\bf b)} Now consider ${\mathbb E}[I^{\ell, n}_{3;{\tt A}_2}]$. Let $\xi \in [t_n, t_{n+1}]$; we use
the mean-value theorem twice, {\bf (A4)} for $m=1, 2$,  to conclude
\begin{equation}\label{sig-u-esti}
\begin{split}
&\bigl\Vert \sigma \bigl( u(\xi) \bigr)  
- \sigma\bigl(u(t_{\ell})\bigr) -D_u \sigma\bigl(u(t_{\ell})\bigr)v(t_{\ell})(\xi - t_{\ell}) \bigr\Vert^2_{{\mathbb L}^2}
 \\
& = \Bigl\Vert  [ D_u \sigma(\widetilde{u}_\zeta) - D_u \sigma\bigl(u(t_{\ell})\bigr)]\int_{t_{\ell}}^{\xi}
v(\eta)\, {\rm d}\eta + D_u \sigma\bigl(u(t_{\ell})\bigr)\int_{t_{\ell}}^{\xi}
\bigl[v(\eta)- v(t_{\ell})\bigr]\, {\rm d}\eta \Bigr\Vert^2_{{\mathbb L}^2} \\
& \leq C \Vert \nabla u(\xi) - \nabla u(t_{\ell}) \Vert_{{\mathbb L}^2}^4 + C k^2\, 
\sup_{t_{\ell} \leq \xi \leq t_{\ell+1}} \Vert v(\xi) - v(t_{\ell}) \Vert^2_{{\mathbb L}^2}\, ,
\end{split}
\end{equation}
where $\widetilde{u}_{\zeta} = \zeta u(\xi) + (1-\zeta) u(t_{\ell})$, for some
$\zeta \in [0,1]$. Thus, we have
\begin{equation*}
\begin{split}
{\mathbb E}[I^{\ell, n}_{3;{\tt A}_2}] 
&\leq Ck \, {\mathbb E}\Bigl[ \sup_{t_{\ell} \leq \xi \leq t_{\ell+1}} \bigl( \Vert \nabla u(\xi) - \nabla u(t_{\ell}) \Vert_{{\mathbb L}^2}^4 + k^2 \, 
\Vert v(\xi) - v(t_{\ell}) \Vert^2_{{\mathbb L}^2}\bigr) \Bigr] + k \, {\mathbb E} \big[ \Vert e^{n+1/2}_u\Vert^2_{{\mathbb L}^2} \big] 
\\ &\leq C k^4 + k \, {\mathbb E} \big[\Vert e^{n+1/2}_u\Vert^2_{{\mathbb L}^2} \big]\, .
\end{split}
\end{equation*}

\noindent
{\bf c)} We now consider ${\mathbb E}[I^{\ell, n}_{3;{\tt A}_3}]$. This term can be estimated by using the same arguments as for ${\mathbb E}[I^{\ell, n}_{3; {\tt A}_1}]$; see \eqref{I3;1_Itois} and \eqref{eq-embe}. Since we do not have the summation in this term, we will get
\[
{\mathbb E}[I^{\ell, n}_{3; {\tt A}_3}] \leq C k^6 + C k \, \bE\Bigl[\|e_u^{n+1}\|^2_{\bL^2} + \|e_u^{n}\|^2_{\bL^2}\Bigr]\,.
\]

\noindent
{\bf d)} Then, we consider ${\mathbb E}[I^{\ell, n}_{3;{\tt A}_4}]$, which will follow the same arguments as for ${\mathbb E}[I^{\ell, n}_{3; {\tt A}_2}]$; see \eqref{sig-u-esti}. We get the estimate
\[
{\mathbb E}[I^{\ell, n}_{3; {\tt A}_4}] \leq C k^5 + C k \, \bE\bigl[\|e_u^{n+1/2}\|^2_{\bL^2} \bigl]\,.
\]

\noindent
{\bf e)} To estimate the term involving $\mathfrak{I}^{\ell,n}_{3,{\tt B}}$, which is defined in \eqref{dist-th}, we use Young's inequality to write
\begin{equation*}
\begin{split}
{\mathbb E}\bigl[\bigl(\mathfrak{I}^{\ell,n}_{3,{\tt B}}, {e}_u^{n+1/2} \bigr)\bigr]
&\leq \frac{\widehat{\alpha}^2}{k} k^2 \,\bE \Big[ \Big\| \sum_{\ell =1}^n D_u \sigma(u^{\ell}) v^{\ell} \big[ \widehat{\Delta_{\ell} W} - \widetilde{\Delta_{\ell} W} \big] \Big\|_{\bL^2}^2 \Big] + k\, \bE \big[ \Vert {e}_u^{n+1/2}\Vert^2_{{\mathbb L}^2} \big]\,.
\end{split}
\end{equation*} 
Then, we use {\bf (A4)} for $m=1$, and independence of increments $\Delta_nW$ to get
\begin{equation*}
\begin{split}
&\leq \widehat{\alpha}^2\,C_g^2\, k \,\sum_{\ell =1}^n \bE \Big[ \| D_u \sigma(u^{\ell}) v^{\ell} \|_{\bL^2}^2 \big| \widehat{\Delta_{\ell} W} - \widetilde{\Delta_{\ell} W} \big|^2 \Big] + k\, \bE \big[ \Vert {e}_u^{n+1/2}\Vert^2_{{\mathbb L}^2} \big]\,.
\end{split}
\end{equation*} 
Finally, we use \eqref{dist-tilde-hat} and \eqref{energy1} of Lemma \ref{lem:scheme1:stab} to obtain
\begin{equation}\label{untersch1}
\begin{split}
&\leq \widehat{\alpha}^2\,C_g^2\, k^5 \,\sum_{\ell =1}^n \bE \big[ \|v^{\ell} \|_{\bL^2}^2 \big] + k\, \bE \big[ \Vert {e}_u^{n+1/2}\Vert^2_{{\mathbb L}^2} \big]
\leq C k^4 + k\, \bE \big[ \Vert {e}_u^{n+1/2}\Vert^2_{{\mathbb L}^2} \big]\, .
\end{split}
\end{equation} 
\smallskip

\noindent
{\bf f)} We may modify the argument in part {\bf 2)} to improve the order for the for the order limiting term ${\tt I}_7$ in $I^{\ell, n}_7$; see step {\bf f}$_1$).  Using integration by parts and using
Lemma~\ref{lem:Holder} $(iv)$ instead, we verify  (\ref{cond1}) of Lemma \ref{quadrature1} for $\gamma=1/2$ (by choosing $f(\xi) = \bE\bigl[ (\nabla u(\xi), \nabla e_u^{n+1/2})\bigr]$ for all $\xi \in [t_n, t_{n+1}]$ ) to get
\begin{equation*}
\Bigl|\bE\Bigl[\bigl(\nabla [v(t)- v(s)], \nabla e_u^{n+1/2} \bigr)\Bigr]\Bigr|
\le C \Bigl(\bE\bigl[ \| e_u^{n+1/2} \|^2_{\bL^2}\bigr]\Bigr)^{1/2} \vert t-s \vert^{\frac{1}{2}}\, .
\end{equation*}
Using this estimate we may bound the term in ${\bf f}_1)$ by
\begin{equation*}
\le Ck^{\frac{5}{2}} \Bigl(\bE\bigl[ \bigl\| e_u^{n+1/2} \bigr\|^2_{\bL^2}\bigr]\Bigr)^{1/2} \leq Ck\,  {\mathbb E}\bigl[ \Vert e^{n+1/2}_u\Vert^2_{{\mathbb L}^2} \bigr] + Ck^4\, .
\end{equation*}

Thanks to the above estimates, and after summation over all iteration steps in (\ref{e1_2:5}) we may then conclude assertion \eqref{pii}. 
\end{proof}

\bigskip

\section{Computational experiments}\label{sec-6}

In this section, we provide computational studies to check 
\begin{itemize}
\item
how essential the assumptions {\bf (A1)--(A5)} and {\bf (B1)--(B2)} ({\em i.e.,} needed in Sections \ref{Mat-fra}--\ref{sec-4}) are in actual computations. In this respect, we computationally study the impact of rough initial data $(u_0, v_0)$ on the discrete dynamics, as well as of drift nonlinearities $F$ (see Example \ref{exm-alp7}). 
\item
If the diffusion $\sigma \equiv \sigma(v)$ and the drift $F \equiv 0$, then there is a reduction of convergence order as proved in \eqref{eq-5.32} of Theorem \ref{lem:scheme1:con}; see Example \ref{ex-sv}. 
\item
The diffusion $\sigma \equiv \sigma(u,v) = 0$ on the boundary, and satisfies {\bf (A3)}. Example \ref{exm-alp8}  discusses the effect that noise has, which is non-homogeneous on the boundary, or violates {\bf (A3)}.
\item
By Theorem \ref{lem:scheme1:stab}, $\beta$ in the $(\widehat{\alpha}, \beta)$-scheme needs be chosen from $(0, 1/2)$ to ensure stable, accurate simulation of \eqref{stoch-wave1:1a} with $\sigma \equiv \sigma(u,v)$ and $F \equiv F(u,v)$. The  simulations in Example \ref{exa-bet1} evidence a small choice for $\beta$ for faster Monte Carlo approximation.
\end{itemize}

\medskip

We use the lowest order conforming finite element method to simulate the $(\widehat{\alpha}, \beta)-$scheme on a regular triangulation $\mathcal{T}_h$ of $\mathcal{O}$; see \cite{BS_2008}. 
Let the finite element space be
\begin{align*}
\bVh := \bigl\{u_h \in \bH^1_0: \  u_h \bigl\vert_K \in \mathcal{P}_1(K) \quad \forall \, K \in \cT_h \bigr\}\, ,
\end{align*}
where $\mathcal{P}_1(K)$ denotes the space of polynomials of degree one on $K \in \mathcal{T}_h$. \\

\noindent
As initial data, we choose $u^1$ and $v^1$ as
\begin{align}\label{u1+v1}
\begin{cases}
u^1 &= u_0 +k\,v_0 + \frac{k^2}{2} \Delta u_0 + k^2 F(u_0) + (k+k^2) \,\sigma(u_0) \Delta_0 W\,, 
\\ v^1 &= v_0 +k \sigma(u_0) W(t_1)\,,
\end{cases}
\end{align}
where $u_0, v_0$ (not finite element valued) satisfy assumptions {\bf (A1)$_{iv}$} and {\bf (B2)}. Recall the definitions for $\widetilde{u}^{n, 1/2}$ and $\widehat{\Delta_n W}$ in \eqref{tidt1} and \eqref{W-hat}, respectively.  We implement the following scheme:

\begin{scheme}  Let $\widehat{\alpha} \in \{0, 1\}$, and $0 \leq \beta < \frac{1}{2}$. Let $\{ t_n\}_{n=0}^N$ be a mesh of size $k>0$ covering $[0, T]$, and \eqref{u1+v1}.  For every $n \ge 1$, find a
$[{\mathbb V}_h]^2$-valued, 
${\mathcal F}_{t_{n+1}}$-measurable random variable~$(u_h^{n+1}, v_h^{n+1})$ such that 
\begin{align}
\label{scheme1:1FEM}
\bigl(u_h^{n+1} - u_h^n, \phi_h) &= 
k(v_h^{n+1},\phi_h) 
&\forall \phi_h \in {\mathbb V}_h\, , \\ 
\notag
(v_h^{n+1}-v_h^n, \psi_h) &= -k \big(\nabla \widetilde{u}_h^{n, 1/2}, \nabla \psi_h \big)  
+ \Bigl(\sigma(u^n_h, v^{n-\frac{1}{2}}_h)\Delta_n W, \psi_h \Bigr)
&\\ 
&\quad + \widehat{\alpha} \,\Bigl(D_{u} \sigma(u_h^{n}, v_h^{n-\frac{1}{2}}) v_h^{n} \,\widehat{\Delta_n W}, \psi_h \Bigr) \label{scheme1:2FEM}  &\\ 
&\quad +\frac{k}{2} \Bigl( 3 F(u_h^n, v_h^n) 
-  F(u_h^{n-1}, v_h^{n-1}), \psi_h \Bigr) 
&\forall \psi_h \in {\mathbb V}_h\, . \notag
\end{align} 
\end{scheme}

\smallskip

\subsection{Convergence rates}

The numerical experiments are performed using MATLAB. In this section, for all the examples we choose $\cO = (0, 1)$, $T=1$, $A = -\Delta$ in \eqref{stoch-wave1:1a}. We choose $u_0(x) = \sin(2 \pi x)$ and $v_0(x) = \sin(3 \pi x)$, and $u^1, v^1$ are chosen as in \eqref{u1+v1}. A reference solution is computed with a step size $k_{\tt ref}=2^{-7}$ and $h_{\tt ref}=2^{-7}$ to approximate the exact solution and the sample Wiener processes $W$. The expected values are approximated by computing averages over ${\tt MC} = 3000$ number of samples. The plots are shown for the time steps $k= \{2^{-3}, \cdots, 2^{-6}\}$.  

Example \ref{exm-alp0} in Section 1 provides computational evidence for the improved convergence rate $\cO(k^{3/2})$ for the scheme \eqref{scheme2:1}--\eqref{scheme2:2} with $\widehat{\alpha}=1$ in the situations where $\sigma \equiv \sigma(u)$. In the following example, we consider $\sigma \equiv \sigma(v)$, and find a convergence rates of ${\mathcal O}(k^{1/2})$ in simulations {\rm (A)--(C)} of Fig. \ref{conv-F-sv}, which validates \eqref{eq-5.32} of Theorem \ref{lem:scheme1:con}. So we observe a reduction of convergence order if compared to Example \ref{exm-alp0}, where $\sigma \equiv \sigma(u)$.

\smallskip

\del{
\begin{example}[{\bf Additive noise}]\label{ex-additive-s}
In Example \ref{exm-alp0}, take $\sigma(x)= 4x(1-x), \ x \in [0, 1]$. Let
$u_0(x) = \sin(2 \pi x) $ and $ v_0(x) = \sin(3 \pi x)$ and $u^1, v^1$ are chosen as \eqref{u1+v1}.
The $\bL^2$-error for $u$ and $\nabla u$ are computed with the scheme \eqref{scheme2:1}--\eqref{scheme2:2} and Fig. \ref{additive} confirm the convergence order $\cO(k^{3/2})$ and $\bL^2$-error for $v$ confirms $\cO(k)$.

\smallskip

\begin{figure}[h!]
  \centering
   \subfloat[$\bL^2$-error for $u$]{\includegraphics[width=0.33\textwidth]{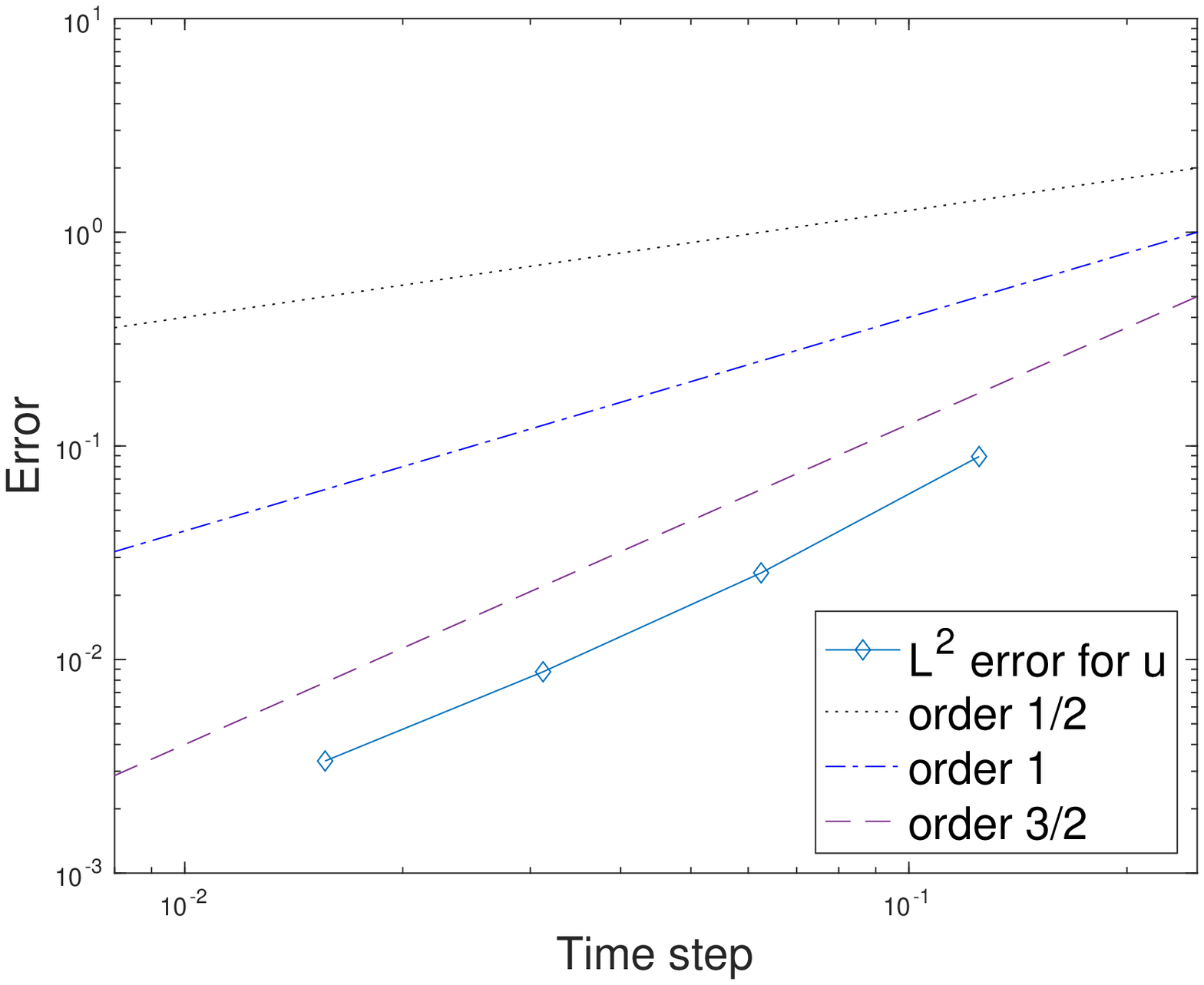}\label{fig:y1}}
  \hfill
  \subfloat[$\bL^2$-error for $\nabla u$]{\includegraphics[width=0.33\textwidth]{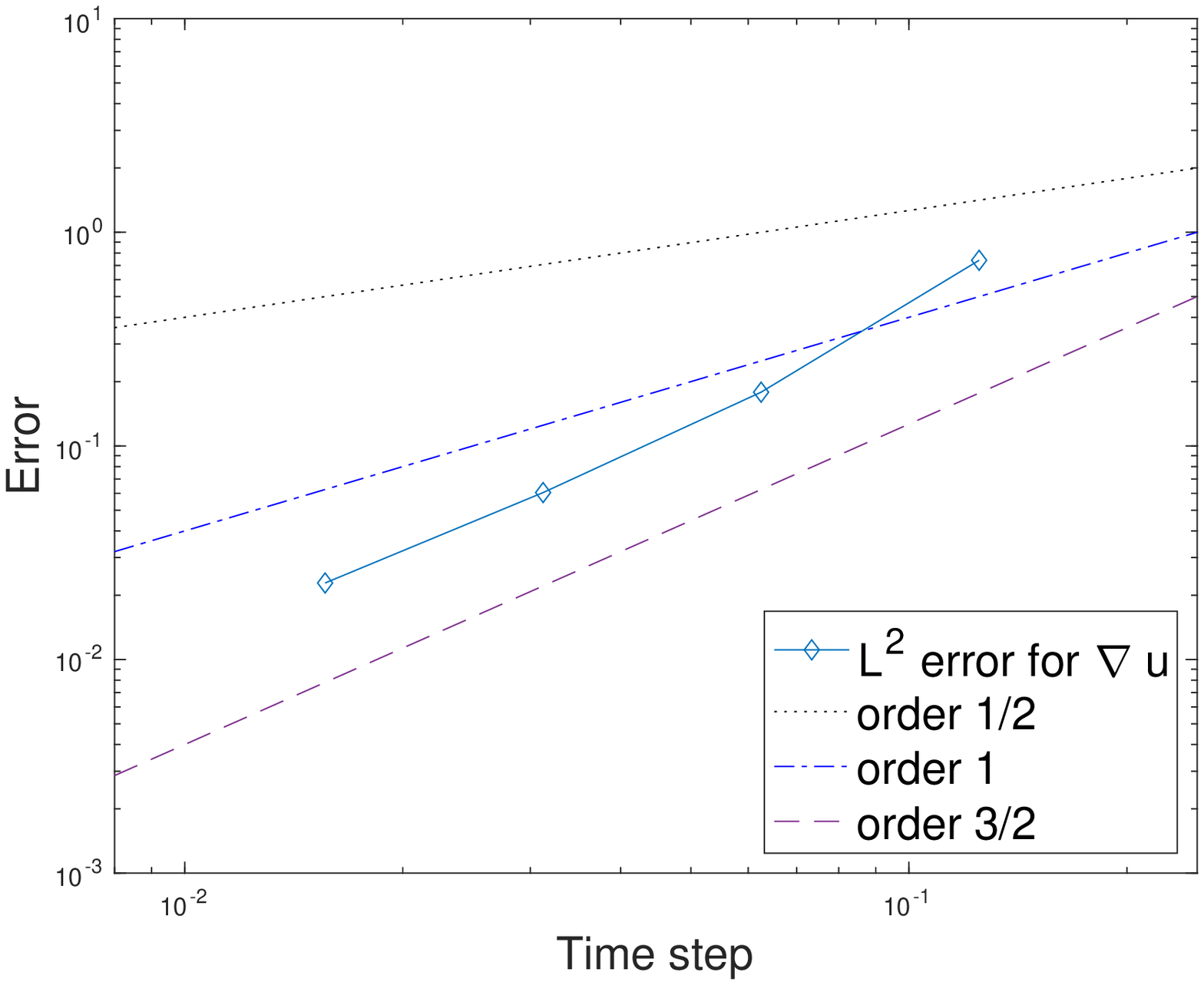}\label{fig:y2}}
  \hfill
  \subfloat[$\bL^2$-error for $v$]{\includegraphics[width=0.33\textwidth]{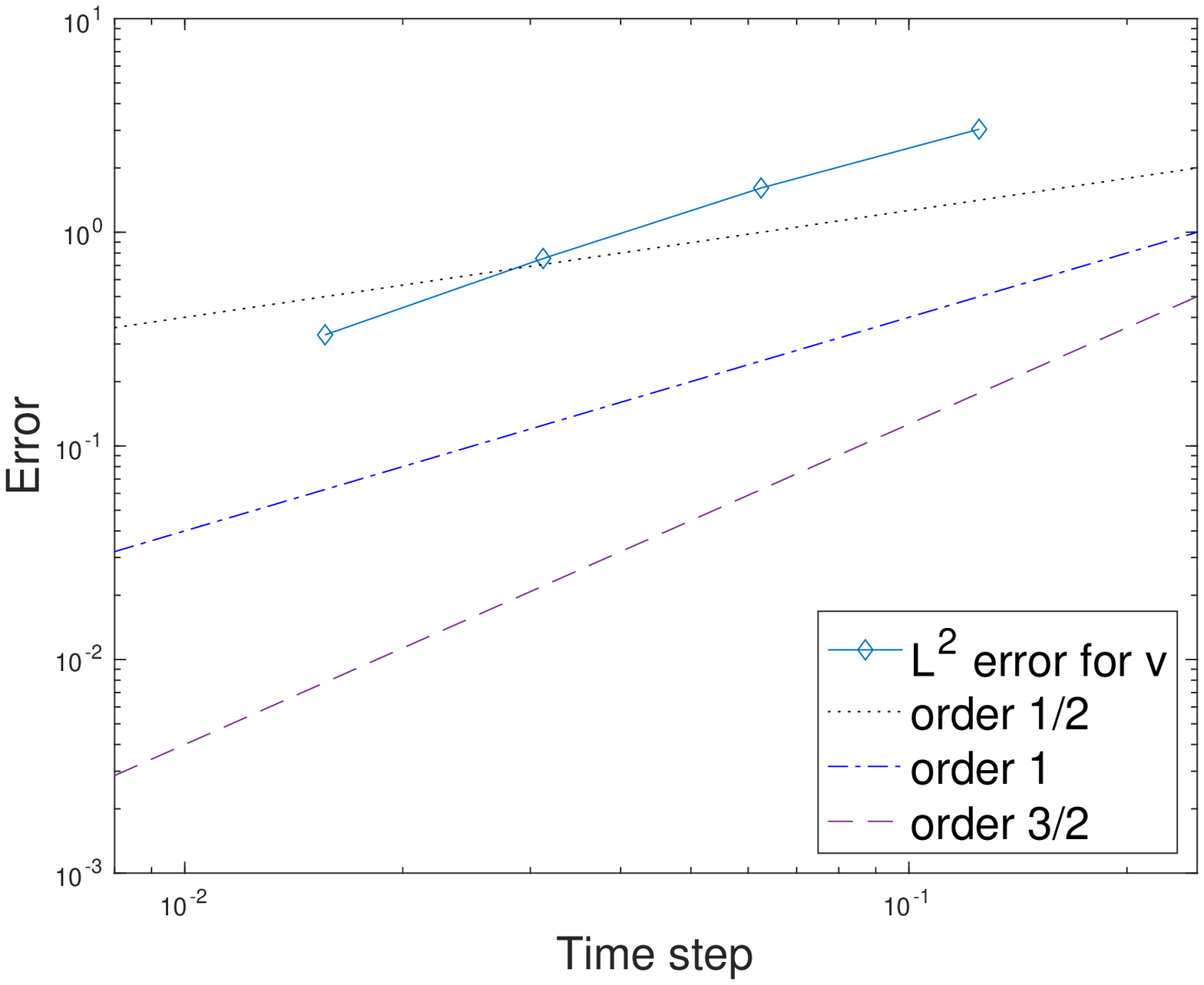}}\label{fig:y3} \caption{{\bf (Example \ref{ex-additive-s})} Temporal rates of convergence for $(\widehat{\alpha}, 0)-$scheme for $\widehat{\alpha}=1$ with $\sigma(x)= 4x(1-x)$; discretization parameters: $h=2^{-7}, k=\{2^{-3}, \cdots, 2^{-6}\}$, ${\tt MC}=3000$.} 
\label{additive}
\end{figure}
\end{example}
}

\begin{example}\label{ex-sv}
Consider $\sigma(v) = \frac{3}{2} v$ and $F \equiv 0$. Fig. \ref{conv-F-sv} displays convergence studies for the $(\widehat{\alpha}, \beta)-$scheme for $\widehat{\alpha}=1$ and $\beta=1/4:$  the plots {\rm (A)--(C)} of $\bL^2$-errors in $u, \nabla u$ and $v$, respectively, confirm convergence order ${\mathcal O}(k^{1/2})$; see \eqref{eq-5.32} of  Theorem \ref{lem:scheme1:con}.
\begin{figure}[h!]
  \centering
  \subfloat[$\bL^2$-error for $u$]{\includegraphics[width=0.33\textwidth]{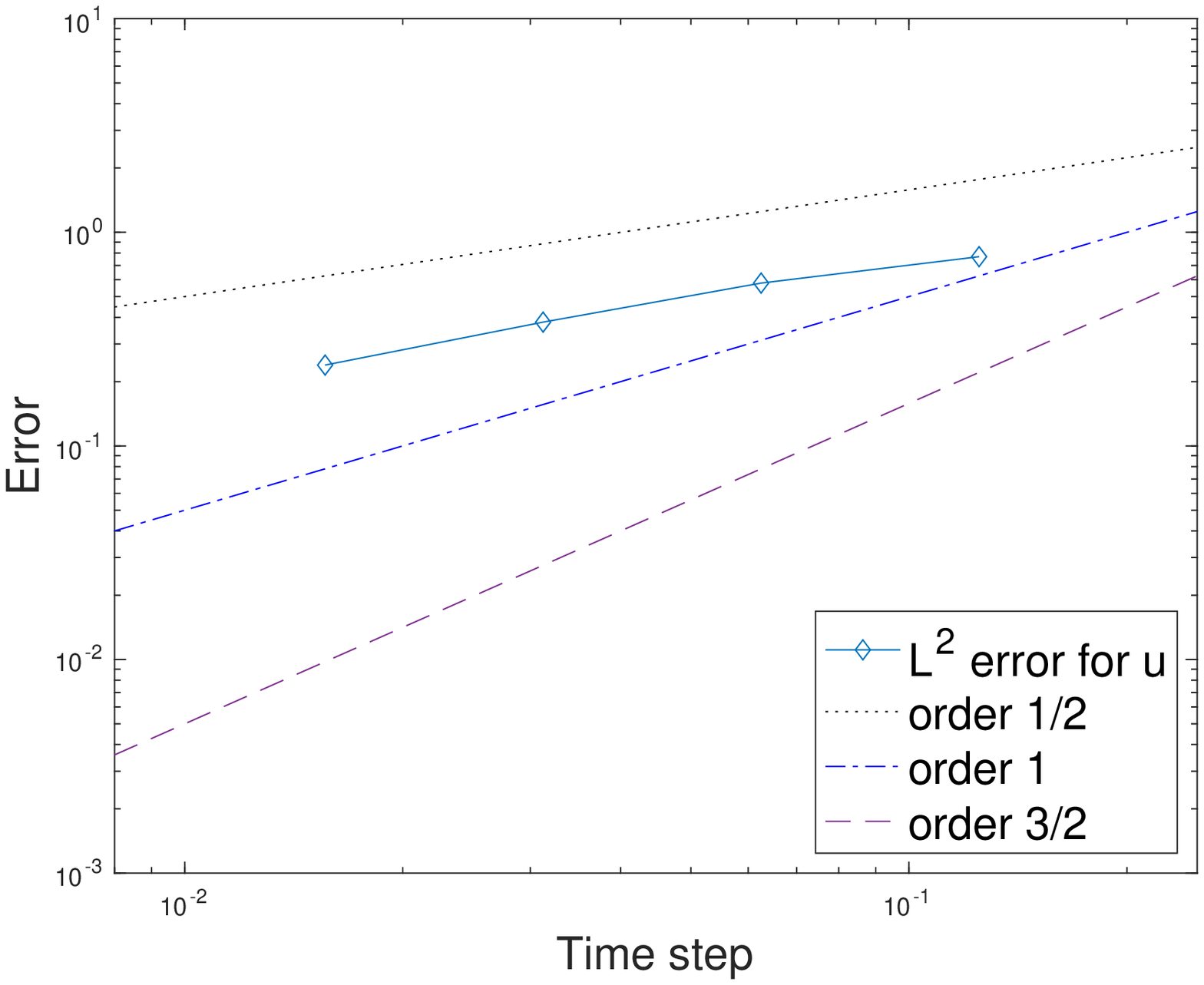}\label{fig--1}}
  \hfill
  \subfloat[$\bL^2$-error for $\nabla u$]{\includegraphics[width=0.33\textwidth]{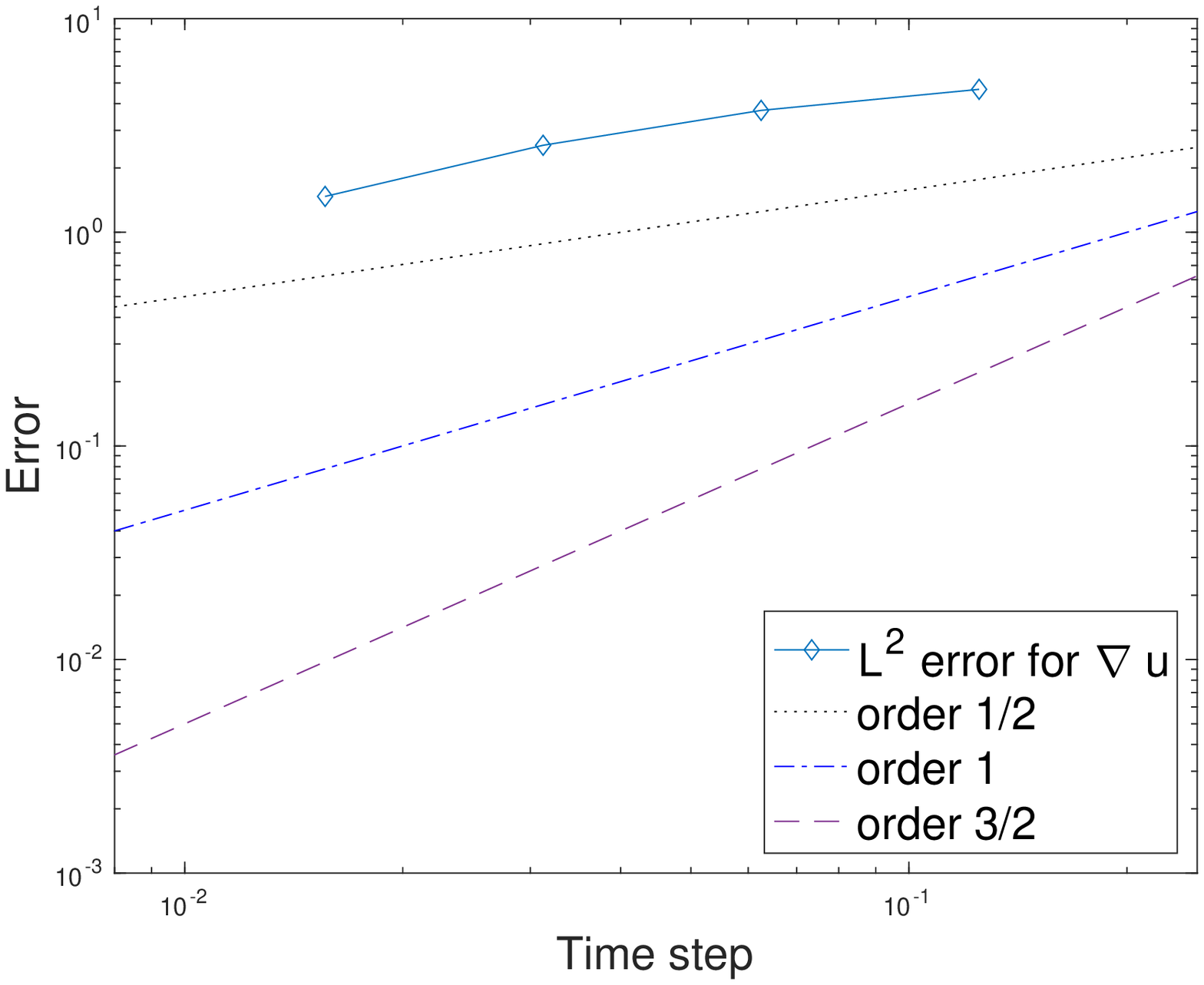}\label{fig--2}}
  \hfill
  \subfloat[$\bL^2$-error for $v$]{\includegraphics[width=0.33\textwidth]{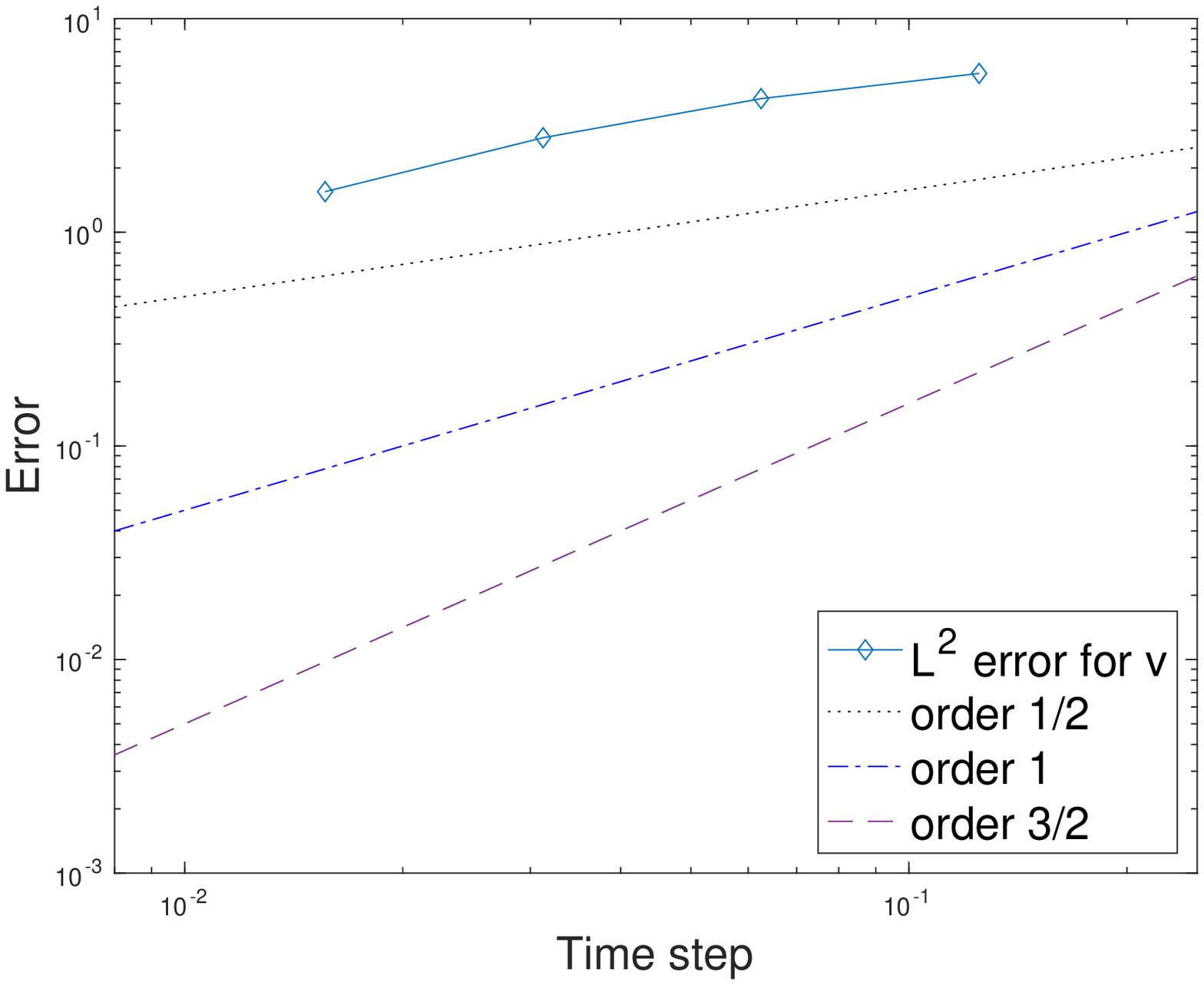}\label{fig--2}}
  \caption{{\bf (Example \ref{ex-sv})} Rates of convergence of the $\big(1, \frac{1}{4} \big)-$scheme with $\sigma(v) = \frac{3}{2} v$ and $F \equiv 0$.}
\label{conv-F-sv}
\end{figure}

\end{example}

\del{
\begin{example}\label{exm-alp6}
{\blue{In Example \ref{ex-additive-s}, take $\sigma(x) = 2(x+1)(x+2), \ x \in [0, 1]$. The errors are computed via the $(\widehat{\alpha}, 0)-$scheme. We compare the Fig. \ref{additive} with Fig. \ref{no+bou} to observe the reduction of convergence order for $\nabla u$ to $\cO(k)$ in plot {\rm (B)}; in plots {\rm (A)} and {\rm (C)}, $\bL^2$-errors in $u$ and $v$ remain the same as $\cO(k^{3/2})$ and $\cO(k)$, respectively.
}}
\smallskip

\begin{figure}[h!]
  \subfloat[$\bL^2$-error for $u$]{\includegraphics[width=0.33\textwidth]{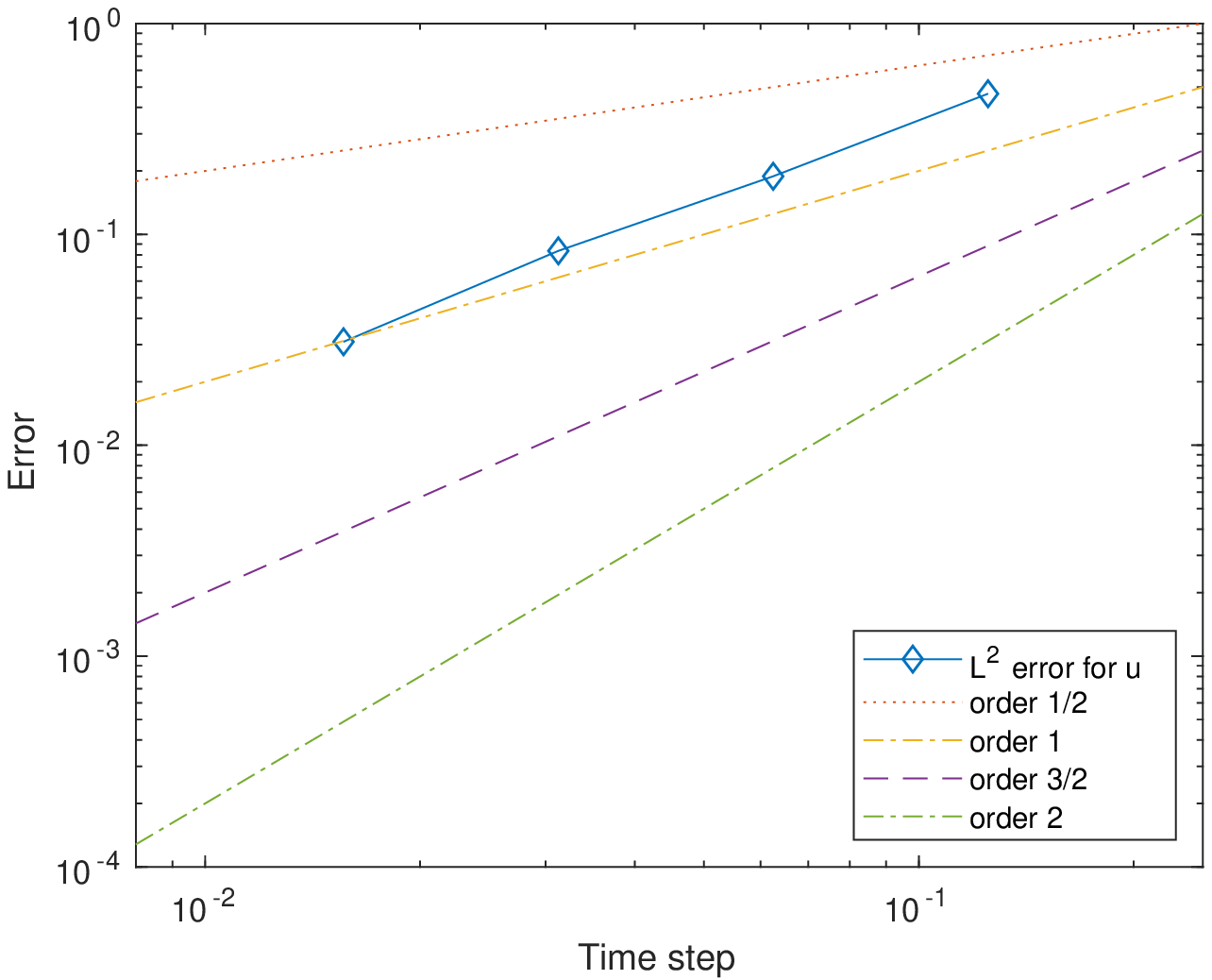}\label{fig:-x--1}}
  \hfill
  \subfloat[$\bL^2$-error for $\nabla u$]{\includegraphics[width=0.33\textwidth]{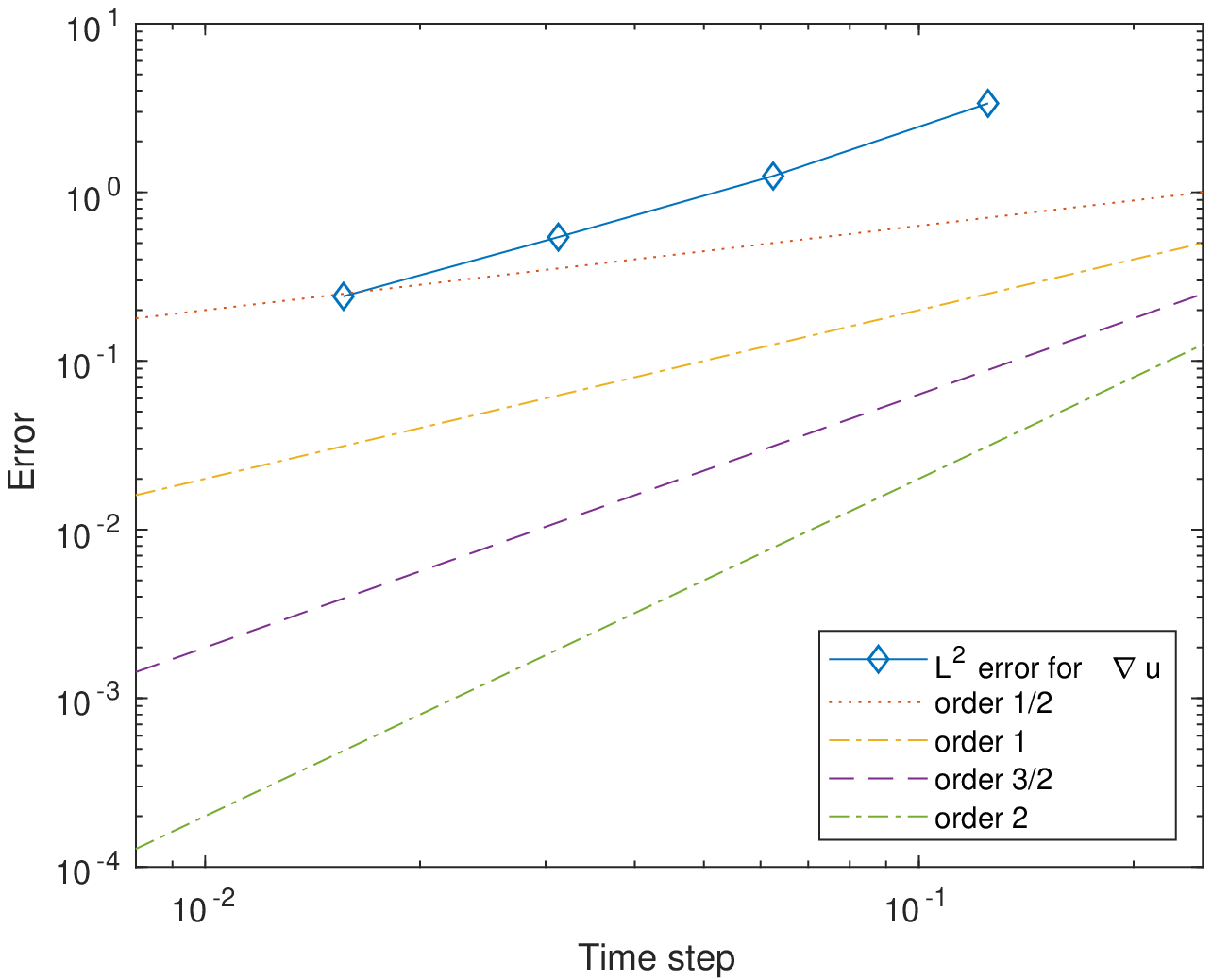}\label{fig:-x--2}}
  \hfill
  \subfloat[$\bL^2$-error for $v$]{\includegraphics[width=0.33\textwidth]{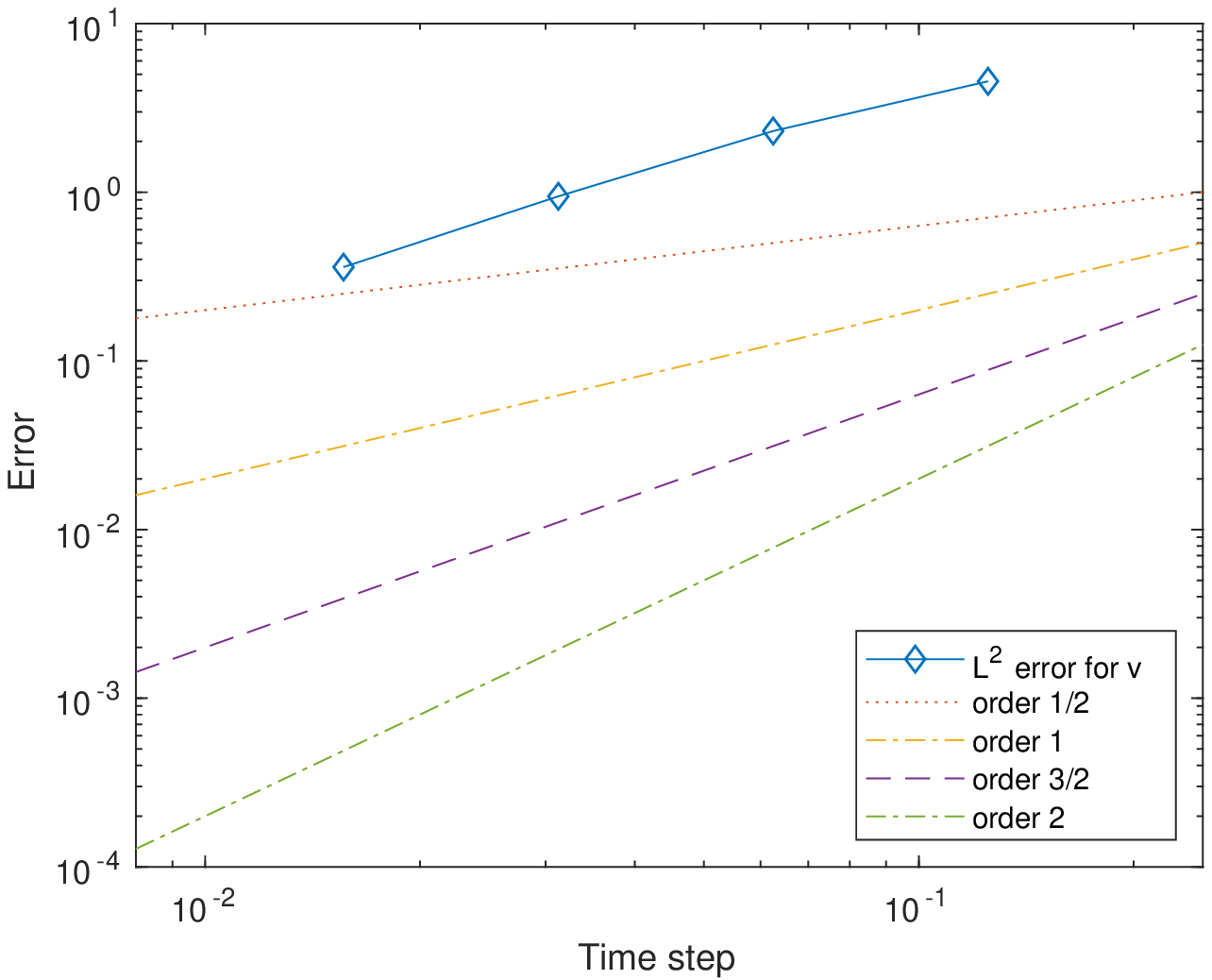}}\label{fig:-x--3}
\label{fig:--y-3} \caption{{\bf (Example \ref{exm-alp6})} {\blue{Temporal rates of convergence for $(\widehat{\alpha}, 0)-$scheme with $\sigma(x) = 2(x+1)(x+2)$; discretization parameters: $h=1/100, k=\{2^{-3}, \cdots, 2^{-6}\}$, ${\tt MC}=1200$.}}} 
\label{no+bou}
\end{figure}

\end{example}
}

\smallskip

\noindent
In the following example, we discuss four different cases where 
\begin{itemize}
\item[$(i)$]
$F \equiv F(u,v)$ is non-zero on the boundary, but Lipschitz and $\sigma \equiv \sigma(u)$;
\item[$(ii)$]
$F \equiv F(u,v)$ only H\"older continuous, and $\sigma \equiv \sigma(u)$;
\item[$(iii)$]
$F \equiv F(u,v)$ is same as $(i)$, and $\sigma \equiv \sigma(u, v)$ satisfying {\bf (A3)};
\item[$(iv)$] 
$F \equiv F(u,v)$ is same as $(ii)$, and $\sigma \equiv \sigma(u, v)$ satisfying {\bf (A3)}.
\end{itemize}
We observe that although $F \equiv F(u,v)$ violates {\bf (A3)} in $(ii)$, we still get improved convergence rates, but if $\sigma \equiv \sigma(u, v)$, we get the convergence order $\cO(k^{1/2})$ as shown in \eqref{eq-5.32} of  Theorem \ref{lem:scheme1:con}.
\begin{example}\label{exm-alp7}

We consider the following cases:
\begin{itemize}
\item[$(i)$]
$\sigma(u) = u$ and $F(u, v) = \cos(u)+2v$;
\item[$(ii)$]
$\sigma(u) = u$ and $F(u, v) = \sqrt{u} + \sqrt{v+2}$;
\item[$(iii)$]
$\sigma(u, v) = \frac{u}{1+u^2} + v$ and $F(u, v) = \cos(u)+2v$;
\item[$(iv)$]
$\sigma(u, v) = \frac{u}{1+u^2} + v$ and $F(u, v) = \sqrt{u} + \sqrt{v+2}$;
\end{itemize}
The errors are computed via the $(\widehat{\alpha}, \beta)-$scheme with $\widehat{\alpha}=1$ for $\beta=1/4 :$ the plots {\rm (A)--(C)} for the problem $(i)$ evidence the convergence order ${\mathcal O}(k^{3/2})$ for $u, \nabla u$, and ${\mathcal O}(k)$ for $v$.  We observe the same convergence rates for the problem $(ii)$ despite the lack of Lipschitzness of $F$ which violates {\bf (A3)}; see plots {\rm (D)--(F)} of Fig. \ref{5.7}. The plots {\rm (G)--(I)} of $\bL^2$-errors in $u, \nabla u$ and $v$, respectively, for the problem $(iii)$ and evidence the convergence order ${\mathcal O}(k^{1/2})$ as shown in \eqref{eq-5.32} of Theorem \ref{lem:scheme1:con}. We observe the same order of convergence for the problem $(iv)$; see plots {\rm (J)--(L)} of Fig. \ref{5.7}. Thus, the above two examples verify that the estimate \eqref{eq-5.32} is sharp in the case of diffusion $\sigma \equiv \sigma(u,v)$.

\smallskip

\begin{figure}[h!]
  \subfloat[$\bL^2$-error for $u$ in $(i)$]{\includegraphics[width=0.24\textwidth]{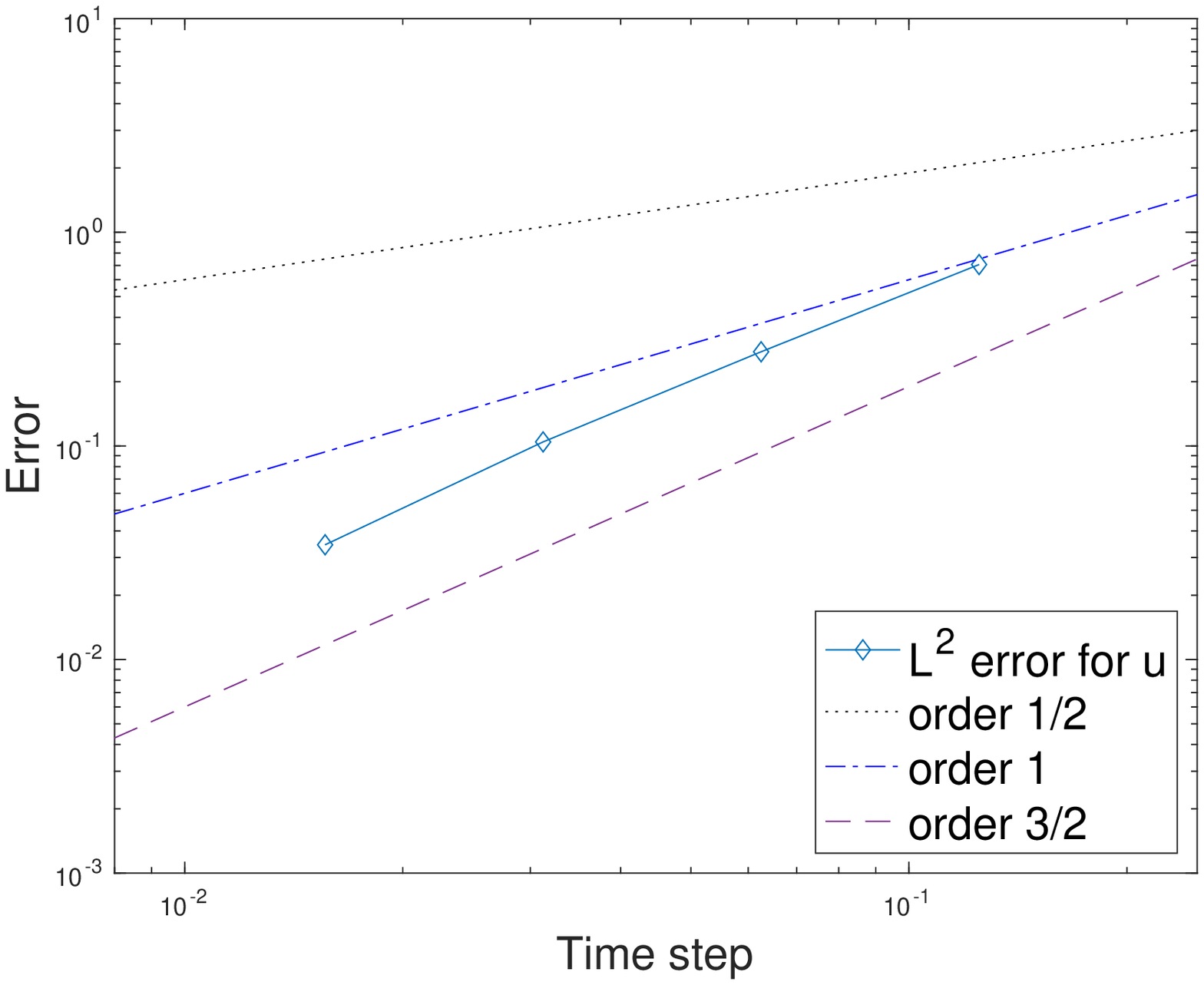}\label{fig:x_-1}}
  \hfill
  \subfloat[$\bL^2$-error for $\nabla u$ in $(i)$]{\includegraphics[width=0.24\textwidth]{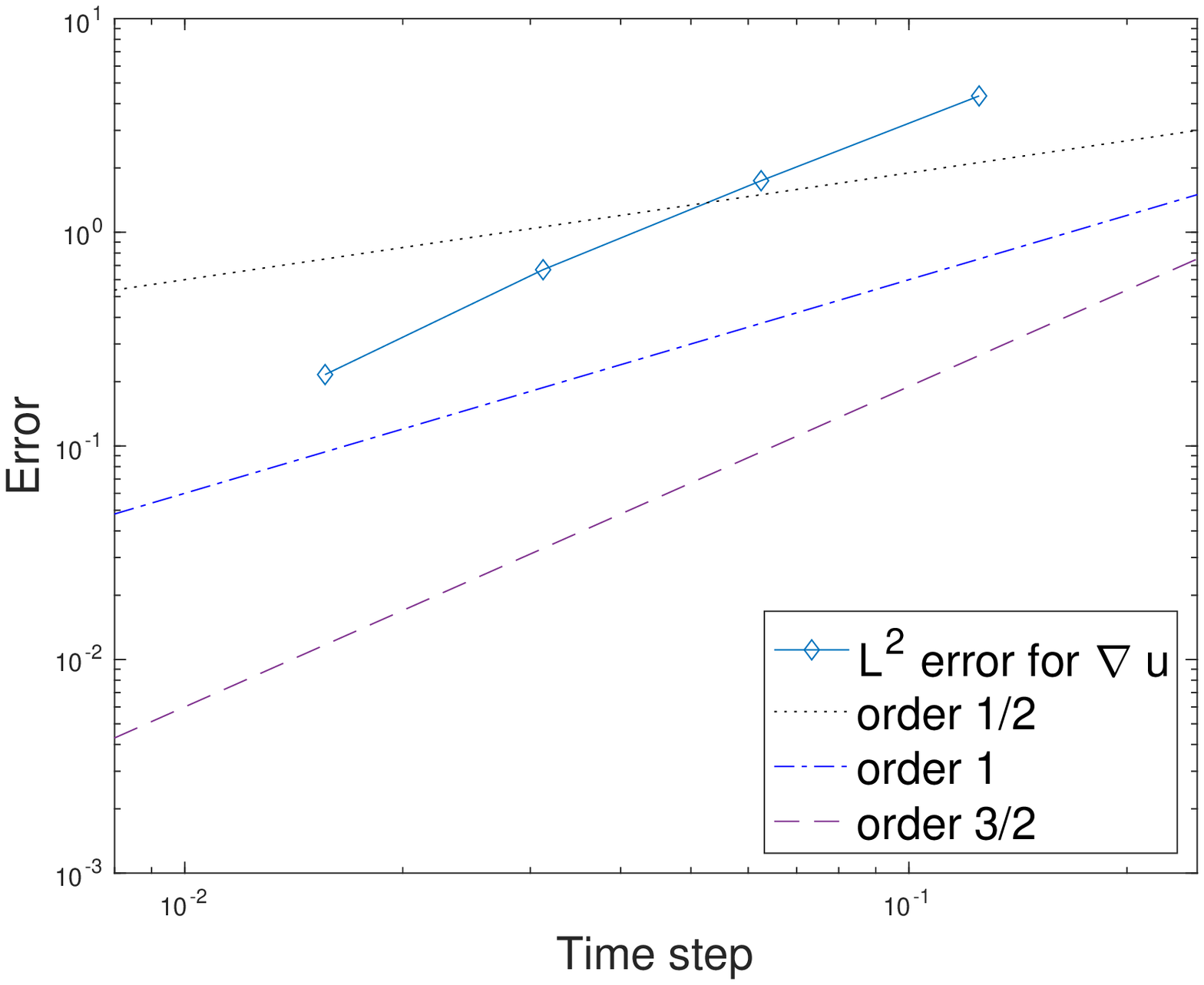}\label{fig:x_-2}}
  \hfill
  \subfloat[$\bL^2$-error for $v$ in $(i)$]{\includegraphics[width=0.24\textwidth]{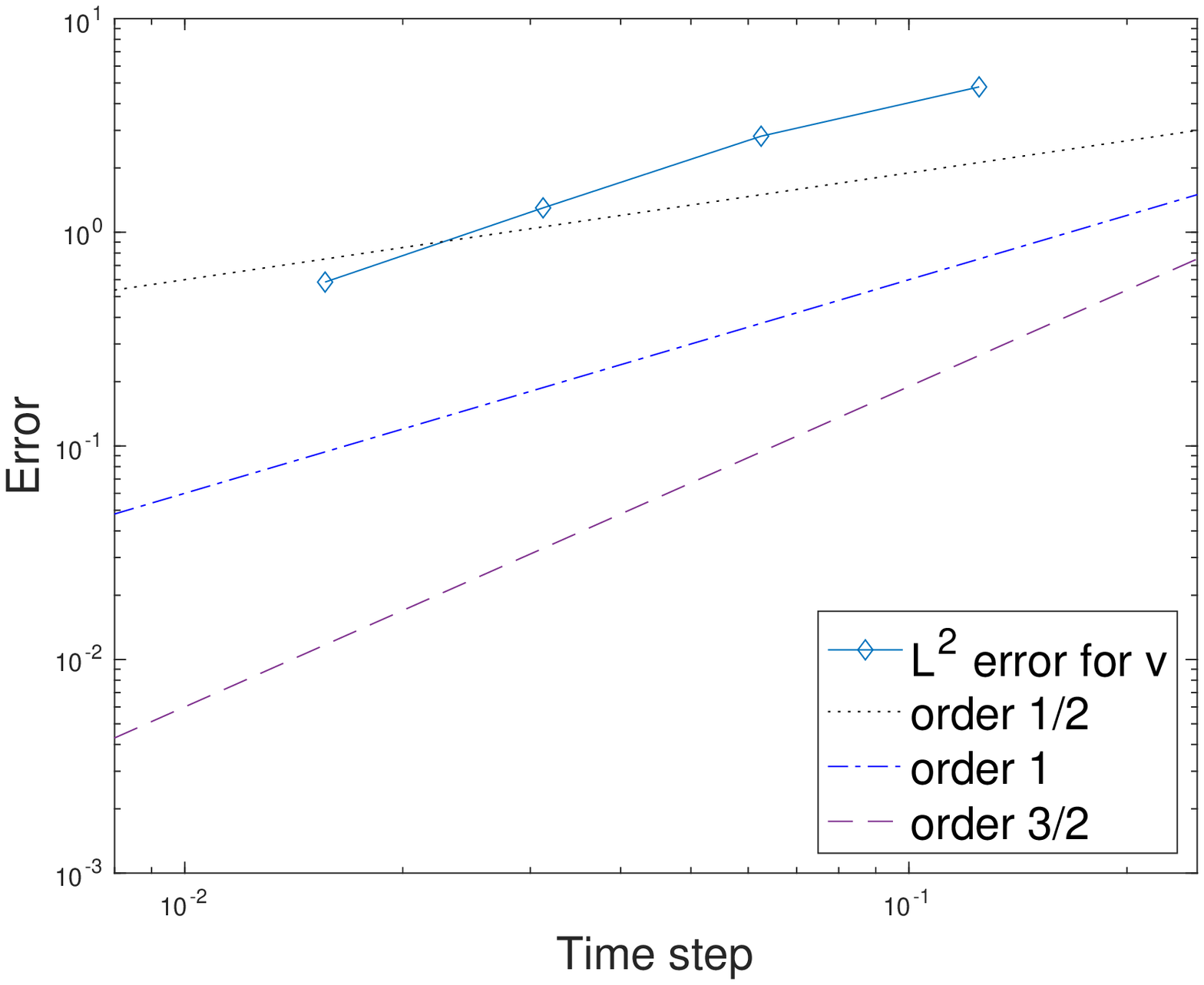}}\label{fig:x_-3}
  \hfill
  \subfloat[$\bL^2$-error for $u$ in $(ii)$]{\includegraphics[width=0.24\textwidth]{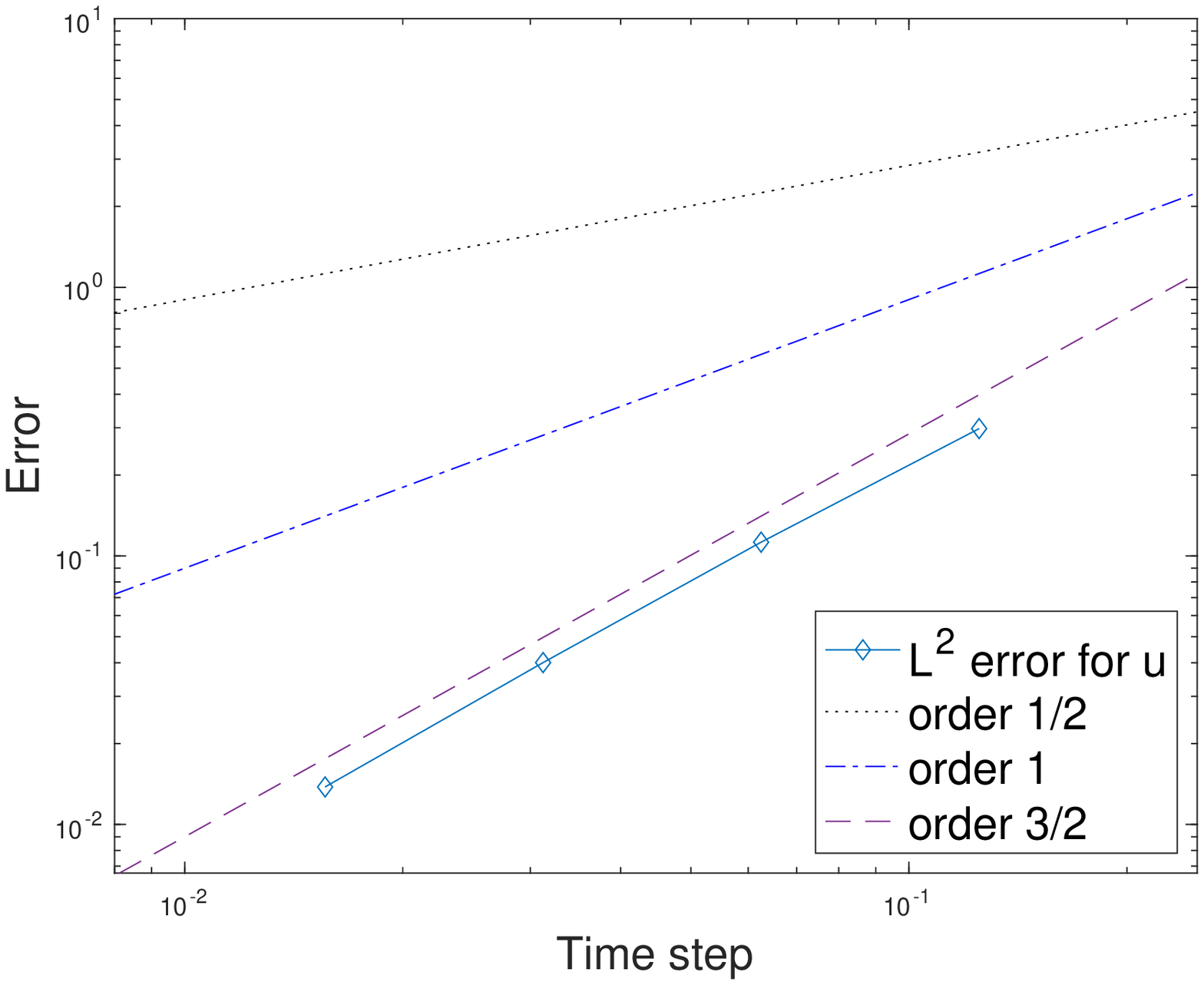}}\label{figu:x_-1}
  \par
  \subfloat[$\bL^2$-error for $\nabla u$ in $(ii)$]{\includegraphics[width=0.24\textwidth]{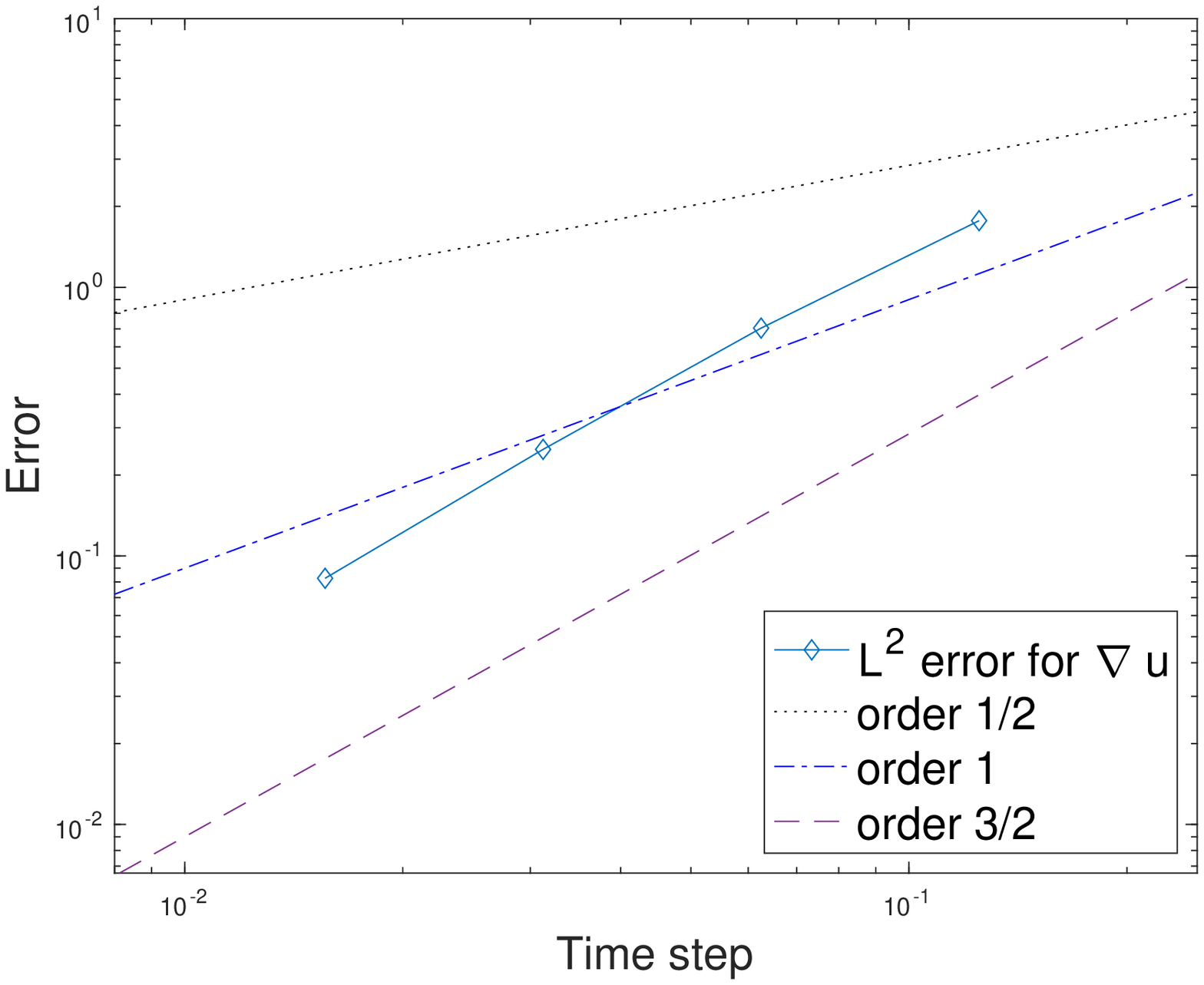}}\label{figu:x_-2}
  \hfill
  \subfloat[$\bL^2$-error for $v$ in $(ii)$]{\includegraphics[width=0.24\textwidth]{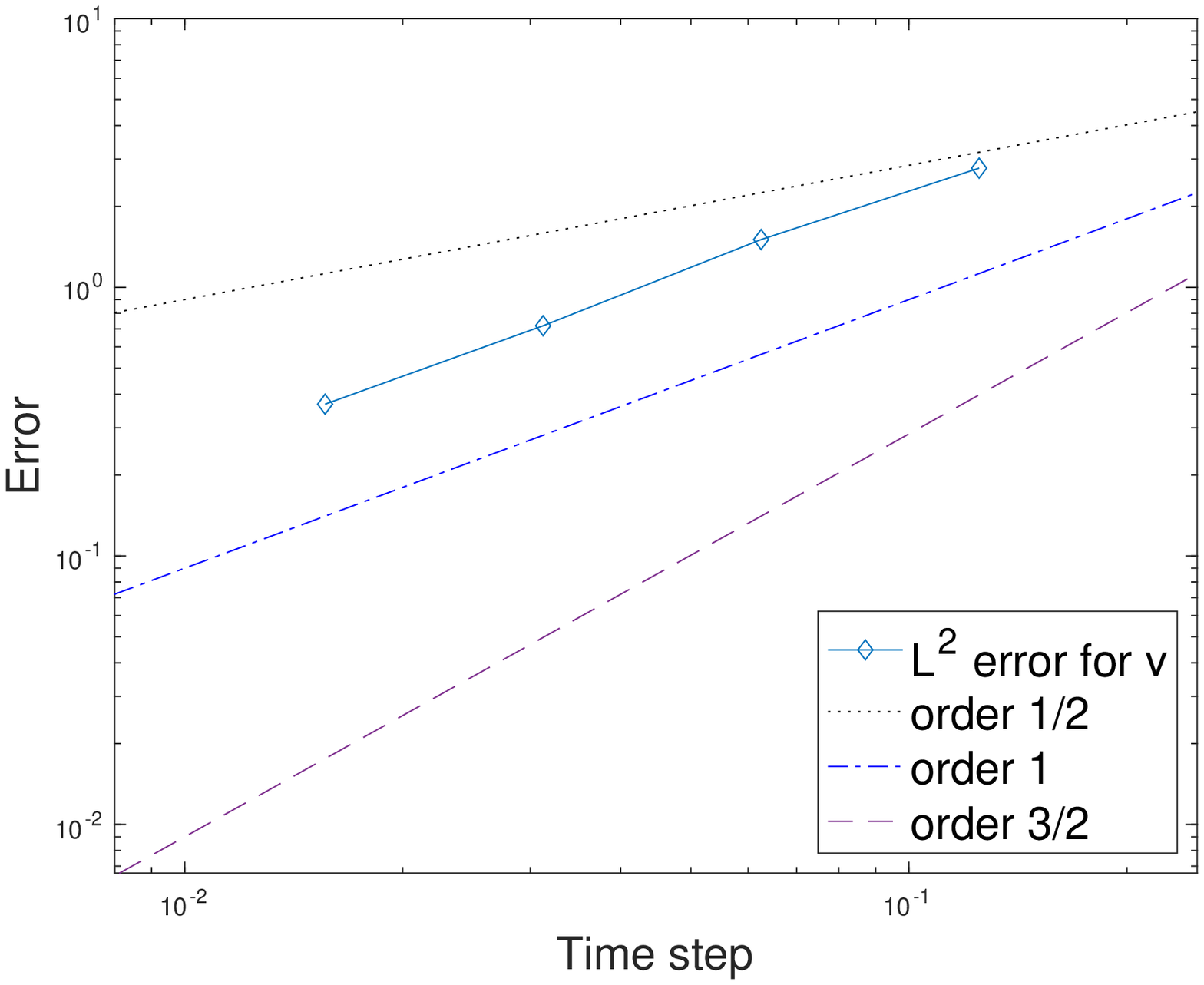}}\label{figu:x_-1}
  \hfill
    \subfloat[$\bL^2$-error for $u$ in $(iii)$]{\includegraphics[width=0.24\textwidth]{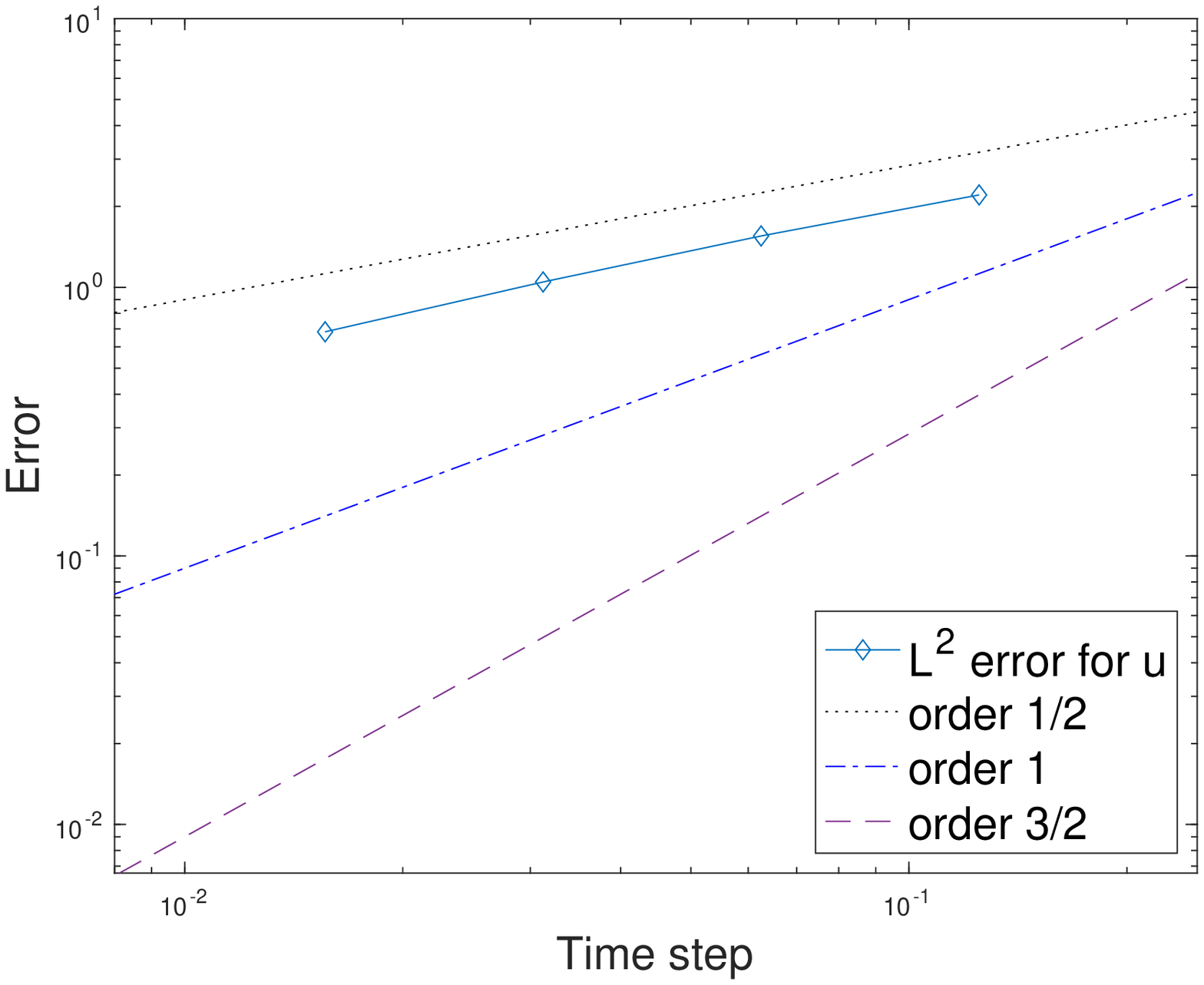}\label{fig:x_-1-}}
  \hfill
  \subfloat[$\bL^2$-error for $\nabla u$ in $(iii)$]{\includegraphics[width=0.24\textwidth]{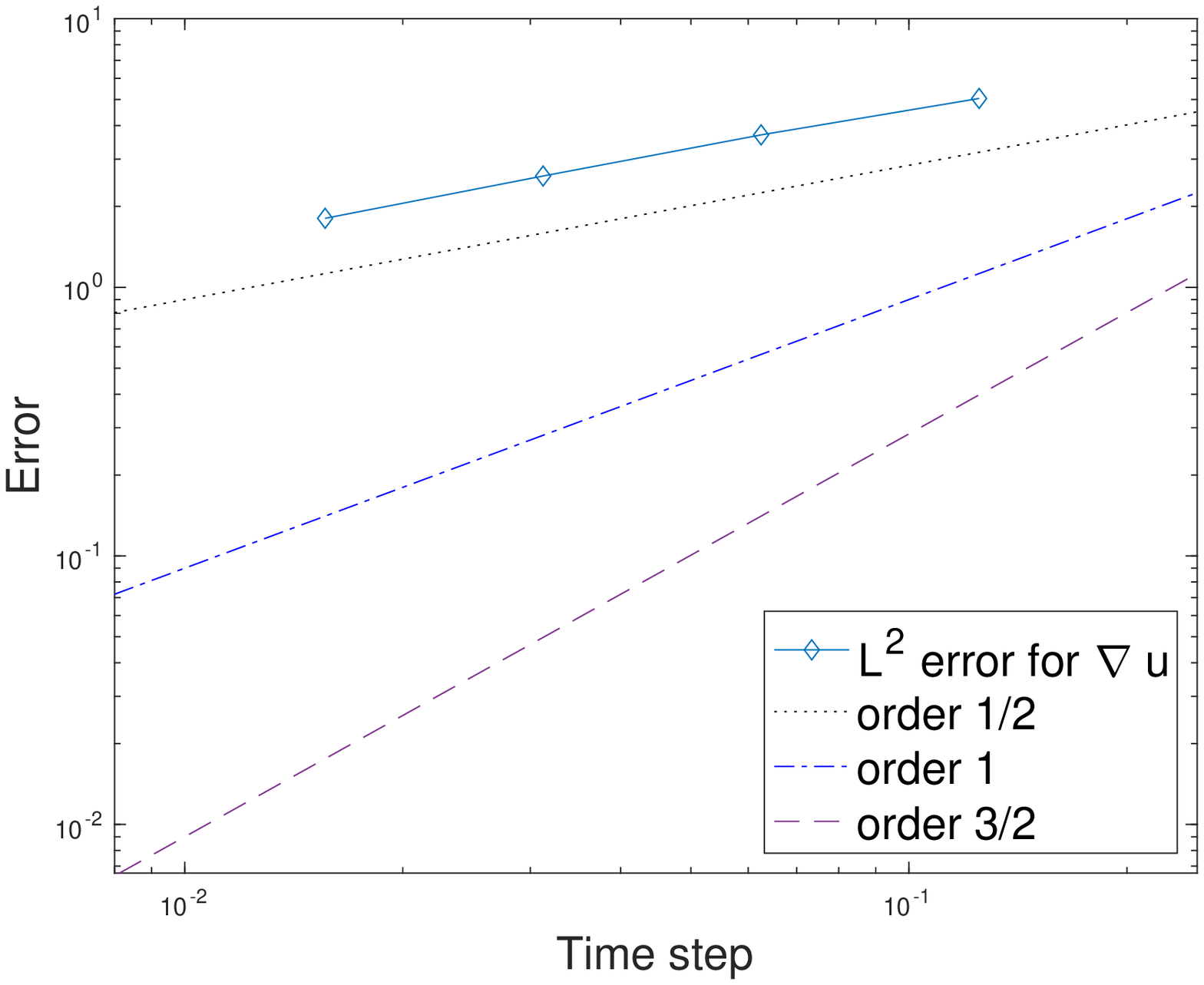}\label{fig:x_-2-}}
  \par
  \subfloat[$\bL^2$-error for $v$ in $(iii)$]{\includegraphics[width=0.24\textwidth]{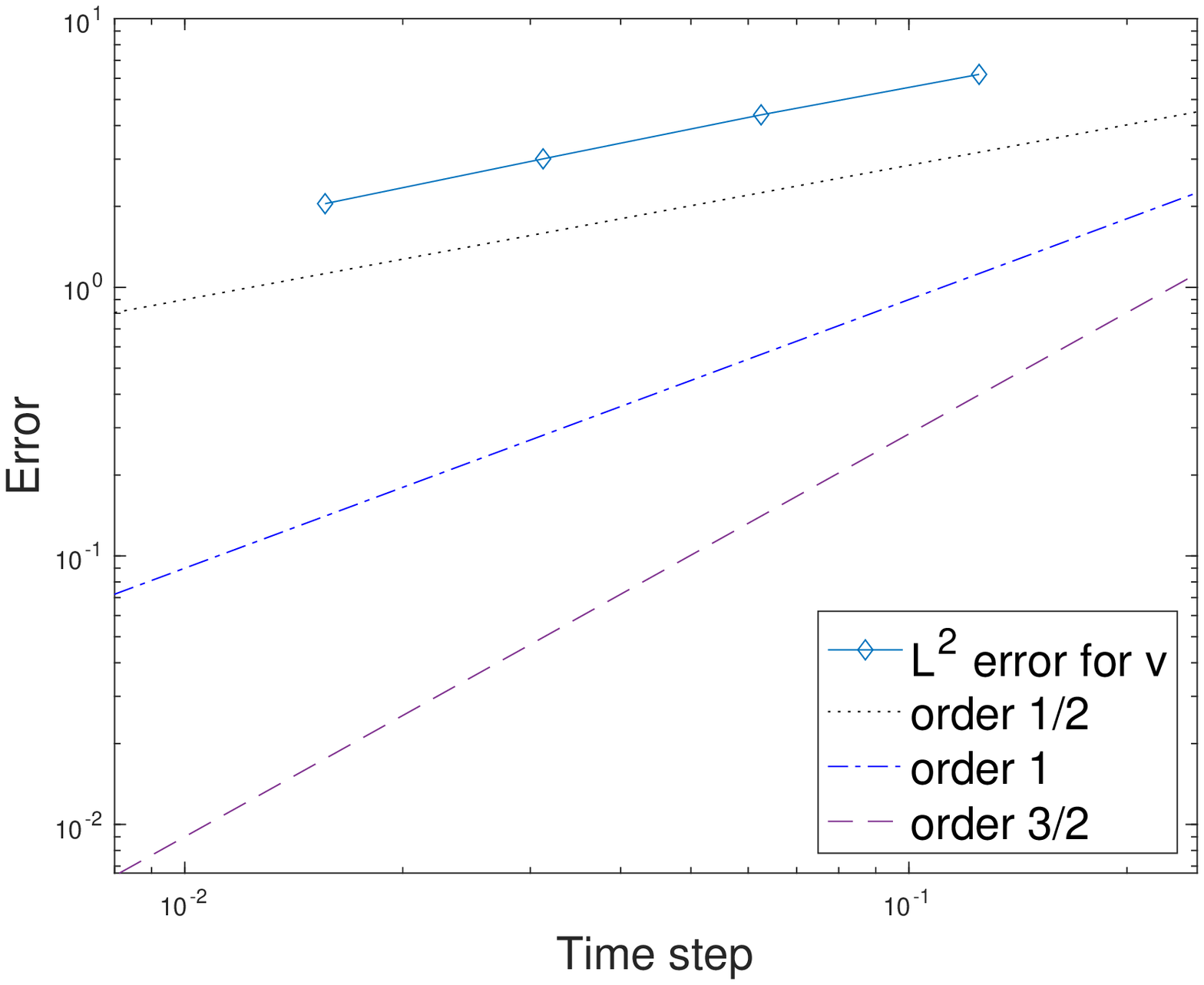}}\label{fig:x_-3-}
  \hfill
  \subfloat[$\bL^2$-error for $u$ in $(iv)$]{\includegraphics[width=0.24\textwidth]{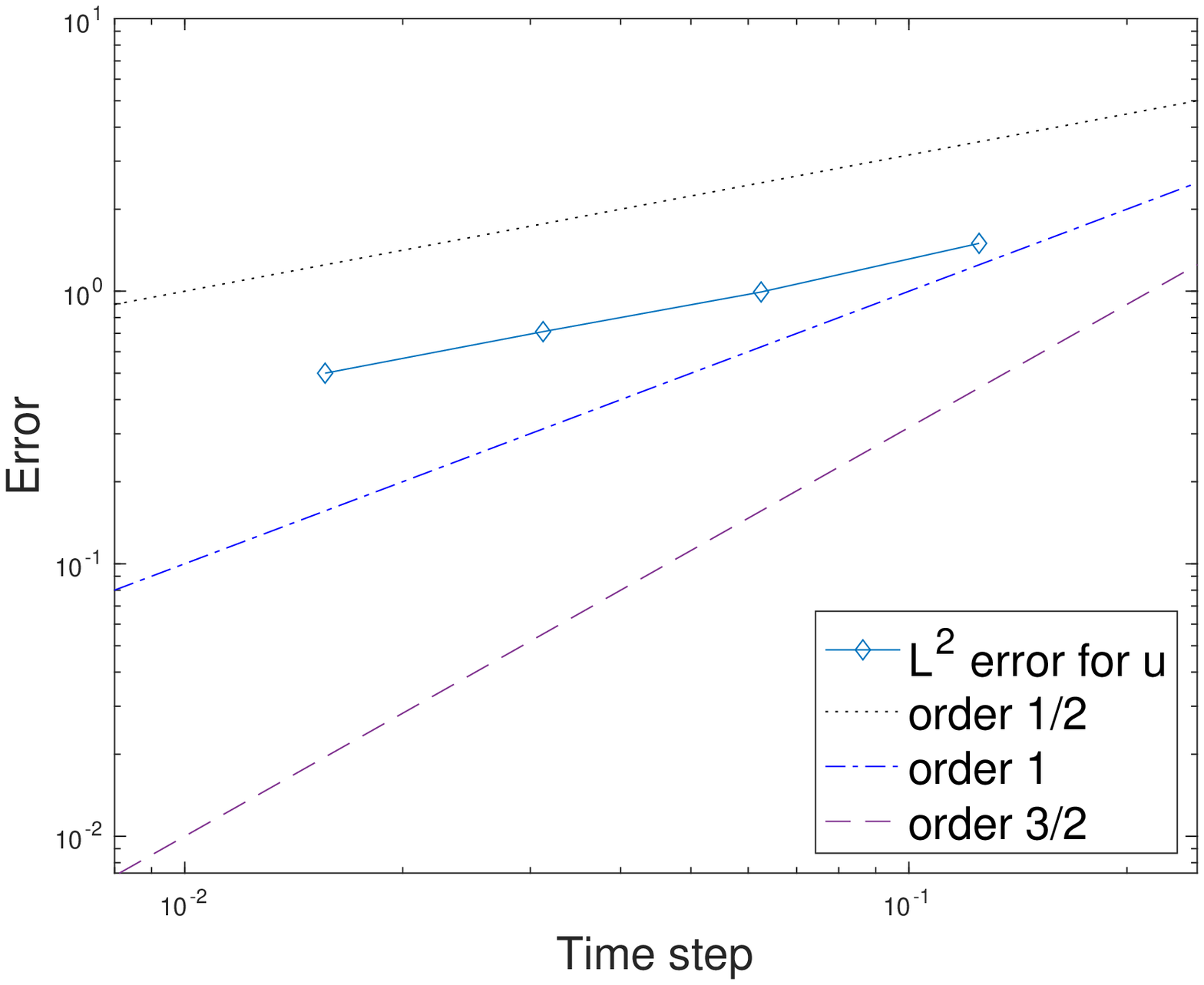}}\label{fi:y_-1-}
  \hfill
  \subfloat[$\bL^2$-error for $\nabla u$ in $(iv)$]{\includegraphics[width=0.24\textwidth]{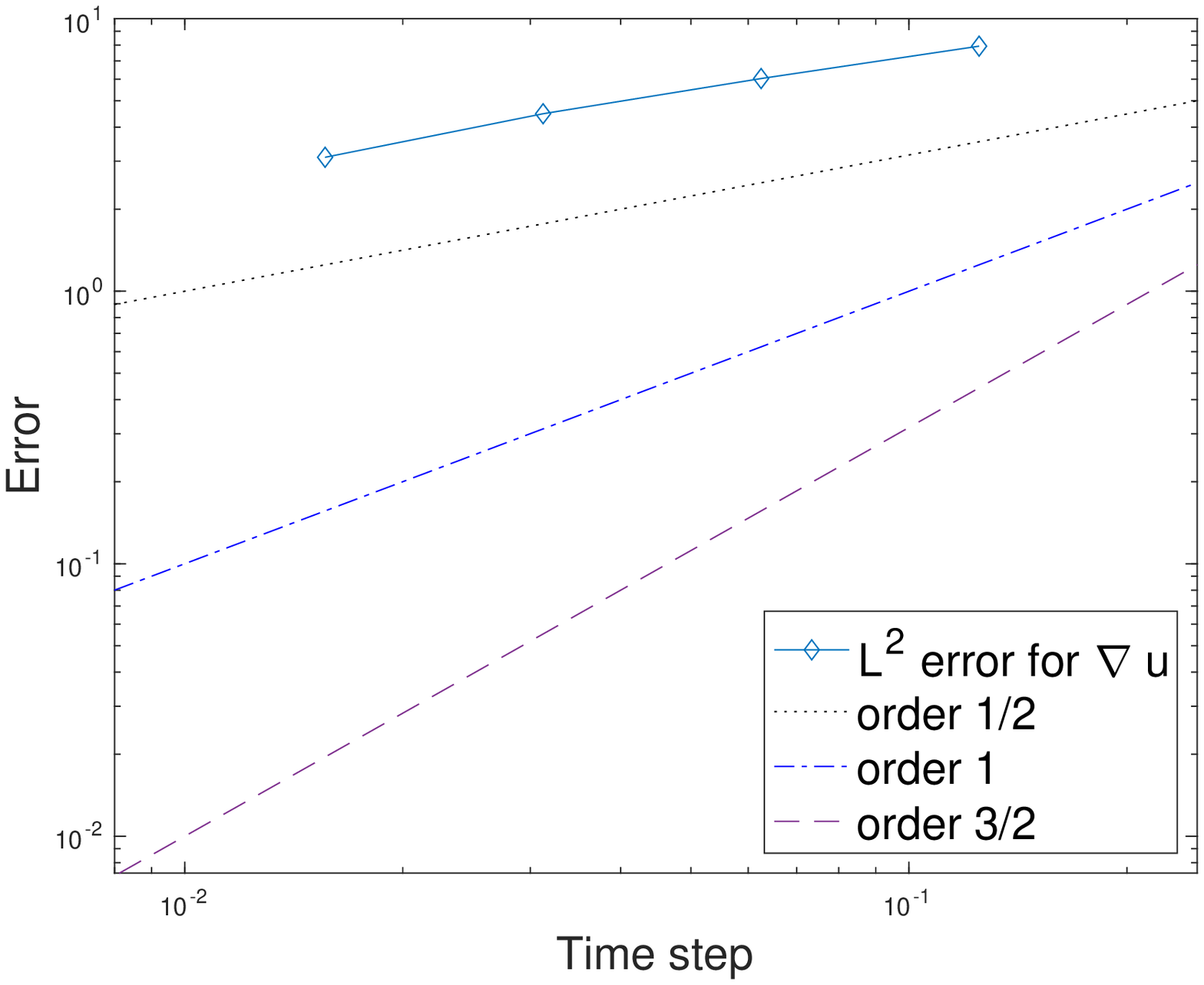}}\label{fi:y_-2-}
  \hfill
  \subfloat[$\bL^2$-error for $v$ in $(iv)$]{\includegraphics[width=0.24\textwidth]{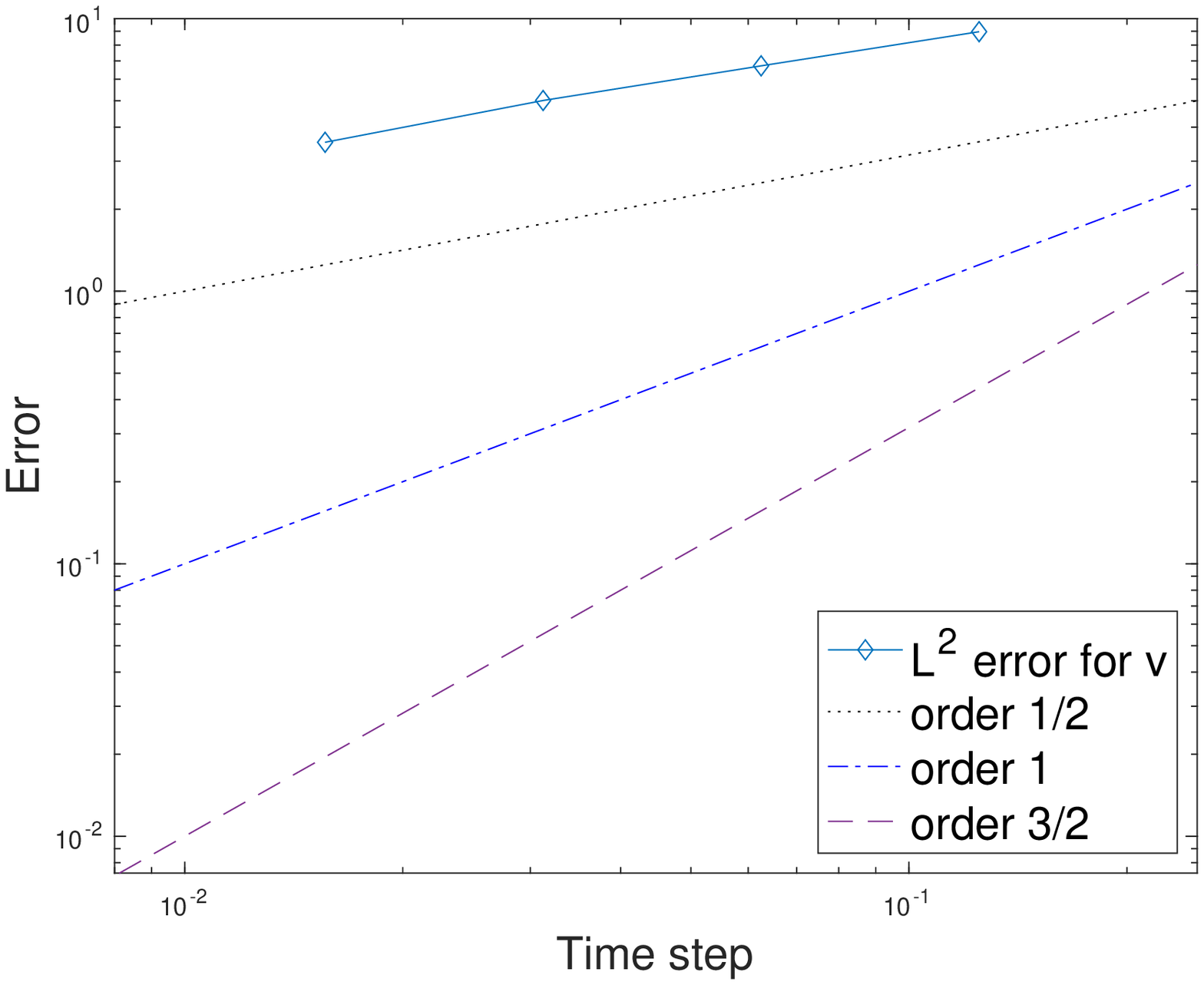}}\label{fi:y_-3-}
\label{fig:--y-3} \caption{{\bf (Example \ref{exm-alp7})} Rates of convergence of the $\big(1, \frac{1}{4} \big)-$scheme.} 
\label{5.7}
\end{figure}

\end{example}

\smallskip

In the next example, we drop the assumption on $\sigma \equiv \sigma(u)$ to be Lipschitz and zero on the boundary to see which of these violations spot the reduction of the convergence order of scheme \eqref{scheme2:1}-\eqref{scheme2:2}.

\begin{example}\label{exm-alp8}

Let $F \equiv 0$. Consider the following cases:
\begin{itemize}
\item[$(i)$]
$\sigma(u) = \frac{1}{1+ u^2}$;
\item[$(ii)$]
$\sigma(u) = \sqrt{|u|}$.
\end{itemize}
In Fig. \ref{nln0}, the errors are computed via the scheme \eqref{scheme2:1}--\eqref{scheme2:2} with $\widehat{\alpha}=1.$ For problem $(i)$ (nonzero boundary), the plots {\rm (A)--(B)} for $\bL^2$-errors in $u, \nabla u$, respectively, show the convergence order ${\mathcal O}(k^{3/2})$ and the plot {\rm (C)} for $\bL^2$-error in $v$ shows ${\mathcal O}(k)$. For the problem $(ii)$ (non-Lipschitz), the convergence rates for $\bL^2$-errors in $u, \nabla u$ are reduced to ${\mathcal O}(k)$; see plots {\rm (D)--(E)}, but $\bL^2$-error in $v$ remains same as ${\mathcal O}(k)$; see plot {\rm (F)}. 

\smallskip

\begin{figure}[h!]
  \subfloat[$\bL^2$-error for $u$ in $(i)$]{\includegraphics[width=0.33\textwidth]{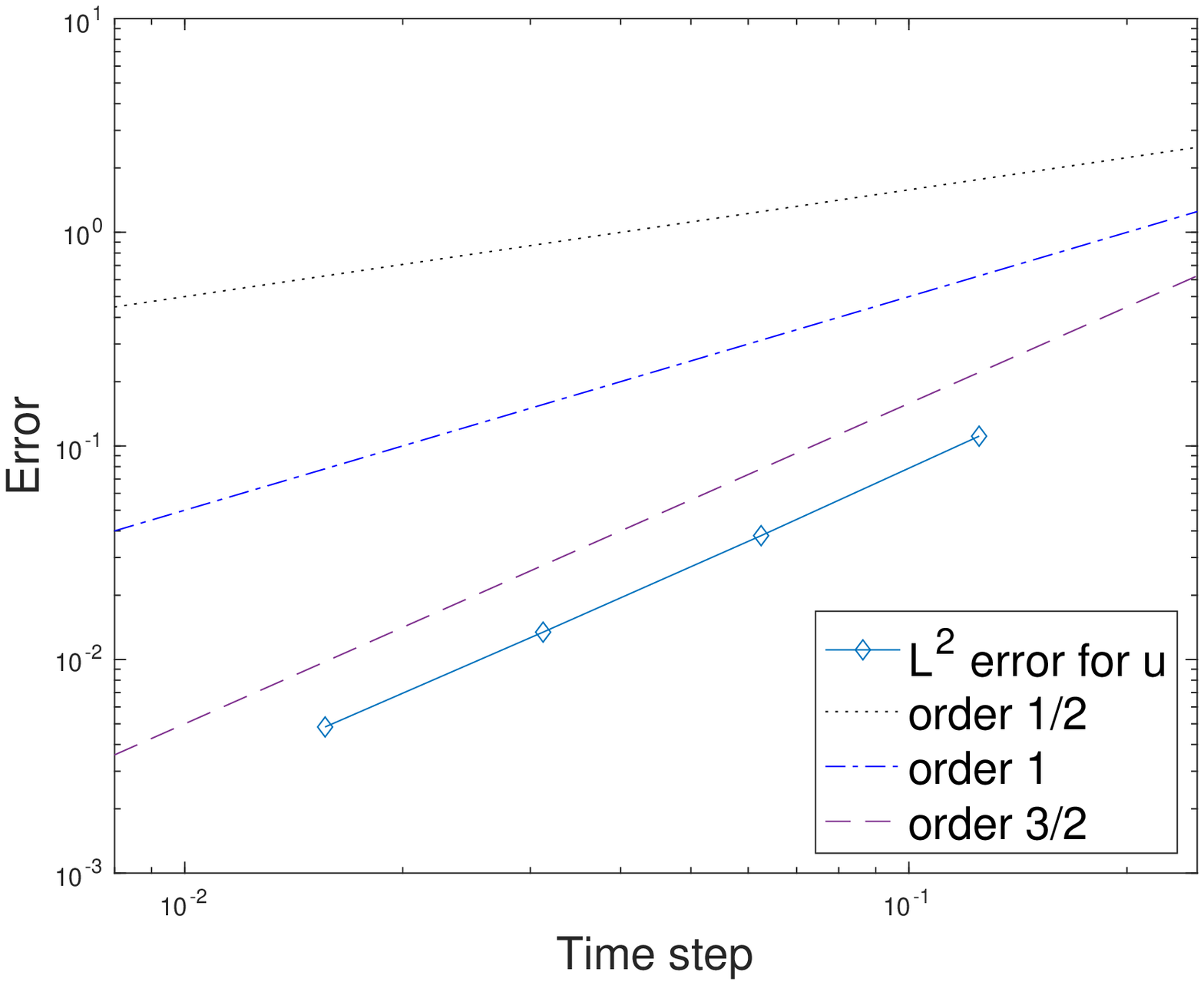}\label{f:x1}}
  \hfill
  \subfloat[$\bL^2$-error for $\nabla u$ in $(i)$]{\includegraphics[width=0.33\textwidth]{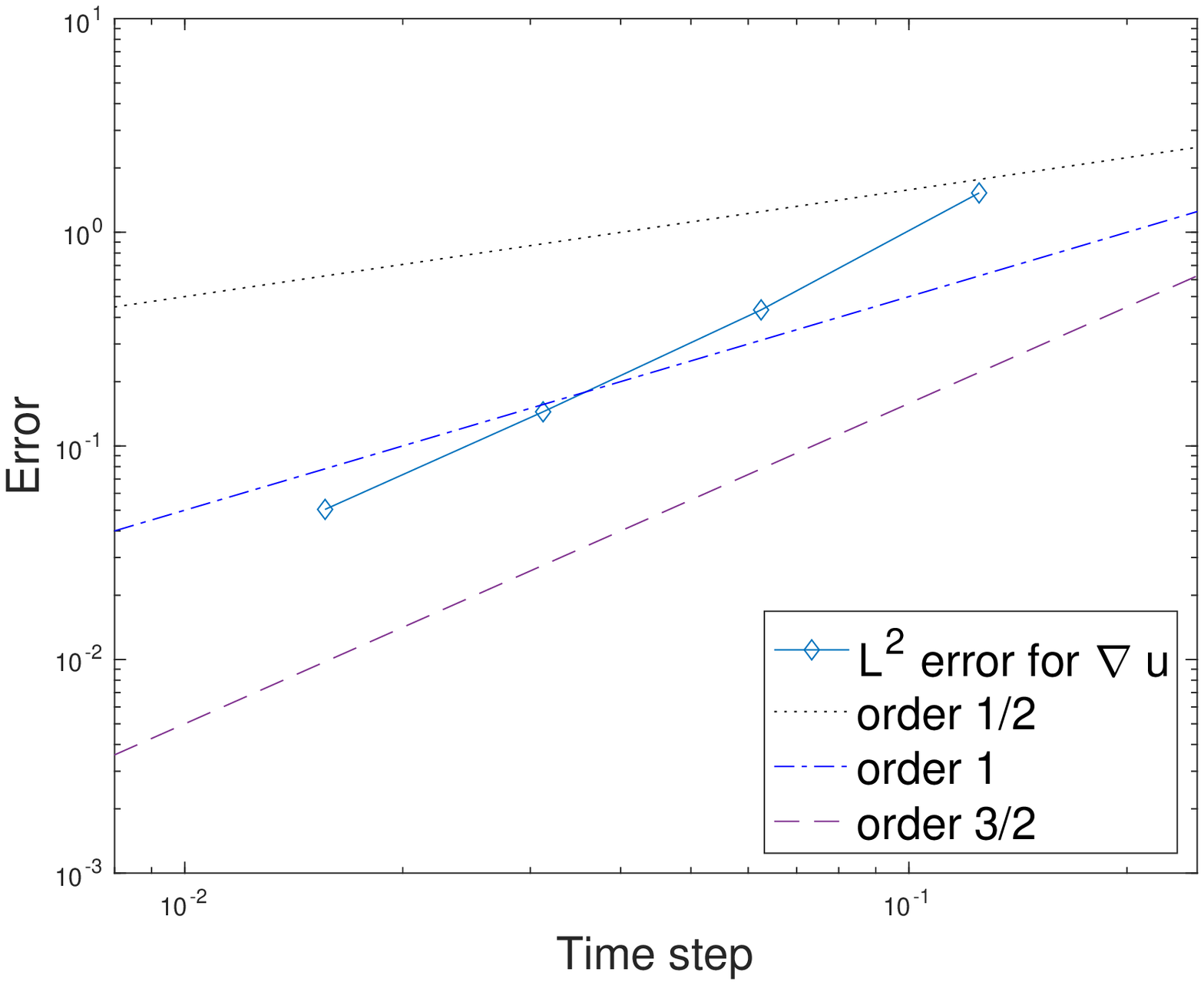}\label{f:x2}}
  \hfill
  \subfloat[$\bL^2$-error for $v$ in $(i)$]{\includegraphics[width=0.33\textwidth]{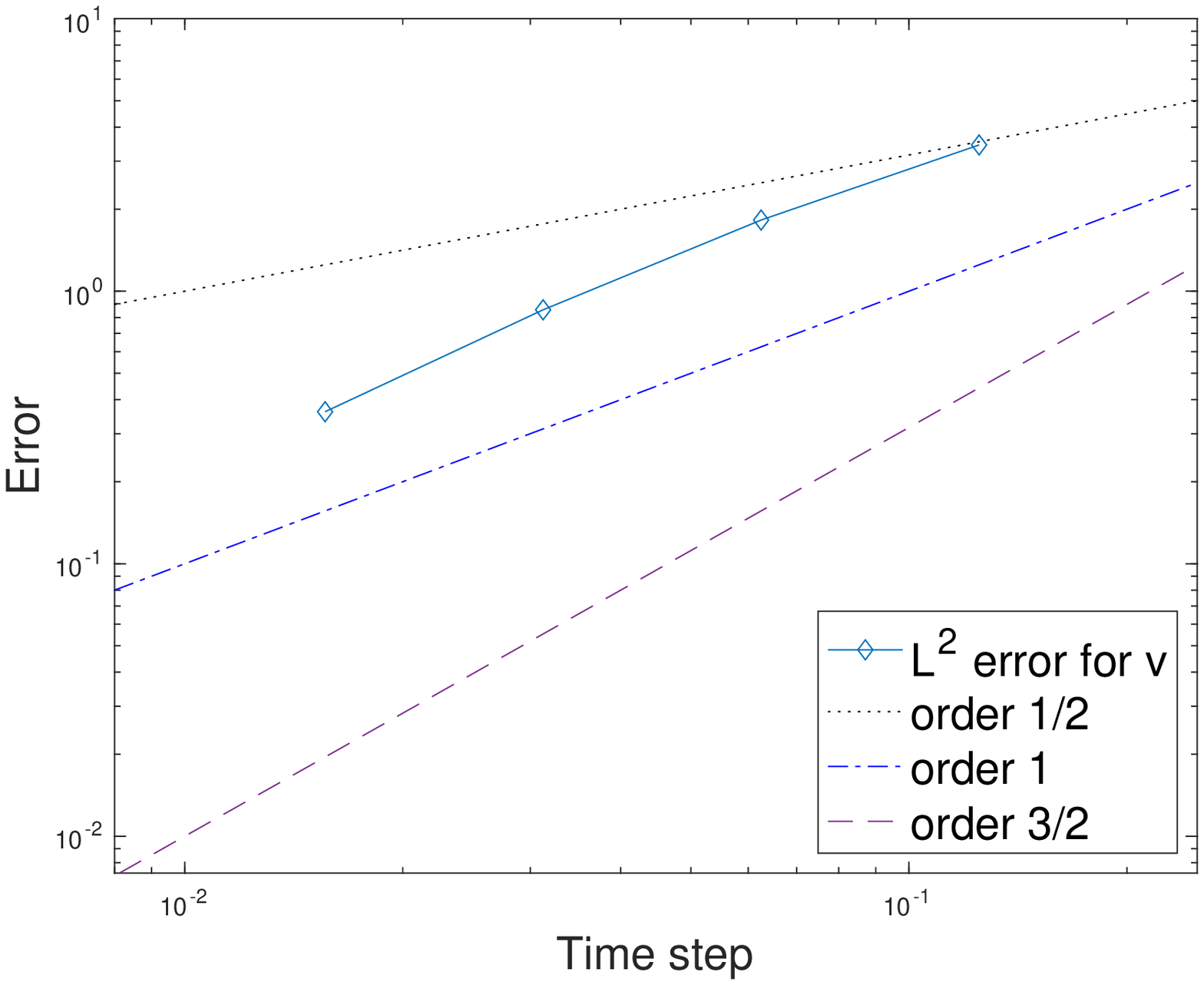}}\label{f:x3}
  \par
  \subfloat[$\bL^2$-error for $u$ in $(ii)$]{\includegraphics[width=0.33\textwidth]{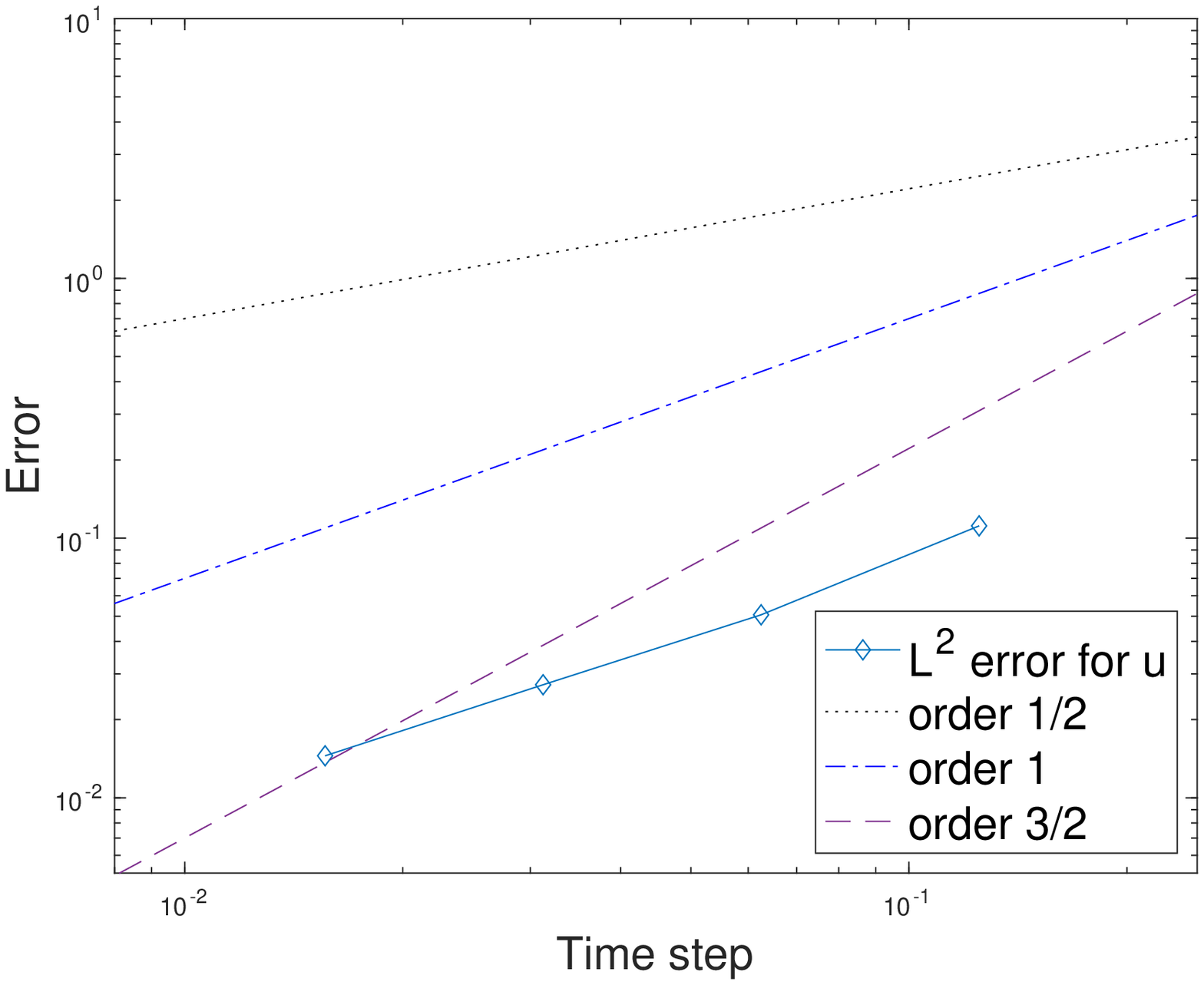}\label{f-:-x1}}
  \hfill
  \subfloat[$\bL^2$-error for $\nabla u$ in $(ii)$]{\includegraphics[width=0.33\textwidth]{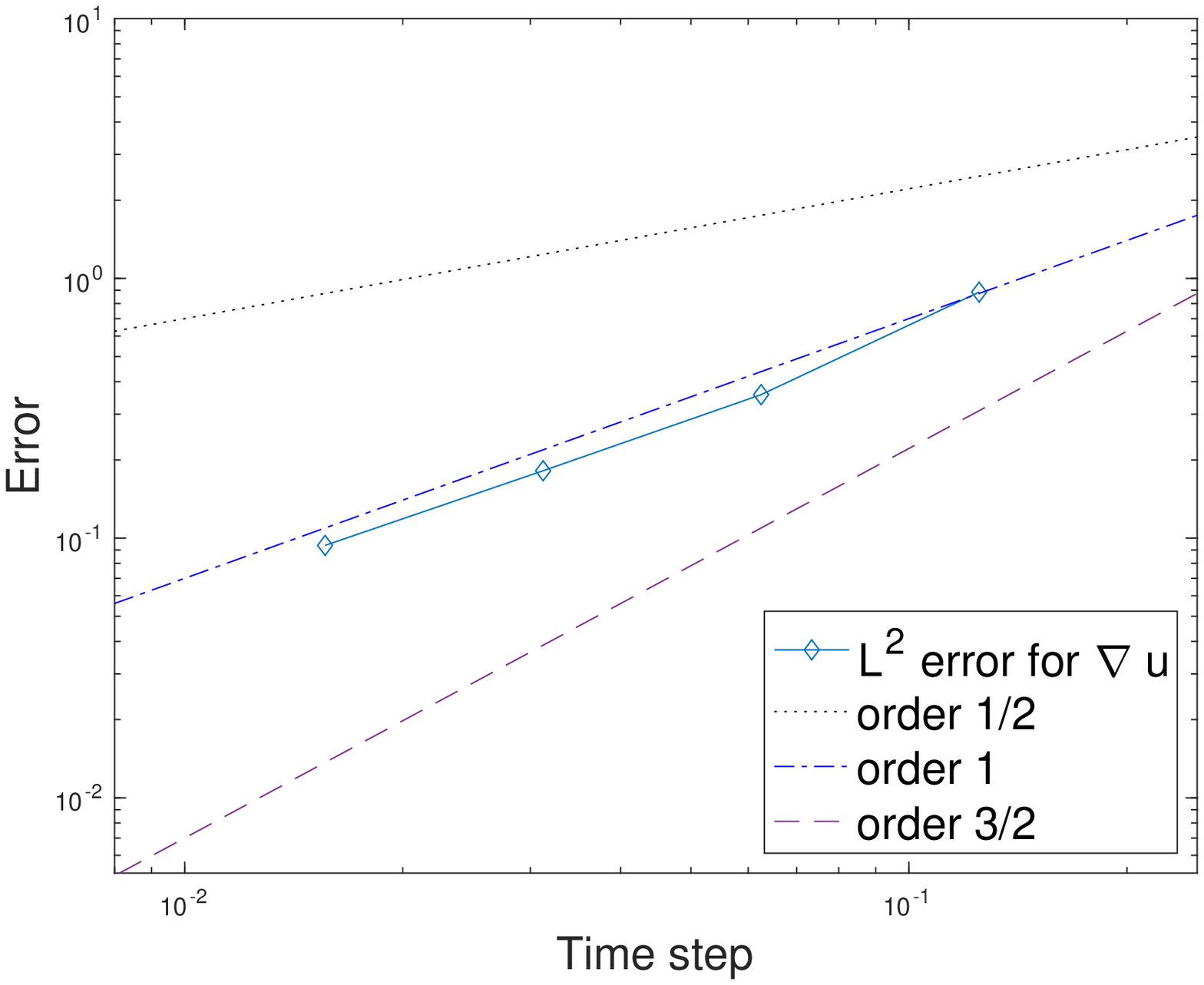}\label{f-:-x2}}
  \hfill
  \subfloat[$\bL^2$-error for $v$ in $(ii)$]{\includegraphics[width=0.33\textwidth]{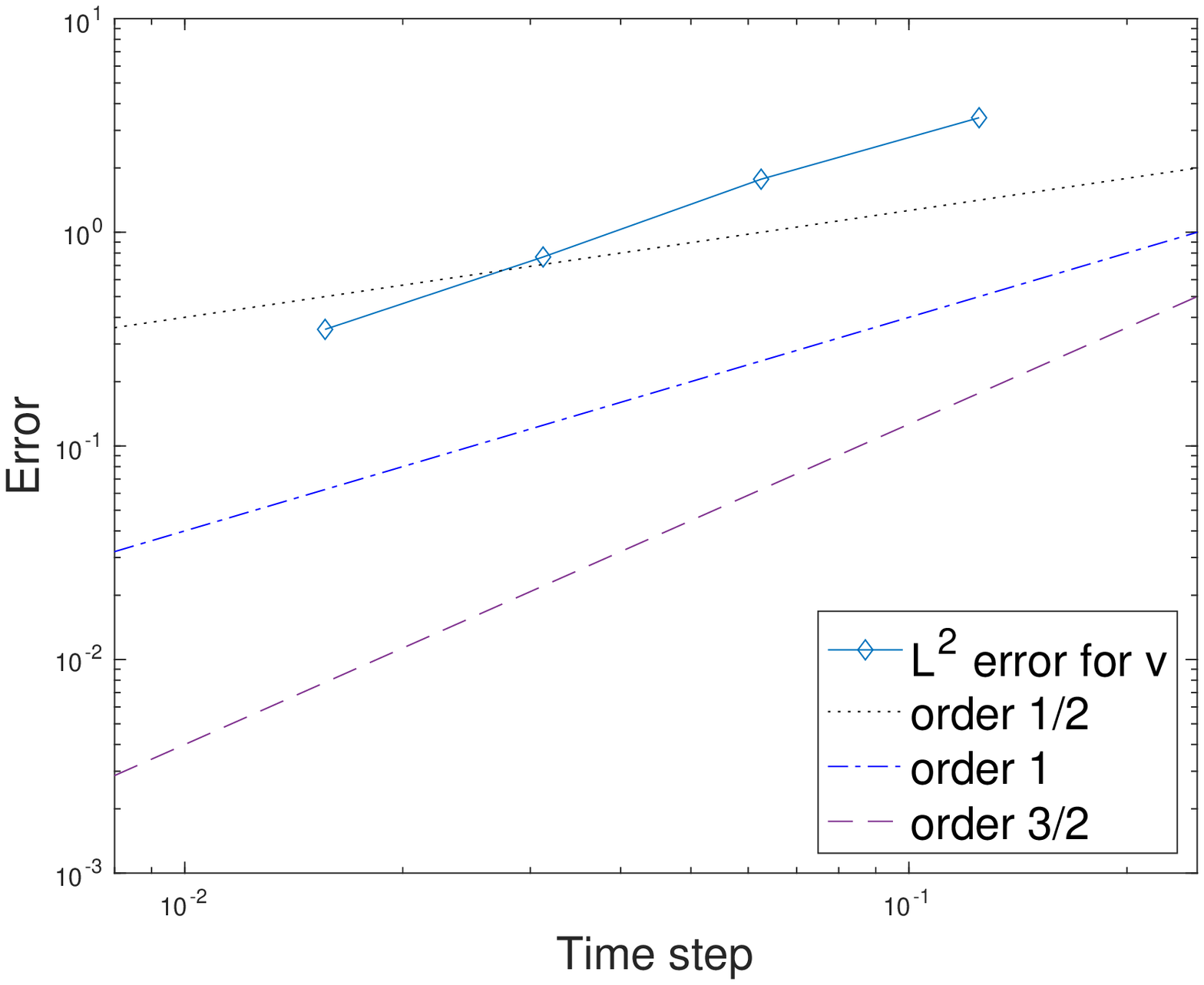}}\label{f-:-x3}
 \label{f:y3} \caption{{\bf (Example \ref{exm-alp8})} Rates of convergence of the the scheme \eqref{scheme2:1}--\eqref{scheme2:2} for $\widehat{\alpha}=1$.} 
\label{nln0}
\end{figure}

\end{example}

\subsection{Choice of $\beta$ and required number of {\tt MC}}

\begin{example}\label{exa-bet1}
Let $\cO=(0, 1)$, $T=0.5$, $A = -\Delta$, $F \equiv 0$, $\sigma(v)= 5v$. 
We compute $W$ on the mesh of size $k = 2^{-12}$ covering $[0, 0.5]$. 
In the $(\widehat{\alpha}, \beta)$-scheme, the term $\widetilde{u}^{n, \frac 12} = {u}^{n, \frac 12} + \beta k^{1+\beta} v^{n+\frac{1}{2}}$ involves $\beta$,  where the last term creates an additional numerical dissipation term  in \eqref{stoch-wave1:1a} to control discretization effect of the noise. For $\beta=0$ with $\sigma \equiv \sigma(u)$ and $F \equiv F(u)$, the scheme \eqref{scheme2:1}--\eqref{scheme2:2} is stable, but for general case we require $\beta \in (0, 1/2)$ for the stability of the $(\widehat{\alpha}, \beta)$-scheme; see Lemma \ref{lem:scheme1:stab}. For increased value of $\beta$, stabilization effect vanishes for small $k$.  Thus, a smaller choice of $\beta$ is preferred to have the stability of the scheme. The snapshot $(A)$ in Fig. \ref{beta-MC} shows for $\beta = 0,  \frac 14, \frac 12, \frac 34, 1$, that at least ${\tt MC}= 400, 600, 800, 1000,1400$, are needed to have a steady of the energy $\mathcal{E}$ at time $T=0.5$. The snapshot $(B)$ evidence a higher number of {\tt MC} as we increase $\beta$ to have a steady energy curve.
\begin{figure}[h!]
  \centering
  \subfloat[Energy for different $\beta$]{\includegraphics[width=0.33\textwidth]{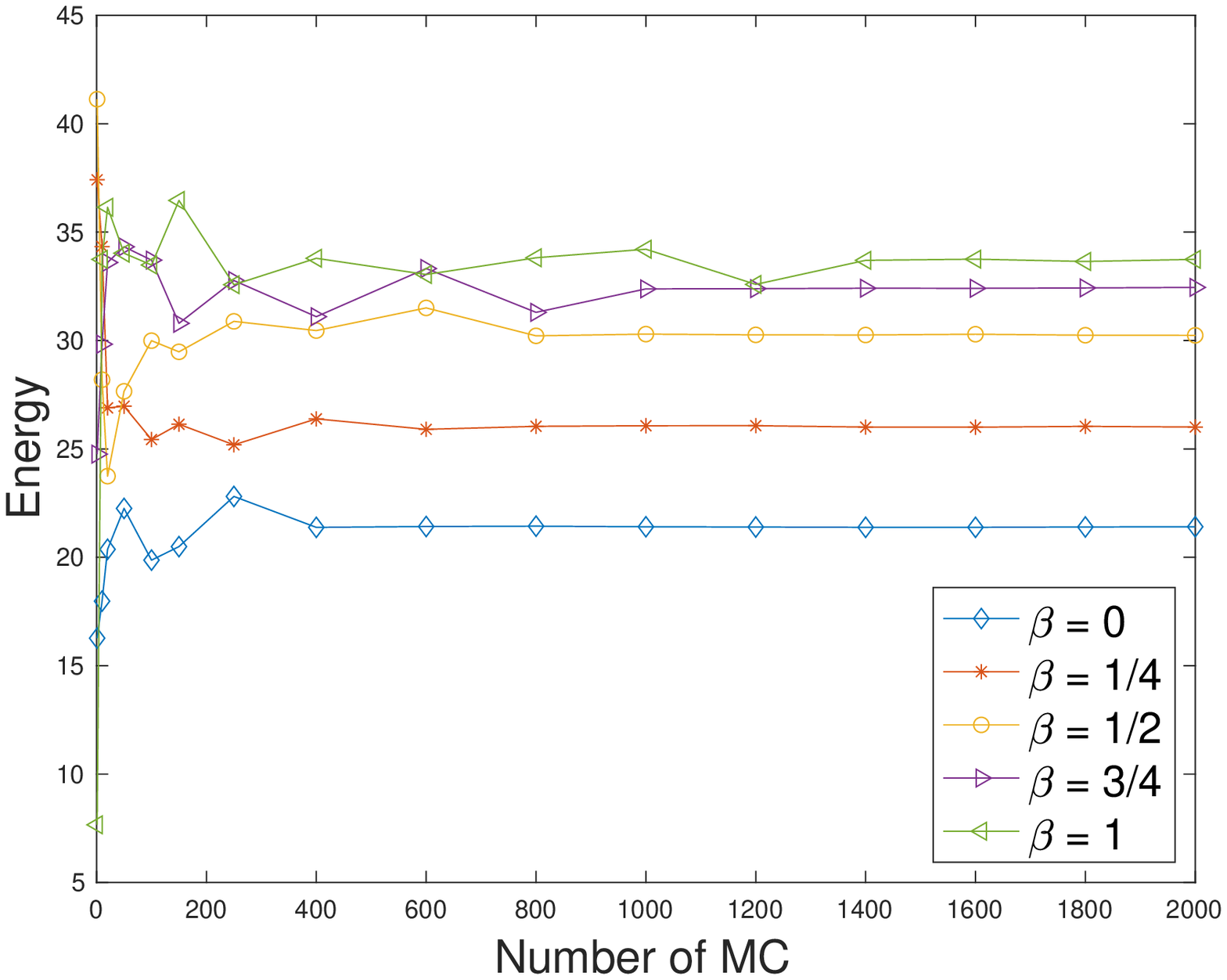}\label{fig:f1}}
  \subfloat[$\beta$ vs number of {\tt MC}]{\includegraphics[width=0.33\textwidth]{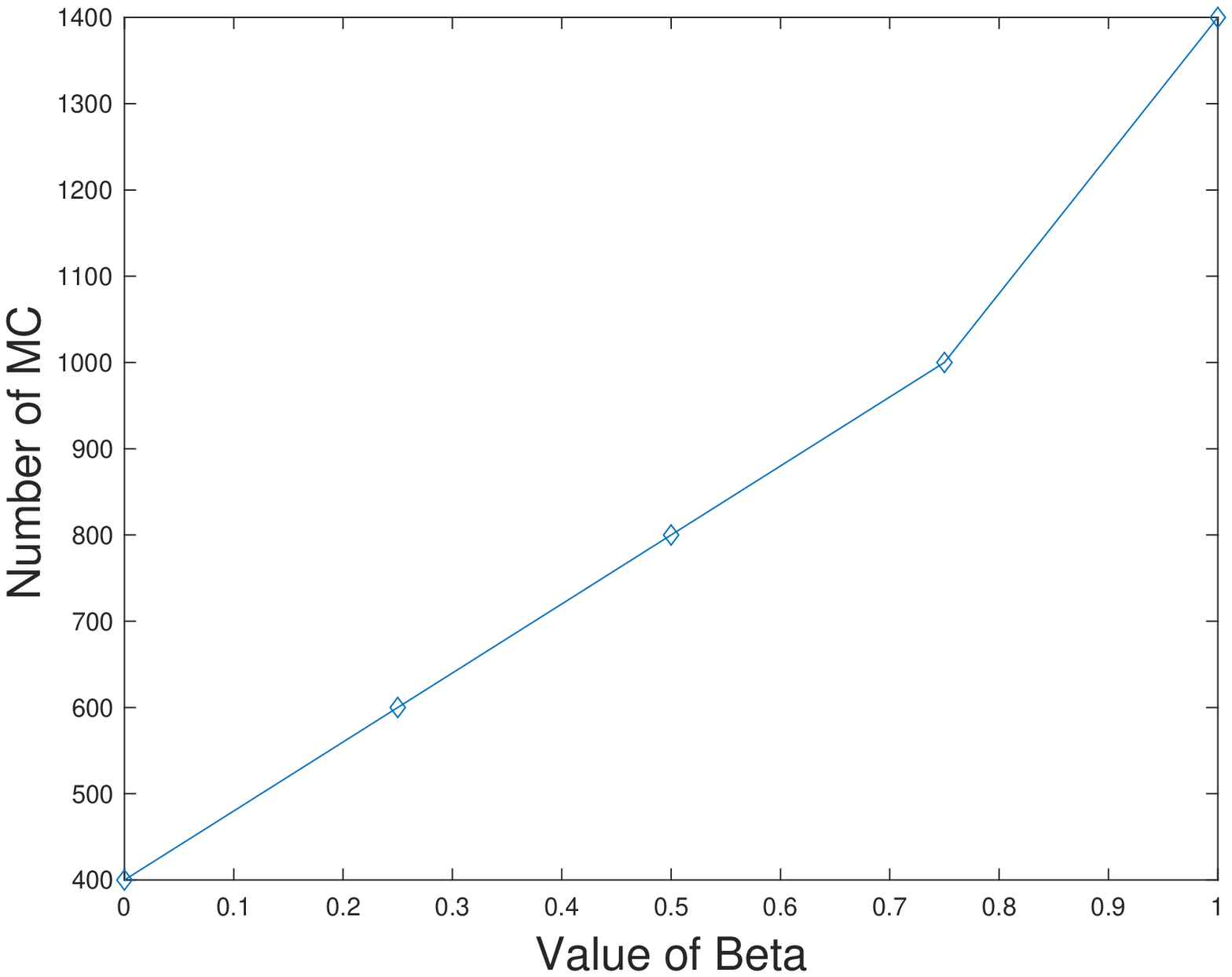}\label{fig:x3}}  \caption{{\bf (Example \ref{exa-bet1})} $(\widehat{\alpha}, \beta)$-scheme with $\sigma(v)= 5v$, and $F \equiv 0$. } 
\label{beta-MC}
\end{figure}

\end{example}

\appendix

\section{Proof of Lemma \ref{lem:L2}}\label{Lem-EE}

\del{
Before going into the proof, we present the main ideas that will be used. The proof in the convergence analysis (Sec. \ref{sec-4}, Thm. \ref{lem:scheme1:con} $(ii)$) does not require the zero boundary conditions on $F$ and $\sigma$, which uses the improved regularity results in Lemmata \ref{lem:L2} and \ref{lem:Holder}. This motivates us to prove these Lemmata under the assumption that $F$ is not zero on the boundary. Although we need to assume $\sigma$  is zero on the boundary due to the regularity results in higher norms, for which we provide a detailed explanation below.
}

We exploit the linearity of the drift operator to decompose the solution $u$ of \eqref{stoch-wave1:1a} with $A = -\Delta$ in the form $u = u_1+u_2$, where $u_1$ solves the following PDE
\begin{align}\label{Tr-prob}
\begin{cases}
\dd \dot{u}_1 - \Delta  u_1 \,\dd t= F(0, 0) \,\dd t \ \qquad
&\mbox{\rm in}\  (0,T) \times \cO\,, \\ 
u_1(0,\cdot) = 0\, , \quad \p_t u_1(0,\cdot) = 0 \ \qquad &\mbox{\rm in}\  \cO\,,\\
u_1(t,\cdot) =0  \ \qquad &\mbox{\rm on}\  \p \cO, \,\forall \,  t \in (0,T)\,,
\end{cases}
\end{align}
where $``\cdot"$ denotes the time derivative, while $u_2$ solves the SPDE
\begin{align}\label{zbv-prob}
\begin{cases}
\dd \dot{u}_2 - \Delta  u_2 \,\dd t= \widehat{F}(u, v) \,\dd t + \sigma(u, v)\,\dd W(t) \ \qquad
&\mbox{\rm in}\  (0,T) \times \cO\,, \\ 
u_2(0,\cdot) = u_0\, , \quad \p_t u_2(0,\cdot) = v_0 \ \qquad &\mbox{\rm in}\  \cO\,,\\
u_2(t,\cdot) =0  \ \qquad &\mbox{\rm on}\  \p \cO, \,\forall \,  t \in (0,T)\,,
\end{cases}
\end{align}
where $\widehat{F}(u, v) := F(u, v) - F(0, 0)$, and $v = \p_t u := \p_t u_1 + \p_t u_2$. The reason for introducing $\widehat{F}$ is to make the drift term has zero trace in \eqref{zbv-prob}$_{1}$. To prove the regularity results, we use the framework of  \cite{Evans} for \eqref{Tr-prob} and we use the Galerkin-based proof for \eqref{zbv-prob}; see, {\em e.g.} \cite[Ch. 6]{Chow_2015}. We need some extra assumptions on $F$ and $\sigma$ ({\em e.g.} {\bf (A3)--(A5)}) and use different arguments than \cite{Chow_2015} as we require improved regularity results.

\smallskip

For the argumentation below to work, for the improved regularity, we assume that $\sigma$ is zero on the boundary, but not $F$. If this is not assumed, then the subproblem \eqref{Tr-prob} will have an extra term $\Big(\int_0^t \sigma(0, 0) \,\dd W(s), \phi \Big)$ in the right-hand side, and in \eqref{zbv-prob}, $\sigma(u, v)$ will be replaced by $\widehat{\sigma}(u, v):= \sigma(u, v) - \sigma(0, 0)$, which is zero on the boundary. In the next step to prove the higher regularity of the modified \eqref{Tr-prob}, we need to consider the following transformation,
$y(t) = u_1(t) - \int_0^{t} \int_0^s \sigma(0, 0) \,\dd W(r) \, \dd s\,.$ Now, $y$ solves a randomized PDE with $y=h$ on the boundary, where $h(t) := \int_0^{t} \int_0^s \sigma(0, 0) \,\dd W(r) \, \dd s\,.$ Since $h$ is of class $\mathrm{C}^{1, \frac 12}$ with respect to the time variable, the standard PDE techniques to show the improved regularity may not be applied. This motivates us to assume that $\sigma$ is zero on the boundary. 

\del{
In the proof of {\bf b)} below, in order to show $(u_{2}, v) \in L^p \big(\Omega; L^{\infty}(0, T; \bH^{2\ell+1} \times \bH^{2 \ell}) \big)$ for $\ell = 1, \frac 32$, we impose {\bf (A5)}, {\em i.e.,} $F \equiv F(v)$ and $\sigma \equiv \sigma(v)$ are linear in $v$, but $F \equiv F(u)$ and $\sigma \equiv \sigma(u)$ are treated as general functions such that the chain rule and the product rule formula of calculus is applied; see {\bf b3)} below. For the case $\ell = 0, \frac 12$, the functions $F \equiv F(u, v)$ and $\sigma \equiv \sigma(u, v)$ are always treated as general functions; see {\bf b1)} and {\bf b2)} below.
}

\smallskip

\begin{proof}[Proof of Lemma \ref{lem:L2}]

We first prove the improved regularity results for $u_1$ and use a bootstrapping argument to prove the improved regularity results for $u_2$.

\smallskip

\noindent
{\bf a) Improved regularity of $u_1$.} By \cite[Sec. 7.2]{Evans}, there exists a unique solution $u_1 \in C([0, T]; \bH^1_0)$ and  $\p_t u_1 \in C([0, T]; \bL^2)$ to \eqref{Tr-prob}. By \cite[Sec. 7.2]{Evans}, for $m = 1, 2, 3$, under the assumption {\bf (A5)}, we get $\big(u_1, \p_t u_1 \big) \in L^{\infty}(0, T; \bH^{m+1}) \times L^{\infty}(0, T; \bH^{m})$, and we have the following estimate
\begin{align}\label{u_1-esti}
\sup_{0 \leq t \leq T} \Big( \|u_1(t) \|_{\bH^{m+1}}^{q} + \|\p_t u_1(t) \|_{\bH^m}^{q} \Big)  \leq C_q\, \| F(0, 0) \|^{q}_{L^2(0, T; \bH^m)}\, \quad (q \geq 2).
\end{align}
We will use this result to prove the improved regularity for $u_2$. 

\smallskip

\noindent
{\bf b) Improved regularity of $u_2$.} By \cite[Thm. 8.4]{Chow_2015}, there exists a unique  $\{\mathcal{F}_t \}_{t \geq 0}-$adapted process $(u_2, \p_t u_2) \in L^2\big(\Omega; C([0, T]; \bH^1_0)\big) \times L^2\big(\Omega; C([0, T]; \bL^2)\big)$, which satisfies \eqref{zbv-prob} $\mathbb{P}-$a.s.. The proof uses a Galerkin approximation, with $\{\rho_i\}_{i=1}^{\infty}$ the orthonormal basis of $\bL^2$, composed of eigenfunctions of $- \Delta$. For any $n \in \mathbb{N}$, we define the finite dimensional space 
$\mathbb{H}_n := \mbox{Span}\{\rho_1, \cdots, \rho_n\}$, and
$\mathcal{P}_n$ be the projection from $\bL^2$ onto $\mathbb{H}_n$. We define $\Delta_n := \mathcal{P}_n \Delta: \mathbb{H}_n \to \mathbb{H}_n$ and use the mappings $\widehat{F}_n(u_{n}, v_n) := \mathcal{P}_n \widehat{F}(u_{n}, v_n) \in \mathbb{H}_n$ and $\sigma_n(u_{n}, v_n) := \mathcal{P}_n \sigma(u_{n}, v_n) \in \mathbb{H}_n$ for $(u_n, v_n) \in [\mathbb{H}_n]^2$, such that $u_n = u_{1n} + u_{2n}$, where $u_{1n} := \mathcal{P}_n u_1$, $v_n := v_{1n} + v_{2n} := \p_t u_{1n} + \p_t u_{2n}$ with $u_{2n}(0)= \mathcal{P}_n u_0$ and $v_{2n}(0)= \mathcal{P}_n v_0$,
\del{
\begin{align*}
F_n : \mathbb{H}_n \times \mathbb{H}_n \ni (u, v) \mapsto \mathcal{P}_n F(u, v) \in \mathbb{H}_n, \\
\sigma_n : \mathbb{H}_n \times \mathbb{H}_n \ni (u, v) \mapsto \mathcal{P}_n \sigma(u, v) \in \mathbb{H}_n\, ,
\end{align*}
}
where $u_{2n}$ and $v_{2n}$ satisfy the following approximated system
 \begin{align}\label{Gal-stoch}
\begin{cases}
\dd u_{2n} &= v_{2n}\, \dd t\,  \\
\dd v_{2n}  &= \left(\Delta_n u_{2n}  + \widehat{F}_n \big(u_{1n} + u_{2n}, v_{1n} + v_{2n} \big) \right) \dd t +  \sigma_n \big(u_{1n} + u_{2n}, v_{1n} + v_{2n} \big) \dd W(t)\,.
\end{cases}
\end{align}
By \cite{Oksendal}, there exists a unique $\{\mathcal{F}_t\}_{t \geq 0}$-adapted process $(u_{2n}, v_{2n})$ on $\big( \Omega, \mathcal{F}, \{\mathcal{F}_t\}_{t \geq 0}, \mathbb{P}\big)$ such that for each $n\in \mathbb{N}$, $(u_{2n}, v_{2n}) \in L^2 \big(\Omega; C([0, T]; [\bH_n]^2) \big)$ for $(\mathcal{P}_n u_0, \mathcal{P}_n v_0) \in [\bH_n]^2$. 

\smallskip

\noindent
{\bf 1) Bounds:}  Let $\ell \in \Big\{ 0, \frac{1}{2}, 1, \frac{3}{2} \Big\}$, which correspond to the parts $(i)-(iv)$ of Lemma \ref{lem:L2}, respectively. Define the map $\Phi_{\ell} : \mathbb{H}_n \times \mathbb{H}_n \to \mathbb{R}$, where 
$$
\Phi_{\ell}(u, v):= \frac 12 \Big[ \|\Delta_n^{\ell +\frac{1}{2}} u\|_{\bL^2}^2 + \|\Delta_n^{\ell} v\|_{\bL^2}^2 \Big]\,.
$$
Thus, $D_u \Phi_{\ell}(u, v), D_v \Phi_{\ell}(u, v) \in \mathcal{L}(\mathbb{H}_n, \mathbb{R})$. For any $\phi \in \mathbb{H}_n$, we have
\del{
we may write
\begin{align*}
\Phi_{\ell}(u_n +h, v_n) &= \frac 12 \big[ \|A_n^{\ell +\frac{1}{2}} (u_n +h)\|_{\bL^2}^2 + \|A_n^{\ell} v_n\|_{\bL^2}^2\big] 
= \frac 12 \big[ \big( A_n^{\ell +\frac{1}{2}} (u_n +h), A_n^{\ell +\frac{1}{2}} (u_n +h)\big) + \|A_n^{\ell} v_n\|_{\bL^2}^2 \big]
\\ &= \frac{1}{2} \big[ \big( A_n^{\ell +\frac{1}{2}} u_n, A_n^{\ell +\frac{1}{2}} u_n \big) + \big( A_n^{\ell +\frac{1}{2}} u_n, A_n^{\ell +\frac{1}{2}} h \big) + \big( A_n^{\ell +\frac{1}{2}} h, A_n^{\ell +\frac{1}{2}} u_n \big) 
\\ &\quad+ \big( A_n^{\ell +\frac{1}{2}} h, A_n^{\ell +\frac{1}{2}} h \big)  + \|A_n^{\ell} v_n\|_{\bL^2}^2 \big]
\\ &= \Phi_{\ell}(u_n, v_n) + \big( A_n^{\ell +\frac{1}{2}} u_n, A_n^{\ell +\frac{1}{2}} h \big) + \|A_n^{\ell +\frac{1}{2}} h\|_{\bL^2}^2\, .
\end{align*}
}
\begin{align*}
D_u \Phi_{\ell}(u, v)(\phi)= \big( \Delta_n^{\ell +\frac{1}{2}} u, \Delta_n^{\ell +\frac{1}{2}} \phi \big) \ \mbox{ and } \ D_v \Phi_{\ell}(u, v)(\phi)= \big( \Delta_n^{\ell} v, \Delta_n^{\ell} \phi \big)\,.
\end{align*}
Applying It\^o's formula to the process $\Phi_{\ell}$ we obtain
\begin{equation}\label{Phi_l}
\begin{split}
\Phi_{\ell} \big(u_{2n}(t), v_{2n}(t) \big) &= \Phi_{\ell}\big(u_{2n}(0), v_{2n}(0) \big) + \int_0^t \Big( \Delta_n^{\ell +\frac{1}{2}} u_{2n}(s), \Delta_n^{\ell +\frac{1}{2}} v_{2n}(s)\Big)\, {\rm d}s 
\\ & \quad + \int_0^t \Big( \Delta_n^{\ell} v_{2n}(s), \Delta_n^{\ell +1} u_{2n}(s) + \Delta_n^{\ell} \widehat{F}_n \big(u_n(s), v_n(s) \big) \Big)\,{\rm d} s 
\\ & \quad + \int_0^t \Big( \Delta_n^{\ell} v_{2n}(s), \Delta_n^{\ell} \sigma_n \big(u_n(s), v_n(s) \big) \,{\rm d}W(s) \Big) 
\\ &\quad+ \frac 12 \int_0^t \big\| \Delta_n^{\ell} \sigma_n \big(u_n(s), v_n(s) \big) \big\|_{\bL^2}^2 \, {\rm d}s \, ,
\end{split}
\end{equation}
where $u_n = u_{1n} + u_{2n}$ and $v_n = v_{1n} + v_{2n}$. We use different arguments for the cases $\ell = 0, \frac 12, 1, \frac 32$, which represnt the parts $(i)-(iv)$ of Lemma \ref{lem:L2}, respectively. 

\smallskip

\noindent
{\bf b1) $F \equiv F(u, v)$ and $\sigma \equiv \sigma(u, v)$ for $\ell = 0$.}
Since $\mathcal{P}_n \sigma(u_n, v_n) = \sum_{i=1}^n \big( \sigma\big(u_n, v_n \big), \rho_i \big) \rho_i$, using {\bf (A3)}, a standard argument gives
\begin{align*}
&\big\| \sigma_n \big(u_n, v_n \big) \big\|_{\bL^2}^2 \leq \big\|\sigma\big(u_n, v_n \big) \big\|_{\bL^2}^2 \leq C_{\tt L} \Big\{ 1+ \| \nabla u_n \|_{\bL^2}^2 + \| v_n \|_{\bL^2}^2 \Big\}
\\ &\leq C_{\tt L} \Big\{ 1+ \| \nabla u_{n} \|_{\bL^2}^2 + \| v_n \|_{\bL^2}^2 \Big\}  \leq C \Big\{ 1+ \| \Delta_n^{1/2} u_{n} \|_{\bL^2}^2 + \| v_n \|_{\bL^2}^2 \Big\} \,.
\end{align*}
A similar estimate will hold for $\| \widehat{F}_n(u_n, v_n) \|_{\bL^2}^2\,.$

\smallskip

\noindent
{\bf b2) $F \equiv F(u, v)$ and $\sigma \equiv \sigma(u, v)$ for $\ell = 1/2$.}
Proceeding similarly as before for $\ell = 0$, and using {\bf (A4)} we infer
\begin{align*}
&\big\|\Delta_n^{1/2} \sigma_n(u_n, v_n) \big\|_{\bL^2}^2 = \sum_{j=1}^n \lambda_j \Big| \Big( \sigma \big(u_n, v_n \big), \rho_j \Big) \Big|^2 \leq \big\|\nabla \sigma \big(u_n, v_n \big) \big\|_{\bL^2}^2 \\
 &\leq C \Big\{ 1 + \big\| \p_u \sigma \big(u_n, v_n \big) (\nabla u_n) + \p_v \sigma \big(u_n, v_n \big)  (\nabla v_n) \big\|_{\bL^2}^2 \Big\} \leq C \Big\{ 1+ \|\Delta_n u_{n} \|_{\bL^2}^2 + \|\Delta_n^{1/2} v_n \|_{\bL^2}^2 \Big\}\,.
\end{align*}
A similar estimate will hold for $\big\|\Delta_n^{1/2} \widehat{F}_n \big(u_n, v_n \big) \big\|_{\bL^2}^2$. The other terms in the right-hand side of \eqref{Phi_l} can be dealt similarly by the use of Cauchy-Schwarz inequality. 

\noindent
Using the above estimates in {\bf b1)} and {\bf b2)} (for $\ell = 0$ and $\frac 12$, respectively) in \eqref{Phi_l} we obtain
{\small{
\begin{equation}\label{ineq-2}
\begin{split}
\Phi_{\ell} \big(u_{2n}(t), v_{2n}(t) \big) &\leq \Phi_{\ell} \big(u_{2n}(0), v_{2n}(0) \big) + C \int_0^t \Big[ 1+ \|\Delta_n^{\ell} v_n(s)\|_{\bL^2}^2 + \|\Delta_n^{\ell+ \frac{1}{2}} u_{n}(s)\|_{\bL^2}^2  \Big] {\rm d}s 
\\ &\quad+ \int_0^t \Big( \Delta_n^{\ell} v_{2n}(s), \Delta_n^{\ell} \sigma_n \big(u_n(s), v_n(s) \big) {\rm d}W(s) \Big)\, .
\end{split}
\end{equation}
}}
Using the definition of $\Phi_{\ell}$, raising the power $p$ in both sides of the inequality for some $p > 2$, taking the supremum over time and then taking expectation, and using the regularity results in {\bf a)}, we get
\begin{equation}\label{ineq-3}
\begin{split}
&\bE \Big[ \sup_{0 \leq s \leq t} \Phi_{\ell}^p \big(u_{2n}(s), v_{2n}(s) \big) \Big] 
\\ &\leq C+ 3^{p-1} \bE \Big[\Phi_{\ell}^p \big(u_{2n}(0), v_{2n}(0) \big) \Big] + 3^{p-1} \int_0^t \bE \Big[\sup_{0 \leq r \leq s} \Phi_{\ell}^p \big( u_{2n}(r), v_{2n}(r) \big) \Big] \,{\rm d}r
\\ &\quad+ 3^{p-1} \bE \bigg[ \sup_{0 \leq s \leq t} \bigg| \int_0^s \Big( \Delta_n^{\ell} v_n(r), \Delta_n^{\ell} \sigma_n \big(u_n(r), v_n(r) \big)\, {\rm d}W(r) \Big) \bigg|^p \bigg]\, .
\end{split}
\end{equation}

Using the Burkholder-Davis-Gundy inequality and previous estimates for $\ell = 0, \frac 12$, and using the regularity results in {\bf a)}, we obtain
\begin{equation}\label{ineq-4}
\begin{split}
&\bE \bigg[ \sup_{0 \leq s \leq t} \bigg| \int_0^s \Big( \Delta_n^{\ell} v_{2n}(r), \Delta_n^{\ell} \sigma_n \big(u_n(r), v_n(r) \big) {\rm d}W(r) \Big) \bigg|^p \bigg] 
\\ &\leq C \bE \bigg[ \bigg( \int_0^t \|\Delta_n^{\ell} v_{2n}(s)\|_{\bL^2}^2 \big\|\Delta_n^{\ell} \sigma_n\big(u_n(s), v_n(s) \big) \big\|_{\bL^2}^2\,{\rm d} s \bigg)^{p/2} \bigg] 
\\ &\leq C + C \bE \bigg[ \sup_{0 \leq s \leq t} \Phi_{\ell}^p \big(u_{2n}(s), v_{2n}(s) \big) \bigg] + C \int_0^t \bE \bigg[ \sup_{0 \leq s \leq t} \Phi_{\ell}^p \big(u_{2n}(s), v_{2n}(s) \big) \bigg] \,{\rm d}s.
\end{split}
\end{equation}
Using \eqref{ineq-4} in \eqref{ineq-3} and using the Gronwall lemma we get for $\ell = 0, \frac 12$ and $p \geq 2$,
\begin{align}\label{Phi-esti}
\bE \bigg[ \sup_{0 \leq s \leq t} \Phi_{\ell}^p \big(u_{2n}(s), v_{2n}(s) \big) \bigg]  \leq C \,\bE\, \Big[ \Phi_{\ell}^p \big(u_{2n}(0), v_{2n}(0) \big) \Big] \, e^{CT} \leq C\, \bE\,\Big[ \Phi_{\ell}^p\big(u_0, v_0 \big) \Big]\, e^{CT} \, .
\end{align}

\smallskip

\noindent
{\bf b3) Dealing of cases $\ell = 1, \frac 32$.}
We assume {\bf (A3)} for these two cases.
If we treat $\sigma_2(v)$ and $F_2(v)$ as general functions, then the chain rule and the product rule formula of calculus will lead us to higher order derivative terms with higher moments in $v$ in the right-hand side as compared to the left-hand side; see \eqref{l=1esti} and \eqref{l=1esti-1} below for the similar estimates in $v$. Then, the Gronwall lemma may not be applied. Thus, $F \equiv F(u)$ and $\sigma \equiv \sigma(u)$ are treated as general functions, but $F \equiv F(v)$ and $\sigma \equiv \sigma(v)$ are assumed to be only affine in $v$.

\smallskip

\noindent
{\bf Case-1:} Let us consider the case $\sigma \equiv \sigma_1(u)$ and $\widehat{F} \equiv F_1(u),$ which can be dealt as general functions. Take $\ell =1$ in \eqref{Phi_l}. Then, using product formula, and chain rule for general functions and by {\bf (A4)} for $m=1, 2$, we infer
{\small{
\begin{align}\label{l=1esti}
&\|\Delta_n \sigma_{n}(u_n) \|_{\bL^2}^2 = \sum_{j=1}^n \lambda_j^2 \big| \big( \sigma(u_n), \rho_j \big) \big|^2 \leq \|\Delta \sigma(u_n)\|_{\bL^2}^2 \leq \tilde{C}_g^2 \,\| (\nabla u_{n})^2 \|_{\bL^2}^2 + C_g \| \Delta u_{n} \|_{\bL^2}^2\,.
\end{align}
}}
Now, using Ladyzhenskaya inequality and Poincar\'e inequality, we estimate the term
\begin{equation}\label{l=1esti-1}
\begin{split}
\| (\nabla u_{n})^2 \|_{\bL^2}^2 \leq C \| \nabla u_{n} \|_{\bL^2}^2 \| \Delta u_{n} \|_{\bL^2}^2 &\leq C \| \nabla u_{n} \|_{\bL^2}^4 \| \Delta u_{n} \|_{\bL^2}^2 + C \| \Delta u_{n} \|_{\bL^2}^2 
\\ &\leq C \| \nabla u_{n} \|_{\bL^2}^8 + C \| \Delta u_{n} \|_{\bL^2}^4 + C \| \nabla \Delta u_{n} \|_{\bL^2}^2\,.
\end{split}
\end{equation}
Similar estimates will hold for $\|\Delta_n \widehat{F}_n(u_n) \|_{\bL^2}^2$. Now, we take $\ell =\frac 32$ in \eqref{Phi_l}. Using the chain rule, {\bf (A4)} for $m=1, 2, 3$, we infer
{\small{
\begin{align}\label{l=3/2esti}
&\|\Delta_n^{3/2} \sigma_n(u_n) \|_{\bL^2}^2 \leq \|\Delta^{3/2} \sigma(u_n)\|_{\bL^2}^2 \leq \bar{C}_g^2 \,\| (\nabla u_{n})^3 \|_{\bL^2}^2 + \tilde{C}_g^2 \,\| \nabla u_{n} \Delta u_n \|_{\bL^2}^2 + C_g \| \nabla \Delta u_{n} \|_{\bL^2}^2\,.
\end{align}
}}
Using the Sobolev embeddings we further esimate
\begin{align}
\| (\nabla u_{n})^3 \|_{\bL^2}^2 \leq C\,\| \nabla u_{n} \|_{\bL^6}^6 \leq C\,\| \Delta u_{n} \|_{\bL^2}^6 \,,
\end{align}
and
\begin{equation}\label{l=3/2esti-1}
\begin{split}
&\| \nabla u_{n} \Delta u_n \|_{\bL^2}^2 \leq C\,\| \nabla u_{n} \|_{\bL^4}^2 \| \Delta u_n \|_{\bL^4}^2 \leq C\,\| \nabla u_{n} \|_{\bL^2}^{1/2} \| \Delta u_{n} \|_{\bL^2}^{3/2} \| \Delta u_{n} \|_{\bL^2}^{1/2} \| \nabla \Delta u_{n} \|_{\bL^2}^{3/2} 
\\ &\leq C \| \nabla u_{n} \|_{\bL^2}^{2} + C \| \Delta u_{n} \|_{\bL^2}^{8} + C \| \nabla \Delta u_{n} \|_{\bL^2}^{3}\,.
\end{split}
\end{equation}
Similar estimates will hold for $\|\Delta_n^{3/2} \widehat{F}_n(u_n) \|_{\bL^2}^2$. 

\smallskip

\noindent
{\bf Case-2:} Let $\sigma \equiv \sigma_2(v)$ and $\widehat{F} \equiv F_2(v)$, such that {\bf (A5)} holds. For $\ell =1$ we have
\begin{align}\label{l=1esti-v}
&\|\Delta_n \sigma_{n}(v_n) \|_{\bL^2}^2 = \sum_{j=1}^n \lambda_j^2 \big| \big( \sigma(v_n), \rho_j \big) \big|^2 \leq \|\Delta \sigma(v_n)\|_{\bL^2}^2 \leq C_g \| \Delta v_{n} \|_{\bL^2}^2\,,
\end{align}
and for $\ell =\frac 32$ we have
\begin{align}\label{l=3/2esti-v}
&\|\Delta_n^{3/2} \sigma_n(v_n) \|_{\bL^2}^2 \leq \|\Delta^{3/2} \sigma(v_n)\|_{\bL^2}^2 \leq C_g \| \nabla \Delta v_{n} \|_{\bL^2}^2\,.
\end{align}
Similar estimates will hold for $\|\Delta_n^{3/2} \widehat{F}_n(v_n) \|_{\bL^2}^2$.

Using the estimates \eqref{l=1esti}, \eqref{l=3/2esti-1}, \eqref{l=1esti-v}, and \eqref{l=3/2esti-v} in \eqref{Phi_l} for $\ell =1, \frac 32$, and using the regularity results proved so far for $u_{1n}$ and $u_{2n}$ and their time derivatives, we get \eqref{ineq-3} for $\ell =1, \frac 32$. Finally, the use of Burkholder-Davis-Gundy inequality yields the assertion for $\ell=1, \frac 32$. 

\del{
Here, we claim that under the assumption $(u_0, v_0) \in \bH^{2\ell+1} \times \bH^{2 \ell}$ for $\ell = 0, \frac 12, 1, \frac 32$, we get
\begin{align}\label{u_2-esti}
\bE \bigg[ \sup_{0 \leq t \leq T} \Big( \|u_2(t) \|_{\bH^{2 \ell+1}}^{p} + \|v_2(t) \|_{\bH^{2 \ell}}^{p} \Big) \bigg] \leq C\big(p, T, \|u_0\|_{\bH^{2\ell+1}}, \|v_0\|_{\bH^{2\ell}} \big)\,.
\end{align}
}

\smallskip

\noindent
{\bf 2) Convergence:}
By step {\bf 1)},  for $p \geq 2$
$$(u_{2n}, v_{2n})_{n} \subset L^p \big(\Omega; L^{\infty}(0, T; \bH^{2\ell+1} \times \bH^{2 \ell}) \big) \cap L^p\big(\Omega; L^{2}(0, T; \bH^{2\ell+1} \times \bH^{2 \ell}) \big)$$ 
is bounded for $\ell = 0, \frac 12, 1, \frac 32$. Here, we need to argue the convergence case by case. First, consider $\ell=0$. Then, there exist subsequences $(u_{2n'})_{n'}$ and $(v_{2n'})_{n'}$, which converge weakly to $u'_2$ and $v'_2$, respectively. Then, using the standard arguments (see \cite{Chow_2015}) it can be shown that $(u'_2, v'_2)$ is a weak solution of \eqref{zbv-prob}. By the uniqueness of the weak solution, we have $(u'_2, v'_2) = (u, v)$.
 By Fatou's lemma, passing to the limit in \eqref{Phi-esti} yields
\begin{align}\label{u2v2_bound}
\bE \bigg[ \sup_{0 \leq s \leq t} \Phi_{\ell}^p \big(u_2(s), v_2(s) \big) \bigg] \leq C \bE\,\Big[ \Phi_{\ell}^p \big(u_0, v_0 \big) \Big]\, e^{CT} \, ,
\end{align}
for $\ell=0$. Now, consider $\ell=\frac 12$. Then, there exist subsequences $(u_{2n''})_{n''}$ and $(v_{2n''})_{n''}$ which converge weakly to some $\tilde{u}_2$ and $\tilde{v}_2$, respectively. By using the standard arguments and the uniqueness of the solution of the system \eqref{zbv-prob}, we claim that $(\tilde{u}_2, \tilde{v}_2) = (\nabla u_2, \nabla v_2)$. Thus, by passing to the limit \eqref{u2v2_bound} holds for $\ell=1/2$.  Similar arguments will yield the result for $\ell = 1, \frac 32$. Combining \eqref{u2v2_bound} with \eqref{u_1-esti} we get the  assertions in Lemma \ref{lem:L2}.
\end{proof}

\section{Proof of H\"older continuity in time}\label{Hoe-con}

The proof of Lemma \ref{lem:Holder} uses the regularity results for the variational solution of \eqref{strong1}--\eqref{strong2} in Lemma \ref{lem:L2}. We obtain a H\"older regularity in time for $u$ which is double the one for $v$: the reason for it is the occurrence of the It\^o integral in \eqref{strong2}, but not in \eqref{strong1}.

\begin{proof} [Proof of Lemma \ref{lem:Holder}] 
{\bf Proof of $(i)$.} Let $r, s \in [0, T]$, and fix $p \in \mathbb{N}$. By Lemma \ref{lem:L2} $(i)$, we have $v \in L^{2} \big(\Omega; L^{\infty}(0, T; \bL^{2}) \big)$. Therefore, $\int_s^r v(\xi)\dd \xi$ is well-defined for {\em a.e.} $x \in \cO$ and $\bP$-a.s.. Thus, we can write the weak formulation \eqref{strong1} in strong form $\bP$-a.s. as 
$$u(r)-u(s) =\int_s^r v(\xi)\dd \xi, \quad \mbox{ for a.e. } x \in \cO, \mbox{ for } r, s \in [0, T]\,.$$
Then, the H\"older inequality yields
\[ \|u(r)-u(s)\|_{\bL^2}^{2p} \leq \Big( \int_s^r \|v(\xi)\|_{\bL^2}\,\dd \xi \Big)^{2p} \leq |r-s|^{2p-1} \int_s^r \|v(\xi)\|^{2p}_{\bL^2}\,\dd \xi \,.\]

\del{
Since the proofs of assertions $\bE \bigl[\|\nabla [u(t)- u(s)]\|_{{\mathbb L}^2}^{2p} \bigr]$ and $\bE \bigl[\|\Delta [u(t)- u(s)]\|_{{\mathbb L}^2}^2 \bigr]$ follow the same lines as that for the $\bE \bigl[\|u(t)-u(s)\|_{{\mathbb L}^2}^{2p} \bigr]$ by using the restated version of \eqref{stoch-wave1:1a}$_1$, $\nabla \bigl(u(t)-u(s)\bigr) =\int_s^t \nabla v(\xi)\dd \xi $, $\Delta \bigl(u(t)-u(s)\bigr) =\int_s^t \Delta v(\xi)\dd \xi $, respectively. 
\begin{align*}
\bE \Bigl[ \sup_{s \leq r \leq t} \|u(r)-u(s)\|_{{\mathbb L}^2}^{2p} \bigr]
&\le \bE \Bigl[\Bigl\|\Bigl(\int_s^t |v(\xi)|^2\dd \xi\Bigr)^{\frac{1}{2}}
\Bigl(\int_s^t 1\dd \xi\Bigr)^{\frac{1}{2}} \Bigl\|_{{\mathbb L}^2}^{2p}\Bigl]
\le |t-s|^p\bE \Bigl[\Bigl(\int_s^t\bigl\|v(\xi)\bigl\|_{{\mathbb L}^2}^2\dd \xi \Bigl)^p\Bigl]
\\ &\le |t-s|^p\bE \Bigl[\Bigl(\int_s^t\bigl\|v(\xi)\bigl\|_{{\mathbb L}^2}^{2p}\dd \xi \Bigl)\Bigl(\int_s^t 1\dd \xi \Bigl)^{p-1}\Bigl]
\le \bE \Bigl[ \supT \|v(t)\|_{\bL^2}^{2p}\Bigr]|t-s|^{2p}\, .
\end{align*}
}
We fix $s, t \in [0, T]$, and take supremum w.r.t. $r$, and then take expectation to get
\begin{align*}
\bE \Bigl[ \sup_{s \leq r \leq t} \|u(r)-u(s)\|_{{\mathbb L}^2}^{2p} \bigr]
&\le |t-s|^{2p-1}\, \bE \Big[ \int_s^t \|v(\xi)\|_{\bL^2}^{2p}\,\dd \xi \Big]
\le |t-s|^{2p}\,\bE \Bigl[ \supT \|v(t)\|_{\bL^2}^{2p}\Bigr]\, .
\end{align*}
Hence, $(i)$ holds by applying \eqref{lem:L2:1} in Lemma~\ref{lem:L2}.

\smallskip

\noindent
{\bf Proof of $(ii)$.} Let $r, s \in [0, T]$, and fix $p \in \mathbb{N}$. The first part follows as {\bf $(i)$}. By Lemma \ref{lem:L2} $(ii)$, we have $u \in L^{2} \big(\Omega; L^{\infty}(0, T; \bH^{2}) \big)$. Therefore, $\int_s^r \Delta u(\xi)\,\dd \xi$ is well-defined for {\em a.e.} $x \in \cO$ and $\bP$-a.s.. By Lemma \ref{lem:L2} $(i)$, we have $(u, v) \in L^{2} \big(\Omega; L^{\infty}(0, T; \bH^{2} \times \bH^1) \big)$. Therefore, by {\bf (A3)}, $\int_s^r F \big(u(\xi), v(\xi) \big) \dd \xi$ is well-defined for {\em a.e.} $x \in \cO$ for $s, r \in [0, T]$ and $\bP$-a.s.. Similarly, by {\bf (A3)} and It\^o isometry, $\int_s^r \sigma \big(u(\xi), v(\xi) \big) \dd W(\xi)$ is well-defined for {\em a.e.} $x \in \cO$ for $s, r \in [0, T]$ and $\bP$-a.s.. Now, from the weak formulation \eqref{strong2} and using the above conclusion, we may rewrite the equation in the strong form as (see \cite[Section 6.3, Remark (ii)]{Evans})
\begin{align}\label{vt-vs}
v(r) - v(s) = \int_s^r \Delta u(\xi) \dd \xi + \int_s^r F \big(u(\xi), v(\xi) \big) \dd \xi + \int_s^r \sigma \big(u(\xi), v(\xi) \big) \dd W(\xi).
\end{align}
\del{
Taking the square, integrating over the domain, and then the expectation on both sides, using It\^o isometry and \eqref{E-bound} and Lemma \ref{lem:L2} $(i)$ and $(ii)$ we infer
\begin{align*}
\bE \Big[ \|v(t) - v(s) \|_{\bL^2}^2 \Big] &\leq C \bE \Big[ \int_{\cO} \Big( \int_s^t \Delta u(\xi) \,\dd \xi \Big)^2 \,\dd x \Big] + C \bE \Big[ \int_{\cO} \Big( \int_s^t F \big(u(\xi), v(\xi) \big) \,\dd \xi \Big)^2 \,\dd x \Big] 
\\ &\quad+ C \bE \Big[ \int_{\cO} \Big( \int_s^t \sigma \big(u(\xi), v(\xi) \big) \,\dd W(\xi) \Big)^2 \,\dd x \Big]
\\ &\leq C \bE \Big[ \int_s^t \|\Delta u(\xi)\|_{\bL^2}^2 \, \dd \xi \Big] (t-s) + C \bE \Big[ \int_s^t \|F \big(u(\xi), v(\xi) \big)\|_{\bL^2}^2 \, \dd \xi \Big] (t-s) 
\\ &\quad+ C \bE \Big[ \int_s^t \|\sigma \big(u(\xi), v(\xi) \big)\|_{\bL^2}^2 \, \dd \xi \Big] 
\\ &\leq C (t-s) + C \bE \Big[ \int_s^t \big( 1 + \| \nabla u(\xi)\|_{\bL^2}^2 + \| v(\xi)\|_{\bL^2}^2 \big) \Big] \leq C (t-s).
\end{align*}
}
By H\"older inequality, we estimate
\begin{equation}\label{v-Hol-L2}
\begin{split}
\|v(r) - v(s) \|_{\bL^2}^2 &\leq C (r-s) \int_s^r \|\Delta u(\xi)\|_{\bL^2}^2 \, \dd \xi  + C (r-s) \int_s^r \|F \big(u(\xi), v(\xi) \big)\|_{\bL^2}^2 \, \dd \xi 
\\ &\quad+ C \int_{\cO} \Big( \int_s^r \sigma \big(u(\xi), v(\xi) \big) \,\dd W(\xi) \Big)^2 \,\dd x \,.
\end{split}
\end{equation}
We fix $s, t \in [0, T]$, and take supremum w.r.t. $r$, then take expectation. Using {\bf (A3)}, It\^o isometry, Lemma \ref{lem:L2} $(i)$ and $(ii)$, we infer
{\small{
\begin{align*}
\bE \Big[ \sup_{s \leq r \leq t} \|v(r) - v(s) \|_{\bL^2}^2 \Big] &\leq C \bE \Big[ \sup_{0 \leq t \leq T} \|\Delta u(t)\|_{\bL^2}^2 \Big] (t-s)^2 + C (t-s)^2 + C \bE \Big[ \sup_{0 \leq t \leq T} \|\nabla u(t)\|_{\bL^2}^2 \Big] (t-s)^2  \nonumber
\\ &\quad+ C \bE \Big[ \sup_{0 \leq t \leq T} \|v(t)\|_{\bL^2}^2 \Big] (t-s)^2 + C \bE \Big[ \int_s^t \big( 1 + \| \nabla u(\xi)\|_{\bL^2}^2 + \| v(\xi)\|_{\bL^2}^2 \big) \Big] 
\\ &\leq C (t-s) \,. \nonumber
\end{align*}
}}

\smallskip

\noindent
{\bf Proof of $(iii)$.} Let $r, s \in [0, T]$, and fix $p \in \mathbb{N}$. The first part follows as {\bf $(i)$}. In order to verify the bound for $\bE \big[\sup_{s \leq r \leq t} \|\nabla[v(r) - v(s)] \|_{\bL^2}^2 \big]$, consider the equation \eqref{vt-vs}. By Lemma \ref{lem:L2}, we have $(u, v) \in L^{2p} \big(\Omega; L^{\infty}(0, T; \bH^{4} \times \bH^{3}) \big)$. Thus, by {\bf (A3)}, {\bf (A4)}, we can take the gradients in \eqref{vt-vs}, since it is a closed operator on $\bH^1$. Then, the terms are well-defined. Proceeding similarly as part $(ii)$ we get
{\small{
\begin{align*}
\bE \Big[ \sup_{s \leq r \leq t} \|\nabla[v(r) - v(s)] \|_{\bL^2}^2 \Big] &\leq C \bE \Big[ \int_s^t \|\nabla \Delta u(\xi)\|_{\bL^2}^2 \, \dd \xi \Big] (t-s) + C \bE \Big[ \int_s^t \big\| \nabla F\big(u(\xi), v(\xi) \big) \big\|_{\bL^2}^2 \, \dd \xi \Big] (t-s) 
\\ &\quad+ C \bE \Big[ \int_s^t \big\|\nabla \sigma \big(u(\xi), v(\xi) \big) \big\|_{\bL^2}^2 \, \dd \xi \Big] \,.
\end{align*}
}}
By {\bf (A4)}, $\big\|\nabla \sigma\big(u(\xi), v(\xi) \big) \big\|_{\bL^2}^2 \leq C_g^2 \big( \|\nabla u(\xi)\|_{\bL^2}^2 + \|\nabla v(\xi)\|_{\bL^2}^2 \big)$. Then, using Lemma \ref{lem:L2} $(i)$, $(ii)$ and $(iii)$, we further estimate
\begin{align*}
\leq C (t-s)^2 + C \bE \Big[ \int_s^t \big( \| \nabla u(\xi)\|_{\bL^2}^2 + \| \nabla v(\xi)\|_{\bL^2}^2 \big) \Big] \leq C (t-s).
\end{align*}

\noindent
{\bf Proof of $(iv)$.} Let $r, s \in [0, T]$, and fix $p \in \mathbb{N}$. The first part follows as {\bf $(i)$}. To verify the bound for $\bE \big[ \sup_{s \leq r \leq t} \|\Delta[v(r) - v(s)] \|_{\bL^2}^2 \big]$, consider \eqref{vt-vs}. Argue similarly as part $(ii)$ to apply the Laplacian to \eqref{vt-vs} due to {\bf (A3)}, {\bf (A4)}, and proceed similarly to obtain
{\small{
\begin{align*}
\bE \Big[ \sup_{s \leq r \leq t} \|\Delta[v(r) - v(s)] \|_{\bL^2}^2 \Big] &\leq C \bE \Big[ \int_s^t \| \Delta^2 u(\xi)\|_{\bL^2}^2 \, \dd \xi \Big] (t-s) + C \bE \Big[ \int_s^t \big\| \Delta F\big(u(\xi), v(\xi) \big) \big\|_{\bL^2}^2 \, \dd \xi \Big] (t-s) 
\\ &\quad+ C \bE \Big[ \int_s^t \big\|\Delta \sigma \big(u(\xi), v(\xi) \big) \big\|_{\bL^2}^2 \, \dd \xi \Big] \,.
\end{align*}
}}
To bound the last two terms requires {\bf (A3)} to {\em e.g.} write $\sigma \big(u(\xi), v(\xi) \big) = \sigma_1\big(u(\xi)\big) + \sigma_2\big(v(\xi)\big)$, where $\sigma_2$ is affine in $v$. Then, $\Delta \sigma \big(u(\xi), v(\xi) \big) = \Delta \sigma_1\big(u(\xi)\big)$ and we can follow the steps of \eqref{l=1esti}-\eqref{l=3/2esti-1} to bound it. Similar techniques may be used to deal with $\| \Delta F \big(u(\xi), v(\xi) \big)\|_{\bL^2}^2$. Lemma \ref{lem:L2} then settles the assertion. 
Thus, the proof is complete.
\end{proof}

\end{document}